\pdfoutput=1
\documentclass[11pt,namelimits,sumlimits,a4paper]{amsart}
\usepackage{cite}
\usepackage{comment}

\usepackage{amssymb,amsmath}
\usepackage[mathscr]{eucal}
\usepackage{slashed}
\usepackage{nameref} 
\usepackage[draft,totoc]{listlbls}

\usepackage[latin1]{inputenc}
\usepackage[T1]{fontenc}
\usepackage[UKenglish]{babel}
\usepackage{amsfonts}
\usepackage{fancyhdr}
\usepackage{graphicx}
\usepackage{amsmath}
\usepackage{amsthm}
\usepackage{amsmath,amscd}
\usepackage{latexsym}
\usepackage{cite}
\usepackage{amssymb,amsmath}
\usepackage{comment}
\usepackage[mathscr]{eucal}
\usepackage{slashed}
\usepackage{times,psfrag}
\usepackage{mathtools}
\usepackage[pdftitle={Complete noncompact G2-manifolds with ALC asymptotics},hidelinks]{hyperref}
\usepackage{multirow}
\usepackage{longtable}
\usepackage{bm}
\usepackage{fancyhdr}
\usepackage{float}
\usepackage{esint}
\usepackage{layout}
\usepackage{enumitem}
\usepackage{lmodern}
\usepackage[rgb]{xcolor}

\numberwithin{equation}{section}
\newtheorem{theorem}[equation]{Theorem}
\newtheorem{lemma}[equation]{Lemma}
\newtheorem{prop}[equation]{Proposition}
\newtheorem{corollary}[equation]{Corollary}

\theoremstyle{definition}
\newtheorem{definition}[equation]{Definition}
\newtheorem{example}[equation]{Example}

\newtheorem{mtheorem}{Theorem}

\theoremstyle{remark}
\newtheorem{remark}[equation]{Remark}
\newtheorem*{remark*}{Remark}

\DeclareFontFamily{U}{mathx}{}
\DeclareFontShape{U}{mathx}{m}{n}{<-> mathx10}{}
\DeclareSymbolFont{mathx}{U}{mathx}{m}{n}
\DeclareMathAccent{\widecheck}{0}{mathx}{"71}

\newcommand{\SigmP}{{\Sigma^+}}
\newcommand{\tsp}{N^+}
\newcommand{\wt}{\widetilde}

\newcommand{\ie}{\emph{i.e.} }
\newcommand{\eg}{\emph{e.g.} }
\newcommand{\cf}{\emph{cf.} }

\newcommand{\beq}{\begin{equation}}
\newcommand{\eeq}{\end{equation}}
\newcommand{\bea}{\begin{eqnarray}}
\newcommand{\eea}{\end{eqnarray}}

\newcommand{\C}{\mathbb{C}}

\newcommand{\R}{\mathbb{R}}

\newcommand{\Z}{\mathbb{Z}}
\newcommand{\N}{\mathbb{N}}

\newcommand{\CP}{\mathbb{CP}}
\newcommand{\PP}{\mathbb{P}}
\newcommand{\Sph}{\mathbb{S}}
\newcommand{\ra}{\rightarrow}

\newcommand{\dvol}{\operatorname{dv}}
\newcommand{\Real}{\operatorname{Re}}
\newcommand{\Imag}{\operatorname{Im}}

\newcommand{\Ric}{\operatorname{Ric}}

\newcommand{\Lie}[1]{\mathfrak{#1}}
\newcommand{\F}{\mathbb{F}}

\newcommand{\BC}{\textup{BC}}
\newcommand{\tu}[1]{\textup{#1}}

\newcommand{\gtwo}{\texorpdfstring{\ensuremath{\textup{G}_2}}{G2}}

\newcommand{\gtstr}{\gtwo--structure}

\newcommand{\gtmfd}{\gtwo--manifold}
\newcommand{\gtmetric}{\gtwo--metric}
\newcommand{\gthol}{\gtwo--holonomy\ }

\newcommand{\unitary}[1]{\textup{U$(#1)$}}

\newcommand{\sunitary}[1]{\textup{SU$(#1)$}}

\newcommand{\suthreestr}{\sunitary{3}--structure}
\newcommand{\sutwostr}{\sunitary{2}--structure}
\newcommand{\sorth}[1]{\textup{SO$(#1)$}}

\def\co{\colon\thinspace}

\begin{document}

\title{Complete noncompact \gtwo--manifolds with ALC asymptotics}
\author{L. Foscolo, M. Haskins and J. Nordstr\"om}

\begin{abstract}
We prove existence, uniqueness and structure results for complete noncompact \linebreak 7-dimensional $\gtwo$--holonomy metrics with ALC (asymptotically locally conical) asymptotics. We regard such spaces as $\gtwo$--analogues of ALF gravitational instantons in 4-dimensional hyperk\"ahler geometry. Our main results include the existence of a $\gtwo$--analogue of the Atiyah--Hitchin metric in 4-dimensional hyperk\"ahler geometry, the existence of a good moduli theory for ALC $\gtwo$--holonomy metrics and rigidity results for ALC $\gtwo$--metrics in terms of the symmetries of their asymptotic model. 

The analytic toolkit needed to prove all these results is a robust Fredholm theory for the natural geometric linear elliptic operators 
on ALC spaces. We provide a self-contained derivation of this Fredholm theory for arbitrary Riemannian manifolds with ALC asymptotics. 
Since our ALC Fredholm theory does not rely on imposing any holonomy reduction or curvature conditions it may also be of utility beyond the setting of ALC special holonomy metrics. As one such application of our general Fredholm theory we prove some Hodge-theoretic results on general ALC spaces.
\end{abstract}

\maketitle

\section*{Introduction}

The acronym ALC (asymptotically locally conical) was introduced in the physics literature by Cveti\v{c}--Gibbons--L\"u--Pope in \cite{CGLP:New:Spin(7)} (and then quickly picked up in a plethora of related works \cite{Brandhuber,BGGG,CGLP:Coho1,CGLP:B8,CGLP:C7:tilde,CGLP:M:conifolds,CGLP:C7,Gukov:Sparks,Gukov:Sparks:Tong,ALC:A7,Kanno:Yasui:I,Kanno:Yasui:II}) to denote an asymptotic geometry that generalises to higher dimensions  that of ALF (asymptotically locally flat) metrics in 4-dimensional geometry; well-known ALF metrics include the Taub-NUT and the Atiyah--Hitchin metrics. A Riemannian metric is said to be ALC if it is asymptotic, with polynomial decay rate, to a Riemannian submersion with base metric the Riemannian cone $\tu{C}(\Sigma)$ over a closed Riemannian manifold $(\Sigma,g_\Sigma)$ and fibre a circle of fixed finite length (see Definition \ref{def:ALC:n}). With this choice of terminology, ALF metrics are then the ALC metrics with tangent cone at infinity $\tu{C}(\Sigma)$ a flat cone.

The focus of this paper is the study of complete Riemannian 7-manifolds with ALC asymptotics and holonomy the exceptional Lie group $\gtwo$. These manifolds are necessarily Ricci-flat and therefore have only one end. In previous work \cite{FHN:ALC:G2:from:AC:CY3,FHN:Coho1:ALC} we constructed infinitely many topological types of \mbox{7-manifolds} that admit ALC $\gtwo$--holonomy metrics. The main results of this paper focus on existence, uniqueness and structure results for general ALC $\gtwo$--holonomy metrics.  These results are motivated on one side by analogies with 4-dimensional ALF hyperk\"ahler metrics and on the other side by extensions to the noncompact setting of foundational results on compact $\gtwo$--manifolds. 

Regarding the analogy between ALC $\gtwo$--holonomy metrics in 7 dimensions and ALF hyper\-k\"ahler metrics in 4 dimensions, we derive direct~\gtwo-analogues of results established by Minerbe \cite{Minerbe:Ak,Minerbe:Mass}, Biquard--Minerbe \cite{Biquard:Minerbe} and Chen--Chen \cite{Chen:Chen:II} (as parts of the recently achieved complete classification of 4-dimensional ALF hyperk\"ahler metrics).

Regarding the second motivation, the third-named author Nordstr\"om \cite{Nordstrom:ACyl:G2} and later Karigiannis--Lotay \cite{Karigiannis:Lotay} extended the deformation theory of $\gtwo$--holonomy metrics developed by Joyce \cite{Joyce:Book} in the compact setting to the asymptotically cylindrical (ACyl) and asymptotically conical (AC) settings respectively; these authors have also
used this deformation theory to derive interesting consequences on the geometry of these complete noncompact spaces. In this paper we describe how this deformation theory for $\gtwo$--holonomy metrics further extends to the noncompact ALC setting and explore some of its consequences.

\subsection*{Main results}

We begin by highlighting three of our geometric results that appear of broadest interest. We will then expand on the contents of the paper and discuss additional results beyond the three highlighted ones.

Recall that ALF gravitational instantons, \ie hyperk\"ahler 4-manifolds with cubic volume growth, come in two infinite families:
\begin{itemize}
    \item the ones of \emph{cyclic type}, where the circle fibration of the end is a principal circle bundle;
    \item the ones of \emph{dihedral type}, where the circle bundle on the end is non-orientable (but the total space is a $\Z_2$--quotient of a principal circle bundle).
\end{itemize}
Minerbe proved that any ALF space of cyclic type, of which the Taub--NUT metric on $\C^2$ is the prototype, must arise from the Gibbons--Hawking ansatz (and hence they are all explicit). The ALF spaces of dihedral type, of which the Atiyah--Hitchin metric is the prototype, are typically not explicit and are much harder to construct. An analogous dichotomy between cyclic and dihedral metrics occurs for (orientable) ALC metrics in any dimension. All the (infinitely many) ALC $\gtwo$--holonomy metrics constructed in our prior work \cite{FHN:ALC:G2:from:AC:CY3,FHN:Coho1:ALC} are of cyclic type. The first highlighted result of this paper is the construction of the first examples of ALC $\gtwo$--metrics of dihedral type. 

\begin{mtheorem}\label{mthm:Dihedral}
There exists a 1-parameter family of complete ALC $\gtwo$--holonomy metrics of dihedral type on~$S^3 \times \mathcal{O}_{\mathbb{P}^1}(-4)$.
\end{mtheorem}

The existence of this family was predicted in the physics literature \cite{ALC:A7}, where it is known as the $\mathbb{A}_7$ family of ALC $\gtwo$--metrics. One end of this 1-parameter family of metrics can be thought of heuristically as a nontrivial fibration over $S^3$ by Atiyah--Hitchin metrics on the total space of $\mathcal{O}_{\mathbb{P}^1}(-4)$. However, we prove Theorem \ref{mthm:Dihedral} by instead rigorously constructing ``the other end'' of the $1$-parameter family; we do so by applying a general desingularisation result for ALC $\gtwo$--holonomy spaces with isolated conical singularities, which we prove in Part \ref{Part3} of the paper. Our desingularisation result is modelled in part on an analogous result in the compact~\gtwo~case due to Karigiannis \cite{Karigiannis} and in part on the construction of dihedral ALF hyperk\"ahler $4$-manifolds by Biquard--Minerbe \cite{Biquard:Minerbe} via desingularisation of ALF hyperk\"ahler $4$-orbifolds 
using ALE dihedral gravitational instantons. Theorem \ref{mthm:Dihedral} is proved by a nontrivial application of our general desingularisation scheme to some of the examples we previously constructed in \cite{FHN:Coho1:ALC} by exploiting cohomogeneity-one methods.

An ALC metric is by definition asymptotic to a locally circle-invariant one on its end, but only in the cyclic case (when the circle bundle is orientable) does the end admits a genuine circle action. 
In the case of ALC $\gtwo$--holonomy metrics, we establish structural results on extending the (local) circle-invariance from the asymptotic limit; these are modelled on analogous results proven in 4-dimensional hyperk\"ahler geometry.

\begin{mtheorem}\label{mthm:Asymptotics} ${}$

\begin{enumerate}
\item Every ALC $\gtwo$-holonomy end metric is asymptotic to a locally circle-invariant $\gtwo$--structure up to an exponentially decaying error.
\item Every complete ALC $\gtwo$--manifold of cyclic type admits a circle action that preserves the $\gtwo$--structure.
\end{enumerate}    
\end{mtheorem}

Part (i) (proven in Theorem~\ref{thm:Exp:Decay:ALC:G2}) is an analogue of a result of Chen--Chen \cite{Chen:Chen:II} in the 4-dimensional hyperk\"ahler setting; the analogue of part (ii) (proven in Theorem~\ref{thm:Circle:Symmetry:Cyclic:ALC}) in that setting is the first step in Minerbe's classification of ALF gravitational instantons of cyclic type \cite{Minerbe:Ak}. The existence of a structure-preserving circle symmetry is not nearly as restrictive for a $\gtwo$--holonomy metric as it is for a 4-dimensional hyperk\"ahler metric; nevertheless,  as a step in deducing part (ii) from part (i) of Theorem \ref{mthm:Asymptotics}, we do obtain the following rigidity result.

\begin{mtheorem}\label{mthm:PositiveMass}
Any ALC $\gtwo$--manifold of cyclic type asymptotic to a flat principal circle bundle over a Calabi--Yau cone
is finitely covered by the Riemannian product of $\Sph^1$ and an AC Calabi-Yau 3-fold.
\end{mtheorem}

In the 4-dimensional hyperk\"ahler case, Minerbe \cite{Minerbe:Mass} deduced a
similar result from %
the Positive Mass Theorem. 
A particular consequence of Minerbe's result is the 
non-existence of non-flat ALF hyperk\"ahler metrics asymptotic to $\R^3\times
S^1$; this non-existence result in the ALF special holonomy setting should be
contrasted with the general $4$-dimensional ALF Ricci-flat setting, where the
Riemannian Schwarzschild metric provides such a metric on $\R^2 \times S^2$.
A similar phenomenon occurs in dimension 7: for any Einstein 5-manifold
$\Sigma$ with positive scalar curvature (in particular, the Sasaki--Einstein
cross-section of a $3$-dimensional Calabi--Yau cone), $\R^2\times\Sigma$
carries a Schwarzschild-like complete ALC Ricci-flat (but not 
special holonomy) metric, see Besse \cite[\S 9.118(a)]{Besse}.

In the rest of this introduction we state further supporting results and discuss
the main ideas of the proofs of our results. This discussion will also serve as a detailed outline of the contents of the paper. Given the length of this paper we hope that the reader will find such a detailed outline useful. 

\subsection*{Analytic tools} This paper establishes a variety of results about arbitrary complete noncompact $\gtwo$--holonomy metrics with ALC asymptotics. One of the threads connecting these results is that they all require a good understanding of ALC $\gtwo$--metrics at the perturbative level, analogous to that first established by Joyce \cite{Joyce:Book} for compact $\gtwo$--manifolds. In the perturbative regime we aim to describe $\gtwo$--holonomy metrics close to a background model (itself torsion-free or approximately torsion-free) by linearising the PDE system characterising the holonomy reduction. The geometric results of the paper therefore rest on a linear-analytic foundation: a  robust Fredholm theory for linear elliptic operators on ALC manifolds. Part \ref{Part1} of the paper provides a self-contained exposition of this analytic framework. Here our main results are Theorems \ref{thm:Fredholm}, \ref{thm:Index:jump} and \ref{thm:Fredholm:2nd:order}.
This section may be read independently of the rest of the paper and we hope its contents 
will be of some interest to geometric analysts beyond the special holonomy community. 

ALC metrics fall into a broader class of asymptotic geometries known as fibred boundary metrics. The general fibred boundary metric is asymptotic to a Riemannian submersion with base metric a Riemannian cone and smooth compact fibres. ALC metrics are thus the simplest fibred boundary metrics (other than ACyl or AC ones
which can both be viewed as somewhat degenerate special cases of fibred boundary metrics). Linear Fredholm theory for fibred boundary metrics is described in the special case of the Hodge-operator $d + d^\ast$ in the work of Hausel--Hunsicker--Mazzeo \cite{HHM} on the $L^2$--cohomology of various classes of complete noncompact spaces and for more general fibred boundary Dirac operators by Kottke--Rochon~\cite[Theorem 5]{Kottke:Rochon:FB}.
These papers take a microlocal approach building on the $\Phi$--calculus developed by Mazzeo--Melrose\cite{Mazzeo:Melrose:Phi:Calculus} and Vaillant \cite{Vaillant}. 
In this paper we provide a more direct and elementary derivation of all the results we need, avoiding the microlocal tools used in the more general theory. We believe there is value in providing such an exposition since it provides clear statements and self-contained proofs for all the refined results we use in the paper. In the very special case of the scalar Laplacian on ALF manifolds this direct approach has already been used by Minerbe \cite[Section 2]{Minerbe:Mass} and again very recently by Kim--Ozuch \cite[Section 2]{Kim:Ozuch}.

As in similar (\ie non microlocal) approaches to linear analysis on AC and ACyl manifolds, the approach we take in this paper is to first analyse certain model operators at infinity and then deduce Fredholm results for ALC manifolds via a perturbation argument. There are two key ingredients to our analysis of the model operators. Central to our exposition is the notion of an adapted connection, a certain distinguished metric connection with torsion on a Riemannian submersion, see Definition~\ref{def:adapted} and Bismut \cite[Definition~1.7]{Bismut}. This adapted connection captures the leading-order behaviour at infinity of the Levi-Civita connection of an ALC metric. Therefore the models at infinity for many natural linear elliptic operators on ALC manifolds (in particular, Dirac-type operators) are provided by the analogous operators associated with the adapted connection of the model metric at infinity. Combined with Fourier decomposition along the fibres, this allows one to work on the end of an ALC manifold as if one were on the direct product of a Riemannian cone with a circle. Fredholm theory in weighted Sobolev and H\"older spaces on Riemannian cones and AC manifolds is very well understood, see for example the summary in \cite[Appendix B]{FHN:ALC:G2:from:AC:CY3}, and it underlies our analysis of the 0-modes, \ie the circle-invariant sections.  Higher Fourier modes are instead easily controlled because of the collapsed asymptotic geometry of ALC manifolds. For example, higher Fourier modes of decaying elements in the kernel of the operators that we consider are typically exponentially decaying (which is the ultimate reason for part (i) of Theorem \ref{mthm:Asymptotics}), while the natural weighted Banach spaces on ALC manifolds we introduce involve polynomial weights.

The direct analytic approach we take in this paper has the advantage that it generalises without complication in two distinct directions, both of which are important 
in further applications of these tools to complete noncompact special and exceptional holonomy metrics.
Firstly, we can allow for more general ALC geometries where the tangent cone at infinity $\tu{C}(\Sigma)$ is an orbifold (but we still require that the total space of the circle fibration be smooth). This generalisation is very natural from the point of view of the theory of Riemannian collapse under bounded curvature, see Fukaya \cite[Proposition 11.5]{Fukaya}.  Furthermore, this orbifold generalisation of ALC has already been observed to be the source of infinitely many new examples of Ricci-flat metrics with exceptional holonomy in \cite{Foscolo:ALC:Spin7}. The analysis we develop here, when combined with the notions and properties of Riemannian foliations recalled in \cite{Foscolo:ALC:Spin7}, covers this generalisation without change. 

Secondly, one can consider asymptotic geometries that generalise the geometry of ALC manifolds by allowing Riemannian submersions over Riemannian cones with fibre a flat torus of higher rank. 
This larger class of asymptotic geometries can be treated with the tools of this paper without any change. 
For example, in the 4-dimensional hyperk\"ahler setting there are many interesting so-called ALG metrics~\cite{Cherkis:Kapustin:ALG,Biquard:Minerbe,Hein} asymptotic to a flat 2-torus bundle over a flat 2-dimensional cone; moreover, a classification result for such $4$-dimensional ALG hyperk\"ahler metrics has been proven by Chen--Chen~\cite{Chen:Chen:III}.
In the higher-dimensional exceptional holonomy context 
there are some recent existence results (but so far no
classification results). 
Cavalleri \cite{Cavalleri} has recently constructed infinitely many complete metrics with holonomy $\tu{Spin}_7$ asymptotic at infinity (with polynomial rate) to a Riemannian submersion over a $3$-dimensional Calabi--Yau cone with fibre a fixed flat 2-torus. 

In principle, analysis in the general fibred boundary case could also be approached using our methods, but the special case of totally geodesic flat torus fibres allows for a simpler exposition because one can always reduce to the case of trivial vertical tangent bundle and trivialise all vertical tensor bundles using parallel tensors. In the general fibred boundary case, one would have to work instead with operators on the base acting on sections of generally nontrivial vector bundles parametrising elements of the kernel of the operator restricted to each fibre. 
However, we do not pursue this extension further in the current paper.

\subsubsection*{Hodge theory on ALC spaces}
As a nontrivial application of the linear analysis developed in the first part of the paper, we prove Hodge-theoretic results for ALC manifolds. More precisely, we describe all the contributions to the space $\mathcal{H}_\lambda^k(M)$ of closed and coclosed $k$-forms on $M$ of decay~$\lambda$ that can be described in terms of topological data. Our main results are Theorems \ref{thm:ALC_hodge_fastdecay} and \ref{thm:ALC_hodge_jump} (which are too technical to be stated in detail in this introduction).
The specialisations of these results to the case of $3$-forms on ALC $\gtwo$--manifolds will be essential in the latter parts of this paper, but we decided to state and prove results that work in any dimension and without any curvature or holonomy reduction assumptions on the ALC manifold involved nor on its asymptotic cone. We hope these Hodge-theoretic results will be useful in other geometric analysis problems.

The first theorem (Theorem~\ref{thm:ALC_hodge_fastdecay}) reproduces the calculation of the $L^2$--cohomology of ALC manifolds  given in the work of Hausel--Hunsicker--Mazzeo \cite[Corollary 1]{HHM}: here we recover their results with independent proofs 
that avoid the use of geometric microlocal analytic tools. Moreover, our proof also applies to the case of orbifold ALC ends. More importantly, in the second theorem (Theorem~\ref{thm:ALC_hodge_jump}) we can extend the results to be able to compute weighted $L^2$--cohomology at rates of decay slower than the one corresponding to (unweighted) $L^2$--integrability, describing  the jump of $\mathcal{H}_\lambda^k(M)$ as $\lambda$ crosses certain distinguished values in terms of topological data. 

\subsection*{Deformation theory and applications} 
 The local smooth structure of the moduli space of $\gtwo$--holonomy metrics was established in the smooth compact case by Joyce \cite{Joyce:Book}. 
 Several extensions of this deformation theory have since been established, 
 where instead one considers metrics on noncompact spaces with tame asymptotic geometry. 
 Specifically, Nordstr\"om \cite{Nordstrom:ACyl:G2} considered the deformation theory in the complete noncompact asymptotically cylindrical case (ACyl) and then Karigiannis--Lotay in \cite{Karigiannis:Lotay} studied complete
noncompact $\gtwo$--holonomy metrics with asymptotically conical ends (AC) (and also compact incomplete $\gtwo$--holonomy spaces with isolated conical singularities). In this paper we give a parallel discussion of this deformation theory in the ALC setting. There are a number of new features in this case that we now explain. These new challenges have also allowed us to refine and clarify some of the methods introduced in \cite{Karigiannis:Lotay}.

First of all, we comment on the asymptotic geometry of ALC $\gtwo$--holonomy metrics. This is described in terms of a circle bundle over a 6-dimensional cone $\tu{C}$. Passing to a double cover of the ALC end in the dihedral case, we can always reduce to the case where the circle bundle is a principal one. The holonomy reduction then forces the base $\tu{C}$ to be a Calabi--Yau cone. The model \gtwo-structure can be written as
\[
\varphi_\BC = \theta_\infty\wedge\omega_\tu{C}+ \Real\Omega_\tu{C},
\]
where $(\omega_\tu{C},\Omega_\tu{C})$ is the conical Calabi--Yau structure and $\theta_\infty$ is a connection on the principal circle bundle, which is further forced to be Hermitian Yang--Mills (HYM). 
(The subscript BC here refers to ``bundle over cone''.)
This model $3$-form $\varphi_\BC$ is \emph{not} closed. There is however a universal correction to it to make it closed:
\[
\varphi'_\BC = \theta_\infty\wedge\omega_\tu{C}+ \Real\Omega_\tu{C} - \tfrac{1}{2}r^2 \eta\wedge d\theta_\infty,
\]
where $\eta$ is the contact 1-form on the Sasaki--Einstein cross section of the Calabi--Yau cone. We then say that a $\gtwo$--structure $\varphi$ is ALC if it is asymptotic with polynomial decay rate to the closed $3$-form $\varphi'_\BC$; it is ALC with rate $\nu$ if the error terms are $O(r^{-\nu})$. Since the correction term $\varphi'_\BC - \varphi_\BC$ is itself of order $O(r^{-1})$ it is natural to impose faster decay on the difference $\varphi-\varphi'_\BC$, \ie to assume $\nu < -1$. This assumption, in particular, implies that all the ALC $\gtwo$--manifolds we consider in the paper are ALC Riemannian manifolds of rate $\leq -1$ (indeed, the rate as an ALC Riemannian manifold is exactly $-1$ unless $d\theta_\infty = 0$). This assumption does not appear to be restrictive in terms of the known examples, and it is convenient as it eliminates various technical complications that would otherwise arise from the construction of a slice for the action of the diffeomorphism group, see Karigiannis--Lotay \cite[\S 5.2.1]{Karigiannis:Lotay}. Note also that, unlike the AC and ACyl settings, the model $\gtwo$--structure $\varphi'_\BC$ is closed but \emph{not} torsion-free  
(unless the HYM connection $\theta_\infty$ is actually flat, \ie $d \theta_\infty=0$.)

The moduli space we wish to study is the quotient of the space of torsion-free \gtwo-structures that are ALC with a fixed rate $\nu < -1$ to a \emph{fixed} end $\varphi'_\BC$ by the identity component of the group of diffeomorphisms that are asymptotic (in an appropriate sense) to the identity on the end.

\begin{remark*}
    We could also consider quotients by diffeomorphisms that are allowed to be asymptotic to nontrivial symmetries of the end (see Definition \ref{def:ALC:automorphism}). The resulting quotients could be regarded as subquotients of the moduli space that we consider (in the context of Theorem \ref{mthm:Symmetries} below it makes no difference though).
\end{remark*}

Now, the basic strategy for understanding $\gtwo$--holonomy metrics close to a given one is to (a) find a good slice to the action of diffeomorphisms, and (b) in the presence of the slice condition turn the torsion-free condition into a nonlinear elliptic PDE that can be solved via the Implicit Function Theorem. This strategy was pioneered in the compact case by Joyce \cite[\S10.4]{Joyce:Book}. An efficient approach to implementing step (a) that exploits the mapping properties of the Dirac operator was later described in the ACyl case by Nordstr\"om in \cite[\S6.3]{Nordstrom:ACyl:G2}. Given the analysis developed in the first part of the paper, the extension of both parts (a) and (b) to the ALC setting is fairly straightforward.

\begin{mtheorem}\label{mthm:Moduli}
Let $(M,\varphi)$ be an ALC $\gtwo$--manifold of rate $\nu\in (-3,-1)$. Then the moduli space of ALC torsion-free $\gtwo$--structures on $M$ of rate $\nu$ is a smooth manifold whose tangent space at the orbit of $\varphi$ is identified with the space $\mathcal{H}^3_\nu(M,g_\varphi)$ of closed and coclosed 3-forms of rate $\nu$.
\end{mtheorem}

The dimension of the space $\mathcal{H}^3_\nu(M,g_\varphi)$ depends on the topology of $M$ \emph{as well as} on the spectral properties of the cross-section of its tangent cone at infinity. We discuss in Section \ref{sec:Dimension} how to compute the dimension of the moduli space for the examples of ALC \gtwo--holonomy metrics we have constructed in our prior work \cite{FHN:ALC:G2:from:AC:CY3,FHN:Coho1:ALC} and \cite[Theorem 4.12]{Foscolo:ALC:Spin7}. 

The topological contributions to the ALC moduli space can be calculated as an application of the general ALC Hodge theory developed in the earlier part of the paper. We exemplify these computations with the cohomogeneity-one ALC $\gtwo$--manifolds that we constructed in \cite{FHN:Coho1:ALC}. In particular, these computations yield local uniqueness of all the cohomogeneity-one ALC $\gtwo$--holonomy metrics constructed in \cite{FHN:Coho1:ALC} and also of the metrics provided by Theorem \ref{mthm:Dihedral}. 

In the particular context of the highly-collapsed ALC $\gtwo$--holonomy metrics constructed in \cite{FHN:ALC:G2:from:AC:CY3} on a principal circle bundle $M$ over an AC Calabi--Yau 3-fold $B$ we can in fact go much further and identify all contributions to the ALC $\gtwo$--moduli space, not only the topological ones. We do this by relating the ALC $\gtwo$--moduli space
on $M$ to the AC Calabi--Yau moduli space on $B$ and exploiting recent results about 
AC Calabi--Yau manifolds.

\begin{mtheorem}\label{mthm:Moduli:AC:CY} Let $M\ra B$ be the total space of a principal circle bundle over a 6-manifold $B$ carrying AC Calabi--Yau structures. Then for any $\nu \in (-3,-1)$ there is a local diffeomorphism between the moduli space of ALC \gtwo--holonomy metrics of rate $\nu$ on $M$ constructed in \cite{FHN:ALC:G2:from:AC:CY3} and a moduli space of AC Calabi--Yau structures $(\omega,\Omega)$ of rate $\nu$ on $B$ satisfying the topological constraint $c_1(M)\cup [\omega]=0\in H^4(B)$.
\end{mtheorem}
Combined with the recent classification of AC Calabi--Yau manifolds obtained by Conlon--Hein \cite{Conlon:Hein:Classification} in terms of complex geometric data, this theorem allows us to compute the dimension of the moduli space of all the ALC $\gtwo$--manifolds produced by \cite{FHN:ALC:G2:from:AC:CY3}. 

\medskip
The proof of part (i) of Theorem \ref{mthm:Asymptotics} is based on a strategy similar to the construction of the moduli space of ALC $\gtwo$--holonomy metrics: we prove a gauge-fixing result on the end of an ALC $\gtwo$--manifold $(M,\varphi)$, see Proposition \ref{prop:Slice:End}, that enables us to write the torsion-free condition for~$\varphi$ as an elliptic system for the difference $\varphi-\varphi'_\BC$. Exponential decay of the higher Fourier modes arises naturally from the structure of these equations and properties of the linearised equation (as in the 4-dimensional hyperk\"ahler setting \cite{Chen:Chen:II}). The gauge-fixing result on the end is analogous to a result proven in the AC setting by Karigiannis--Lotay \cite[\S 6.6]{Karigiannis:Lotay}. Working on an exterior domain in the noncompact manifold $M$ introduces new difficulties due to the need to choose suitable boundary conditions that allow us to understand the mapping properties of the Dirac operator on such domains. Karigiannis--Lotay manage to bypass this point by giving an ad hoc but slightly cumbersome argument that relies on working on the whole manifold $M$, 
rather than an exterior domain in it. In this paper we show that imposing suitable classical boundary conditions for differential forms introduced by Schwarz \cite{Hodge:BVP} in the study of Hodge theory on manifolds with boundary suffices for our purposes (our method also applies in the AC case). This clean approach might turn out to be useful more generally in the study of $\gtwo$--manifolds with boundary as initiated by Donaldson \cite{Donaldson:G2:Boundary}. 

We now discuss the second part of Theorem \ref{mthm:Asymptotics} and Theorem \ref{mthm:PositiveMass}. When studying complete noncompact Ricci-flat manifolds (or manifolds with other curvature constraints), a natural question is whether symmetries of the asymptotic model at infinity extend to symmetries of the manifold itself. In the case of continuous symmetries of AC \gtwo-manifolds, this has already been affirmed by Karigiannis--Lotay \cite[\S 6.3]{Karigiannis:Lotay}.

The model $\varphi'_\BC$ admits two types of symmetry: the circle action arising from the principal circle bundle structure;  and secondly,  those symmetries of the cross-section of the tangent cone at infinity that lift to bundle automorphisms that preserve the HYM connection $\theta_\infty$. Theorem \ref{mthm:Asymptotics} (ii) states that the former type of model symmetry always extends to a symmetry of any ALC $\gtwo$--manifold of cyclic type. In general, even at the linearised level there is no reason why model symmetries of the second type need to act trivially on the tangent space $\mathcal{H}^3_\nu (M,g_\varphi)$ \emph{unless} the latter can be identified with a subspace of de Rham cohomology.

\begin{mtheorem}\label{mthm:Symmetries}
Let $(M,\varphi)$ be an ALC $\gtwo$--manifold of rate $\nu \in (-3,-1)$. Assume that the natural map $\mathcal{H}^3_\nu (M,g_\varphi)\rightarrow H^3(M)\oplus H^4(M)$ given by 
$\rho \mapsto ([\rho],[\ast \rho])$ is injective. Then the action of the identity component of the symmetry group of the asymptotic model at infinity extends to an action by $\gtwo$--automorphisms on~$(M,\varphi)$. 
\end{mtheorem}

This is proved in Theorem \ref{thm:Symmetries}.
Together with the cohomogeneity-one analysis in \cite{FHN:Coho1:ALC}, it allows us to prove strong uniqueness results for ALC $\gtwo$--manifolds with  $\sunitary{2}^2\times\unitary{1}$--invariant ends, see Theorem \ref{thm:Uniqueness:Coho1:ALC}.

The proofs of the symmetry extension results Theorem \ref{mthm:Asymptotics}(ii) and \ref{mthm:Symmetries} involve two main steps. The first step is to show that each end symmetry under consideration is the restriction of a global symmetry of the ALC \gtmfd{} $(M,\varphi)$. Here the main idea (similar to \cite[\S 6.3]{Karigiannis:Lotay} in the AC case) is to extend the end symmetry to an arbitrary ALC diffeomorphism $F$ and deduce from Theorem \ref{mthm:Moduli} that $F^\ast \varphi$ represents the same point in the ALC moduli space. That means there exists a diffeomorphism $G$ of $M$ asymptotic to the identity such that $G \circ F$ is the desired automorphism. To justify this argument here we need to use either part (i) of Theorem \ref{mthm:Asymptotics}, if the symmetry in question is the principal circle action, or otherwise invoke the injectivity assumption in Theorem \ref{mthm:Symmetries} in general. 
 
 However, even if we can extend the elements of some group of symmetries of the end elementwise, that does not immediately imply that we can extend the action of the whole group; conceivably the result could instead be an action of the universal cover of the group of end symmetries. (This was not an issue in the particular AC examples considered in \cite[\S 6.3]{Karigiannis:Lotay} since  the end symmetry groups were all simply connected). The second main step in the proof is therefore to show that the restriction map from symmetries of the \gtmfd{} to symmetries of the end is always injective. We prove this in Proposition \ref{prop:Decaying:symmetries}, using the fact that, unless the tangent cone at infinity is flat, ALC $\gtwo$--manifolds admit a canonical foliation by the gradient flow-lines of the function $u$ (unique up to an additive constant) satisfying $\triangle u = 1$. The idea of using such a function to construct canonical foliations of the end of a complete Ricci-flat manifold is due to Biquard--Hein \cite{Biquard:Hein}. Finally, the case when the tangent cone at infinity is flat is trivial as a consequence of Theorem \ref{mthm:PositiveMass}. The proof of the latter theorem is based on known $\gtwo$--identities and only uses the existence of an automorphic vector field asymptotic to the vector field that generates the principal circle action on the model at infinity and the fast decay of its covariant derivative implied by the assumption of the flatness of the circle bundle on the asymptotic model.

\subsection*{Existence of ALC \gtwo--manifolds of dihedral type}

In previous work \cite{FHN:ALC:G2:from:AC:CY3} we gave a general construction of ALC \gtmetric s on principal circle bundles over AC Calabi--Yau $3$-folds. The ALC metrics we constructed are close to the \emph{collapsed limit} in which the circle fibres shrink to zero length. More precisely, we constructed families of ALC \gtmetric s of cyclic type that depend on a parameter $\epsilon$ that measures the length of the circle fibre at infinity; the metrics we constructed exist for $\epsilon>0$ sufficiently small.

Now, we would also like to study degenerations of ALC manifolds at the opposite extreme where instead $\epsilon\ra \infty$ and the circle fibres at infinity get longer and longer. In \cite{FHN:Coho1:ALC} we used ODE methods to study the ``large circle'' limit of $\sunitary{2}^2\times\unitary{1}$--invariant ALC \gtmetric s and in that setting we found a uniform qualitative behaviour: after suitable rescalings, as $\epsilon\ra\infty$ we see ALC \gtmfd s that converge to an AC \gtmfd. For example, the explicit ALC {\gtmetric} on $S^3\times \R^4$ constructed by Brandhuber et al \cite{BGGG} admits a one-parameter family of cohomogeneity-one ALC {\gtwo}-deformations (up to scale): see Bogoyavlenskaya \cite{Bogoyavlenskaya}, or the alternative proof given in~\cite{FHN:Coho1:ALC}. By normalising this family so that the size of the zero-section has fixed volume, the parameter can be interpreted as the length of the circle fibres at infinity. As the parameter grows larger and larger the ALC metrics approach Bryant--Salamon's AC metric on the spinor bundle of $S^3$ \cite{Bryant:Salamon}.

In \cite{FHN:Coho1:ALC} we also provided a very clear picture for the transition from ALC to AC asymptotic geometry. The scaling freedom inherent in the problem furnishes a way of regarding the ``large circle'' limit from a different perspective, where instead \emph{ALC manifolds with isolated conical singularities} appear in the limit. For example, the aforementioned $1$-parameter family of cohomogeneity-one ALC \gtmetric s on $S^3\times\R^4$ of \cite{FHN:Coho1:ALC,Bogoyavlenskaya} can alternatively be normalised so that the length of the circle fibres at infinity have fixed length. Under this rescaling the ``large circle limit'' corresponds to the limit of a family of ALC \gtmetric s on $S^3\times\R^4$  in which the zero-section has vanishing size and there
a conical singularity appears. In fact, in \cite[Theorem A(ii)]{FHN:Coho1:ALC} we proved the existence of a unique up to scale cohomogeneity-one ALC {\gtmfd} with an isolated conical singularity in its interior modelled on the cone over $S^3\times S^3$ endowed with its (unique) homogeneous nearly K\"ahler structure. The smooth ALC metrics with large but finite $\epsilon$ can then be thought of as arising from desingularising the isolated singularity of (a rescaled copy of) the conically singular ALC space by using the Bryant--Salamon AC \gtmfd.

The purpose of the last part of the paper is to make this picture precise
in a much broader setting where we do not assume any continuous symmetries. 
To that end, we study the problem of desingularising isolated conical singularities of ALC {\gtwo}--spaces by gluing in AC \gtmfd s. Applying this construction to (suitable finite quotients of) the conically singular {\gtwo}--space constructed in \cite{FHN:Coho1:ALC} and the AC \gtmfd s of \cite{Bryant:Salamon} and \cite{FHN:Coho1:ALC} will then allow us to recover analytically the families of cohomogeneity-one \gtmetric s of \cite{FHN:Coho1:ALC} in the ``large circle'' limit, previously constructed by ODE methods.  More importantly, by exploiting  the discrete symmetries of the cone over ${S^3\times S^3}$ we will also be able to construct a new family of ALC \gtmetric s beyond the ones studied in~\cite{FHN:ALC:G2:from:AC:CY3,FHN:Coho1:ALC}. This new family, denoted $\mathbb{A}_7$ in the physics literature and whose existence is stated in Theorem~\ref{mthm:Dihedral}, provides the first known example of ALC \gtmetric s of dihedral type. At the opposite end of the 1-dimensional moduli space the $\mathbb{A}_7$ family is expected to collapse to the quotient of Stenzel's AC Calabi--Yau metric on $T^\ast S^3$ by the standard antiholomorphic involution and with an adiabatic fibration of Atiyah--Hitchin metrics over $S^3$ bubbling off in the limit. We plan to describe  this other degeneration of the $\mathbb{A}_7$ family rigorously in future work.

Our construction of dihedral ALC $\gtwo$--manifolds can be regarded as a $\gtwo$--version of the  Biquard--Minerbe~\cite{Biquard:Minerbe} gluing construction of 4-dimensional ALF hyperk\"ahler metrics of dihedral type:  they desingularised finite quotients of the Taub--NUT metric on $\C^2$ by gluing in ALE gravitational instantons of dihedral type. 
Foscolo~\cite{Foscolo:K3} gave a gluing construction of sequences of hyperk\"ahler metrics on the K3 surface collapsing to the 3-dimensional limit space $\mathbb{T}^3/\Z_2$.
In view of the central role played by the $4$-dimensional dihedral ALF hyperk\"ahler spaces in Foscolo's construction we expect that ALC spaces of dihedral type will be essential building blocks in the future construction of sequences of \emph{compact} \gtmfd s collapsing to a compact Calabi--Yau $3$-fold (possibly with isolated conical singularities). 

The desingularisation problem for \emph{compact} \gtmfd s with isolated conical singularities has already been studied by Karigiannis \cite{Karigiannis}. His construction has three steps: (i) understand the obstructions to smoothing a conically singular {\gtmfd} $M_0$ by studying the linearised problem; (ii) under appropriate topological conditions, construct a manifold $M$ with an approximately torsion-free {\gtstr} by using AC \gtmfd s to resolve the singularities of $M_0$; (iii) deform the approximate solution into a genuine {\gtmetric} on $M$ using a general deformation result for closed $\gtwo$--structures on compact manifolds with small torsion due to Joyce \cite{Joyce:Book}. Our goal is to adapt this strategy to the noncompact ALC setting. The construction of an approximate solution in step (ii) does not present any additional difficulty in our ALC setting.  In contrast, steps (i) and (iii) require some substantial changes when 
passing from the compact to the ALC setting.

In particular, when adapting to the ALC setting Joyce's existence result for torsion-free $\gtwo$--structures close to a closed $\gtwo$--structure with small torsion, which we do in Theorem \ref{thm:Joyce:ALC}, we face two main new issues. Firstly, in the compact case measures of the smallness of the torsion of the background $\gtwo$--structure in the $C^0$ and $L^p$--norms are equivalent. In the noncompact case the necessary control of the $L^p$--norm forces the torsion not only to be small but also to be of sufficiently fast decay. This is not unexpected and parallels for example extensions of Yau's proof of the Calabi Conjecture to the noncompact case, in the refined form proven by Hein \cite{Hein}: in that case, even though one need not assume the failure of the initial K\"ahler metric to solve a complex Monge--Amp\`ere equation to be small, it is still necessary to require that the Ricci potential of this background metric decays sufficiently fast at infinity. Secondly, the linear problem one needs to solve to prove Theorem \ref{thm:Joyce:ALC} is a Poisson equation on 2-forms. This is particularly delicate for ALC manifolds of dimension $7$, since the Laplacian on 2-forms on a 6-dimensional cone has a double indicial root.  Therefore it has analytic features analogous to those of a scalar Poisson equation on non-parabolic manifolds such as $\R^2$.    

While we have to work harder on the analytic side, our noncompact setting has two significant advantages compared with the compact case considered by Karigiannis \cite{Karigiannis}: firstly, as is often the case, the freedom to change the asymptotic geometry at infinity allows us to overcome some of the topological obstructions present in the compact setting; secondly, in \cite{FHN:Coho1:ALC} we have already proven the existence of at least one conically singular ALC {\gtwo}--holonomy space (in contrast, at present there are no known compact {\gtwo}--holonomy spaces with isolated conical singularities). Theorem \ref{mthm:Dihedral} is proven by applying our general desingularisation result Theorem \ref{thm:Desing:CS:ALC} to some of the new examples (a certain quotient of the conically singular ALC space as well as the AC $\gtwo$--manifold used in the desingularisation) constructed in \cite{FHN:Coho1:ALC}. This application is nontrivial because to obtain a dihedral (as opposed to a cyclic) ALC space requires us to exploit special features of the discrete symmetries of both the building blocks and of the tangent cone at the isolated conical singularity.  

\subsection*{Organisation of the paper} 

The paper is organised into three parts: Part 1 is devoted to a discussion of Fredholm theory for linear elliptic operators on weighted Banach spaces on general ALC manifolds; Part 2 contains the main results about the deformation theory of ALC $\gtwo$--holonomy metrics and the proofs of our main Theorems \ref{mthm:Asymptotics} and \ref{mthm:PositiveMass}; Part 3  explains the general desingularisation construction underlying main Theorem \ref{mthm:Dihedral}.  

Part \ref{Part1} of the paper is organised as follows. Section \ref{sec:ALC} discusses the asymptotic geometry of ALC manifolds and introduces the adapted connection. Sections \ref{sec:Dirac} and \ref{sec:weighted:spaces} introduce, respectively, Dirac-type operators and weighted Banach spaces on ALC manifolds.  Section~\ref{sec:fredholm} establishes Fredholmness of Dirac-type (and natural second-order) operators on ALC manifolds in weighted Banach spaces outside of a discrete set of indicial roots and index jump formulas across such indicial walls. Generalisations of these results are discussed at the end of the section. Finally, Section~\ref{sec:Hodge:ALC} applies our ALC Fredholm theory to study  the Hodge theory of ALC manifolds and identifies all topological contributions to spaces of closed and coclosed forms on an ALC manifold.

Part \ref{Part2} is organised into four sections. Section \ref{sec:ALC:G2} introduces ALC $\gtwo$--manifolds and their asymptotic geometry and summarises indicial roots calculations from \cite{FHN:ALC:G2:from:AC:CY3} that arise from the holonomy reduction. Section \ref{sec:def_theory} contains the construction of the moduli space of ALC $\gtwo$--holonomy metrics and the proof of part (i) of %
Theorem \ref{mthm:Asymptotics}. 
Section~\ref{sec:Symmetries} discusses the problem of extending symmetries of the end of an ALC \gtwo--metric to the interior, thereby concluding the proofs of Theorems \ref{mthm:Asymptotics} and~\ref{mthm:PositiveMass}.
Section~\ref{sec:Dimension} contains calculations of the dimension of various moduli spaces of ALC $\gtwo$--holonomy metrics and the proof of Theorem~\ref{mthm:Moduli:AC:CY}.

Part \ref{Part3} consists of two sections: Section \ref{sec:CS:ALC:Def} introduces ALC $\gtwo$--spaces with isolated conical singularities and establishes some of their Hodge-theoretic properties; Section \ref{sec:CS:ALC:Gluing} provides a proof of our general desingularisation result, Theorem \ref{thm:Desing:CS:ALC}, and its application to  the proof of Theorem~\ref{mthm:Dihedral}.

\subsection*{Acknowledgements} During this work MH and JN were partially supported by the Simons Collaboration on Special Holonomy in Geometry, Analysis and Physics (grants \#488620, Mark Haskins, and \#488631, Johannes Nordstr\"om) and LF was partially supported by a Royal Society University Research Fellowship, the INdAM (Istituto Nazionale di Alta Matematica) research group GNSAGA and the Ministero dell'Universit\`a e della Ricerca of Italy -- Avviso FIS2 grant FIS-2023-0395. The paper is based upon work partially supported by the National Science Foundation under Grants No. DMS-1440140 and No. DMS-1928930, since it started while all three authors were in residence at MSRI, Berkeley, California, during the Spring 2016 semester and continued when all three authors were in residence at the same institute (with the new name of Simons Laufer Mathematical Sciences Institute) in the Fall 2024 semester. The authors are grateful to all these funders and to MSRI/SLMath and their staff for the support.  

\part{Analysis on ALC manifolds}\label{Part1}

This first part of the paper establishes a general Fredholm theory for elliptic operators (focusing on first-order operators, in view of our applications) acting on weighted Sobolev and H\"older spaces on Riemannian manifolds with an ALC end. We want to stress that the results 
of this part of the paper make no assumption on any kind of holonomy reduction of the ALC metric. We will apply this general machinery to ALC \gtwo--holonomy metrics in Parts \ref{Part2} and \ref{Part3} of the paper, which contain our main new results. As already explained in the introduction 
similar Fredholm results can be deduced from results about linear elliptic theory for more general fibred boundary metrics proven using geometric microlocal methods. We have instead chosen to exploit fully the ALC structure to provide a self-contained treatment of the requisite Fredholm theory on ALC spaces that avoids microlocal methods. We hope this choice will make the paper accessible to a wider range of geometric analysts. 
Moreover, as we will explain later another benefit is that our approach also extends with only cosmetic changes in two separate directions, both of which turn out
to arise naturally in other special holonomy related problems. 

\section{The asymptotic geometry of ALC manifolds}\label{sec:ALC}

The aim of this section is to introduce the notion of an ALC manifold, \ie a complete Riemannian manifold with an end asymptotic to a circle fibration over a Riemannian cone with fibres of fixed finite length. We will give our initial definitions in the category of smooth Riemannian manifolds and specialise to ALC \gtmfd s only in the second part of the paper.

\subsection{ALC manifolds}
\label{subsec:ALC}

Let $(\Sigma, g_\Sigma)$ be a closed $(n-1)$-dimensional smooth Riemannian manifold. The \emph{Euclidean cone} over $\Sigma$ is the manifold $\tu{C}(\Sigma)=\R^+\times \Sigma$ endowed with the (incomplete) Riemannian metric $g_\textup{C}=dr^2 + r^2 g_\Sigma$. Here $r$ is a coordinate on the $\R^+$--factor.

Let $\pi\co N\ra \Sigma$ be a circle bundle, which we assume to have a fixed $O(2)$ structure. While we will always \emph{assume the total space $N$ to be orientable}, we will not assume $\Sigma$ or the bundle to be orientable (\ie that the structure group can be reduced to $SO(2)$, to make it a principal circle bundle). Nevertheless, the restriction of the bundle to any small neighbourhood can always be reduced to a principal circle bundle, so there is a well-defined notion of tensors on the total space~$N$ being ``locally circle-invariant''.

We find it expedient to say that a connection on $N$ is a locally circle-invariant, symmetric bilinear form $\theta^2$ of rank 1 on the total space, whose restriction to each fibre is a metric of circumference~$2\pi$. If there exists a square root $\theta \in \Omega^1(N)$, then that defines an orientation on the bundle, and a connection in the usual sense on the resulting principal $O(2)$-bundle. (In general, we can in any case say that $\ker \theta^2$ defines choices of horizontal subspaces in $TN$.)

Often it is convenient to deal with the non-orientable case by noting that, because we assume the total space $N$ to be orientable, the pull-back $\wt N$ of the circle bundle to the oriented double cover $\wt \Sigma$ is always a principal circle bundle. The total space $\wt N$ is then also a double cover of $N$, and the nontrivial deck transformation $\iota$ is an orientation-preserving involution that reverses the orientations of the fibres, and lifts a fixed-point free involution on 
$\Sigma$. We shall refer to such an $\iota$ as a \emph{standard involution}.

Any locally circle-invariant Riemannian metric on the total space $N$ can be written as
\begin{equation}
g_N = \pi^* g_\Sigma + \ell^2 \theta^2
\end{equation}
for some Riemannian metric $g_\Sigma$ on $\Sigma$, function $\ell : \Sigma \to \R_{>0}$ and connection $\theta^2$. For the purposes of defining the ALC end models, we will take $\ell$ to be a positive \emph{constant}. Equivalently, we assume that the circle fibres of $\pi$ are totally geodesic. Most of the time we will further assume that $\ell=1$. In our applications to Ricci-flat manifolds this can always be achieved by scaling.

We can now define the model for an ALC end. Let $\BC (\Sigma)=\BC (\Sigma, \pi,\theta^2,\ell)$ be the product $\R^+ \times N$ endowed with the metric
\begin{equation}\label{eq:FB:N} 
g_{\BC}=dr^2 + r^2 g_\Sigma + \ell^2 \theta^2.
\end{equation}
By radial extension of $\pi$ and $\theta^2$ from $\Sigma$ to $\tu{C}(\Sigma)$ we think of $\BC(\Sigma)$ as the total space of a circle bundle $\pi\co \BC (\Sigma)\ra \tu{C}(\Sigma)$ endowed with a connection $\theta^2$. Then $\pi\co \big( \BC(\Sigma),g_{\BC}\big)\ra \big( \tu{C}(\Sigma),g_{\tu{C}}\big)$ becomes a Riemannian submersion with totally geodesic fibres of length $2\pi\ell$. The structure of $\BC(\Sigma)$ as a \emph{bundle over a cone} explains our choice of notation. 

\begin{remark*}
By changing coordinates $x=\tfrac{1}{r}$ the metric $g_{\BC}$ becomes $\frac{dx^2}{x^4} + \frac{g_\Sigma}{x^2} + \ell^2 \theta^2$, which is a special case of an exact \emph{fibred boundary} metric in the language of \cite{HHM}.
\end{remark*}

An ALC manifold is a complete Riemannian manifold $(M,g)$ with finitely many ends all of which are asymptotic to $\BC (\Sigma)$ of the above form. For simplicity of exposition and since we are mainly interested in complete Ricci-flat manifolds we will restrict to the case when $M$ has only one end.

\begin{definition}\label{def:ALC:n}
Let $(M^{n+1},g)$ be a complete Riemannian manifold with only one end. We say that it is an \emph{ALC manifold} asymptotic to $\BC (\Sigma) = \BC (\Sigma, \pi, \theta^2,\ell)$ with rate $\tau<0$ if there exists a compact set $K\subset M$, a positive number $R>0$ and a diffeomorphism $f\co \BC(\Sigma) \cap \{ r> R\} \ra M\setminus K$ such that for all $j\geq 0$
\[
|\nabla_{g_\BC}^j (f^\ast g - g_\BC) |_{g_\BC}=O(r^{\tau-j}).
\]
We say that $M$ is of \emph{cyclic type} if $\Sigma$ is orientable, and that it is of \emph{dihedral type} if $\Sigma$ is not orientable.
\end{definition}

For $\SigmP$ oriented, we can instead write $\BC(\SigmP, \pi, \theta, \ell)$ for the model, where $\theta$ is the honest $SO(2)$-connection.
If we need to ensure that we work with such a cyclic model, then we set $\SigmP = \Sigma$ or the oriented double cover $\wt \Sigma$ according to whether $\Sigma$ itself is orientable or not, and $f_+ = \textrm{Id}$ or its double cover map.

\begin{remark*}
The name ALC, which stands for \emph{asymptotically locally conical}, first arose in the physics literature in connection with the construction of the first examples of such manifolds with $\tu{Spin}_7$--holonomy \cite{CGLP:New:Spin(7)}. In the $4$-dimensional case, if we further assume that the cone $\tu{C}(\Sigma)$ is (Ricci-)flat then the acronym ALF is more commonly used, \cf \cite{Minerbe:ALF}. In fact, in \cite[\S 7.3]{HHM} the explicit ALC metric with \gthol of \cite{BGGG} is also called ALF. However, in \cite[pp. 34--35]{CGLP:New:Spin(7)} the name ALF is reserved for the special case when the asymptotic cone $\tu{C}(\Sigma)$ is \emph{flat}, in the same way that ALE metrics are often distinguished from the more general class of asymptotically conical (AC) ones. We have chosen to adopt the same convention. The terminology ABC (asymptotically a bundle over a cone) has also been put forward \cite{Apostolov:Salamon}.
\end{remark*}

\begin{remark*}
The use of the terminology cyclic/dihedral type is motivated by the $4$-dimensional case: with the exceptions of $\R^3\times S^1$ and the $D_2$ ALF metrics, the boundaries of large geodesic balls in a hyperk\"ahler ALF space of cyclic (dihedral) type are diffeomorphic to $S^3/\Gamma$ where $\Gamma$ is a finite cyclic (binary dihedral) subgroup of $\sunitary{2}$. 
\end{remark*}

Note that our chosen definition of ALC manifolds rules out the possibility of circle fibrations over orbifolds. These could arise for example by taking the quotient of $\BC(\Sigma)$ as above by a finite group $\Gamma$ of isometric bundle automorphisms; if the underlying diffeomorphisms of $\Sigma$ act freely then this simply recovers $\BC(\Sigma/\Gamma)$ in the sense defined above, but if the action on $\Sigma$ has fixed points then $\BC(\Sigma)/\Gamma$ is a Seifert circle fibration over the orbifold $\Sigma/\Gamma$. The extension to Seifert fibrations will be discussed in Section \ref{subsec:seifert}.

\subsection{The Riemannian geometry of \texorpdfstring{$\BC (\Sigma)$}{BC(Sigma)} and the adapted connection}

In the rest of the section we study the geometry of the model at infinity $\BC(\Sigma)$ of an ALC manifold. Following Bismut \cite[Definition 1.7]{Bismut} we will introduce an adapted connection $\overline{\nabla}$ which is better suited to the fibration structure than the Levi-Civita connection of $g_\BC$. In particular, $\overline{\nabla}$ will be used in the next sections to describe the leading-order behaviour of natural differential operators on $\BC (\Sigma)$.

We assume that $\Sigma$ is oriented, so that we can consider the connection 1-form $\theta \in \Omega^1(\BC(\Sigma))$. (This is not a restriction, since the work here is entirely local; in particular Definition \ref{def:adapted} makes sense also when $\Sigma$ is not orientable.)
Let $\xi$ be the vertical vector field on $\BC (\Sigma)=\BC (\Sigma,\pi,\theta,\ell)$ dual to $\theta$. For every vector field $X$ on the cone $\tu{C}(\Sigma)$ we denote by $\widetilde{X}$ its horizontal lift to $\BC (\Sigma)$. Finally, to ease the notation we will not distinguish between $d\theta$ as a $2$-form on $\BC (\Sigma)$ and as a $2$-form on the cone $\tu{C}(\Sigma)$.

Since $\wt X$ is circle-invariant, we have $[\xi,\widetilde{X}]=0$. The standard formulae for the Levi-Civita connection of a Riemannian submersion \cite[Definition 9.25]{Besse} yield
\begin{equation}\label{eq:Levi:Civita}
\begin{aligned}
\nabla _\xi \xi=0, &\quad \nabla _{\widetilde{X}}\xi=\frac{\ell^2}{2}(X\lrcorner d\theta)^\sharp, \\ \nabla _\xi \widetilde{Y}=\frac{\ell^2}{2}(Y\lrcorner d\theta)^\sharp,& \quad \nabla _{\widetilde{X}} \widetilde{Y} = -\frac{1}{2}d\theta (X,Y)\, \xi + \widetilde{\nabla^\tu{C} _X Y},
\end{aligned}
\end{equation}
where $\nabla^\tu{C}$ denotes the Levi-Civita connection of the cone metric $g_\tu{C}$. Using \eqref{eq:Levi:Civita} we can then calculate the leading terms of the Riemannian curvature, Ricci tensor and scalar curvature of $g_\BC$ in terms of curvature of $g_C$.
\begin{prop}\label{prop:Curvature:Model}
The Riemannian curvature $\textup{Rm}_{g_\BC}$ of $g_\BC$ satisfies
\[
\begin{gathered}
\langle \,\textup{Rm}_{g_\BC}(\widetilde{X},\xi)\,\widetilde{Y},\xi\,\rangle = O(r^{-4}), \qquad \langle\, \textup{Rm}_{g_\BC}\,(\widetilde{X},\widetilde{Y})\widetilde{Z},\xi\,\rangle = O(r^{-3}),\\
\langle\, \textup{Rm}_{g_\BC}(\widetilde{X},\widetilde{Y})\,\widetilde{Z},\widetilde{W}\,\rangle =  \langle\, \textup{Rm}_{g_\tu{C}} (X,Y)\,Z,W\,\rangle + O(r^{-4}).
\end{gathered}
\]
In particular, the Ricci tensor $\Ric_{g_\BC}$ of $g_\BC$ satisfies
\[
\Ric_{g_\BC}(\xi,\xi)=O(r^{-4}), \qquad \Ric_{g_\BC}(\widetilde{X},\xi)=O(r^{-3}) , \qquad \Ric_{g_\BC}(\widetilde{X},\widetilde{Y})= \Ric_{g_\tu{C}}(X,Y) + O(r^{-4}),
\]
and the scalar curvature $\textup{Scal}(g_\BC)$ satisfies $\textup{Scal}(g_\BC) = \tu{Scal}(g_\tu{C})+O(r^{-4})$.
\proof
Formulae for the Riemannian curvature, Ricci tensor and scalar curvature of a Riemannian submersion are given in \cite[\S 9.D]{Besse}. These formulae involve terms controlled by $|d\theta|^2$ and $|\nabla ^\tu{C} d\theta|$. Since $d\theta$ is a $2$-form on $\Sigma$ radially extended to $\tu{C}(\Sigma)$ and pulled-back to $\BC (\Sigma)$, we have $|d\theta|^2=O(r^{-4})$ and $|\nabla ^\tu{C} d\theta|=O(r^{-3})$ with respect to the metric $g_\BC$.
\endproof
\end{prop}

\begin{remark*}
The term of order $O(r^{-3})$ in the expression for $\Ric(\widetilde{X}, \xi)$ only involves $d^\ast (d\theta)$ on $\tu{C}(\Sigma)$. If $\theta$ is a Yang--Mills connection on the cone $\tu{C}(\Sigma)$ we can therefore replace $O(r^{-3})$ with $O(r^{-4})$ in the formula for the Ricci tensor.
\end{remark*}

\begin{remark}\label{rmk:Curvature:Cone}
Since $\tu{C}(\Sigma)$ is a Riemannian cone with smooth cross-section, we have $| \tu{Rm}_{g_\tu{C}}| \leq C r^{-2}$ for a constant $C$ that depends only on the Riemannian curvature of the metric $g_\Sigma$ on $\Sigma$.
\end{remark}

In order to capture the leading-order behaviour of natural differential operators associated with a metric asymptotic to $g_\BC$ it is more convenient to introduce a new \emph{metric} connection which preserves the decomposition into vertical and horizontal subspaces.

\begin{definition}
\label{def:adapted}
The \emph{adapted connection} of $\BC (\Sigma)$ is the metric connection $\overline{\nabla}$  characterised by
\begin{equation}\label{eq:Adapted:connection}
\overline{\nabla}_\xi \xi=0, \qquad \overline{\nabla}_{\widetilde{X}}\xi=0, \qquad \overline{\nabla}_\xi \widetilde{Y}=0, \qquad \overline{\nabla}_{\widetilde{X}}\widetilde{Y}=\widetilde{\nabla^\tu{C} _X Y}.
\end{equation}
\end{definition}
\begin{remark}\label{rmK:Difference:connections}
Denote by $S$ the difference operator $\nabla -\overline{\nabla}$. Comparing \eqref{eq:Levi:Civita} and \eqref{eq:Adapted:connection} and using the fact that $|d\theta|=O(r^{-2})$ (with analogous estimates on the derivatives) we deduce that $r^j |\overline{\nabla}^j S| \leq C r^{-2}$ for all $j\geq 0$. 
\end{remark}

\begin{remark*}
Note that the torsion $\overline{T}$ of $\overline{\nabla}$ is $\overline{T}(U,V)=d\theta (\pi_\ast U,\pi_\ast V)\, \xi$.
\end{remark*}

\begin{remark}
\label{rmk:motivate_adapted}
In order to motivate the introduction of the adapted connection $\overline{\nabla}$ consider the first-order operator $d+d^\ast$ acting on differential forms on $\BC (\Sigma)$. Define an operator $d_{\overline{\nabla}} + d_{\overline{\nabla}}^\ast$ by
\begin{equation}\label{eq:Adapted:d+dstar}
d_{\overline{\nabla}} + d_{\overline{\nabla}}^\ast =\sum_{i=1}^{n+1}{e_i\wedge \overline{\nabla}_{e_i}} -\sum_{i=1}^{n+1}{e_i\lrcorner\overline{\nabla}}_{e_i}
\end{equation}
where $e_1,\dots,e_{n+1}$ is an orthonormal frame for $g_\BC$. One finds (see Lemma \ref{lem:comparison:Dirac:operators} below for the proof for a general class of Dirac-type operators) that
\begin{equation}\label{eq:comparison:d+dast}
d+d^\ast = d_{\overline{\nabla}} + d_{\overline{\nabla}}^\ast + Q,
\end{equation}
where $Q$ is a zeroth-order operator whose coefficients depend on $d\theta$. Since $|d\theta|=O(r^{-2})$, for large $r$ the operator $Q$ should be regarded as a lower-order perturbation.
\end{remark}

\section{Dirac-type operators on ALC manifolds}\label{sec:Dirac}

The operator $d+d^\ast$ of the previous remark is an example of a Dirac-type operator. The main goal of the first part of the paper is to establish a Fredholm theory for a natural class of Dirac-type operators on ALC manifolds. In this preliminary section we recall basic facts about the spin geometry of principal circle bundles and then define a class of Dirac-type operators on ALC manifolds that is well adapted to the asymptotic geometry.
\subsection{Spin geometry of principal circle bundles}

Dirac-type operators on Riemannian submersions have been the subject of research by many authors. In the special case of principal circle bundles, the approach of Ammann--B\"ar \cite[\S 4]{Ammann:Bar}, which we follow closely, allows one to avoid some of the more technical aspects of the general theory; in particular, one can avoid the language of superconnections.

In order to simplify the notation, in this section we work on an arbitrary (with structure group $\tu{O}(2)$) circle bundle $\pi\co M^{n+1}\ra B^n$ endowed with a Riemannian metric $g$ that makes $\pi$ into a Riemannian submersion over the oriented Riemannian manifold $(B,g_B)$ with totally geodesic fibres of length $2\pi\ell$. The reader should keep in mind that $M=\BC (\Sigma)$ and $B=\tu{C}(\Sigma)$ in our applications.

Let $P_{\tu{SO}}(M)$ be the bundle of oriented orthonormal frames of $(M,g)$. 
Because the vertical tangent bundle of $M$ is a pull-back of a real line bundle $R \to B$, a choice of connection $\theta^2$ on $M$ identifies $P_{\tu{SO}}(M)$ with the pull-back of an $\tu{SO}(n+1)$-bundle $\overline P_\tu{SO}(B)$ from $B$. Indeed, $R$ must be $\det TB$ because we assume $M$ to be oriented, and $\overline P_\tu{SO}(B)$ is associated to the $\tu{O}(n)$-frame bundle of $B$ by the embedding $(\textup{Id} \times \det) \co \tu{O}(n) \to \tu{O}(n) \times \Z/2 \to SO(n+1)$.

Now consider a spin structure on $TB+\det TB$, defined by a $\tu{Spin}(n+1)$-bundle $\overline P_{\tu{Spin}}(B)$ with a double cover map to $\overline P_{\tu{SO}}(B)$. (In other words, $\overline P_{\tu{Spin}}(B)$ is a $\tu{Pin}^-$--structure on $B$. If $B$ is itself oriented, so that $\det TB$ is trivial, then this is equivalent to a spin structure on $B$ itself.) Then $P_{\tu{Spin}}(M):=\pi^\ast \overline P_{\tu{Spin}}(B)$
defines a spin structure on $M$, which is locally $S^1$--invariant. Call such spin structures \emph{projectable}. (Not all spin structures on $M$ must be lifts of $\tu{Pin}^-$--structure on $B$ when the second Stiefel--Whitney class of the principal $\tu{O}(2)$--bundle associated with $M$ is trivial, since then $H^1(M;\Z_2)$ is larger than $H^1(B;\Z_2)$.)

\begin{remark*}
In our application to $\gtwo$--manifolds, since a locally $S^1$--invariant $\gtwo$--structure on the total space of a circle bundle induces an $(\sunitary{3}\rtimes \Z_2)$--structure on the base and $\sunitary{3}\rtimes \Z_2$ is a subgroup of $\tu{Pin}^-(6)$, we will always be considering projectable spin structures. Here $\Z_2$ is generated by complex conjugation.
\end{remark*}

We now want to understand Dirac-type operators on $M$ and relate them to Dirac-type operators on the base $B$. First, we identify the relevant bundles on $M$ with pull-backs of ones on $B$.

\begin{definition}\label{def:Clifford:Module}
Let $G$ be either $\tu{SO}(n)$ or $\tu{Spin}(n)$. A complex finite-dimensional representation $V$ of $G$ is called a \emph{$G$--Clifford module} if $V$ carries
\begin{enumerate}
\item a Hermitian metric preserved by $G$;
\item a Clifford multiplication $\gamma$ preserved by $G$, \ie a $G$--equivariant map $\gamma\co \R^{n}\ra \tu{End}(V)$ such that $\gamma (v)^\ast = -\gamma (v)$ and $\gamma (v)\,\gamma (w)+\gamma (w)\,\gamma (v)=-2\langle v,w\rangle\, \text{id}$.
\end{enumerate}
\end{definition}

The basic example of an $\tu{SO}(n)$--Clifford module is $V=\Lambda^\bullet(\R^{n})^\ast \otimes \C$ with Clifford multiplication $\gamma (v)\omega = v^\flat \wedge \omega - v\lrcorner\,\omega$. The basic example of a $\tu{Spin}(n)$--Clifford module is given by a complex representation of the Clifford algebra $Cl(\R^n)$ and therefore also of the complexified Clifford algebra $\C l (\R^n)$. Since $\tu{Spin}(n)\subset Cl(\R^n)$, $V$ is also a representation of $\tu{Spin}(n)$. The complexified Clifford algebra $\C l(\R^n)$ has one (two) irreducible complex representation(s) if $n$ is even (odd). When $n$ is odd, the two irreducible representations are equivalent as $\tu{Spin}(n)$--representations and we denote either of the two by $\slashed{S}(\R^n)$; when $n$ is even, the irreducible representation of $\C l(\R^n)$ splits as two inequivalent $\tu{Spin}(n)$--representations $\slashed{S}^{\pm}(\R^n)$ and we set $\slashed{S}(\R^n)=\slashed{S}^+(\R^n)\oplus\slashed{S}^-(\R^n)$. 

\medskip

Let $P_G (M)$ be a \emph{projectable} $G$--structure on $M$ and $\overline P_{G}$ the corresponding $G$--structure on $B$. Given a $G$--Clifford module $V$ we construct vector bundles $V_M = P_G(M)\times_G V$ and $V_B=P_{G}(B)\times_{G}V$ on $M$ and $B$ respectively. Since $P_G (M)=\pi^\ast \overline P_{G}(B)\times_{G}G$, we have $V_M \simeq \pi^\ast V_B$.

\begin{remark*}
Assume $n$ is even, $B$ is spin and $V=\slashed{S}(\R^{n+1})$. Then $V_M=\slashed{S}(M)$ and $V_B=\slashed{S}(B)$ are the spinor bundles of $M$ and $B$ respectively.
\end{remark*}

In particular, the bundle $V_M$ is trivial on each fibre of $\pi\co M\ra B$. Then on each fibre we have the $L^2$--orthogonal projection $\Pi_0$ from the space of sections of $V$ over the fibre onto the finite-dimensional space of constant $V$--valued functions. Set $\Pi_\perp = \text{id}-\Pi_0$. Extend $\Pi_0$ and $\Pi_\perp$ to the space of all sections of $V_M$ which are $L^2$--integrable when restricted to almost every fibre.

Finally, we allow for the possibility of twisting by an additional Hermitian vector bundle. Let $E\ra B$ be a rank $m$ Hermitian vector bundle. Then we consider $V_B \otimes E$ and $V_M \otimes\pi^\ast E=\pi^\ast (V_B \otimes E)$.

Note that the connection $\nabla$ on $V_M$ induced by the Levi-Civita connection of $M$ is compatible with the Clifford multiplication $\gamma$ on $V_M$, in the sense that $\nabla_X \left( \gamma(v)\psi\right) = \gamma (\nabla_X v)\psi + \gamma (v)\nabla_X\psi$. If $E$ is endowed with a Hermitian connection $\nabla_E$ we can then define the \emph{Dirac-type operator} $D$ acting on sections of $V_M \otimes \pi^\ast E$ by
\begin{equation}\label{eq:Dirac:type:operator}
D\psi = \sum_{i=1}^{n+1}{\gamma(e_i)\nabla_{e_i}\psi},
\end{equation}
where $\{ e_1,\dots, e_{n+1}\}$ is an orthonormal frame of $TM$. Here, by abuse of notation, $\nabla$ is the connection on $V_M\otimes\pi^\ast E$ obtained from the Levi-Civita connection on $M$ and $\pi^\ast\nabla_E$. For example, if $V=\Lambda^\bullet (\R^{n+1})^\ast\otimes\C$ then $D = d+d^\ast$. If $V=\slashed{S}(\R^{n+1})$ then $D$ is the Dirac operator of $M$.

An important tool is the Weitzenb\"ock formula for the Dirac Laplacian
\begin{equation}\label{eq:Weitzenbock}
D^2\psi = \nabla ^\ast \nabla \psi + \mathcal{R}\cdot\psi.
\end{equation}
Here the curvature $\mathcal{R}$ of $\nabla$ acts on $\psi$ by
\[
\mathcal{R}\cdot\psi = \sum_{i<j}{\gamma (e_i)\gamma (e_j)\mathcal{R}(e_i,e_j)_\ast \psi}
\]
where $\mathcal{R}(e_i,e_j)_\ast\psi$ is the representation of $\tu{Lie}(G)\oplus \mathfrak{u}(m)$ on the tensor product of $V$ with the fibre~$\C^m$ of $E$. Formula \eqref{eq:Weitzenbock} follows from a direct calculation.

Now, besides the Levi-Civita connection on $M$ we also have the adapted connection $\overline{\nabla}$ defined in \eqref{eq:Adapted:connection}. This adapted connection and $\nabla_E$ then define a new connection on $V_M$, still denoted by~$\overline{\nabla}$, and we have an associated Dirac-type operator $\overline{D}$ defined as in \eqref{eq:Dirac:type:operator}. (If $B$ is non-oriented, this involves a local choice of square root $\theta \in \Omega^1(M)$ for the connection $\theta^2$, yielding a locally defined dual vector field $\xi$.)

\begin{lemma}\label{lem:comparison:Dirac:operators}
For every smooth section $\psi$ of $V_M\otimes \pi^\ast E$ we have
\[
D\psi = \overline{D}\psi -\tfrac{1}{4}\gamma(\xi)\,\gamma (d\theta)\,\psi.
\]
Here Clifford multiplication by the $2$-form $d\theta$ is defined by
\[
\gamma (d\theta)\,\psi = \tfrac{1}{2}\sum_{i,j}{d\theta (e_i,e_j)\,\gamma (e_i)\,\gamma(e_j)\,\psi}.
\]
Furthermore we have the Weitzenb\"ock formula
\begin{equation}\label{eq:Weitzenbock:Adapted}
\overline{D}^2\psi = \overline{\nabla}^\ast\overline{\nabla}\psi + \overline{\mathcal{R}}\cdot\psi - \gamma (d\theta)\overline{\nabla}_\xi\psi,
\end{equation}
where $\overline{\mathcal{R}}$ is the curvature of $\overline{\nabla}$.
\proof
Over a small open set $U$ in $B$ we consider an orthonormal frame of $\pi^{-1}(U)$ obtained by adding the vector $\tilde{e}_{n+1} = \ell^{-1}\xi$ to the horizontal lift $\{ \tilde{e}_1,\dots, \tilde{e}_n\}$ of an orthonormal frame $\{ e_1,\dots, e_{n}\}$ of $B$. Then we have \cite[Theorem 4.14]{Lawson:spin}
\[
D\psi = \overline{D}\psi + \tfrac{1}{4}\sum_{i=1}^n{\gamma (\tilde{e}_i)\,\gamma (S_{\tilde{e}_i})\,\psi} +\tfrac{1}{4}\ell^{-2}\gamma (\xi)\,\gamma (S_\xi)\,\psi,
\]
where, using \eqref{eq:Levi:Civita} and \eqref{eq:Adapted:connection},
\[
\begin{gathered}
\gamma (S_{\tilde{e}_i}) = -\tfrac{1}{2}\sum_{j=1}^n{ d\theta (e_i,e_j)\, \gamma (\tilde{e}_j)\, \gamma (\xi)} + \tfrac{1}{2}\sum_{j=1}^n{d\theta (e_i,e_j)\, \gamma (\xi)\, \gamma (\tilde{e}_j)} = \sum_{j=1}^n{d\theta (e_i,e_j)\, \gamma (\xi)\, \gamma (\tilde{e}_j)},\\
\gamma (S_{\xi}) = \frac{\ell^2}{2}\sum_{i,j=1}^n{d\theta (e_i,e_j)\, \gamma (\tilde{e}_i)\, \gamma (\tilde{e}_j)}.
\end{gathered}
\]
Here $S=\nabla - \overline{\nabla}$. The first part of the Lemma follows immediately from these formulas using simple algebraic manipulations using the properties of the Clifford multiplication $\gamma$.

The Weitzenb\"ock formula \eqref{eq:Weitzenbock:Adapted} is obtained by a direct computation using an adapted frame $\{ \tilde{e}_1, \dots, \tilde{e}_{n+1}\}$ as above. Observe that $[\xi,\tilde{e}_i]=0$ and that we can further assume that $[\tilde{e}_i,\tilde{e}_j]=-d\theta(e_i,e_j)\,\xi$ over a fixed point $p\in U$ by requiring that $[e_i,e_j]_p=0$. With respect to the classical Weitzenb\"ock formula \eqref{eq:Weitzenbock}, the additional term in \eqref{eq:Weitzenbock:Adapted} involving $d\theta$ arises from the non-vanishing commutator $[\tilde{e}_i,\tilde{e}_j]$, \ie the non-vanishing torsion of $\overline{\nabla}$. 
\endproof
\end{lemma}

\begin{remark*}
Formula \eqref{eq:comparison:d+dast} is a special case of the Lemma with $V=\Lambda^\bullet(\R^{n+1})^\ast\otimes\C$.
\end{remark*}

Using adapted orthonormal frames as in the proof of Lemma \ref{lem:comparison:Dirac:operators}, \ie frames $\{ e_1,\dots, e_{n+1}\}$ on~$M$ with $e_{n+1}=\ell^{-1}\xi$ and the first $n$ vectors obtained by horizontal lift of an orthonormal frame of~$B$, we can further decompose
\begin{equation}\label{eq:Decomposition:adapted:Dirac}
\overline{D} = D_B + \ell^{-2}\gamma (\xi) \overline{\nabla}_\xi,
\end{equation}
where $D_B$ is defined by $D_B \psi = \sum_{i=1}^n{\gamma (e_i)\, \overline{\nabla}_{e_i} \psi}$ for an orthonormal frame $\{ e_1, \dots, e_n\}$ of $B$.

\begin{remark*}
The operator $D_B$ can be regarded as a Dirac-type operator on the infinite-dimensional vector bundle over $B$ whose fibre over a point $b$ is the space of $L^2$--sections of $V_M\otimes E|_{\pi^{-1}(b)}$, \ie the space of $V\times E_b$--valued $L^2$ functions on the circle $\pi^{-1}(b)$. In fact, by Fourier decomposition along the fibres one can decompose this infinite-dimensional bundle into a sum of finite-dimensional ones.

Assume that $M$ and $B$ are compact for simplicity (or restrict to a slice $r=\text{const}$ in the case $M=\BC (\Sigma)$ and then extend radially). Assume also that $B$ is oriented, so that $M$ is a principal $S^1$-bundle. Then we have an $S^1$--action on $L^2 (M;V_M\otimes \pi^\ast E)$ and therefore a decomposition into irreducible $\tu{U}(1)$--representations:
\[
L^2(M;V_M\otimes \pi^\ast E)=\bigoplus_{k\in\Z}V_k,
\]
where $\psi\in V_k$ if and only if $\mathcal{L}_\xi\psi = ik\psi$. By definition, every element $\psi\in V_k$ corresponds to a section $\check{\psi}$ of $V_B \otimes E \otimes L^{-k}$ over $B$, where $L$ is the complex line bundle on $B$ associated to the principal $\tu{U}(1)$--bundle $\pi\co M\ra B$. Moreover, using the fact that $\mathcal{L}_\xi=\overline{\nabla}_\xi$, this identification intertwines connections, \ie for every $\psi\in V_k$ the element $\overline{\nabla}_{\tilde{e}_i}\psi \in V_k$ corresponds to the section $\widecheck{\nabla}_{e_i}\check{\psi}$ of $V_B\otimes E\otimes L^{-k}$, where $\widecheck{\nabla}$ is the connection on $V_B\otimes E\otimes L^{-k}$ induced by the Levi-Civita connection of $B$, $\nabla_E$ and $\theta$. We conclude that $D_B$ acts on $V_k$ as the twisted Dirac-type operator on $V_B \otimes E \otimes L^{-k}$, \cf \cite[Lemma 4.4]{Ammann:Bar}. In the following we will not use this full Fourier decomposition along the fibres but only the projections $\Pi_0$ and $\Pi_\perp$
\end{remark*}

\begin{remark}
\label{rmk:nonproj}
    The spinor bundle of a \emph{non}-projectable spin structure on $M$ cannot be naturally identified with the pull-back of a bundle on $B$ (the complexified spinor bundle is always isomorphic to a bundle associated to a spin$^c$ structure on $B$, but not naturally so). In this case, there are no $S^1$-invariant sections over a fibre; the $\overline \nabla$-parallel transport around a fibre has monodromy $-1$.

    One can still consider a Fourier decomposition for a section of the spinor bundle, but in this case the modes are \emph{half}-integers, \cf Ammann--B\"ar \cite[Theorem 4.5]{Ammann:Bar}. However, interpreting $\Pi_0$ as projection to sections that are $\overline \nabla$-constant along the fibres, we should simply set $\Pi_0 = 0$ since there are no such sections.
\end{remark}

\subsection{Admissible Dirac-type operators on ALC manifolds}

We now specialise the discussion above to the case of ALC manifolds and define the class of Dirac-type operators that we will study in the following sections.

Let $(M^{n+1},g)$ be a complete oriented ALC manifold in the sense of Definition \ref{def:ALC:n}. Then there exists a compact set $K$  with a diffeomorphism $f$ from the open set $\{r\geq R\}$ in $\BC (\Sigma)=\BC (\Sigma,\pi,\theta^2,\ell)$ to $M \setminus K$ such that $f^\ast g$ is asymptotic to $g_\BC$ with rate $\tau<0$ in the sense of Definition \ref{def:ALC:n}. Let $G$ be a subgroup of $\tu{SO}(n+1)$ or $\tu{Spin}(n+1)$ and assume that $M$ carries a $G$--structure $P_G$. Then we get an induced $G$--structure $\overline{P}_G = f^\ast P_G$ on~$\BC (\Sigma)$. Similarly, if $E\ra M$ is a Hermitian vector bundle, we obtain an induced vector bundle $\overline{E}=f^\ast E$ over $\BC (\Sigma)$.

\begin{definition}\label{def:Admissible:bundle}
We say that the Hermitian vector bundle $E\ra M$ is \emph{admissible} if $\overline{E}$ is identified with the pull-back of a Hermitian vector bundle $\widecheck{E}$ on the cone $\tu{C}(\Sigma)$.
\end{definition}

If $B$ is oriented, this is equivalent to a choice of lift of the $S^1$--action on $\BC (\Sigma)$ to $\overline{E}$.
Since $\tu{C}(\Sigma)=\R^+\times\Sigma$ we can think of $\widecheck{E}$ as a vector bundle on $\Sigma \simeq \{ r=1\} \subset \tu{C}(\Sigma)$ extended radially to the cone.

\begin{remark*}
Every bundle associated to the orthonormal frame bundle of $M$ is admissible. In the spin case, a bundle associated to the spin structure on $M$ is admissible if the induced spin structure on $\BC (\Sigma)$ is projectable.
\end{remark*}

Suppose now that the admissible bundle $E$ carries a connection $\nabla_E$ that preserves the Hermitian structure.

\begin{definition}\label{def:Admissible:connection}
We say that the connection $\nabla_E$ is \emph{admissible} if there exists $\tau<0$ such that up to gauge transformations $f^\ast \nabla_E = \nabla_{\overline{E}} + O(r^{\tau-1})$ where $\nabla_{\overline{E}}$ is in \emph{radial gauge}, \ie it is the pull-back to $\BC (\Sigma)$ of the radial extension to $\tu{C}(\Sigma)$ of a connection on $\widecheck{E}|_\Sigma$. Here $f^\ast\nabla_E-\nabla_{\overline{E}}=O(r^{\tau-1})$ means that the difference of the two connections is a $1$--form $a$ with values in the bundle of skew-Hermitian endomorphisms of $E$ such that
\[
|\nabla_{\overline{E}}^j a| \leq C r^{\tau-1-j}.
\]
\end{definition}

\begin{remark*}
If $E$ is a bundle of (complexified) tensors on $M$ then the Levi-Civita connection on $E$ is admissible by \eqref{eq:Levi:Civita} and the relation between the Levi-Civita connections of the cone $\tu{C}(\Sigma)$ and of~$\Sigma$, \cf for example \cite[Lemma 2.2]{Chou}.
\end{remark*}

Now, for $G=\tu{SO}(n+1)$ or $G=\tu{Spin}(n+1)$, let $V$ be a $G$--Clifford module and denote by $V_M$ and $\overline{V}=f^\ast V_M$ the associated vector bundles on $M$ and $\BC (\Sigma)$ respectively. From the Levi--Civita connection of $M$ and an admissible connection $\nabla_E$ we construct a connection $\nabla$ on $V_M\otimes E$ and a corresponding Dirac-type operator $D$. Similarly, from the adapted connection~$\overline{\nabla}$ of \eqref{eq:Adapted:connection} and $\nabla_{\overline{E}}$ we obtain a connection $\overline{\nabla}$ and a Dirac-type operator $\overline{D}$ on $\overline{V}\otimes\overline{E}$. By Lemma \ref{lem:comparison:Dirac:operators} and the definition of an admissible connection we have
\begin{equation}\label{eq:comparison:Dirac}
D = \overline{D} + Q,
\end{equation}
where $Q$ is a zeroth-order operator with $|\overline{\nabla}^jQ| = O(r^{-\min \{|\tau|,1\}-1-j})$. Here we used Remark \ref{rmK:Difference:connections} to estimate the difference between the Levi-Civita connection of $\BC (\Sigma)$ and the adapted connection \eqref{eq:Adapted:connection} and we suppressed from the notation the fact that we are using $\phi$ to think of $D$ as an operator on $\overline{V}\otimes \overline{E}$. We will refer to $D$ as an \emph{admissible} Dirac-type operator \emph{asymptotic} to $\overline{D}$.

Finally, pulling back bundles to $\BC (\Sigma)$ using $f$, outside the compact set $K$ we have projections $\Pi_0$ and $\Pi_\perp$ into locally $S^1$--invariant and ``oscillatory'' sections of $\overline{V}\otimes \overline{E}=f^\ast (V_M\otimes E)$. Note that $[\overline{D},\Pi_0]=0$ and therefore we have that
\begin{equation}\label{eq:Asymptotic:Dirac:decomposition}
\overline{D} = \Pi_0\, \overline{D}\,\Pi_0 + \Pi_\perp\, \overline{D}\,\Pi_\perp.
\end{equation}
The operator $\Pi_0\, \overline{D}\,\Pi_0$ can be identified with a Dirac-type operator $D_{\tu{C}}$ on the cone $\tu{C}(\Sigma)$. By Definition \ref{def:Admissible:connection} the operator $D_{\tu{C}}$ is ``adapted'' to the conical geometry of $\tu{C}(\Sigma)$: after conformally rescaling the cone to a cylinder, $D_{\tu{C}}$ is a translation-invariant operator of the type studied by Lockhart--McOwen \cite{Lockhart:McOwen}. 

\section{Weighted Banach spaces}
\label{sec:weighted:spaces}
The aim of this section is to introduce weighted Sobolev and H\"older spaces adapted to the asymptotic geometry of an ALC manifold. These are the most natural spaces on which to study mapping properties of the admissible Dirac-type operators introduced in the previous section.

\subsection{Set-up}\label{sec:Set:up}

Let $(M^{n+1},g)$ be a complete ALC manifold in the sense of Definition \ref{def:ALC:n}. Let $f$ be the identification of the complement of a compact set $K\subset M$ with its asymptotic model $\BC (\Sigma)=\BC (\Sigma,\pi,\theta^2,\ell)$. In the following we will use $f$ to identify %
$M\setminus K$ and all the structures it carries with an exterior region $\{ r \geq R\}$ in $\BC (\Sigma)$. We will often need to assume that $R$ is large enough and therefore reserve the freedom to enlarge $K$ as needed.

Using the map $f$ we push-forward the radial function $r$ on $\BC (\Sigma)$ to $M \setminus K$. We extend $r$ to the whole manifold $M$ so that $r\geq 1$ on $M$ and $r\equiv 1$ on a compact subset $K'\subset K$. We call such a function a \emph{radial function} and by abuse of notation still denote it with the symbol $r$. 

Let $P_G$ be a $G$--structure on $M$ with $G\subset \tu{SO}(n+1)$ or $G\subset \tu{Spin}(n+1)$, $V$ a $G$--Clifford module and $V_M = P_G \times_G V$ the associated vector bundle. Given an admissible Hermitian bundle $E$ as in Definition \ref{def:Admissible:bundle}, we consider $V_M\otimes E$ endowed with an admissible connection $\nabla$ in the sense of Definition \ref{def:Admissible:connection}. On the pull-back $\overline{V}\otimes\overline{E}=f^\ast (V_M\otimes E)\ra \BC (\Sigma)$ let $\overline{\nabla}$ be the Hermitian connection obtained from the adapted connection \eqref{eq:Adapted:connection} and the limit $\nabla_{\overline{E}}$ of $f^\ast\nabla_E$. Recall that $f^\ast\nabla - \overline{\nabla}=O(r^{\tau-1})$ for some $\tau<0$ by Definition \ref{def:Admissible:connection}. To ease the notation we will often drop explicit reference to the identification $f$ between $M\setminus K$ and $\BC(\Sigma)$: thus $(V_M\otimes E,\nabla)$ will denote both the bundle with connection on $M$ as well as its pull-back to $\BC (\Sigma)$. $D$ and $\overline{D}$ will denote the Dirac-type operators associated with the connections $\nabla$ and $\overline{\nabla}$ respectively. Recall that $D$ and $\overline{D}$ are related by \eqref{eq:comparison:Dirac}.

Finally, we define projections $\Pi_0$ and $\Pi_\perp = \text{id}-\Pi_0$ onto the space of fibrewise constant sections of $\overline{V}\otimes\overline{E}$. Using the map $f$, we can think of $\Pi_0$ and $\Pi_\perp$ as defined on sections of $V_M\otimes E$ outside the compact set $K$.

\subsection{Definition of weighted norms}

In the set-up of the previous section we now define weighted Sobolev and H\"older spaces of sections $\psi$ of $V_M \otimes E$.

\begin{definition}\label{def:Weighted:Sobolev}
Given $\nu\in\R$, $p\geq 1$ and $l\in\Z_{\geq 0}$ define the weighted Sobolev $L^p_{l,\nu}$--norm of a section $\psi$ of $V_M\otimes E$ by
\begin{equation}\label{eq:Weighted:Sobolev}
\| \psi \| ^p_{L^p_{l,\nu}} = \| \psi \| ^p_{L^p_l(K)}+ \sum_{j=0}^{l}{ \int_{M\setminus K}{\left( |r^{-\nu+j}\nabla^j \Pi_0\psi |^p + |r^{-\nu+l}\nabla^j \Pi_\perp\psi |^p \right) r^{-n}\dvol_M}}.
\end{equation}
The weighted Sobolev space $L^p_{l,\nu}(M;V_M \otimes E)$ is defined as the closure of the space of smooth compactly supported sections of $V_M \otimes E$ on $M$ with respect to the norm \eqref{eq:Weighted:Sobolev}. For ease of notation we will often write $L^p_{l,\nu}$ for $L^p_{l,\nu}(M;V_M \otimes E)$ and $L^p_\nu$ for $L^p_{0,\nu}$.
\end{definition} 

\begin{remark}
\label{remark:lp:unweighted}
If $|\psi|=O(r^a)$ then $\psi \in L^p_{\nu}$ for all $\nu >a$. Moreover, $L^p_{\nu} \subseteq L^p$ if and only if $\nu \leq -\tfrac{n}{p}$.
\end{remark}

\begin{definition}\label{def:Weighted:Holder}
Given $\nu\in\R$, $k\in\N_0$ and $\alpha\in (0,1)$ define the weighted H\"older $C^{k,\alpha}_\nu$--norm of a section $\psi$ of $V_M\otimes E$ by
\begin{equation}\label{eq:Weighted:Holder}
\| u\|_{C^{k,\alpha}_\nu} = \sum_{j=0}^k{ \left( \|r^{-\nu+j}\nabla^j \Pi_0 u\|_{C^0} + \|r^{-\nu+k}\nabla^j \Pi_\perp u\|_{C^0}\right) } + [r^{-\nu+k+\alpha}\nabla^k u]_\alpha.
\end{equation}
Here $[r^{-\nu+k+\alpha}\nabla^k u]_\alpha$ is the H\"older seminorm defined using parallel transport of $\nabla$ to identify fibres of $V_M\otimes E$ along minimising geodesics in a small neighbourhood of each point in $M$. The weighted H\"older space $C^{k,\alpha}_\nu(M;V\otimes E)$ is defined as the closure of the space of smooth compactly supported sections of $V_M \otimes E$ on $M$ with respect to the norm \eqref{eq:Weighted:Holder}. For ease of notation we will often write $C^{k,\alpha}_{\nu}$ for $C^{k,\alpha}_{\nu}(M;V_M \otimes E)$. By dropping the H\"older seminorm $[r^{-\nu+k}\nabla^k u]_\alpha$ in the definition of the $C^{k,\alpha}_\nu$--norm, we obtain the definition of the space of sections of $V_M\otimes E$ of class $C^k_\nu$. We also set $C^\infty_\nu = \bigcap_{k\geq 0}{C^k_\nu}$.
\end{definition}

Note that definitions \ref{def:Weighted:Sobolev} and \ref{def:Weighted:Holder} impose a stronger decay on the oscillatory part $\Pi_\perp\psi$ of a spinor field $\psi$ than they do on its locally $S^1$--invariant component $\Pi_0\psi$. In fact, the following lemma shows that this is automatic.

\begin{lemma}\label{lem:Control:Oscillatory}
Assume that $\psi \in L^p_{\nu}$ satisfies $\Pi_0\psi=0$ and $\nabla\psi\in L^p_{\nu-1}$. Then $\psi \in L^p_{\nu-1}$ and there exist a constant $C>0$ and a compact set $K\subset M$ independent of $\psi$ such that
\[
\| \psi|_{M\setminus K} \| _{L^p_{\nu-1}} \leq C \| \nabla\psi \|_{L^p_{\nu-1}}.
\]
Similarly, if $\Pi_0\psi=0$ and $\nabla\psi\in C^{0,\alpha}_{\nu-1}$ for some $\alpha\in (0,1)$ then
\[
\| \psi|_{M\setminus K} \| _{C^{0}_{\nu-1}} \leq C \| \nabla\psi \|_{C^{0,\alpha}_{\nu-1}}.
\]
\proof
We work on the region $\{ r\geq R\}$ in $\BC (\SigmP)$ for $R>0$ sufficiently large. Since $\Pi_0\psi=0$, we have
\begin{equation}\label{eq:Control:Oscillatory}
\int_{\pi^{-1}(x)}{|\psi|^p} \leq C\int_{\pi^{-1}(x)}{|\overline{\nabla}_\xi\psi|^p}
\end{equation}
for some uniform constant $C$. Here $\pi\co\BC (\Sigma)\ra \tu{C}(\Sigma)$ is the bundle projection and $\xi$ is the locally defined vertical unit vector field. Integration over the whole region $\{ r\geq R\}$ then yields the estimate
\[
\| \psi|_{\{ r\geq R\}} \|_{L^p_{\nu-1}} \leq C \| \overline{\nabla}\psi|_{\{ r\geq R\}} \| _{L^p_{\nu-1}}.
\]
By Remark \ref{rmK:Difference:connections} and Definition \ref{def:Admissible:connection} we can replace $\overline{\nabla}$ with $\nabla$ for $R$ sufficiently large.

The proof in the setting of H\"older spaces is analogous.
\endproof
\end{lemma}

\begin{remark}\label{rmk:Weighted:Sobolev}
In view of Lemma \ref{lem:Control:Oscillatory}, since $\nabla - \overline{\nabla}=O(r^{-\min\{|\tau|,1\}-1})$ and $\overline{\nabla}$ commutes with $\Pi_0$, the norms defined by
\[
\| \psi \| _{L^p_l(K)}+ \sum_{j=0}^{l}{ \| \nabla^j\psi\|_{L^p_{l-j,\nu-j}}}, \qquad \| u\|_{C^{k,\alpha}_\nu} = \sum_{j=0}^k{\|r^{-\nu+j}\nabla^j u\|_{C^0}} + [r^{-\nu+k+\alpha}\nabla^k u]_\alpha
\]
define equivalent norms to the weighted Sobolev and H\"older norms in \eqref{eq:Weighted:Sobolev} and \eqref{eq:Weighted:Holder} respectively. While the definition in \eqref{eq:Weighted:Sobolev} and \eqref{eq:Weighted:Holder} emphasise the actual decay of each component, the formulas above are sometimes more convenient when deriving estimates.  
\end{remark}

\begin{remark}\label{rmk:Scaling:Technique}
Below we will extend basic estimates (continuous embeddings, elliptic estimates) from the standard compact setting to weighted estimates on ALC manifolds. The following observations will be crucial to derive these results. Write $M$ as the union of a compact set $K$ together with an exterior domain $\{ r\geq R\}$ in $\BC (\Sigma)$. The weighted norms of a spinor field $\psi$ restricted to $K$ coincide with the standard Sobolev/H\"older norms on $K$. Consider then the restriction of $\psi$ to $M\setminus K$.

Assume first that $\Pi_\perp\psi=0$: we can work on the cone $\tu{C}(\Sigma)$ and apply the scaling trick of Bartnik \cite[Theorem 1.2]{Bartnik}. Decompose the region $\{ r\geq R \}$ in $\tu{C}(\Sigma)$ into the union of annuli ${ \{ 2^k R \leq r \leq 2^{k+1}R\} }$. Up to a factor of $(2^k R)^{-\nu}$, on each annulus the weighted norms are equivalent (with constants independent of $R$ and $k$) to the standard Sobolev/H\"older norms on the rescaled annulus. Weighted estimates can the be obtained by applying standard estimates on these rescaled annuli, rescaling back and summing/taking supremums over $k\in\Z_{\geq 0}$.

If $\Pi_\perp\psi \neq 0$, we can modify this scaling technique as follows. Each annulus $\{ 2^k R \leq r \leq 2^{k+1}R\}$ can be covered with the same finite number of open subsets with the property that $\BC (\Sigma)$ restricts to a trivial circle bundle over each of them. The fact that the number of open subsets is independent of the radius of the annulus follows from the fact that $\BC (\Sigma)$ (and all the structure it carries) is the radial extension of a circle bundle over $\Sigma$. The same scaling technique can then be applied replacing annuli with $(2^k R)^{-1}$--covers of the restriction of $\BC (\Sigma)$ to such subsets.
\end{remark}

\subsection{Basic properties}

We now state a few basic facts about weighted Sobolev and H\"older spaces that can be deduced easily from Definitions \ref{def:Weighted:Sobolev} and \ref{def:Weighted:Holder}.

\begin{lemma}\label{lem:Weighted:Embedding}
Fix $p,q\geq 1$, $h,k\in\mathbb{N}_0$, $\alpha,\beta\in (0,1)$ and $\nu, \nu'\in\R$. There are continuous embeddings
\begin{enumerate}
\item $L^p_\nu\subset L^q_{\nu'}$ whenever $q\leq p$ and $\nu'>\nu$;
\item $C^{k,\alpha}_\nu\subset C^{h,\beta}_{\nu'}$ whenever $h+\beta \leq k+\alpha$ and $\nu'\geq \nu$;
\item $C^{0,\alpha}_\nu\subset L^p_{\nu'}$ whenever $\nu'>\nu$.
\end{enumerate}
 \end{lemma}

Nonlinearities later in the paper will be controlled thanks to the following weighted H\"older inequality, whose proof is immediate. We state it here for product of functions.

\begin{lemma}\label{lem:Weighted:Holder}
Fix $p_1,p_2\geq 1$ and $\nu_1,\nu_2\in\R$ and define $p$ and $\nu$ by $\tfrac{1}{p_1}+\tfrac{1}{p_2}=\tfrac{1}{p}$ and $\nu_1 +\nu_2 =\nu$. Then for every $u\in L^{p_1}_{\nu_1}$ and $v\in L^{p_2}_{\nu_2}$ we have
\[
\| u v \| _{L^p_\nu} \leq \| u \| _{L^{p_1}_{\nu_1}} \| v \| _{L^{p_2}_{\nu_2}}.
\]
Similarly, given $k\geq 0$ and $\alpha\in (0,1)$, the product $C^{k,\alpha}_{\nu_1}\times C^{k,\alpha}_{\nu_2}\ra C^{k,\alpha}_\nu$ is continuous.
\end{lemma}

In the rest of the paper we will mainly use weighted H\"older and $L^2$ spaces. We state here two results that will be used repeatedly throughout the paper.

\begin{lemma}\label{lem:Dual:weighted:Sobolev}
The $L^2$--inner product induces an isomorphism between $(L^2_{\nu})^\ast$ and $L^2_{-\nu-n}$. 
\end{lemma}

\begin{lemma}\label{lem:integration:parts}
Let $\psi_1\in L^2_{1,\nu_1}$ and $\psi_2 \in L^2_{1,\nu_2}$ with $\nu_1 + \nu_2 +n-1 \leq 0$. Then
\[
\langle D\psi_1, \psi_2 \rangle_{L^2} = \langle \psi_1, D\psi_2\rangle_{L^2}.
\]
\end{lemma}
\begin{proof}
Note that for $\nu_1+\nu_2\leq -n+1$ the $L^2$--inner products in the statement are finite.

Choose a sequence $\{ \epsilon_i\} \subset \R^+$ with $\epsilon_i\ra 0$ as $i\ra \infty$. Since $\| \psi_1 \| _{L^2_{\nu_1}}$ is finite, there exists a sequence $\{R_i\}$ with $R_i \ra \infty$ such that $\| \psi_1|_{\{ r\geq R_i\}} \| _{L^2_{\nu_1}}\leq \epsilon_i$. Fix cutoff functions $\chi_i$ such that $\chi_i\equiv 1$ for $r\leq R_i$, $\chi_i\equiv 0$ for $r\geq 2R_i$ and $|d\chi_i|\leq CR_i^{-1}$. Then integration by parts shows that

\begin{align*}
\left| \langle \chi_i D\psi_1,\psi_2\rangle_{L^2} - \langle \psi_1, \chi_iD\psi_2\rangle_{L^2}\right| &\leq \int_{r\geq R_i}{r^{\nu_1+\nu_2+n}|d\chi_i|\left( r^{-\nu_1-\frac{n}{2}}|\psi_1| \right) \left( r^{-\nu_2-\frac{n}{2}}|\psi_2| \right) \dvol}\\
 &\leq C R_i^{\nu_1+\nu_2+n-1}\|\psi_2\| _{L^2_{\nu_2}} \|\psi_1\| _{L^2_{\nu_1}(r\geq R_i)} \leq C\epsilon_i \stackrel{i\ra\infty}{\longrightarrow} 0. \qedhere 
\end{align*}
\end{proof}

\begin{remark*}
The Dirac-type operator $D$ can be replaced in this integration-by-parts formula with any other first-order operator and its adjoint.
\end{remark*}

\begin{remark*}
We observe explicitly that Lemmas \ref{lem:Weighted:Embedding} (iii) and \ref{lem:integration:parts} imply that
\[
\langle D\psi_1, \psi_2 \rangle_{L^2} = \langle \psi_1, D\psi_2\rangle_{L^2}
\]
for all $\psi_i\in C^{1,\alpha}_{\nu_i}$, $i=1,2$, with $\nu_1 + \nu_2 +n-1 <0$.
\end{remark*}

The next two results contain deeper estimates that are proved using the scaling technique of Remark \ref{rmk:Scaling:Technique}. The first result is a weighted version of the Sobolev Embedding Theorem. 

\begin{theorem}\label{thm:Sobolev}
Let $(M^{n+1},g)$ be an ALC manifold and fix $\nu\in\R$, $k,l\in\N_0$, $p\geq 1$ and $\alpha\in (0,1)$.

\begin{enumerate}
\item If $lp<n+1$, then for every $\gamma \in [1,\tfrac{n+1}{n+1-lp}]$ there exists a constant $C>0$ such that
\[
\| u \| _{L^{\gamma p}_{k,\nu}} \leq C \| u \| _{L^p_{k+l,\nu}}
\]
for all $u\in C^\infty_0(M)$.
\item If $lp \geq n+1+\alpha$, then there exists a constant $C>0$ such that
\[
\| u \|_{C^{k,\alpha}_\nu} \leq C \| u \| _{L^p_{k+l,\nu}}
\]
for all $u\in C^\infty_0(M)$.
\end{enumerate}
\end{theorem}

\begin{remark*}
By Kato's inequality the theorem is valid for functions as well as for sections of Riemannian and Hermitian vector bundles endowed with a metric connection.
\end{remark*}

Finally, the following theorem contains weighted elliptic estimates for admissible Dirac-type operators. 

\begin{theorem}\label{thm:Weighted:Regularity}
Let $D$ be an admissible Dirac-type operator on a complete ALC manifold $M^{n+1}$. Then for every $l\geq 0$, $p\geq 1$, $\alpha\in (0,1)$ and $\nu\in \R$ there exists $C>0$ such that
\[
\begin{aligned}
\| \psi \|_{L^p_{l+1,\nu+1}} &\leq C\left( \| D\psi\|_{L^p_{l,\nu}} + \| \psi \|_{L^p_{0,\nu+1}}\right) ,\\
\| \psi \|_{C^{l+1,\alpha}_{\nu+1}} &\leq C\left( \| D\psi\|_{C^{l,\alpha}_{\nu}} + \| \psi \|_{C^{0}_{\nu+1}}\right),\\
\| \psi \|_{C^{l+1,\alpha}_{\nu+1}} &\leq C \left( \| D\psi\|_{C^{l,\alpha}_{\nu}} + \| \psi \|_{L^2_{\nu+1}}\right)
\end{aligned}
\]
for all smooth compactly supported spinor fields $\psi$.
\end{theorem}

\section{Fredholm theory for Dirac-type operators on ALC manifolds}
\label{sec:fredholm}

This section contains the main results of the first part of the paper: a complete Fredholm theory in weighted $L^2$ and H\"older spaces for admissible Dirac-type operators on ALC manifolds. We continue working with the notation and set-up of Section \ref{sec:Set:up}. Thus $(M^{n+1},g)$ is an ALC manifold and $D\co V_M \otimes E \ra V_M \otimes E$ is an admissible Dirac-type operator. Since $M$ is noncompact, ellipticity is not sufficient to guarantee that $D$ acting on the weighted Banach spaces of Definitions \ref{def:Weighted:Sobolev} and~\ref{def:Weighted:Holder} is a Fredholm operator. As is well understood, $D$ is Fredholm if and only if it is ``invertible at infinity''. Recall that by \eqref{eq:Asymptotic:Dirac:decomposition} we can decompose $\overline{D}$ as $\Pi_0\,\overline{D}\,\Pi_0 + \Pi_\perp\,\overline{D}\,\Pi_\perp$. In view of the inequality \eqref{eq:Control:Oscillatory}, one expects that $\Pi_\perp\,\overline{D}\,\Pi_\perp$ is always invertible. Thus the Fredholm theory of $D$ is controlled by the mapping properties of the operator $\Pi_0\,\overline{D}\,\Pi_0$. Recall that the latter can be identified with a Dirac-type operator $D_\tu{C}$ on the cone $\tu{C}(\Sigma)$ that is adapted to the conical geometry. The theory developed by Lockhart--McOwen \cite{Lockhart:McOwen} gives in particular a Fredholm theory for this kind of operator.

\subsection{The model operator}

In this preliminary section we study mapping properties of the model operator $\overline{D}$ acting on weighted Sobolev spaces on an exterior region of $\BC (\Sigma)$. We split $\overline{D}=\Pi_0\overline{D}\Pi_0 + \Pi_\perp\overline{D}\Pi_\perp$ and study the two summands separately.

\subsubsection{Dirac-type operators on the cone\texorpdfstring{ $\tu{C}(\Sigma)$}{}}\label{sec:Model:Invariant}

By Definitions \ref{def:Admissible:bundle} and \ref{def:Admissible:connection} the operator $\Pi_0\, \overline{D}\, \Pi_0$ can be identified with a Dirac-type operator $D_\tu{C}$ acting on sections of a bundle $V_\tu{C}\otimes \widecheck{E}$ on an exterior region of the cone $\tu{C}(\Sigma)$. Here $V_\tu{C}=P_G\big( \tu{C}(\Sigma)\big)\times_GV$, where $P_G\big( \tu{C}(\Sigma)\big)$ is the $G$--structure on $T\tu{C}(\Sigma) \oplus \det T\tu{C}(\Sigma)$ that pulls back to the $G$--structure on $\BC (\Sigma)$ and $V$ is a $G$--Clifford module. The Hermitian bundle $V_\tu{C}\otimes \widecheck{E}$ is endowed with a connection $\overline{\nabla}$ in radial gauge.

We will understand the mapping properties of $\Pi_0\, \overline{D}\, \Pi_0$ by separation of variables. Therefore we now relate $D_\tu{C}$ to a Dirac-type operator on $\Sigma$, thought of as the hypersurface $\{ r=1\}$ in $\tu{C}(\Sigma)$ with outward unit normal $\partial_r$. A spinor $\psi$ on $\tu{C}(\Sigma)$ defines by restriction a section of the bundle $V_\tu{C}\otimes \widecheck{E}|_{\Sigma}$.

\begin{remark*}
If $V_{\tu{C}}$ is the bundle of complex differential forms then its restriction to $\Sigma$ is two copies of the bundle of complex differential forms on $\Sigma$, exchanged by Clifford multiplication by $\partial_r$ (\ie $dr\wedge\,\cdot\, + \partial_r\lrcorner\cdot$). Similarly, if $\SigmP$ is spin and $V_\tu{C}$ is the spinor bundle of $\tu{C}(\SigmP)$ then
\[
V_\tu{C}|_{\SigmP} \simeq 
\begin{dcases}
\slashed{S}(\SigmP)=\slashed{S}^+(\SigmP)\oplus\slashed{S}^-(\SigmP) & \mbox{ if } n \mbox{ is odd},\\
\slashed{S}(\SigmP)\oplus\slashed{S}(\SigmP) & \mbox{ if } n \mbox{ is even}.
\end{dcases}
\] 
Clifford multiplication by $\partial_r$ exchanges the two factors in these decompositions.
\end{remark*}
 
The Dirac-type operator $D_{\tu{C}}$ on $\tu{C}(\Sigma)$ induces a \emph{hypersurface Dirac-type operator} $D_\Sigma$ on $V_\tu{C}\otimes \widecheck{E}|_{\Sigma}$ defined by
\begin{equation}\label{eq:Hypersurface:Dirac}
D_\Sigma\psi = \sum_{i=1}^{n-1}{\gamma(\partial_r)\gamma (e_i)\overline{\nabla}^\Sigma_{e_i}\psi} -\frac{n-1}{2}\psi,
\end{equation}
where $\{ e_1,\dots, e_{n-1}\}$ is a orthonormal frame for $(\Sigma,g_\Sigma)$ and $\overline{\nabla}^\Sigma$ is the connection on $V_\tu{C}\otimes \widecheck{E}|_{\Sigma}$ induced from the Levi-Civita connection on $\Sigma$ and the connection $\nabla_{\overline{E}}$ on $\widecheck{E}$. Note that $D_\Sigma$ is an elliptic self-adjoint operator. Using the fact that the second fundamental form of $\Sigma$ in $\tu{C}(\Sigma)$ in the direction of $\partial_r$ is the identity, one has

\begin{align*}
D_\tu{C}\psi = & \gamma (\partial_r)\overline{\nabla}_{\partial_r}\psi -\gamma (\partial _r) \sum_{i=1}^{n-1}{\gamma(\partial_r)\gamma (e_i)\overline{\nabla}_{e_i}\psi}\\
=& \gamma (\partial_r)\overline{\nabla}_{\partial_r}\psi -\gamma (\partial _r) \left( \sum_{i=1}^{n-1}{\gamma(\partial_r)\gamma (e_i)\overline{\nabla}^\Sigma_{e_i}\psi} -\frac{n-1}{2}\psi\right)\\
=& \gamma (\partial_r)\overline{\nabla}_{\partial_r}\psi -\gamma (\partial _r)D_\Sigma\psi.
\end{align*}

Finally, exploiting the conical structure of $\tu{C}(\Sigma)$, every section $\psi$ of $V_\tu{C}\otimes \widecheck{E}|_{\Sigma}$ can be extended to a section of $V_\tu{C}\otimes \widecheck{E}$, still denoted by $\psi$, by parallel transport along radial geodesics. Then for any radial function $u(r)$ we have \cite[Proposition 2.5]{Chou}
\begin{equation}\label{eq:Dirac:Cone}
D_{\tu{C}}\big( u(r)\psi \big) = \gamma (\partial_r) \left( u'(r)\psi - \frac{u}{r}D_\Sigma\psi\right),
\end{equation}
where $u'$ denotes differentiation with respect to $r$. In particular, the eigenvalues of $D_\Sigma$ control the growth of harmonic spinors on $\tu{C}(\Sigma)$ since $D_{\tu{C}}(r^\lambda\psi)=0$ if and only if $D_\Sigma\psi=\lambda\psi$.

\begin{definition}\label{def:Indicial:roots}
Let $\mathcal{D}_{D_\tu{C}} \subset \R$ be the (discrete) set of eigenvalues of $D_\Sigma$. We refer to $\mathcal{D}_{D_\tu{C}}$ as the set of \emph{indicial roots} of $D_\tu{C}$.
\end{definition} 
(If $\Sigma$ is non-orientable, this is equivalent to considering eigenvalues of $D_\SigmP$ acting on $\Z_2$--invariant spinors on the oriented double cover $\SigmP$.)

For $\lambda\in\mathcal{D}_{\tu{C}}$ let $d(\lambda)$ be the (finite) dimension of the space of solutions to $D_\tu{C}\psi=0$ of the form $\psi = r^\lambda\sum_{i=1}^k{(\log{r})^i\psi_i}$, where $\psi_i$ are sections of $V_\tu{C}\otimes\widecheck{E}|_{\Sigma}$ extended radially to $\tu{C}(\Sigma)$. Then for $\nu_1 < \nu_2$ in $\R\setminus \mathcal{D}_{D_\tu{C}}$ set
\begin{equation}\label{eq:Jump:Index:Count}
N(\nu_1,\nu_2) = \sum_{\lambda\in \mathcal{D}_{\tu{C}}\cap (\nu_1,\nu_2)}{d(\lambda)}.
\end{equation}

\begin{remark*}
Exploiting \eqref{eq:Dirac:Cone}, one can check that in fact there are no elements in the kernel of $D_\tu{C}$ which are nontrivial polynomials in $\log{r}$. In particular $d(\lambda)$ is the dimension of the eigenspace of~$D_\Sigma$ of eigenvalue $\lambda$. For higher-order operators though $\log{r}$ terms can indeed appear.
\end{remark*}

Separation of variables using Fourier decomposition along the eigenspaces of $D_\Sigma$ allows one to understand completely the mapping properties of the operator $D_\tu{C}$ in weighted Sobolev spaces.

\begin{prop}\label{prop:Estimate:Cone}
Fix $R_0>0$ and $\nu\notin \mathcal{D}_{D_\tu{C}}$ (and $0<\alpha<1$). For every $\phi\in L^2_{\nu-1}$ ($\phi\in C^{0,\alpha}_{\nu-1}$) there exists a solution $\psi\in L^2_{1,\nu}$ ($\psi\in C^{1,\alpha}_{\nu}$) of the inhomogeneous Dirac equation $D_\tu{C}\psi=\phi$ in the exterior region $\{ r\geq R_0\}$. Moreover, there exists a constant $C>0$ independent of $\psi$ and $\phi$ such that
\[
\| \psi \|_{L^2_{1,\nu}} \leq C \| \phi \| _{L^2_{\nu-1}}, \qquad \| \psi \|_{C^{1,\alpha}_{\nu}} \leq C \| \phi \| _{C^{0,\alpha}_{\nu-1}}.
\]
\proof
Let $\{ \psi_j\}_{j\in\Z}$ be an $L^2$--orthonormal basis of eigenspinors for $D_\Sigma$ with $D_\Sigma\psi_i=\lambda_i\psi_i$. Write $-\gamma (\partial_r)\phi = \sum_{i\in\Z}{b_i(r)\, \psi_i}$ and look for a solution $\psi$ of the form $\psi=\sum_{i\in\Z}{a_i(r)\,\psi_i}$. Exploiting \eqref{eq:Dirac:Cone}, one can check that 
\[
a_i (r) = \begin{dcases}
 r^{\lambda_i}\int_{R_0}^{r}{s^{-\lambda_i}b_i (s)ds} & \mbox{if }\lambda_i <\nu,\\
 -r^{\lambda_i}\int_{r}^{\infty}{s^{-\lambda_i}b_i (s)ds} & \mbox{if }\lambda_i >\nu.
\end{dcases}
\] 
yields a solution to $D_\tu{C}\psi=\phi$. It is equally straightforward to verify that 
\[
\| \psi \| _{L^2_{0,\nu}} \leq C \| \phi\| _{L^2_{\nu-1}}, \qquad \| \psi\|_{C^0_\nu}\leq C  \| \phi\| _{C^{0,\alpha}_{\nu-1}}
\]
for a uniform constant $C$ that only depends on (a lower bound for) $R_0$ and the distance of $\nu$ from~$\mathcal{D}_{D_\tu{C}}$. Theorem \ref{thm:Weighted:Regularity} then concludes the proof.
\endproof
\end{prop}

\begin{remark}\label{rmk:Boundary:Conditions:Cone}
Note that the solution $\psi$ to $D_\tu{C}\psi=\phi$ on $\{r \geq R_0\}$ given by Proposition \ref{prop:Estimate:Cone} has the property that $\psi|_{\{ r=R_0\}}$ lies in the sum of eigenspaces of $D_\Sigma$ of eigenvalues $\lambda>\nu$.
\end{remark}

\begin{remark}
    We could carry out some of the above theory for some non-admissible operators $\overline D$, \eg the Dirac operator of an ALC spin manifold such that the spin structure on the end is non-projectable. As indicated in Remark \ref{rmk:nonproj}, one would then set $\Pi_0 = 0$, so in \ref{eq:Asymptotic:Dirac:decomposition} there would be only the ``oscillatory'' term, and no $S^1$-invariant term $D_C = \Pi_0 \overline D \Pi_0$. Therefore in this case there are no indicial roots, and the argument below leads to the conclusion that $\overline D$ is Fredholm for any rate $\nu$.
\end{remark}

\subsubsection{The model operator acting on $\ker{\Pi_0}$}\label{sec:Model:Oscillatory}

We now move on to study the mapping properties of the model operator $\overline{D}$ acting on spinors $\psi$ satisfying $\Pi_0\psi=0$. The key ingredients in this case are Lemma \ref{lem:Control:Oscillatory} and the Weitzenb\"ock formula \eqref{eq:Weitzenbock}.

Fix $r\geq R_0$ and work on the exterior domain $\{ r \geq R_0\}$ in $\BC (\Sigma)$. The crucial consequence of the Weitzenb\"ock formula \eqref{eq:Weitzenbock:Adapted} is the following integration-by-parts formula, \cf for example
\linebreak \cite[Equation~(2.6)]{Bartnik:Chrusciel},
\begin{equation}\label{eq:Bochner}
\int_{r\geq R_0}{|\overline{\nabla}\psi|^2 + \langle \overline{\mathcal{R}}\cdot\psi,\psi\rangle - \langle \gamma(d\theta)\overline{\nabla}_\xi\psi,\psi\rangle } = \int_{r\geq R_0}{|\overline{D}\psi|^2} + \int_{r=R_0}{\langle \psi, \overline{D}_\partial\,\psi \rangle}.
\end{equation}
Here $\overline{\mathcal{R}}$ is the curvature of the connection $\overline{\nabla}$ and $\overline{D}_\partial$ is the boundary Dirac-type operator analogous to the hypersurface Dirac-type operator $D_\Sigma$ of \eqref{eq:Hypersurface:Dirac}, $\theta$ is the locally defined connection 1-form and $\xi$ its dual. $\overline{D}_\partial$ is a self-adjoint elliptic operator.

The following proposition contains the crucial estimate.

\begin{prop}\label{prop:Estimate:Oscillatory}
For every $\nu\in\R$ there exists $R_0>0$ and $C>0$ such that the following holds. Assume that $\psi \in L^2_{1,\nu}(r\geq R_0)$ satisfies $\Pi_0\psi=0$ and $\langle\psi, \overline{D}_\partial\, \psi \rangle_{L^2(r=R_0)} \leq 0$. Then
\[
\| \psi \| _{L^2_{1,\nu}(r\geq R_0)} \leq C \|\overline{D}\psi\|_{L^2_{\nu-1}(r\geq R_0)}.
\]
\proof
We apply \eqref{eq:Bochner} to $r^{-\nu +1 -\frac{n}{2}}\psi$. The boundary term is non-positive and therefore
\[
\| \overline{\nabla} (r^{-\nu +1 -\frac{n}{2}}\psi)\|^2_{L^2} \leq \| \overline{D}(r^{-\nu +1 -\frac{n}{2}}\psi)\|^2_{L^2} + \| \overline{\mathcal{R}}\|_{L^\infty(r\geq R_0)} \| \psi\|^2_{L^2_{\nu-1}} + \|d\theta\|_{L^\infty(r\geq R_0)} \| \overline{\nabla}_\xi\psi\|_{L^2_{\nu-1}} \| \psi\|_{L^2_{\nu-1}}.
\]
Note that the $L^2$--norms are well defined since $\psi\in L^2_{1,\nu}$ and $\Pi_0\psi=0$.

Now, using the existence of a uniform constant $C>0$ such that $\overline{\nabla}r \leq C$, $\| \overline{\mathcal{R}}\|_{L^\infty(r\geq R_0)} \leq Cr^{-2}$ (by Proposition \ref{prop:Curvature:Model} and Remark \ref{rmk:Curvature:Cone}) and $\| d\theta\|_{L^\infty(r\geq R_0)} \leq Cr^{-2}$ (as in the proof of Proposition~\ref{prop:Curvature:Model}), we deduce
\[
\| \overline{\nabla}\psi \|^2_{L^2_{0,\nu-1}} \leq C\left( \| \overline{D}\psi\|^2_{L^2_{0,\nu-1}} + R_0^{-2} \left( \|\overline{\nabla}_\xi\psi\| _{L^2_{0,\nu-1}} + \|\psi \|_{L^2_{0,\nu-1}}\right) \|\psi \|_{L^2_{0,\nu-1}}\right).
\]
In view of Lemma \ref{lem:Control:Oscillatory} we now choose $R_0$ large enough so that the last terms can be absorbed in the left-hand-side of the inequality.
\endproof
\end{prop}

\begin{remark}\label{rmk:APS}
The requirement that $\langle \psi, \overline{D}_\partial\, \psi\rangle_{L^2(r=R_0)} \leq 0$ is the Atiyah--Patodi--Singer boundary condition \cite{APS} for the Dirac-type operator $\overline{D}$.
\end{remark}

\subsection{The Fredholm property}

Propositions \ref{prop:Estimate:Cone} and \ref{prop:Estimate:Oscillatory} are the key results needed to establish a Fredholm theory for the admissible Dirac-type operator $D$. In order to extend these estimates for the model operator~$\overline{D}$ on an exterior region in $\BC (\Sigma)$ to estimates for the Dirac-type operator $D$ on the whole ALC manifold $M$ we use cutoff functions and perturbation arguments.

\begin{theorem}\label{thm:Fredholm}
Fix $\nu \in \R \setminus \mathcal{D}_{D_\tu{C}}$ and $\alpha\in (0,1)$. Then there exists a compact set $K\subset M$ and a constant $C>0$ such that
\[
\| \psi \| _{L^2_{1,\nu}} \leq C \left( \| D\psi \| _{L^2_{\nu-1}} + \| \psi \| _{L^2 (K)}\right), \qquad \| \psi \| _{C^{1,\alpha}_{\nu}} \leq C \left( \| D\psi \| _{C^{0,\alpha}_{\nu-1}} + \| \psi \| _{L^2 (K)}\right)
\]
for all smooth compactly supported spinor field $\psi$. In particular, $D\co L^2_{1,\nu}\ra L^2_{\nu-1}$ and $D\co C^{1,\alpha}_\nu\ra C^{0,\alpha}_{\nu-1}$ are Fredholm operators.
\proof
The fact that the Fredholm property follows from the claimed estimate (for $D$ and $D^\ast= D$) is standard, \cf for example \cite[\S 2]{Lockhart:McOwen}. We therefore only explain how to use Propositions \ref{prop:Estimate:Cone} and \ref{prop:Estimate:Oscillatory} to prove the estimates.

Consider first the weighted $L^2$ estimate. Let $K'\subset K$ be compact subsets of $M$ such that $M \setminus K'$ is identified with the exterior region $\{r \geq R_0\}$ in $\BC (\Sigma)$. Here $R_0$ is sufficiently large so that Proposition \ref{prop:Estimate:Oscillatory} holds. We also assume that under the identification of $M \setminus K$ with an exterior region in $\BC (\Sigma)$, $K\setminus K'$ is identified with the region $\{ R_0 \leq r \leq R_0 +1\}$. Let $\chi, 1-\chi$ be a partition of unity subordinate to the cover $M = \left(M\setminus K'\right) \cup \text{int}(K)$, where $\text{int}(K)$ denotes the interior of $K$.

Propositions \ref{prop:Estimate:Cone} and \ref{prop:Estimate:Oscillatory} imply that
\[
\| \chi \psi \| _{L^2_{1,\nu}} \leq C \| \overline{D}(\chi\psi) \|_{L^2_{0,\nu-1}}.
\]
Since $D-\overline{D}=O(r^{\tau-1})$ for some $\tau<0$ (and similarly $\nabla - \overline{\nabla}=O(r^{\tau-1})$), by taking $K',K$ (and therefore $R_0$) larger if necessary, we can replace $\overline{D}$ with $D$ (and $\overline{\nabla}$ with $\nabla$). Combining the resulting estimate with standard elliptic regularity for $(1-\chi)\psi$ in the compact set $K$ yields the estimate of the theorem.

The proof of the weighted H\"older estimate is similar. Fix an auxiliary rate $\nu'>\nu$ such that $\nu'<\nu+1$ and $[\nu,\nu']$ does not contain any indicial root. By Lemma \ref{lem:Weighted:Embedding} (iii) $\|D\psi\|_{L^2_{\nu'-1}}\leq C \| D\psi\|_{C^{0,\alpha}_{\nu-1}}$. Hence the weighted $L^2$ estimate already established shows that
\[
\| \psi \| _{L^2_{1,\nu'}} \leq C \left( \| D\psi \| _{C^{0,\alpha}_{\nu-1}} + \| \psi \| _{L^2 (K)}\right).
\]
We want to apply the regularity estimate of Theorem \ref{thm:Weighted:Regularity}, but for this we need to improve from rate $\nu'$ to rate $\nu$. Consider $\chi\psi$, where $\chi$ is the cutoff function defined earlier. Now, on one side Proposition \ref{prop:Estimate:Cone} shows that $\|\Pi_0 (\chi\psi)\|_{C^{0,\alpha}_\nu}$ is controlled by the $C^{0,\alpha}_{\nu-1}$--norm of its image under $D$. On the other side, in view of Lemma \ref{lem:Control:Oscillatory} we have $\| \Pi_\perp (\chi\psi)\|_{L^2_\nu} \leq \| \Pi_\perp (\chi\psi)\|_{L^2_{\nu'-1}} \leq \|\chi\psi\|_{L^2_{\nu'}}$. Combining these observations with the third estimate in Theorem \ref{thm:Weighted:Regularity} concludes the proof. 
\endproof
\end{theorem}

\subsection{The index jump formula}

For $\nu \notin \mathcal{D}_{D_\tu{C}}$ let $i(\nu)$ denote the index of the Fredholm operator $D\co L^2_{1,\nu}\ra L^2_{\nu-1}$. It is not difficult to prove that $i(\nu)$ also coincides with the index of the Fredholm operator $D\co C^{1,\alpha}_\nu\ra C^{0,\alpha}_{\nu-1}$. The goal of this final part of the section is to prove the following index jump formula.

\begin{theorem}\label{thm:Index:jump}
For every $\nu,\nu'\notin \mathcal{D}_{D_\tu{C}}$ with $\nu < \nu'$ we have
\[
i(\nu')-i(\nu)=N(\nu,\nu'),
\]
where $N(\nu, \nu')$ is the number defined in \eqref{eq:Jump:Index:Count}.
\end{theorem}

It is enough to consider the case where $\nu = \nu_\ast -\epsilon$ and $\nu'=\nu_\ast + \epsilon$, where $\epsilon>0$ is chosen small enough so that $[\nu_\ast-\epsilon, \nu_\ast+\epsilon]$ does not contain any indicial root other than possibly $\nu_\ast$. 
The proof is modelled on \cite[Chapter 11]{Pacard}. We isolate two initial steps in the following two lemmas, which are as important as the theorem itself.

\begin{lemma}\label{Step1}
For every $\nu\notin\mathcal{D}_{D_\tu{C}}$ there exists a compact set $K\subset M$ and a constant $C>0$ such that for every $\phi\in L^2_{\nu-1}$ there exists a solution $\psi\in L^2_{1,\nu}$ to the inhomogeneous Dirac equation $D\psi=\phi$ on $M \setminus K$ and satisfying
\[
\| \psi \| _{L^2_{1,\nu}(M \setminus K)} \leq C \| \phi \| _{L^2_{\nu-1}}.
\]
\proof
For $K$ sufficiently large we can identify $M\setminus K$ with an exterior region $\{ r\geq R_0 \}$ in $\BC (\Sigma)$. The statement of the Lemma with $\overline{D}$ in place of $D$ follows immediately from Propositions \ref{prop:Estimate:Cone} and~\ref{prop:Estimate:Oscillatory}. Indeed, we decompose $\phi$ as $\Pi_0 \phi+ \Pi_\perp \phi$ and accordingly look for $\Pi_0 \psi$ and $\Pi_\perp \psi$ separately. Proposition \ref{prop:Estimate:Cone} guarantees the existence of $\Pi_0 \psi$ with the required estimate. Proposition \ref{prop:Estimate:Oscillatory} shows that the elliptic boundary value problem $D\Pi_\perp \psi = \Pi_\perp \phi$ with the Atiyah--Patodi--Singer boundary condition of Remark \ref{rmk:APS} is Fredholm with trivial kernel and cokernel (since the adjoint problem is the analogous boundary value problem with $-\nu-n+1$ in place of $\nu$). Thus $\overline{D}\co L^2_{1,\nu}(M \setminus K)\ra L^2_{\nu-1}(M \setminus K)$ has a bounded right inverse $\overline{G}$.

Since $D-\overline{D}=O(r^{\tau-1})$ for some $\tau<0$, for $R_0$ sufficiently large we can define a right inverse $G$ for $D$ by $G=\overline{G}(D\overline{G})^{-1}$. The inverse of $D\overline{G}$ is well defined since $D\overline{G}-\text{id}=(D-\overline{D})\overline{G}$ has arbitrarily small norm for $r\geq R_0$ with $R_0$ large. 
\endproof
\end{lemma}

Now let $\psi_1,\dots,\psi_N$ be a basis of the $N=d(\nu_*) = N(\nu,\nu')$--dimensional space of solutions to $D_{\tu{C}}\psi=0$ of the form $\psi = r^{\nu_\ast}\sum_{i=1}^k{(\log{r})^i\psi'_i}$ (where $\psi'_i$ are sections of $V_\tu{C}\otimes\widecheck{E}|_{\Sigma}$ as before \eqref{eq:Jump:Index:Count}). If $\nu_\ast\notin \mathcal{D}_{D_\tu{C}}$ then $N=0$. Using a cutoff function $\chi$ supported in the region $\{ r> R_0\}$ and satisfying $\chi\equiv 1$ on $\{ r\geq 2R_0\}$, we regard $\chi\,\psi_1,\dots, \chi\,\psi_N$, initially defined on the cone $\tu{C}(\SigmP)$, as spinors on $M\setminus K$. Since $\chi\,\psi_i\in L^2_{1,\nu'}$ and $\overline{D}\psi_i=0$, we have $D(\chi\,\psi_i) = (D-\overline{D})(\chi\,\psi_i)+\gamma (\overline{\nabla}\chi)\,\psi_i\in L^2_{\nu-1}$, provided $\epsilon = \nu_\ast-\nu=\nu'-\nu_\ast$ is chosen sufficiently small. We can then use Lemma \ref{Step1} to deform $\chi\,\psi_i$ to a solution $\widetilde{\psi}_i$ of $D\psi=0$ on $M\setminus K$ with
\[
\| \widetilde{\psi}_i -\chi\,\psi_i \| _{L^2_{1,\nu}} \leq C. 
\]

\begin{lemma}\label{Step3}
For every $\phi\in L^2_{\nu-1}$ satisfying $\phi=D\psi$ for some $\psi\in L^2_{1,\nu'}$ there exist $a=(a_1,\dots,a_N)\in\R^N$ and $\psi'\in L^2_{1,\nu}$ such that $\psi=\psi'+\chi\left( \sum_{i=1}^N{a_i\,\widetilde{\psi}_i}\right)$. Here $\chi$ is the cutoff function defined in the previous paragraph. Moreover, there exists a constant $C>0$ independent of $\phi,\psi,\psi'$ and $a$ such that
\[
\| \psi' \|_{L^2_{1,\nu}} + \|a\| \leq C\left( \| \phi \| _{L^2_{\nu-1}} + \| \psi \|_{L^2_{1,\nu'}}\right).
\]
\proof
As usual we identify $M \setminus K$ with an exterior region in $\BC (\SigmP)$.

Let $\chi'$ be a smooth function supported in $\{ r> R_0\}$ with $\chi'\equiv 1$ on $\{ r\geq R_0+1\}$. Then $\overline{D}(\chi' \psi)\in L^2_{\nu-1}$ and we can define $\psi_0 = \overline{G}\,\overline{D}(\chi' \psi)\in L^2_{1,\nu}$, where $\overline{G}$ is the right inverse of $\overline{D}$ constructed in the proof of Lemma \ref{Step1}. In particular, $\| \psi_0\|_{L^2_{1,\nu}}$ is controlled by $\| \phi \|_{L^2_{\nu-1}}$ and the $L^2$--norm of $\psi$ on the compact set $\{ R_0\leq r \leq 2R_0\}$. By Remarks \ref{rmk:Boundary:Conditions:Cone} and \ref{rmk:APS}, comparison of the boundary values on the hypersurface $\{ r=R_0\}$ implies that $\Pi_\perp (\chi' \psi) = \Pi_\perp \psi_0$, while there exists $a=(a_1,\dots,a_N)\in\R^N$ such that $\Pi_0 (\chi' \psi)=\Pi_0 \psi_0 + \sum_{i=1}^N{a_i\psi_i}$. Moreover, $\| a\|$ is controlled by $\| \phi \| _{L^2_{0,\nu-1}} + \| \psi \|_{L^2_{1,\nu'}}$. Thus $\psi=\psi_0 + \sum_{i=1}^N{a_i\psi_i}$ for $r\geq R_0$ and we can write
\begin{align*}
\psi = (1-\chi)\,\psi + \chi\, \psi &= (1-\chi)\, \psi + \chi\, \psi_0 + \chi \left( \sum_{i=1}^N{a_i\psi_i}\right)\\
&= (1-\chi)\, \psi + \chi\, \psi_0 + \chi \left( \sum_{i=1}^N{a_i(\psi_i-\widetilde{\psi}_i)}\right) + \chi \left( \sum_{i=1}^N{a_i\widetilde{\psi}_i}\right).
\end{align*}
Setting $\psi'=(1-\chi)\, \psi + \chi\, \psi_0 + \chi \left( \sum_{i=1}^N{a_i(\psi_i-\widetilde{\psi}_i)}\right)$ yields the result.
\endproof
\end{lemma}

\begin{remark*}
The asymptotic expansions established in this lemma will be used repeatedly in the rest of the paper. In practice, they are often more helpful than Theorem \ref{thm:Index:jump} itself in studying the jump of the kernel and cokernel of differential operators as $\nu$ crosses an indicial root.
\end{remark*}

We now have all the ingredients to conclude the proof of Theorem \ref{thm:Index:jump}.
\begin{proof}[Proof of Theorem \ref{thm:Index:jump}]
For any $\mu\notin \mathcal{D}_{D_\tu{C}}$ let $D_\mu$ denote the Fredholm operator $D_\mu\co L^2_{1,\mu}\ra L^2_{\mu-1}$. Set $W=\textup{Span}\{ \widetilde{\psi}_1,\dots, \widetilde{\psi}_N\}$ and define
\[
D'\co L^2_{1,\nu} \oplus W \ra L^2_{\nu-1}, \qquad D'(u,\psi) = D(u+\chi\psi).
\]
We have $\tu{index}\, D'  = \tu{index}\, D_\nu  + N$ so we have to show that $\tu{index}\, D' = \tu{index}\, D_{\nu'}$.

Now, the inclusion $L^2_{1,\nu}\oplus W\subset L^2_{1,\nu'}$ induces an inclusion $\tu{ker}(D')\subseteq \tu{ker}(D_{\nu'})$. Lemma \ref{Step3} applied to $\phi=0$ shows that $\tu{ker}(D')= \tu{ker}(D_{\nu'})$. On the other hand, the inclusions $L^2_{1,\nu}\oplus W\subset L^2_{1,\nu'}$ and $L^2_{\nu-1}\subset L^2_{\nu'-1}$ induce a map $\tu{coker}(D') \ra \tu{coker}(D_{\nu'})$. Lemma \ref{Step1} with $\nu'$ in place of $\nu$ shows that every class in $\tu{coker}(D_{\nu'})$ can be represented by a compactly supported spinor. In particular, the map $\tu{coker}(D') \ra \tu{coker}(D_{\nu'})$ is surjective. Lemma \ref{Step3} shows that this map is also injective.
\end{proof}

We conclude this section by considering natural extensions of the theory described here.

\subsection{Second-order elliptic operators}

For applications in the next sections we will also need to consider second-order elliptic operators such as the square $D^2$ of an admissible Dirac-type operator $D$ on an ALC manifold, and in particular the Bochner Laplacian acting on differential forms. We state here a theorem analogous to Theorems \ref{thm:Fredholm} and \ref{thm:Index:jump}. We omit the proof since it follows similar lines to the arguments used in the previous sections; for the case of the scalar Laplacian when $\Sigma$ is a round sphere see for example \cite[\S 2.2]{Minerbe:Mass}.

As for first-order operators, the leading-order behaviour at infinity (and therefore the Fredholm property) of $D^2$ is captured by the operator $\overline{D}^2$ on the asymptotic model $\BC (\Sigma)$. The operator $\Pi_0 \overline{D}^2 \Pi_0$ can be identified with an elliptic second-order operator $D_\tu{C}^2$ on the cone $\tu{C}(\Sigma)$. In order to understand the mapping properties of $D_\tu{C}^2$ we use separation of variables. For each $\lambda\in \R$ let $d(\lambda)$ be the dimension of the finite-dimensional vector space consisting of solutions to $D_\tu{C}^2\psi=0$ of the form $\psi = r^\lambda\sum_{i=1}^k{(\log{r})^i\psi_i}$ with $\psi_i$ sections of $V_\tu{C}\otimes\widecheck{E}|_{\Sigma}$ extended radially to $\tu{C}(\Sigma)$. There is a discrete set $\mathcal{D}_{D_\tu{C}^2}\subset\R$ of \emph{indicial roots} such that $d(\lambda)>0$. For $\nu_1 < \nu_2$ in $\R\setminus \mathcal{D}_{D_\tu{C}^2}$ set
\begin{equation}\label{eq:Jump:Index:Count:2nd:Order}
N_{D_\tu{C}^2}(\nu_1,\nu_2) = \sum_{\lambda\in \mathcal{D}_{D_\tu{C}^2}\cap (\nu_1,\nu_2)}{d(\lambda)}.
\end{equation}

\begin{theorem}\label{thm:Fredholm:2nd:order}
The linear elliptic operator $D^2\co L^2_{l,\nu}\ra L^2_{l-2,\nu-2}$ is Fredholm if and only if $\nu \in \R \setminus \mathcal{D}_{D_\tu{C}^2}$. Moreover, if $\nu,\nu'\notin \mathcal{D}_{D_\tu{C}^2}$ with $\nu < \nu'$ then the Fredholm index of $D^2\co L^2_{l,\nu'}\ra L^2_{l-2,\nu'-2}$ is equal to the sum of the Fredholm index of $D^2\co L^2_{l,\nu}\ra L^2_{l-2,\nu-2}$ and $N_{\mathcal{D}_C^2}(\nu,\nu')$.
\end{theorem}

Analogous statements hold for the operator $D^2$ acting on weighted H\"older spaces.

\begin{remark}
In fact Theorem \ref{thm:Fredholm:2nd:order} can be proven more generally also for second-order elliptic operators which are not the squares of a Dirac-type operator. In particular, we will need to use the fact that the statement of Theorem \ref{thm:Fredholm:2nd:order} holds for operators of the form $a\,d^\ast d + b\,dd^\ast$ ($a,b>0$) acting on differential forms on an ALC manifold. In this case the proof also uses similar arguments to the ones we exploited in the proof of Theorems \ref{thm:Fredholm} and \ref{thm:Index:jump}. 
\end{remark}

\subsection{Seifert ALC ends}
\label{subsec:seifert}

Foscolo~\cite{Foscolo:ALC:Spin7} has constructed infinitely many complete noncompact Ricci-flat manifolds (in fact, with exceptional holonomy) where the underlying manifold is the smooth total space of a principal circle orbibundle over an AC orbifold. These examples suggest a natural generalisation of Definition \ref{def:ALC:n} where now one assumes $\tu{BC}(\Sigma)$ to be a circle Seifert fibration over an orbifold cone $\tu{C}(\Sigma)$ with smooth total space. We will now briefly explain how, using the analytic framework introduced in \cite[\S 2]{Foscolo:ALC:Spin7}, the analysis of the previous sections extends naturally to this generalised setting, which we will refer to as \emph{manifolds with Seifert ALC ends}.

Let $(\Sigma, g_\Sigma)$ be a closed $(n-1)$-dimensional orbifold endowed with a smooth orbifold Riemannian metric. Following \cite[Definition 1.2]{Goette:Seifert}, we say that $\pi\co N\ra\Sigma$ is a Seifert circle fibration if $N$ is a smooth manifold with the structure of an orbibundle over $\Sigma$ with fibre $S^1$. We will always assume that $N$ is the associated bundle to a fixed principal orbibundle on $\Sigma$ with structure group $\tu{O}(2)$ and that the smooth manifold $N$ is oriented. Alternatively, we could define a Seifert circle fibration $\pi\co N\ra \Sigma$ to be a smooth manifold endowed with a foliation with 1-dimensional compact leaves and leaf space $\Sigma$. In the special case where the structure group of $\pi\co N\ra\Sigma$ reduces to $\tu{SO}(2)$, corresponding to Seifert ALC ends of cyclic type,
 the leaves of the foliation are the orbits for an $S^1$--action on $N$ and the orbit space $\Sigma$ is then a global quotient orbifold.

In any case, as in \S\ref{subsec:ALC} we let $\theta^2$ denote a connection on $\pi\co N\ra \Sigma$, interpreted as a locally circle-invariant symmetric bilinear form of rank $1$ on $TN$ which is positive definite on the vertical space $\ker\,\pi_\ast$ and gives length $2\pi$ to generic fibres, fix $\ell\in\R_{>0}$ and define $\tu{BC}(\Sigma)$ and its Riemannian metric as in \eqref{eq:FB:N}. Definition \ref{def:ALC:n}, the definition of a smooth (one-ended) ALC manifold, extends immediately to this setting. If we want to emphasize the fact that $\Sigma$ is an orbifold rather than a smooth manifold we will say that the complete manifold $M$ has a \emph{Seifert ALC} end.

The analysis of Dirac-type operators and their Laplacians on $\tu{BC}(\Sigma)$ has an immediate generalisation to the case where $\pi\co N\ra \Sigma$ is a Seifert circle fibration. Indeed, the analysis of Section \ref{sec:Model:Oscillatory} on $\ker\Pi_0$ goes through unchanged, and the only difference is the interpretation of operators and Banach spaces in Section \ref{sec:Model:Invariant}, \ie the analysis on the orbifold $\tu{C}(\Sigma)$ which is now allowed to have a singular set that extends to infinity. As explained in \cite[\S 2]{Foscolo:ALC:Spin7}, a convenient framework to interpret the discussion of Section \ref{sec:Model:Invariant} with only cosmetic changes is provided by the fact that analysis on~$\tu{C}(\Sigma)$ can be replaced by analysis on (locally) basic tensors on the smooth manifold $\tu{BC}(\Sigma)$ using the adapted connection (recall Definition~\ref{def:adapted}). With this change of language, the main results of this section, Theorems \ref{thm:Fredholm}, \ref{thm:Index:jump} and \ref{thm:Fredholm:2nd:order} go through without any modification of the statements or their proofs to the case of manifolds with a Seifert ALC end.

Finally, in view of the applications to weighted $L^2$--cohomology described in the following section, we note that there is a natural extension of many fundamental objects in algebraic and differential topology, such as homotopy and cohomology groups, to orbifolds. These notions can be used to interpret, for example, the cohomology of the orbifold $\Sigma$ and its relation with the cohomology of the total space $N$ of the Seifert fibration $\pi\co N\ra \Sigma$ via the Gysin sequence of the fibration. Here potentially different topological definitions of $H^\bullet(\Sigma)$ pose no subtleties and all coincide since the orders of the local uniformising groups of points in $\Sigma$ are bounded (since $\Sigma$ is compact), the relation with de Rham cohomology goes via the notion of basic cohomology of $N$ and a Hodge theorem for basic cohomology identifies the latter with spaces of $\overline{\nabla}$--covariantly closed and coclosed basic forms. For details, we refer the reader to the summary of the relevant literature given in \cite[\S 2.1.3]{Foscolo:ALC:Spin7}.

\subsection{Fibred-boundary metrics with totally geodesic flat torus fibres}

Finally, one can also consider a generalisation of the ALC asymptotic geometry to ends modelled on the product of a cone $\tu{C}(\Sigma)$ with a flat torus of rank $k$. We briefly explain how the main results of this section, Theorems \ref{thm:Fredholm}, \ref{thm:Index:jump} and \ref{thm:Fredholm:2nd:order}, also generalise to this setting. This generalisation includes the $4$-dimensional hyperk\"ahler 4-manifolds with so-called ALG and ALH asymptotics (but not those with ALG$^\ast$ and ALH$^\ast$ asymptotics) classified by Chen--Chen (in fact \cite[\S4 4]{Chen:Chen:I} and \cite[\S 4]{Chen:Chen:II} contain analysis of Laplacians on such manifolds along the lines we describe here) and also the complete $\tu{Spin}_7$--holonomy metrics recently produced by Cavalleri \cite{Cavalleri}.

Consider the total space $\pi\co N\ra \Sigma$ of a principal rank-$k$ torus bundle over a smooth compact manifold $\Sigma$. More generally, we can consider $N$ to be the smooth total space of a principal torus orbi-bundle over an orbifold $\Sigma$. Fix a principal connection $\theta=(\theta_1,\dots,\theta_k)$ on $\pi\co N\ra \Sigma$, a flat metric $\hat{g}=\hat{g}_{ij}$ on the $k$-torus and a Riemannian metric $g_\Sigma$ on $\Sigma$. We can then endow $N$ with the torus-invariant metric $g_N=g_\Sigma + \sum_{ij}\hat{g}_{ij}\,\theta_i\otimes\theta_j$. Thus $\pi\co (N,g_N)\ra (\Sigma,g_\Sigma)$ is a Riemannian submersion with totally geodesic fibres all isometric to the fixed flat torus $(T^k,\hat{g})$. As in \eqref{eq:FB:N} we can then construct a Riemannian submersion $\tu{BC}(\Sigma)=\tu{BC}(\Sigma,\pi,\theta,\hat{g})$ with the same totally geodesic flat torus fibres and horizontal metric given by the Riemannian cone $\tu{C}(\Sigma)$. We can also allow finite quotients of this model by a finite group $\Gamma$ that acts on the torus fibres via (not necessarily orientation-preserving) isometries.

\begin{remark*}
Let $(T,\hat{g})$ be a flat $k$-dimensional torus; its isometry group is isomorphic to a semidirect product $T\rtimes\Gamma$ with $\Gamma$ a subgroup of $\tu{GL}(k,\Z)\cap O(k)$. Let $\pi\co M\ra B$ be a smooth fibre bundle with fibre a $k$-dimensional torus and assume that the vertical tangent bundle $\ker\pi_\ast$ is endowed with a metric making each fibre of $\pi$ isometric to $(T,\hat{g})$. Then $\pi\co M\ra B$ is associated with a principal $(T\rtimes\Gamma)$--bundle over $B$. Since $\Gamma$ is discrete, by passing to the universal cover of $B$ we can therefore always assume that $\pi\co M\ra B$ is a principal torus bundle, while the map $\pi_1(B)\ra \Gamma$ is the monodromy of the flat bundle arising from the first cohomology of the fibres of $\pi$. In applications to higher-dimensional Ricci-flat manifolds, since the cross-section of the tangent cone at infinity must then have finite (orbifold) fundamental group, by passing to a finite cover we can always reduce to the situation of a principal torus-bundle. If instead the cross-section of the tangent cone at infinity has infinite fundamental group (\eg if the cross-section is a circle, which is relevant in $4$-dimensional hyperk\"ahler geometry) then in general one can only pass to a principal torus bundle by allowing $\Sigma$ to be noncompact. In the 4-dimensional hyperk\"ahler case, however, the ALG metrics we are aiming to generalise here have finite monodromy $\pi_1(S^1)\ra \tu{SL}(2,\Z)$ \cite{Hein,Chen:Chen:II,Chen:Chen:III} and so they also fit into our general scheme.
\end{remark*}

Now, with this generalised definition of $\tu{BC}(\Sigma)$ we can again introduce an adapted connection that captures the leading-order behaviour at infinity of the Levi-Civita connection of the model metric and  allows one to work on $\tu{BC}(\Sigma)$ as if on a product $\tu{C}(\Sigma)\times T^k$ (with the understanding that analysis on $\tu{C}(\Sigma)$ is interpreted as analysis on basic tensors on $\tu{BC}(\Sigma)$ when $\Sigma$ is an orbifold). The definition of an end modelled on $\tu{BC}(\Sigma)$ up to polynomially decaying errors can then be given exactly as in Definition \ref{def:ALC:n}. We will refer to this class of complete metrics with an asymptotic geometry modelled on the Riemannian submersion over a cone with totally geodesic flat torus fibres and polynomially decaying error as \emph{ALC metrics with $k$-dimensional torus fibres}. (Cavalleri \cite{Cavalleri} uses $AT^kC$  to denote this asymptotic geometry.) As observed in the introduction, interesting Ricci-flat metrics with this generalised type of ALC asymptotic geometry also  exist in abundance: for example, it includes 4-dimensional ALG hyperk\"ahler metrics \cite{Cherkis:Kapustin:ALG,Biquard:Minerbe,Hein,Chen:Chen:III} and all the complete $\tu{Spin}_7$--holonomy metrics constructed by Cavalleri \cite{Cavalleri} 
on principal $T^2$-bundles over AC Calabi--Yau $3$-folds. 

Given a Dirac-type operator $D$ satisfying similar admissibility conditions to the ones in Definitions \ref{def:Admissible:bundle} and \ref{def:Admissible:connection}, we can define projections $\Pi_0$ and $\Pi_\perp$ as in the ALC case via Fourier decomposition along the torus fibres. Here we are implicitly using the fact that the vertical 1-forms $\theta_1,\dots,\theta_k$ are parallel with respect to the adapted connection $\overline{\nabla}$ and therefore the frame bundle of $\tu{BC}(\Sigma)$ has a $\overline{\nabla}$--parallel reduction to the frame bundle of $\tu{C}(\Sigma)$. In the spin case, we restrict to the case of projectable spin-structures (this always occurs in the special holonomy case) so there is no difficulty in defining $\Pi_0$ and $\Pi_\perp$. Given these projections we again reduce the analysis of the model operator $\overline{D}$ on $\tu{BC}(\Sigma)$ to studying separately $\Pi_0 \overline{D} \Pi_0$ and $\Pi_\perp \overline{D} \Pi_\perp$. The former coincides with the conical Dirac-type operator studied in Section \ref{sec:Model:Invariant}, while the latter can be studied using the same methods of Section \ref{sec:Model:Oscillatory}. The statements and proofs of Theorems \ref{thm:Fredholm}, \ref{thm:Index:jump} and \ref{thm:Fredholm:2nd:order} therefore extend without difficulty to the more general class of Seifert ALC metrics with $k$-dimensional torus fibres.

\section{Weighted Hodge theory for ALC manifolds}\label{sec:Hodge:ALC}

As a nontrivial application of the Fredholm theory discussed in the previous sections, we now prove Hodge-theoretic results for arbitrary ALC manifolds. More precisely, we describe the space $\mathcal{H}_\lambda^k(M)$ of closed and coclosed $k$-forms on $M$ of decay $\lambda$ for $\lambda$ in ranges where $\mathcal{H}_\lambda^k(M)$ can be described in terms of topological data. The $L^2$--cohomology of ALC manifolds has already been computed in the work of Hausel--Hunsicker--Mazzeo \cite[Corollary 1]{HHM}: here we recover their results with independent proofs 
that avoid the geometric microlocal analytic tools used in their proofs. Moreover, our proof will also apply to the case of ALC ends of Seifert type and we can extend the discussion to weighted $L^2$--cohomology at rates of decay slower than the one corresponding to (unweighted) $L^2$--integrability (recall Remark~\ref{remark:lp:unweighted}).

\subsection{Indicial roots of \texorpdfstring{$\Delta$}{the Laplacian}}
Let $M$ be an ALC manifold of dimension $n+1$ and asymptotic to $\BC(\Sigma)$. When we treat the Laplacian $\triangle_M$ on $k$-forms as in \S\ref{sec:fredholm}, then $\Pi_0\overline{\triangle}\Pi_0$, where $\overline{\triangle}=d_{\overline{\nabla}}d_{\overline{\nabla}}^\ast + d_{\overline{\nabla}}^\ast d_{\overline{\nabla}}$ with $\overline{\nabla}$ the adapted connection of $\tu{BC}(\Sigma)$, is the direct sum of the Laplacian of $\tu{C}(\Sigma)$ on $k$-forms and on $(k-1)$-forms. In particular, the indicial roots of $\triangle_M$ are the indicial roots of those Laplacians on $\tu{C}(\Sigma)$. 
Moreover, the obvious analogue of  Lemma \ref{Step3} for $\Delta$ implies that if $(\lambda_0, \lambda_1)$ contains a unique indicial root $\lambda$ and $\rho \in \Omega^k(M)$ has rate $\lambda_1$ but $\Delta \rho$ has rate $\lambda_0$ (in particular, this certainly includes the case where $\rho$ is itself harmonic), then
\begin{equation}
    \label{eq:harmonic_leading}
\rho = \alpha + \theta \wedge \beta + O(r^{\lambda_0})
\end{equation} 
for harmonic forms $\alpha \in \Omega^k(\tu{C}(\Sigma))$ and $\beta \in \Omega^{k-1}(\tu{C}(\Sigma))$, both polynomials in $\log{r}$ with coefficients in the space of homogeneous forms of rate $\lambda$. A similar discussion holds for the Hodge operator $d+d^\ast$ replacing the Hodge Laplacian on $\tu{C}(\Sigma)$ with the first-order operator $d+d^\ast$: indeed, $\Pi_0 (d_{\overline{\nabla}}+ d^\ast_{\overline{\nabla}})\Pi_0$ is the direct sum of two copies of the operator $d+d^\ast$ on $\tu{C}(\Sigma)$.

Motivated by the previous paragraph we now collect some basic facts about harmonic or closed and coclosed forms homogeneous of order $\lambda$ on a Riemannian cone $\tu{C}(\Sigma)$ of dimension $n$. These results hold without any 
further geometric assumptions about the cross-section $\Sigma$ of the cone, \ie they apply to arbitrary 
ALC manifolds without assuming $\tu{C}(\Sigma)$ to be Ricci flat or to have special holonomy. 
These results can all be deduced from~\cite[Appendix A]{FHN:ALC:G2:from:AC:CY3}. The framework and language described in Section \ref{subsec:seifert} allow one to make sense of these results also in the case of ALC ends of Seifert type.

Let $\tu{C}(\Sigma)$ be a Riemannian cone of dimension $n$ over a smooth compact Riemannian $(n-1)$-manifold $\Sigma$.
Recall that a $k$-form $\gamma$ on $\tu{C}=
\tu{C}(\Sigma)$
is said to be homogeneous of order $\lambda$
if there exist smooth forms $\alpha$ and $\beta$ on $\Sigma$ so that
$\gamma = r^{\lambda+k} (\frac{dr}{r}\wedge \alpha+ \beta)$.

\begin{prop}\label{prop:Indicial:Roots:pforms}
Let $\tu{C}(\Sigma)$ be a Riemannian cone of dimension $n$.
\begin{enumerate}
\item There are no forms in the kernel of $d+d^\ast$ that are non-constant polynomials in $\log{r}$ with coefficients in the space of homogeneous forms of order $\lambda$.
\item If $p<\frac{n}{2}$ there are no closed and coclosed $p$-forms homogeneous of order $\lambda\in (-n+p, -p)
=(-\frac{n}{2}-\frac{n-2p}{2},-\frac{n}{2}+\frac{n-2p}{2}).$ 
Equivalently, if $p>\frac{n}{2}$ then there are no closed and coclosed $p$-forms homogeneous of order $\lambda\in (-p, -n+p)$.
\item If $p<\frac{n}{2}-1$ there are no harmonic $p$-forms homogeneous of order $\lambda\in (-n+p+2, -p)$.
\item If $p\neq \frac{n}{2}$ then every closed and coclosed $p$-form homogeneous of order $\lambda=-p$ is the pull-back of a harmonic $p$-form on $\Sigma$. If $p=\frac{n}{2}$ then every closed and coclosed $p$-form homogeneous of order $\lambda=-p$ is a linear combination of the pull-back of a harmonic $p$-form on $\Sigma$
and the image of the pull-back of a harmonic $p$-form on $\Sigma$ under the Hodge-star operator on $\tu{C}(\Sigma)$.
\item For $p\leq \frac{n}{2}$ all closed and coclosed $p$-forms of rate $\lambda >-p$ are exact and coexact. More precisely, any such $p$-form can be written as $d\left( \frac{r^{\lambda+p}}{\lambda+p}\alpha\right)$ where $\alpha$ is a $(p-1)$--form on $\Sigma$ that is coclosed and an eigenform of the Laplacian with eigenvalue $(\lambda+p)(\lambda+n-p)$.
\item Any form of rate $\lambda = -\frac{n}{2}+1$ in the kernel of $d+d^\ast$ is closed and coclosed.
\end{enumerate}
\end{prop}
\begin{proof}
Statement (i) is \cite[Proposition A.8]{FHN:ALC:G2:from:AC:CY3}. Statements (ii)--(v) are immediate consequences of \cite[Theorem A.2]{FHN:ALC:G2:from:AC:CY3} (which describes all possible harmonic forms homogeneous of order $\lambda$) and \cite[Remark A.5]{FHN:ALC:G2:from:AC:CY3} (which  describes which of these harmonic forms are closed and coclosed). The excluded rates in parts (ii) and (iii) follow from the fact that the Hodge Laplacian on $\Sigma$ on forms of any degree is a non-negative operator. See also \cite[Lemma 2.32]{Foscolo:ALC:Spin7}.

The final statement is not explicitly observed in \cite[Appendix A]{FHN:ALC:G2:from:AC:CY3}. 
To establish it first observe that every pure-degree component of $\gamma$ is a harmonic form of rate $\lambda$. Then by part (i) the assumption that $\gamma$ is in the kernel of $d+d^\ast$ implies that no logarithmic terms appear. Now, if $\lambda = -\frac{n}{2}+1$ then $\lambda +p = -(\lambda + n-p-2)$ and $\lambda +p -2 = -(\lambda + n-p)$ for any~$p$: then in the decomposition of each pure-degree component of $\gamma$ as a sum of forms of type (i), (ii) and (iv) in \cite[Theorem A.2]{FHN:ALC:G2:from:AC:CY3} the component of type (i) necessarily vanishes unless $\lambda=-n+p$ (that is, $p=\frac{n}{2}+1$) and the component of type (iv) vanishes unless $\lambda=-p$ (that is, $p=\frac{n}{2}-1$). (Of course these cases can only occur if $n$ is even.) The statement then follows from \cite[Remark A.5]{FHN:ALC:G2:from:AC:CY3}.
\end{proof}

In the following we will repeatedly use the following results, which are immediate applications of integration by parts and the excluded rates in Proposition \ref{prop:Indicial:Roots:pforms}.

\begin{lemma}\label{lem:Harmonic:to:Closed}
Let $M$ be an ALC manifold of dimension $n+1$.
\begin{enumerate}
\item A harmonic $k$-form $\gamma$ of rate $\lambda$ is necessarily closed and coclosed if either
\begin{enumerate}
\item $\lambda < -\frac{n}{2}+1$, or
\item $k\leq\frac{n}{2}-1$ and $\lambda<-k$.
\end{enumerate}
\item Any form $\gamma$ (possibly of mixed degree) of rate $\lambda\leq -\frac{n}{2}+1$ in the kernel of $d+d^\ast$ is closed and coclosed.
\end{enumerate}
\proof
Consider statement (i). Under hypothesis (a) the statement is a simple consequence of integration by parts Lemma \ref{lem:integration:parts}. 
Hypothesis (b) in the case of equality, \ie with $k=\frac{n}{2}-1$ in fact coincides with (a).
Under hypothesis (b) with strict inequality $k < \frac{n}{2}-1$ we apply Proposition \ref{prop:Indicial:Roots:pforms} (iii) with $p=k$ and $p=k-1$ to deduce that in fact $\gamma$ is of rate $-n+k+2<-\frac{n}{2}+1$ and so we can now apply part (a) to conclude the proof of (i).

Consider now statement (ii). Since each pure-degree component of $\gamma$ is harmonic of rate $\lambda$, 
if $\lambda<-\frac{n}{2}+1$ then we fall under hypothesis (a) of part (i). If $\lambda=-\frac{n}{2}+1$ then 
a priori we should expect that $d\gamma, d^\ast\gamma = O(r^{-\frac{n}{2}})$. 
But in fact, thanks to Proposition \ref{prop:Indicial:Roots:pforms} (vi), 
$d\gamma, d^\ast\gamma = O(r^{-\frac{n}{2}-\delta})$ for some $\delta>0$, and thus integration by parts is justified.
\endproof
\end{lemma}

The indicial roots calculations in Proposition \ref{prop:Indicial:Roots:pforms} focused mostly on closed and coclosed forms. In order to study $\mathcal{H}^k_{\lambda}(M)$ for $k=\lceil\frac{n}{2}\rceil$
we will also need the following result about harmonic forms. 

\begin{prop}\label{prop:Laplacian:BC:MidDegree}
Let $\tu{BC}(\Sigma)$ an ALC model of cyclic type of dimension $n+1$, let $\overline{\nabla}$ denote its adapted connection and set $k=\lceil\frac{n}{2}\rceil$, \ie $k=\frac{n}{2}$ if $n$ is even and $k=\frac{n+1}{2}$ if $n$ is odd. Let $\xi_\infty$ be an $S^1$--invariant $(k-1)$-form on $\tu{BC}(\Sigma)$ in the kernel of $d^\ast_{\overline{\nabla}}d_{\overline{\nabla}}+ d^\ast_{\overline{\nabla}}d_{\overline{\nabla}}$ which can be expressed as a polynomial in $\log{r}$ with coefficients in the space of homogeneous forms of rate $\lambda$.
\begin{enumerate}
\item If $n$ is even and $\lambda=-\frac{n}{2}+1$ then
\[
\xi_\infty = \alpha_1 + \alpha_2\log{r},
\]
where $\alpha_1,\alpha_2\in\mathcal{H}^{k-1}(\Sigma)$.
\item If $n$ is odd and $\lambda=-\frac{n}{2}+\frac{1}{2}+1$ then
\[
\xi_\infty = \theta\wedge\alpha_1 + r\alpha_2 + d(r\alpha_3),
\]
where $\alpha_1\in\mathcal{H}^{k-2}(\Sigma)$, $\alpha_2\in\mathcal{H}^{k-1}(\Sigma)$ and $\alpha_3$ is a coclosed $(k-2)$--form on $\Sigma$ such that $\triangle_\Sigma\alpha_3 = 2\alpha_3$.
\end{enumerate}
\proof
The statements are direct consequences of~\cite[Theorem A.2]{FHN:ALC:G2:from:AC:CY3} and~\cite[Proposition A.6]{FHN:ALC:G2:from:AC:CY3} that describe harmonic $p$-forms on $\tu{C}(\Sigma)$ which can be written as polynomials in $\log{r}$ with coefficients in the space of homogeneous forms of rate $\lambda$ applied with $p=k-1$ and $p=k-2$.
\endproof
\end{prop}

\subsection{Adjusting a closed form to a closed and coclosed representative of its cohomology class}

In order to understand the space of closed and coclosed forms on an ALC manifold $M$ we will need analytic results that allow us to adjust a closed form to a closed and coclosed one while maintaining some control on its decay rate, ideally by adding an exact form.

As above, we assume that we are working on an ALC manifold $M$ of dimension $n+1$.

\begin{prop}\label{prop:Closed:to:Coclosed:Laplacian}
Let $d^\ast\rho\in C^{0,\alpha}_{\lambda-1}$ be a coexact $(k-1)$--form with coprimitive $\rho\in C^{1,\alpha}_{\lambda+\epsilon}$ for $\epsilon>0$ small. If either
\begin{enumerate}
\item $\lambda\in (-\frac{n}{2},-\frac{n}{2}+1)$ and $\langle d^\ast\rho,\xi\rangle_{L^2}=0$ for all $\xi\in\mathcal{H}^{k-1}_{-n+1-\lambda}(M)$, or
\item $k \leq \frac{n}{2}$ and $\lambda\in (-n+k,-k+2)$
\end{enumerate}
then the Poisson equation
\[
\triangle\gamma = -d^\ast\rho
\]
has a solution $\gamma\in C^{2,\alpha}_{\lambda+1}$, and any such solution has $dd^\ast\gamma=0$. 

In particular, if $\rho$ is also assumed to be closed then the $k$-form $\rho+d\gamma$ is closed and coclosed.
\proof
To motivate the proof observe that if $\rho$ is closed then $\rho + d \gamma$ is coclosed precisely when $d^\ast d \gamma= - d^\ast \rho$. 
So if we can solve the Poisson equation $\Delta \gamma =- d^\ast \rho$ and additionally prove that $dd^\ast \gamma=0$ then $\rho + d\gamma$ 
is indeed coclosed as desired. 
So now we study the properties and solvability of this Poisson equation. 

Notice that any solution $\gamma\in C^{2,\alpha}_{\lambda+1}$ of it has the property that $d^\ast\gamma$ is a harmonic $(k-2)$--form of rate $\lambda$. The conditions $\lambda<-\frac{n}{2}+1$ or $k-2 \leq \frac{n}{2}-2 \leq \frac{n}{2}-1$ and $\lambda<-k+2$ guarantee that Lemma \ref{lem:Harmonic:to:Closed} (i) can be applied to conclude that 
($d^\ast \gamma$~is necessarily closed and coclosed and so) $dd^\ast\gamma=0$. 

Now, the obstructions to solve $\triangle\gamma=-d^\ast\rho$ with $\gamma\in C^{2,\alpha}_{\lambda+1}$  are identified with harmonic $(k-1)$--forms $\xi$ in $C^\infty_{-n-\lambda+1}$. The hypotheses $\lambda>-\frac{n}{2}$ or $k\leq \frac{n}{2}$ and $\lambda>-n+k$ then allow one to apply Lemma \ref{lem:Harmonic:to:Closed} (i) to conclude that any such harmonic form $\xi$ is in fact closed and coclosed. In particular,  under hypothesis (i) the result then follows immediately. 

Under hypothesis (ii), the conditions $k\leq \frac{n}{2}$ and $\lambda>-n+k$, imply that $\mathcal{H}^{k-1}_{-n+1-\lambda}(M)=\mathcal{H}^{k-1}_{-n+k-1+\delta}(M)$ for every $\delta>0$ by Proposition \ref{prop:Indicial:Roots:pforms} (ii). So we may assume 
that the obstructions $\xi$ belong to $\mathcal{H}^{k-1}_{-n+k-1+\delta}(M)$ 
for any $\delta>0$. Hence under the assumption $\lambda < -k+2$ we can also integrate by parts
\[
\langle d^\ast\rho,\xi\rangle_{L^2} = \langle \rho,d\xi\rangle_{L^2}=0
\]
because, by choosing $\delta>0$ small enough so that $\lambda +\epsilon +\delta < -k+2$, we have  $\lambda+\epsilon -n+k-1+\delta < -n+1$.
\endproof
\end{prop}

\begin{remark}\label{rmk:Closed:to:Cocolosed:Laplacian} The first part of the proof %
shows that any $(k-1)$--form $\gamma\in C^{2,\alpha}_{\lambda+1}$ that solves $\triangle\gamma=-d^\ast\rho$ automatically satisfies $dd^\ast\gamma=0$ as soon as either $\lambda<-\frac{n}{2}+1$ or $k \leq \frac{n}{2}+1$ and $\lambda<-k+2$.
\end{remark}

As we will see, the previous result will be enough to study closed and coclosed forms of degree $k<\frac{n}{2}$ and, hence by duality, also of degree $k>\frac{n}{2}+1$. However, more work is needed to deal with ``middle degree'' forms: the cases $k=\frac{n}{2}, \frac{n}{2}+1$ for $n$ even and $k=\frac{n+1}{2}$ for $n$ odd. The following result provides the simplest tool to study the change in the space of closed and coclosed forms as we move from $L^2$--rates ($\lambda < - \frac{n}{2}$) to slower decay rates.

\begin{prop}\label{prop:Closed:to:Coclosed:dd*}
There exist $\delta, \epsilon>0$ such that for every $\lambda\in (-\frac{n}{2}-\delta, -\frac{n}{2}+1)$ the following holds. Let $\rho\in C^{1,\alpha}_{\lambda+\epsilon}$ be a closed $k$-form such that $d^\ast\rho\in C^{0,\alpha}_{\lambda-1}$. If $\langle d^\ast\rho,\xi\rangle_{L^2}=0$ for all $\xi\in\mathcal{H}^{k-1}_{-n+1-\lambda}(M)$ then there exists a $k$--form $\rho'\in C^{1,\alpha}_\lambda$ such that $\rho+\rho'$ is closed and coclosed. 
\proof
Consider a solution of the equation $(d+d^\ast)\rho'=-d^\ast\rho$ for a multi-degree form $\rho'\in C^{1,\alpha}_{\lambda}$. Since $\lambda<-\frac{n}{2}+1$, integration by parts shows that $d\rho'$ is  $L^2$--orthogonal to $d^\ast\rho'$ and also to  $d^\ast\rho$, and therefore any solution $\rho'$ satisfies $d\rho'=0$ and $d^\ast\rho'=-d^\ast\rho$. Then replacing $\rho'$ with its degree-$k$ component  we deduce that the $k$-form $\rho+\rho'$ is closed and coclosed.

For $\lambda \in (-\frac{n}{2},-\frac{n}{2}+1)$ we can apply Proposition \ref{prop:Closed:to:Coclosed:Laplacian} (i) to solve $\triangle \gamma=-d^\ast \rho$ and $\gamma$ automatically satisfies $d d^\ast \gamma=0$. 
Then $\rho'=d \gamma$  is the desired $(k-1)$-form.

It remains to treat the case where $\lambda\in (-\frac{n}{2}-\delta,-\frac{n}{2}]$. The obstructions to solve $(d+d^\ast)\rho'=-d^\ast\rho$ for $\rho'\in C^{1,\alpha}_{\lambda}$ arise from multi-degree forms $\xi\in C^{\infty}_{-n+1-\lambda}$ in the kernel of $d+d^\ast$. If $\delta>0$ is small enough so that there are no indicial roots in $(-\frac{n}{2}+1,-\frac{n}{2}+1+\delta)$, then any such $\xi$ is in fact of rate $\mu\leq -\frac{n}{2}+1$ and so by Lemma \ref{lem:Harmonic:to:Closed} (ii) it is then closed and coclosed.
\endproof
\end{prop}

We illustrate how to use these propositions by accounting for all the contributions to $\mathcal{H}^k_\lambda(M)$, $\lambda\in (-k,-k+1]$, that do not have a topological interpretation. Motivated by Proposition~\ref{prop:Indicial:Roots:pforms} (v) we introduce the following notation.

\begin{definition}
Denote by $\mathcal{H}^k_\lambda(\Sigma)$ the space of $(k-1)$--forms on $\Sigma$ which are coclosed and eigenforms of the Laplacian with eigenvalue $(\lambda+k)(\lambda+n-k)$.
\end{definition}
By Proposition \ref{prop:Indicial:Roots:pforms} (v) every nonzero $\alpha\in \mathcal{H}^k_\lambda(\Sigma)$ defines a homogeneous closed and coclosed (in fact, exact and coexact) $k$-form $d\left( \frac{r^{\lambda+k}}{\lambda+k}\alpha\right)$ on $\tu{C}(\Sigma)$ of rate $\lambda$.

\begin{prop}\label{prop:ALC:closed:coclosed:exact}
Let $k\leq \frac{n}{2}$ and $\lambda_0\in (-k,-k+1]$. Then there is an injective linear map  $f\co\mathcal{H}^k_{\lambda_0}(\Sigma) \to \mathcal{H}^k_{\lambda_0}(M)$ such that $f(\alpha)\in \mathcal{H}^k_{\lambda_0}(M)$ is an exact form with $f(\alpha) = d\left( \frac{r^{\lambda_0+k}}{\lambda_0+k}\alpha\right) + O(r^{-\lambda_0-\epsilon})$ for some $\epsilon>0$.
\proof
Fix a cutoff function $\chi$ with $\chi\equiv 1$ on the ALC end. Then $\rho=d\left( \chi\, \frac{r^{\lambda_0+k}}{\lambda_0+k}\alpha\right)$ is a closed $k$-form on $M$ which coincides with $d\left( \frac{r^{\lambda_0+k}}{\lambda_0+k}\alpha\right)$ on the ALC end. Also,  if $\epsilon>0$ is small enough then $d^\ast \rho \in C^{0,\alpha}_{\lambda-1}$.

Consider now the Poisson equation $\triangle\gamma = -d^\ast\rho$ for $\gamma\in C^{2,\alpha}_{\lambda+1}$. Since $k\leq\frac{n}{2}$,  Proposition \ref{prop:Closed:to:Coclosed:Laplacian} (ii) immediately guarantees the existence of a solution $\gamma$ and that $dd^\ast\gamma=0$ holds automatically. 
By requiring that $\gamma$ be $L^2_{\lambda+1}$--orthogonal to harmonic forms in $C^{2,\alpha}_{\lambda+1}$ 
we can ensure uniqueness of $\gamma$ and hence deduce its linear dependence on $\rho$. We then define $f(\alpha)=\rho + d\gamma$.
\endproof
\end{prop}
 \begin{remark}
     \label{rmk:accounted}
     Together with Proposition \ref{prop:Indicial:Roots:pforms}(v), this tells us that for $k\leq \frac{n}{2}$ the changes in $\mathcal{H}^k_\lambda (M)$ for $\lambda\in (-k,-k+1)$ are \emph{all} accounted for by the subspaces $f(\mathcal{H}^k_\lambda(\Sigma))$. Note that we do not have an analogue of Proposition \ref{prop:ALC:closed:coclosed:exact} for rates $\lambda\in (-n+k-1,-n+k)$ for $k\leq \frac{n}{2}$, despite that by Proposition \ref{prop:Indicial:Roots:pforms}(v) all the indicial roots in this range also depend on the eigenspaces $\mathcal{H}^k_\lambda(\Sigma)$.
 \end{remark}
It turns out that the case where $k\in [\frac{n}{2},\frac{n}{2}+1]$ and $\lambda<-\frac{n}{2}$ is more delicate and to deal with it we will need the following extension of Proposition \ref{prop:Closed:to:Coclosed:Laplacian}.  The proof strategy rests on solving the same Poisson equation as previously, with 
the complication now being that we encounter genuine obstructions to solving it. 
To identify exactly what obstructions arise and then understand how we can compensate for them
Proposition~\ref{prop:Laplacian:BC:MidDegree} plays a key role: it determines the leading-order asymptotics
of the (nonclosed) harmonic forms responsible for these obstructions. 

\begin{prop}\label{prop:Closed:to:Coclosed:Laplacian:Mid:deg}
Let $k=\frac{n}{2}$ if $n$ is even or $k=\frac{n+1}{2}$ if $n$ is odd. Then there exists $\delta\in (0,1)$ with the following significance. 
Let $\rho\in C^{1,\alpha}_{\lambda}$ be a closed $k$-form
with $\lambda\in (-k-\delta,-k)$.
Then there exists a coclosed $(k-1)$--form $\gamma$ of the form $\gamma = \gamma_\infty + O(r^{\lambda+1})$
such that $\rho+d\gamma$ is closed and coclosed.
Here
\begin{enumerate}
\item if $n$ is even, $\gamma_\infty = \tau$ with $\tau\in\mathcal{H}^{k-1}(\Sigma)$, and
\item if $n$ is odd, $\gamma_\infty = \tau_1 + \frac{1}{r}\theta\wedge\tau_2$ with $\tau_1\in\mathcal{H}^{k-1}(\Sigma)$ and $\tau_2\in \mathcal{H}^{k-2}_-(\Sigma)$. Here the subscript $_-$ indicates that we consider harmonic forms with values in the flat orientation bundle of $\Sigma$.
\end{enumerate}
\proof
We again consider the equation $\triangle\gamma = -d^\ast\rho$ for $\gamma\in C^{2,\alpha}_{\lambda+1}$. The obstructions to solve it are identified with harmonic $(k-1)$--forms $\xi\in C^{\infty}_{-n+1-\lambda}$, \ie a solution exists if and only if $\langle d^\ast\rho,\xi\rangle_{L^2}=0$ for all such $\xi$. A harmonic form $\xi$ of this rate need not be  closed, but it must satisfy $dd^\ast\xi=0$ by Remark \ref{rmk:Closed:to:Cocolosed:Laplacian}, so that if a solution $\gamma$ exists then $\rho+d\gamma$ is closed and coclosed.

On the other hand, if $\xi$ is in fact closed then
\[
\langle d^\ast\rho,\xi\rangle_{L^2} = \langle \rho, d\xi\rangle_{L^2} =0.
\]
Here integration by parts is justified since, if $\delta$ is small enough so that $-n+1-\lambda$ is not an indicial root, then $\xi = O(r^{-n+1-\lambda-\epsilon})$ for some $\epsilon>0$ and $\rho = O(r^{\lambda})$. Hence the obstructions to solve $\triangle \gamma = -d^\ast\rho$ with $\gamma\in C^{2,\alpha}_{\lambda+1}$ arise from inner products
\[
\langle d^\ast\rho,\xi\rangle_{L^2}
\]
for all harmonic $(k-1)$--forms $\xi$ of rate $-n+1-\lambda$ which are \emph{not} closed.

We look more closely at such harmonic forms. If $\delta$ is sufficiently small Proposition \ref{prop:Laplacian:BC:MidDegree} implies that $\xi = \xi_\infty + O(r^{-n+1-\lambda-\mu})$ for some $\mu>0$ and
\begin{enumerate}
\item if $n$ is even,
\[
\xi_\infty = \alpha_1\log{r} + \alpha_2,
\]
where $\alpha_1,\alpha_2\in\mathcal{H}^{k-1}(\Sigma)$, and
\item if $n$ is odd,
\[
\xi_\infty = r\alpha_1 + \theta\wedge\alpha_2 + d(r\alpha_3),
\]
where $\alpha_1\in\mathcal{H}^{k-1}(\Sigma)$, $\alpha_2\in\mathcal{H}^{k-2}(\Sigma)$ and $\alpha_3$ is a coclosed $(k-2)$--form on $\Sigma$ such that $\triangle_\Sigma\alpha_3 = 2\alpha_3$.
\end{enumerate}
In particular,
\[
d\xi = \frac{dr}{r}\wedge\alpha_1 + O(r^{-n-\lambda-\mu}), \qquad d\xi = dr\wedge\alpha_1 + d\theta\wedge\alpha_2 + O(r^{-n-\lambda-\mu}) = dr\wedge\alpha_1 + O(r^{-n-\lambda-\mu})
\]
if $n$ is even and, respectively, odd. If $n$ is even, an integration by parts as in Lemma \ref{lem:Harmonic:to:Closed} implies that $d\xi=0$ whenever $\alpha_1=0$. If $n$ is odd, first of all note that the term $d(r\alpha_3)$ is the leading term in the asymptotic expansion of one of the exact forms in $f(\mathcal{H}^{k-1}_{-k+2}(\Sigma))$ constructed in Proposition \ref{prop:ALC:closed:coclosed:exact}. When $n$ is odd we can therefore assume without loss of generality that $\alpha_3=0$ by adding to $\xi$ a closed and coclosed form in $f(\mathcal{H}^{k-1}_{-k+2}(\Sigma))$. Then an integration by parts as in Lemma \ref{lem:Harmonic:to:Closed} shows that $d\xi=0$ whenever $\alpha_1=0=\alpha_2$ since if $\alpha_1=0$ then $d\xi$ is actually of order $O(r^{-k+\epsilon'})$ for any $\epsilon'>0$ by Proposition \ref{prop:Indicial:Roots:pforms}(ii).

Now, let $\chi$ be a cutoff function with $\chi\equiv 1$ on the ALC end and consider $k$-forms $d(\chi\gamma_\infty)$ with $\gamma_\infty$ as in the statement of the Proposition. Using the fact that $d^\ast d\xi=0$, we have
\[
\langle d^\ast d(\chi\gamma_\infty),\xi\rangle_{L^2}=\pm \int_N {\xi_\infty \wedge \ast d\gamma_\infty} \pm \int_N {\gamma_\infty\wedge\ast d\xi_\infty}.
\]
\begin{enumerate}
\item If $n$ is even, the first boundary term vanishes since $d\gamma_\infty=0$, while the second term 
\[
\int_N {\gamma_\infty\wedge\ast d\xi_\infty}= \pm \int_\Sigma{\tau\wedge\ast_\Sigma\alpha_1}.
\]
\item If $n$ is odd then both boundary terms are non-zero and an explicit calculation shows
\[
\int_N {\gamma_\infty\wedge\ast d\xi_\infty}= \pm \int_\Sigma{\tau_1\wedge\ast_\Sigma\alpha_1}, \qquad \int_N {\xi_\infty \wedge \ast d\gamma_\infty} = \pm \int_\Sigma{\alpha_2\wedge\ast_\Sigma\tau_2}.
\]
\end{enumerate}
We conclude that we can compensate for all the obstructions to solve $\triangle\gamma=-d^\ast\rho$ by adding an exact form $d(\chi\gamma_\infty)$ to $\rho$.
\endproof
\end{prop}

\begin{remark}\label{rmk:Closed:to:Coclosed:Laplacian:Mid:deg}
Note that $d(\chi\gamma_\infty)$ is compactly supported if $n$ is even, while if $n$ is odd then $d(\chi\gamma_\infty) = \frac{1}{r^2}\theta\wedge dr\wedge\tau_2 +O(r^{-k-1})=O(r^{-k})$.
\end{remark}

\subsection{Cohomology}

\newcommand{\bdrincl}{j_M}

We now come to discuss closed forms on $M$ and we begin with some observations about the cohomology of $M$. Considering $M$ as the interior of a compact manifold with boundary $N$ we have a long exact sequence:
\begin{equation}\label{eq:LES:M:boundary}
\cdots \ra H^{k-1}(N) \stackrel{\delta}{\ra} H^k_c(M) \ra H^k(M) \stackrel{\bdrincl^*}{\ra} H^k(N)\ra\cdots
\end{equation}
At the level of de Rham representatives, the map $H^k_c(M) \ra H^k(M)$ is induced by inclusion of the compactly supported subcomplex, the map $\bdrincl^*: H^k(M)\ra H^k(N)$ is restriction to the boundary and the snake map $\delta : H^{k-1}(N)\ra H^k_c(M)$ is induced by $\tau \mapsto d(\chi\, \tau)$, where $\chi$ is a cutoff function with $\chi\equiv 1$ on the ALC end.

Now, since $N$ is a circle fibration over $\Sigma$ we can also consider the Gysin sequence for $\pi\co N\ra \Sigma$. In the cyclic case, this reads
\begin{subequations}\label{eq:Gysin:bdry}
\begin{equation}\label{eq:Gysin:bdry:cyclic}
\cdots\ra H^{k-2}(\Sigma)\stackrel{\cup [d\theta]}{\ra} H^k(\Sigma) \stackrel{\pi^*}{\ra}  H^k(N) \stackrel{I}{\ra} H^{k-1}(\Sigma)\ra \cdots 
\end{equation}
where the last map $I$ is integration along the circle fibres. In the dihedral case, we can either work $\iota$--equivariantly on the principal circle bundle $\wt{N}\ra \wt{\Sigma}$ or equivalently consider cohomology of $\Sigma$ with values in the real line bundle with first Stiefel--Whitney class $w_1(N)$. Denoting by $H^\bullet_-(\Sigma)$ this cohomology, \eqref{eq:Gysin:bdry:cyclic} becomes
\begin{equation}\label{eq:Gysin:bdry:dihedral}
\cdots \ra H^{k-2}_-(\Sigma)\stackrel{\cup [d\theta]}{\ra} H^k(\Sigma) \stackrel{\pi^*}{\ra}  H^k(N) \stackrel{I}{\ra} H^{k-1}_-(\Sigma)\ra \cdots 
\end{equation}
\end{subequations}
If $N$ is a Seifert fibration over the orbifold $\Sigma$, the cohomology of $\Sigma$ is understood as the basic cohomology of $N$.

Now we wish to combine \eqref{eq:LES:M:boundary} and the Gysin sequence \eqref{eq:Gysin:bdry}. In particular, we shall be interested in the maps 
\begin{equation}\label{eq:Cohomology:X}
\delta \circ \pi^* : H^{\bullet}(\Sigma)\ra H^{\bullet+1}_c(M), \qquad I \circ \bdrincl^* : H^\bullet(M)\ra H^{\bullet-1}_-(\Sigma),
\end{equation}
that can be obtained by composing maps in the sequences
\eqref{eq:LES:M:boundary} and \eqref{eq:Gysin:bdry},
where abusing notation we set $H^{k-1}_-(\Sigma)=H^{k-1}(\Sigma)$ in the cyclic case. In that case we also set $\wt{\Sigma}=\Sigma$.

An efficient way to understand the kernels and images of the maps in \eqref{eq:Cohomology:X} is to consider the natural compactification $X$ of $M$ obtained by collapsing the circle fibres at infinity. Thus $X$ contains a copy of $\Sigma$ and $M=X\setminus\Sigma$. In other words, $X$ is obtained by gluing $M$ with the real rank-2 vector (orbi)bundle $E\ra\Sigma$ in which $N\ra\Sigma$ is the unit circle bundle. We can endow $X$ with a (orbi)smooth structure identifying the ALC end $(r_0,\infty)\times N$ with the complement $E^\ast = (0,x_0)\times N$ of the $0$-section in $E$ via $(r,p)\mapsto (\frac{1}{r},p)$. Then $X$ is a smooth manifold when the cross-section $\Sigma$ is smooth, but
when $N$ is a Seifert fibration over an orbifold~$\Sigma$, $X$ is itself an orbifold. However, since the singularities of $X$ are contained in the compact orbifold $\Sigma$ (so the minimum common multiple of the orders of the stabilisers of points in $X$ is finite) there is no issue in defining the cohomology of $X$, which can be taken equivalently as its (Haefliger) orbifold cohomology (\ie the singular cohomology of the classifying space of the groupoid representing $X$) or as the singular cohomology of the underlying topological space. Concretely, $H^\bullet (X)$ can be defined via the Mayer--Vietoris sequence associated to the cover $X=M\sqcup E$.

Now, the long exact sequence for the pair $(X,\Sigma)$, identifying $H^\bullet (X,\Sigma)$ with $H^\bullet_c(M)$, reads
\begin{equation}\label{eq:pair_exact}
\dots \ra H^{k-1}(\Sigma) \ra H^{k}_c(M) \stackrel{i_*}{\ra} H^{k}(X) \stackrel{j_X^*}{\ra} H^k(\Sigma) \ra H^{k+1}_c(M)\ra \dots
\end{equation}
where the maps $H^\bullet (\Sigma)\ra H^{\bullet +1}_c(M)$ coincide with the first map $\delta \circ \pi^*$ in \eqref{eq:Cohomology:X}. Thus we have isomorphisms
\begin{gather*}
\ker\, \delta \circ \pi^* : H^k(\Sigma)\ra H^{k+1}_c(M)\simeq \tu{im}\, j_X^*: H^k(X)\ra H^k(\Sigma) \\ H^k_c(M)/\left(\tu{im}\, \delta \circ \pi^* : H^{k-1}(\Sigma)\ra H^k_c(M) \right)\simeq \tu{im}\, i_* : H^k_c(M)\ra H^k(X).
\end{gather*}
Similarly, the long exact sequence for the
pair $(X,M)$ coupled with the Thom isomorphism $H^k(X,M)\simeq H^{k-2}_-(\Sigma)$ on (the double cover of) a tubular neighbourhood of $\Sigma$ in $X$ reads:
\begin{equation}\label{eq:thom_exact}
\dots\ra H^{k-2}_-(\Sigma)\ra H^k(X) \stackrel{i^*}{\ra} H^k(M)\ra H^{k-1}_-(\Sigma)\ra \dots
\end{equation}
The maps $H^\bullet(M)\ra H^{\bullet-1}_-(\Sigma)$ coincide with the second map $I \circ \bdrincl^*$ in \eqref{eq:Cohomology:X} and therefore we have isomorphisms
\[
\ker\, I \circ \bdrincl^* : H^k(M)\ra H^{k-1}_-(\Sigma)\simeq \tu{im}\, i^* : H^k(X)\ra H^k(M).
\]
We will use these isomorphisms freely in the rest of the section.

\begin{remark}\label{rmk:Extension:Closed}
We observe explicitly that
\begin{enumerate}
\item for every representative $\tau$ of a class in the kernel of $\delta \circ \pi^*: H^{k}(\Sigma)\ra H^{k+1}_c(M)$ there exists a closed $k$--form $\rho$ on $M$ with $\rho = \tau$ on the ALC end;
\item for every representative $\tau$ of a class in the image of $I \circ \bdrincl^*: H^k(M)\ra H^{k-1}_-(\Sigma)$ (thus $\tau$ can be lifted to a $\check{\iota}$--anti-invariant form on $\widetilde{\Sigma}$ in the dihedral case) there exists a closed $k$-form $\rho$ on $M$ with $\rho = \theta\wedge\tau + O(r^{-k})$ on the ALC end.
\end{enumerate}
Indeed, in the first case since $\ker\, \delta \circ \pi^*\simeq\tu{im}\, j_X^* : H^k(X)\ra H^k(\Sigma)$ there always exists a closed $k$-form $\tilde{\tau}$ on $X$ that in the tubular neighbourhood $E$ of $\Sigma$ takes the form $\tilde{\tau}=\tau + d\xi$ for a smooth $(k-1)$-form $\xi$ on $X$; we then set $\rho = \tilde{\tau}|_M + d(\chi\, \xi|_M)$ for a cutoff function $\chi$ with $\chi\equiv 1$ on the ALC end. In the second case, by assumption there exists a closed form on $N$ of the form $\theta\wedge\tau + u$ with $u\in\Omega^{k}(\Sigma)$ satisfying $du=-d\theta\wedge \tau$ and $[\theta\wedge\tau + u]$ in the image of $H^k(M)$. A procedure similar to the one given above yields a closed $k$-form $\rho$ on $M$ with $\rho = \theta\wedge\tau + u$ on the ALC end. Now note that $u=O(r^{-k})$.
\end{remark}

Finally, we make explicit the following simple consequences of Poincar\'e duality.

\begin{lemma}\label{lem:Poincare:Duality}
For all $k$ we have
\begin{enumerate}
\item $\tu{im}\, \bdrincl^* : H^k(M)\ra H^k(N)$ is the annihilator of $\tu{im}\, \bdrincl^* : H^{n-k}(M)\ra H^{n-k}(N)$,
\item $\tu{im}\, j_X^* : H^k(X)\ra H^k(\Sigma)\simeq \ker\, \delta \circ \pi^* : H^k(\Sigma)\ra H^{k+1}_c(M)$ is the annihilator of $\tu{im}\, I \circ \bdrincl^* : H^{n-k}(M)\ra H^{n-k-1}_-(\Sigma)$.
\end{enumerate}
\proof
The first statement is a a consequence of the fact that \eqref{eq:LES:M:boundary} is Poincar\'e self-dual, while the second one follows from the fact that the sequence \eqref{eq:pair_exact} is Poincar\'e dual to \eqref{eq:thom_exact}.
\endproof
\end{lemma}

Return now to the space $\mathcal{H}^k_\lambda(M)$ of closed and coclosed $k$-forms on the ALC manifold $M$. Since elements of $\mathcal{H}^k_\lambda(M)$ are closed, we have the natural projection map to de Rham cohomology
\[
\Psi : \mathcal{H}^k_\lambda (M) \to H^k(M),\qquad \rho \mapsto [\rho].
\]
For $\lambda < -k$, there is in addition a canonical map $\Phi$ to compactly supported de Rham cohomology $H^k_c(M)$. In order to define it we need the following observation.

\begin{lemma}\label{lem:Exact:forms:cone}
Let $\alpha$ be a smooth closed $k$-form on $\{ r\geq R\} \subset \tu{C}(\Sigma)$. If $|\alpha|_{g_{\tu{C}}}=O(r^\nu)$ for some $\nu<-k$ then $\alpha = d\beta$ for some $(k-1)$-form $\beta$ with $|\beta|_{g_{\tu{C}}}=O(r^{\nu+1})$.

Similarly, if $\alpha$ is a smooth closed $k$-form on $\{ r\geq R\} \subset \BC(\Sigma)$ such that $|\alpha|_{g_{\BC}}=O(r^\nu)$ for some $\nu<-k$, then $\alpha = d\beta$ for some $(k-1)$-form $\beta$ with $|\beta|_{g_{\BC}}=O(r^{\nu+1})$.
\begin{proof}
The statement about closed forms on $\tu{C}(\Sigma)$ is proved in \cite[Lemma 2.11]{Karigiannis}. The proof adapts easily to forms on $\BC (\Sigma)$. Indeed, write $\alpha = dr\wedge \alpha_{k-1} + \alpha_k$. Closure of $\alpha$ is equivalent to
\begin{equation}\label{eq:Exact:forms:cone}
\partial_r\alpha_k = d_N\alpha_{k-1}, \qquad d_N\alpha_k=0.
\end{equation}
Here $d_N$ denotes the differential on the slices $\{ r=\text{const}\}\simeq N$.

Considering the decomposition $\Omega^{l} (N) = \bigoplus_{p+q=l}{\pi^\ast \Omega^p(\Sigma)\otimes\Omega^q(S^1)}$ of the differential forms of the circle bundle $\pi\co N\ra \Sigma$ we further decompose
\[
\alpha_{k-1} = \theta \wedge \zeta_{k-2} + \zeta_{k-1}, \qquad \alpha_k = \theta \wedge \xi_{k-1} + \xi_{k}.
\]
The assumption $|\alpha|_{g_{\BC}}=O(r^\nu)$ implies that
\[
|\zeta_{k-2}|_{g_N} = O(r^{\nu+k-2}), \qquad |\zeta_{k-1}|_{g_N}=O(r^{\nu+k-1}),\qquad |\xi_{k-1}|_{g_N} = O(r^{\nu+k-1}), \qquad |\xi_k|_{g_N}=O(r^{\nu+k})
\]
where $g_N$ is the fixed metric $g_N = \pi^\ast g_\Sigma + \theta^2$. We then define
\[
\beta = -\int_r^\infty{\alpha_{k-1}}.
\]
The condition $\nu<-k$ guarantees that the previous integral converges and that $|\alpha_k|\ra 0$ as $r\ra \infty$. It is then straightforward to verify that $d\beta =\alpha$ using \eqref{eq:Exact:forms:cone}.
\end{proof}
\end{lemma}

This lemma not only says that if $\rho \in \mathcal{H}^k_\lambda(M)$ for $\lambda<-k$ then $\Psi(\rho)\in\tu{ker}\, \bdrincl^* : H^k(M)\ra H^k(N)\simeq\tu{im}\, H^k_c(M)\ra H^k(M)$ but also provides a canonical primitive for $\rho$ on the ALC end and therefore a canonical lift $\Phi(\rho)$ to $H^k_c(M)$. Indeed, for $\lambda<-k$ we can use Lemma \ref{lem:Exact:forms:cone} to write $\rho \in \mathcal{H}^k_\lambda(M)$ as $\rho=d\gamma' + \rho'$ where $\gamma'=O(r^{\lambda+1})$ and $\rho'$ is a closed compactly supported form. While $\gamma'$ involves a choice of cutoff function, changing that choice only changes $\rho'$ by addition of the differential of a compactly supported form. We can thus define a linear map
\[
\Phi\co \mathcal{H}^k_\lambda(M)\longrightarrow H^k_c(M), \qquad \rho \longmapsto [\rho'].
\]

For $\lambda < -k$, we can moreover consider $\overline \Phi : \mathcal{H}^k_\lambda (M) \to H^k(X)$, the composition of $\Phi$ with the push-forward $i_* : H^k_c(M) \to H^k(X)$. We can think of $\overline \Phi$ as extending a closed form smoothly to~$X$ (as zero on $\Sigma$).

In fact for each $k$ there exists $\delta\in (0,1)$ such that $\overline{\Phi}\co\mathcal{H}^k_\lambda(M)\ra H^k(X)$ is well defined for all $\lambda<-k+\delta$. Indeed, by Lemma \ref{Step3} on the ALC end we have $\rho = \theta\wedge\alpha +\beta +O(r^{-k-\epsilon})$ for closed and coclosed forms $\alpha,\beta$ on the cone $\tu{C}(\Sigma)$ homogeneous of order $\lambda$ and some $\epsilon>0$. (In fact $\alpha=0$ if $k\leq\frac{n}{2}$.) Since $\alpha$ is a $(k-1)$-form and $\lambda<-k+1$, by the statement of Lemma \ref{lem:Exact:forms:cone} for $\tu{C}(\Sigma)$ we have $\alpha = d\gamma'$ with $\gamma'=O(r^{\lambda+1})$. Thus $\theta\wedge\alpha = d(-\theta\wedge\gamma') + d\theta\wedge\gamma'$. The first term is the differential of a form of order $O(r^{\lambda+1})$ and the second term is of order $O(r^{\lambda-1})$ since $d\theta=O(r^{-2})$. We can therefore rewrite $\rho = \beta + d\gamma' + O(r^{-k-\epsilon})$ with $\gamma'=O(r^{\lambda+1})$. Applying Lemma \ref{lem:Exact:forms:cone} to $\rho-\beta-d\gamma'$ we conclude that along the ALC end $\rho = \beta + d\gamma$ with $\gamma=O(r^{\lambda+1})$. Finally, note that if $\delta>0$ is sufficiently small then $\beta$ must be a closed and coclosed $k$-form homogeneous of order $-k$. For $k\neq \frac{n}{2}$ Proposition \ref{prop:Indicial:Roots:pforms}(iv) implies that $\beta$ is the pull-back of a harmonic $k$-form on $\Sigma$. By cutting off $\gamma$ we can therefore smoothly extend $\rho$ to a closed form $\overline{\rho}$ on $X$ that restricts to $\beta$ on $\Sigma$. As before, changing the cutoff function only changes $\overline{\rho}$ by the differential of a form (with compact support in~$M$) and therefore $\overline{\Phi}(\rho)=[\overline{\rho}]\in H^k(X)$ is well-defined.

The same construction can be extended to the exceptional case $k=\frac{n}{2}$ since by Proposition \ref{prop:Indicial:Roots:pforms}(iv) now $\beta = \tau_1 + \frac{dr}{r}\wedge\tau_2= \tau_1 + d(\tau_2\log{r})$ for $\tau_1\in\mathcal{H}^k(\Sigma)$ and $\tau_2\in\mathcal{H}^{k-1}(\Sigma)$. In order to define $\overline{\rho}$ we then cut off both $\gamma$ and $\tau_2\log{r}$.

\begin{lemma}\label{lem:Injectivity:Compactly:Supported}
\hfill
\begin{enumerate}
    \item For $\lambda < -k$, the map $ \Phi\co \mathcal{H}^k_\lambda(M)\ra H^k_c(M)$ is injective.
    \item If $k \geq \frac{n}{2}+1$ and $\lambda < -n-1+k$ then $\Psi : \mathcal{H}^k_\lambda (M) \to H^k(M)$ is injective.
    \item If $\frac{n}{2}\leq k \leq \frac{n}{2}+1$ and $\lambda<-\frac{n}{2}$ then $\overline{\Phi}\co\mathcal{H}^k_\lambda(M)\ra H^k(X)$ is injective.
    \item $\overline \Phi$ is also injective if $k = \frac{n+1}{2}$ and $\lambda < -k+1$.
\end{enumerate}
\proof
We prove (i). First note that for $\lambda<-k$ we have $\mathcal{H}^k_\lambda(M)\subset L^2\mathcal{H}^k(M)$. This is obvious if $k\geq\frac{n}{2}$ and it is also true for $k<\frac{n}{2}$ by Proposition \ref{prop:Indicial:Roots:pforms} (ii). Thus we can always assume $\lambda<-\frac{n}{2}$.

Now, if $\rho\in \mathcal{H}^k_\lambda(M)$ write $\rho=d\gamma' + \rho'$ where $\gamma'=O(r^{\lambda+1})$ and $\rho'$ is a closed compactly supported form such that $[\rho']=\Phi (\rho)$. If $\rho=d\gamma''$ where $\gamma''$ is compactly supported, then $\rho=d\gamma$ with $\gamma=O(r^{\lambda+1})$. Then
\[
\|\rho\|^2_{L^2} = \langle d\gamma,d\gamma\rangle_{L^2} = \langle \gamma, d^\ast \rho\rangle_{L^2}=0.
\]

The other statements are variations of this argument. Consider first the case $\lambda <-\frac{n}{2}$, so that any $\rho\in\mathcal{H}^k_\lambda(M)$ is $L^2$--integrable. As above, integration by parts shows that $\int_M{\rho\wedge d\gamma}=0$ for all $\rho\in\mathcal{H}^k_\lambda(M)$ and $\gamma=O(r^{\lambda+1})$. Furthermore, note that $\Psi\co\mathcal{H}^k_\lambda(M)\ra H^k(M)$ is always defined, but, without any further assumption on $\lambda$ other than $\lambda<-\frac{n}{2}$, $\overline{\Phi}\co \mathcal{H}^k_\lambda(M)\ra H^k(X)$ is only defined for $k\leq \frac{n}{2}+1$ and $\Phi\co \mathcal{H}^k_\lambda(M)\ra H^k_c(M)$ is only defined if $k\leq \frac{n}{2}$. Dually, if $\rho\in\mathcal{H}^k_\lambda(M)$ then $\Phi(\ast\rho)$ is only defined for $k\geq \frac{n}{2}+1$ and $\overline{\Phi}(\ast\rho)$ only defined if $k\geq\frac{n}{2}$. These observations show that
\[
\|\rho\|_{L^2}^2=\begin{dcases}
\langle \Phi(\rho)\cup \Psi(\ast\rho),[M]\rangle & \mbox{ if } k\leq\tfrac{n}{2},\\
\langle \overline{\Phi}(\rho)\cup \overline{\Phi}(\ast\rho), [X]\rangle & \mbox{ if } \tfrac{n}{2}\leq k\leq \tfrac{n}{2}+1,\\
\langle\Psi(\rho)\cup \Phi(\ast\rho),[M]\rangle & \mbox{ if } k\geq\tfrac{n}{2}+1.\\
\end{dcases}
\]
The injectivity statements in the Lemma follow from these identities when $\lambda<-\frac{n}{2}$.

Finally, the excluded indicial roots in Proposition \ref{prop:Indicial:Roots:pforms}(ii) allow one to extend part (ii) up to $\lambda<-n-1+k$, and to deduce part (iv).
\endproof
\end{lemma}

\subsection{Weighted Hodge theory}

We are now ready to state and prove the main results of this section. Recall that we are working on an ALC manifold M of dimension $n + 1$, and we have defined
\begin{itemize}
\item injective maps $\Psi\co \mathcal{H}^k_\lambda(M)\ra H^k(M)$, $\overline{\Phi}\co \mathcal{H}^k_\lambda(M)\ra H^k(X)$ and $\Phi\co \mathcal{H}^k_\lambda(M)\ra H^k_c(M)$ for suitable choices of $k$ and $\lambda$ as described in Lemma \ref{lem:Injectivity:Compactly:Supported};
\item for $k\leq \frac{n}{2}$ and $\lambda\in (-k,-k+1]$ the subspace $f(\mathcal{H}^k_\lambda(\Sigma))\subset \mathcal{H}^k_\lambda (M)$ of Proposition \ref{prop:ALC:closed:coclosed:exact} (which by Remark \ref{rmk:accounted} account for all the changes in $\mathcal{H}^k_\lambda (M)$ for $\lambda\in (-k,-k+1)$).
\end{itemize}
By Proposition \ref{prop:Indicial:Roots:pforms} applied with $p=k$ and $p=k-1$ we know that
\begin{itemize}
\item $\mathcal{H}^k_\lambda (M)$ is constant for $\lambda\in (-n+k,-k)$ if $k<\frac{n}{2}$;
\item $\mathcal{H}^k_\lambda (M)$ is constant for $\lambda\in (-k+1,-n+k-1)$ if $k>\frac{n}{2}+1$;
\item $\mathcal{H}^\frac{n+1}{2}_\lambda (M)$ is constant for $\lambda\in (-\frac{n+1}{2},-\frac{n-1}{2})$ if $n$ is odd.
\end{itemize}

Our goal is to first describe the $L^2$--cohomology $L^2\mathcal{H}^k(M)$ corresponding to $\mathcal{H}^k_\lambda(M)$ with rates $\lambda\in (-\frac{n}{2}-\delta_1,-\frac{n}{2}+\delta_2)$ for suitable $\delta_1,\delta_2\geq 0$ depending on $k$ and spectral properties of $\Sigma$. Then we describe the change in $\mathcal{H}^k_\lambda(M)$ as $\lambda$ crosses
\begin{itemize}
    \item 
the lower and upper bounds of this range (depending on $k$),
\item the value $\lambda=-k+1$  for $k\leq\frac{n+1}{2}$
\item the value $\lambda=-n+k$  for $k\geq\frac{n+1}{2}$.
\end{itemize}
By \eqref{eq:harmonic_leading}, understanding these changes amounts to deciding which $k$-forms $\alpha + \theta \wedge \beta$, with $\alpha$ and $\beta$ harmonic on $\tu{C}(\Sigma)$ and homogeneous of rate $\lambda$, appear as highest-order terms of some $\rho \in \mathcal{H}^k_\lambda(M)$.

We begin by describing $\mathcal{H}^k_\lambda(M)$ for $\lambda$ in the $L^2$--range $(-\frac{n}{2}-\delta_1,-\frac{n}{2}+\delta_2)$, where we can use the maps $\Phi,\overline{\Phi},\Psi$ to identify $\mathcal{H}^k_\lambda(M)$ in terms of the cohomology groups of $M$ and $X$. 

\begin{theorem}\label{thm:ALC_hodge_fastdecay}
Let $M$ be an ALC manifold of dimension $n+1$ asymptotic to the cone $\tu{C}(\Sigma)$. Denote by $X$ the compactification of $M$ obtained by collapsing the circle fibres at infinity.
\begin{enumerate}
        \item For $k < \frac{n}{2}$, $\Phi\co \mathcal{H}^k_\lambda (M) \to H^k_c(M)$ is an isomorphism for all $\lambda\in (-n+k,-k)$.
        \item For $k=\frac{n}{2}$, there exists $\delta>0$ depending only on $\Sigma$ such that $\overline \Phi\co \mathcal{H}^k_\lambda (M) \to H^k(X)$ is injective with image equal to that of $H^k_c(M)$ (equivalently, $\Phi$ induces an isomorphism of $\mathcal{H}^k_\lambda (M)$ with the quotient $H^k_c(M)/\left(\tu{im}\, \delta \circ \pi^* : H^{k-1}(\Sigma)\ra H^k_c(M)\right)$) for all $\lambda\in (-{k-\delta},-k)$.
        \item For $k =\frac{n+1}{2}$, $\overline \Phi\co \mathcal{H}^k_\lambda (M) \to H^k(X)$ is an isomorphism for all $\lambda\in (-k,-k+1)$.
        \item For $k=\frac{n}{2}+1$, there exists $\delta>0$ depending only on $\Sigma$ such that $\Psi\co \mathcal{H}^k_\lambda (M) \to H^k(M)$ is injective with image equal to that of $H^k(X)$ (equivalently, $\overline{\Phi}$ induces an isomorphism between $\mathcal{H}^k_\lambda(M)$ and the quotient $H^k(X)/\tu{im}\, H^{k-2}_-(\Sigma)\ra H^k(X)$) for all $\lambda\in (-k+1-\delta,-k+1)$.
        \item For $k > \frac{n}{2} + 1$, $\Psi \co \mathcal{H}^k_\lambda (M) \to H^k(M)$ is an isomorphism for all $\lambda\in (-k+1,-n+k-1)$.
    \end{enumerate}
\proof
If $k<\frac{n}{2}$, the injectivity of $\Phi\co\mathcal{H}^k_\lambda(M)\ra H^k_c(M)$ is stated in Lemma \ref{lem:Injectivity:Compactly:Supported}(i) and the surjectivity is a consequence of Proposition \ref{prop:Closed:to:Coclosed:Laplacian}(ii).

If $k=\frac{n}{2}$, Lemma \ref{lem:Injectivity:Compactly:Supported} (ii) implies that $\overline{\Phi}$ is injective. Now let $[\rho]\in \tu{im}\, H^k_c(M)\ra H^k(X)$ be a cohomology class represented by a closed $k$-form $\rho$ on $M$ with compact support. Proposition \ref{prop:Closed:to:Coclosed:Laplacian:Mid:deg} implies that there exists a $(k-1)$--form $\gamma=\chi\tau + O(r^{\lambda+1})$ with $\tau\in\mathcal{H}^{k-1}(\Sigma)$ (and $\chi$ a cutoff function for the ALC end like in Remark \ref{rmk:Extension:Closed}) such that $\rho+d\gamma\in\mathcal{H}^k_\lambda(M)$. By the definition of $\Phi$ and $\overline{\Phi}$ we have $\Phi(\rho+d\gamma)=\Phi(\rho+d(\chi\tau))$ and, since $\chi\tau$ has a smooth extension to $X$, $\overline{\Phi}(\rho+d(\chi\tau))=\overline{\Phi}(\rho)\in H^k(X)$. This shows that the injective map $\overline{\Phi}$ is also surjective onto the subspace $\tu{im}\, H^k_c(M)\ra H^k(X)$. Equivalently, the map $\mathcal{H}^k_\lambda(M)\ra H^k_c(M)/\tu{im}\, \delta \circ \pi^* :  H^{k-1}(\Sigma)\ra H^k_c(M)$ induced by $\Phi$, which is injective since both $\Phi$ and $\overline{\Phi}$ are, is in fact an isomorphism. 

The proof for $k=\frac{n+1}{2}$ is similar. Firstly, Lemma \ref{lem:Injectivity:Compactly:Supported} (ii) implies that $\overline{\Phi}$ is injective. By Remark \ref{rmk:Extension:Closed}(i) every cohomology class on $X$ can be represented by a closed form $\rho$ whose restriction to the ALC end of $M$ coincides with the harmonic representative $\rho_\infty\in\mathcal{H}^k(\Sigma)$ of the image of $[\rho]$ under the map $H^k(X)\ra H^k(\Sigma)$. We would like to apply Proposition \ref{prop:Closed:to:Coclosed:Laplacian:Mid:deg} to solve $\triangle\gamma = -d^\ast\rho$ for $\gamma\in C^{2,\alpha}_{\lambda+1}$ with $\lambda\in (-k-\delta,-k)$; this is not immediately possible since now we do not assume that $\rho=O(r^\lambda)$, but one can check explicitly that the initial integration by parts
\[
\langle d^\ast\rho,\xi\rangle_{L^2}=\langle \rho,d\xi\rangle_{L^2} \pm \int_N{\xi\wedge \ast\rho} = \langle \rho,d\xi\rangle_{L^2}
\]
still holds for all $\xi\in C^{\infty}_{-n+1+k}$ since $\ast\rho = \pm r^{-2}dr\wedge\theta\wedge\ast_\Sigma\rho_\infty + O(r^\lambda)$ on the ALC end and the first term restricts to zero on the boundary at infinity $N$. Thus $d^\ast\rho$ is automatically $L^2$--orthogonal to closed forms $\xi$ of the appropriate decay and the proof of Proposition \ref{prop:Closed:to:Coclosed:Laplacian:Mid:deg} goes through unchanged. Thus we can add to $\rho$ a term of the form $d\gamma$ with $\gamma = \tau_1 + \frac{1}{r}\theta\wedge\tau_2 + O(r^{-k+1-\epsilon})$ on the ALC end (for $\epsilon>0$) to make it coclosed. Since $\rho+d\gamma = \rho + O(r^{-k})$ by Remark \ref{rmk:Closed:to:Coclosed:Laplacian:Mid:deg}, we conclude that the closed and coclosed form $\rho+d\gamma$ is still in $L^2$, and $\overline{\Phi}(\rho+d\gamma)=[\rho]$ by the definition of $\overline{\Phi}$. Thus $\overline \Phi$ is surjective.

In the cases $k\geq \frac{n}{2}+1$, we first note that injectivity of $\Psi$ follows Lemma \ref{lem:Injectivity:Compactly:Supported}(ii). Surjectivity then follows from the dimensions of the domain and codomain matching, which can be deduced from (i) and (ii) by duality. Indeed, working at the $L^2$--rate included in the allowed ranges, (i) and (ii) state
\[
L^2\mathcal{H}^{n+1-k}(M) \simeq \begin{dcases} H^{n+1-k}_c(M) & \mbox{if } k>\tfrac{n}{2}+1,\\
\tu{im}\, H^{n+1-k}_c(M)\ra H^{n+1-k}(X) & \mbox{if }k=\tfrac{n}{2}+1\end{dcases}
\]
via $\Phi$ and $\overline{\Phi}$ respectively. Since the Poincar\'e pairing is nondegenerate for $L^2$--integrable forms, we deduce
\[
L^2\mathcal{H}^{k}(M) \simeq \begin{dcases} H^{k}(M) & \mbox{if } k>\tfrac{n}{2}+1,\\
H^{k}(X)/\tu{im}\, H^{k-2}_-(M)\ra H^k(X)  & \mbox{if } k=\tfrac{n}{2}+1 \end{dcases}
\]
via $\Phi$ and $\overline{\Phi}$ respectively. Here we used the fact that the sequence \eqref{eq:pair_exact} is Poincar\'e dual to \eqref{eq:thom_exact}. Moreover, since by Lemma \ref{lem:Injectivity:Compactly:Supported} $\Psi\co L^2\mathcal{H}^k(M)\ra H^k(M)$ remains injective for $k=\frac{n}{2}+1$, by exactness of the sequence \eqref{eq:thom_exact}  we deduce that $\Psi$ induces an isomorphism between $L^2\mathcal{H}^k(M)$ and the image of $H^k(X)$ in $H^k(M)$.
\endproof
\end{theorem}

\begin{remark*}
Theorem \ref{thm:ALC_hodge_fastdecay} provides an independent calculation of the $L^2$--cohomology of ALC manifolds previously carried out by  Hausel--Hunsicker--Mazzeo~\cite[Corollary 1]{HHM}. The microlocal methods of~\cite{HHM} apply to more general fibred boundary metrics, while instead our ALC proof extends without change to the case of a Seifert ALC end.
\end{remark*}

We denote by $\triangle \mathcal{H}^k_\lambda(M)$ the quotient $\mathcal{H}^k_{\lambda+\epsilon}(M)/\mathcal{H}^k_{\lambda-\epsilon}(M)$ for $\epsilon>0$ small enough (so $\triangle\mathcal{H}^k_\lambda (M) = 0$ unless $\lambda$ is an indicial root). We already know that $\triangle \mathcal{H}^k_\lambda(M) = 0$ for $\lambda \in (-n+k, -k)$ if $k < \frac{n}{2}$, for $\lambda \in (-k+1, -n + k -1)$ for $k \geq \frac{n}{2}$, and for $\lambda \in (-k, -k+1)$ for $k = \frac{n+1}{2}$.

By \eqref{eq:harmonic_leading}, $\triangle \mathcal{H}^k_\lambda(M)$ embeds into the direct sum of homogeneous rate $\lambda$ harmonic $k$ and $k-1$-forms on $C(\Sigma)$, and we need to understand the image.

Topological contributions to $\triangle \mathcal{H}^k_\lambda(M)$ only occur for $\lambda\in \{ -k, -k+1\}$ and (by Hodge-star duality) $\lambda\in\{ -n+k-1,-n+k\}$. If $k \leq \frac{n}{2}$ then Proposition \ref{prop:ALC:closed:coclosed:exact}/Remark \ref{rmk:accounted} controls the jumps in the interval $(-k, -k+1)$, and dually in the interval $(-n+k-1, n+k)$ for $k \geq \frac{n}{2}+1$. However, since there cannot be any analogue of Proposition \ref{prop:ALC:closed:coclosed:exact} for rates $\lambda<-\frac{n}{2}$ it is hopeless to describe the jump $\triangle \mathcal{H}^k_\lambda(M)$ for $\lambda<-\frac{n}{2}$ as soon as indicial roots corresponding to the eigenspaces $\mathcal{H}^k_\lambda(\Sigma)$ play a role (\ie $\lambda < -n+k$ for $k \leq \frac{n}{2}$, or $\lambda < -k$ for $k \geq \frac{n}{2}+1$). For even if $\rho$ is a closed compactly supported form then when trying to make it coclosed terms due to indicial roots corresponding to $\mathcal{H}^k_\lambda(\Sigma)$ for $\lambda<-\frac{n}{2}$ could arise to compensate global obstructions. Similarly, when $n$ is odd and $k=\frac{n+1}{2}$ due to Remark \ref{rmk:Closed:to:Coclosed:Laplacian:Mid:deg} we cannot say anything also about the jump as $\lambda$ crosses the indicial root $-k$. Indeed, even if $\rho$ is a closed compactly supported form then Proposition \ref{prop:Closed:to:Coclosed:Laplacian:Mid:deg} only allows us to make $\rho$ closed and coclosed by adding a term $d\gamma$ which could introduce terms of order $O(r^{-k})$ of which we do not have explicit control.

Excluding these cases and observing that the case $k\geq\frac{n}{2}+1$ can be dealt with by an application of the Hodge-$\ast$ operator, the next theorem describes all topological contributions to $\triangle \mathcal{H}^k_\lambda(M)$.  

\begin{theorem}\label{thm:ALC_hodge_jump}
Let $M$ be an ALC manifold of dimension $n+1$ asymptotic to the cone $\tu{C}(\Sigma)$. Denote by $X$ the compactification of $M$ obtained by collapsing the circle fibres at infinity.
\begin{enumerate}
\item If $k<\frac{n}{2}$ then there are isomorphisms
\[
\triangle\mathcal{H}^k_\lambda (M)/f(\mathcal{H}^k_\lambda (\Sigma)) \simeq \begin{dcases}
\tu{im}\, H^{n+1-k}(M)\ra H^{n-k}_-(\Sigma) & \mbox{ if } \lambda=-n+k,\\
\tu{im}\, H^k(X)\ra H^k(\Sigma) & \mbox{ if } \lambda=-k,\\
\tu{im}\, H^k(M)\ra H^{k-1}_-(\Sigma)
& \mbox{ if }\lambda=-k+1,
\end{dcases}
\]
induced by $\Psi\circ\ast$, $\overline{\Phi}$ and $\Psi$, respectively. For $\lambda\in (-k,-k+1)$, $\triangle\mathcal{H}^k_\lambda (M) = f(\mathcal{H}^k_\lambda (\Sigma))$.
\item If $n$ is even and $k=\frac{n}{2}$ then there are isomorphisms
\[
\triangle\mathcal{H}^k_\lambda (M)/f(\mathcal{H}^k_\lambda (\Sigma)) \simeq \begin{dcases}
\tu{im}\, H^k(X)\ra H^k(\Sigma) \oplus \tu{im}\, H^{k+1}(M)\ra H^{k}_-(\Sigma) & \mbox{ if } \lambda=-k,\\
\tu{im}\, H^k(M)\ra H^{k-1}_-(\Sigma) & \mbox{ if }\lambda=-k+1.
\end{dcases}
\]
induced by $\overline{\Phi}\oplus (\Psi\circ\ast)$ and $\Psi$ respectively.
\item If $n$ is odd, $k=\frac{n+1}{2}$ and $\lambda=-k+1$ then there is an isomorphism
\[
\triangle\mathcal{H}^k_\lambda (M) \simeq 
\left( \tu{im}\, H^k(M)\ra H^{k-1}_-(\Sigma)\right)^{\oplus 2} 
\]
induced by $\Psi\oplus (\Psi\circ\ast)$.
\end{enumerate}
\proof
The scheme of proof is similar in all cases. Firstly, for $k<\frac{n}{2}$ we describe the change in $\mathcal{H}^k_{\lambda}(M)$ as $\lambda$ crosses the indicial root $\lambda=-n+k$ as a simple consequence of Theorem \ref{thm:ALC_hodge_fastdecay}. We then consider the possible asymptotic expansions of new elements in $\mathcal{H}^k_\lambda(M)$ as $\lambda$ crosses the indicial roots $\lambda=-k,-k+1$,  deriving natural injective maps to the spaces on the right-hand-sides of the statements. Surjectivity of these maps is then proved using Propositions \ref{prop:Closed:to:Coclosed:Laplacian} and/or \ref{prop:Closed:to:Coclosed:dd*}. The cases $k=\frac{n}{2}$ and $k=\frac{n+1}{2}$ follow similar lines, though some of the details differ, as already apparent in Theorem \ref{thm:ALC_hodge_fastdecay}.

If $k<\frac{n}{2}$ and $\lambda=-n+k$ then Theorem \ref{thm:ALC_hodge_fastdecay}(v) states that $\Psi\circ\ast$ induces an isomorphism between $\mathcal{H}^k_{\lambda+\epsilon}(M)$ and $H^{n+1-k}(M)$. Moreover, Proposition \ref{prop:Indicial:Roots:pforms}(iv) implies that any $\rho\in \mathcal{H}^k_{\lambda+\epsilon}(M)$ has an asymptotic expansion $\ast\rho = \theta\wedge\alpha + O(r^{\lambda-\epsilon})$ with $[\alpha]\in H^{n-k}_-(\Sigma)$ the image of $[\ast\rho]\in H^{n+1-k}(M)$.

We now study the new closed and coclosed forms that appear as $\lambda$ crosses the indicial roots $-k$ and $-k+1$, working modulo $f(\mathcal{H}^k_\lambda(\Sigma))$ for $\lambda=-k+1$. If $k<\frac{n}{2}$ then any $\rho\in \mathcal{H}^k_{\lambda}(M)$ (modulo $f(\mathcal{H}^k_\lambda(\Sigma))$ for $\lambda=-k+1$) is of the form $\rho = \rho_\infty + O(r^{\lambda-\epsilon})$ for some $\epsilon>0$ and
\[
\rho_\infty = \begin{dcases}
\alpha , \quad \alpha\in\mathcal{H}^k(\Sigma) & \mbox{ if } \lambda = -k,\\
\theta\wedge\beta, \quad \beta\in\mathcal{H}^{k-1}_-(\Sigma) & \mbox{ if } \lambda = -k+1.
\end{dcases}
\]
Thus if $\lambda=-k$ then $\overline{\Phi}(\rho)$ is well defined with $[\alpha]\in H^k(\Sigma)$ the image of $\overline{\Phi}(\rho)$ under $H^k(X)\ra H^k(\Sigma)$, while if $\lambda=-k+1$ then $[\beta]\in H^{k-1}_-(\Sigma)$ is the image of $\Psi(\rho)$ under $H^k(M)\ra H^{k-1}_-(\Sigma)$. The maps $\rho\mapsto [\alpha]$ and $\rho\mapsto [\beta]$ then induce canonical injective maps from $\triangle\mathcal{H}^k_\lambda(M)/f(\mathcal{H}^k_\lambda (\Sigma))$ into the claimed subspaces. These maps are also surjective: the existence of a closed form with the desired expansion at infinity is guaranteed by Remark \ref{rmk:Extension:Closed} and Proposition \ref{prop:Closed:to:Coclosed:Laplacian} allows one to modify an arbitrary choice of closed form with prescribed asymptotics into a closed and coclosed one. This completes the proof of (i).

\smallskip

Now consider the case where $n$ is even and $k=\frac{n}{2}$. As before, modulo $f(\mathcal{H}^k_\lambda(\Sigma))$, $\rho\in\mathcal{H}^k_{\lambda} (M)$ has an asymptotic expansion $\rho=\rho_\infty + O(r^{\lambda-\epsilon})$ where 
\[
\rho_\infty = \begin{dcases}
\alpha + \frac{dr}{r}\wedge\beta , \quad \alpha\in\mathcal{H}^k(\Sigma), \beta\in\mathcal{H}^{k-1}(\Sigma) & \mbox{ if } \lambda = -k,\\
\theta\wedge\tau, \quad \beta\in\mathcal{H}^{k-1}_-(\Sigma) & \mbox{ if } \lambda = -k+1.
\end{dcases}
\]
Note that $\frac{dr}{r}\wedge\beta = d(\beta\log{r})$ and
\[
\ast\rho = \pm \theta\wedge\ast_\Sigma\beta \pm \frac{dr}{r}\wedge\theta\wedge\ast_\Sigma\alpha + O(r^{\lambda-\epsilon})= \pm \theta\wedge\ast_\Sigma\beta \pm d\left(\log{r}\,\theta\wedge\ast_\Sigma\alpha\right) + O(r^{\lambda-\epsilon}).
\]
Then $[\alpha]\in H^k(\Sigma)$ is the image of $\overline{\Phi}(\rho)$ under $H^k(X)\ra H^k(\Sigma)$, while $\beta$ and $\tau$ are harmonic representatives of the images of $\Psi(\ast\rho)$ and $\Psi(\rho)$ in $H^{k-1}_-(\Sigma)$ for $\lambda=-k$ and $\lambda=-k+1$, respectively. The maps $\rho\mapsto [\alpha], [\ast_\Sigma\beta],[\tau]$ induce canonical injective maps from $\triangle\mathcal{H}^k_\lambda(M)/f(\mathcal{H}^k_\lambda(\Sigma))$ to these cohomological subspaces. We can prove surjectivity of these maps using Proposition \ref{prop:Closed:to:Coclosed:dd*}. Indeed, by Remark \ref{rmk:Extension:Closed} we can always find a closed form $\rho$ that satisfies the required asymptotic expansions. In order to apply Proposition \ref{prop:Closed:to:Coclosed:dd*} we need to check that $\langle d^\ast\rho,\xi\rangle_{L^2}=0$ for all $\xi\in\mathcal{H}^{k-1}_{-n+1-\lambda+\epsilon}(M)$.
\begin{itemize}
\item When $\lambda=-k+1$, one has $-n+1-\lambda+\epsilon = -\frac{n}{2}+\epsilon$ and therefore $\xi$ is in fact of rate $-\frac{n}{2}-1$: then the integrations by parts $\langle d^\ast\rho,\xi\rangle_{L^2} = \langle \rho,d\xi\rangle_{L^2}=0$ is completely justified.
\item When $\lambda=-k$, one has $-n+1-\lambda+\epsilon = -\frac{n}{2}+1+\epsilon$ and therefore, by part (i) of the Theorem which we have already proved, $\xi = \tau + O(r^{-\frac{n}{2}+1-\epsilon})$ for $[\tau]\in H^{k-1}(X)\ra H^{k-1}(\Sigma)$. Then, using the expansion above for $\ast\rho$, we calculate
\[
\langle d^\ast\rho,\xi\rangle_{L^2} = \pm\int_N{\theta\wedge\ast_\Sigma\beta\wedge\tau}.
\]
The latter integral vanishes by Lemma \ref{lem:Poincare:Duality}(ii).
\end{itemize}

Finally consider the case where $n$ is odd, $k=\frac{n+1}{2}$ and $\lambda=-k+1$. Any $\rho\in\mathcal{H}^k_{\lambda}(M)$ has an expansion $\rho=\rho_\infty + O(r^{\lambda-\epsilon})$ where
\[
\rho_\infty = dr\wedge\alpha + \theta\wedge\beta, \quad \alpha\in\mathcal{H}^{k-1}(\Sigma), \beta\in\mathcal{H}^{k-1}_-(\Sigma)
\]
We calculate
\[
\ast\rho = \pm dr\wedge\ast_\Sigma\beta \pm \theta\wedge\ast_\Sigma\alpha + O(r^{\lambda-\epsilon})
\]
and observe that
\[
dr\wedge\alpha+\theta\wedge\beta = d(r\alpha)+\theta\wedge\beta.
\]
We conclude that $[\ast_\Sigma\alpha],[\beta]\in \tu{im}\, H^k(M)\ra H^{k-1}_-(M)$ are the images of $\Psi(\ast\rho)$ and $\Psi(\rho)$, respectively. Conversely, given $\alpha,\beta$ as above we can always find a closed form $\rho$ with $\rho=\rho_\infty+O(r^{\lambda-\epsilon})$: we can just take the sum of $d(\chi r\alpha)$, where $\chi$ is the usual cutoff function such that $\chi\equiv 1$ on the ALC end, and a closed form on $M$ asymptotic to $\theta\wedge\beta$ whose existence is guaranteed by Remark \ref{rmk:Extension:Closed}(ii). We apply Proposition \ref{prop:Closed:to:Coclosed:dd*} with $\lambda=-\frac{n}{2}+1-\epsilon$. In order to conclude that we can perturb $\rho$ to a closed and coclosed form with the same expansion at infinity, we only need to check that $\langle d^\ast\rho,\xi\rangle_{L^2}=0$ for all $\xi\in \mathcal{H}^{k-1}_{-n+1-\lambda}(M)$. Since $-n+1-\lambda = -\frac{n}{2}+1-\epsilon$, any such $\xi$ has an expansion
\[
\xi = \tau +O(r^{\lambda+1}).
\]
Then a computation shows that
\[
\langle d^\ast\rho,\xi\rangle_{L^2} = \pm\int_N{\theta\wedge\ast_\Sigma\alpha'\wedge\tau}=\pm 2\pi \int_\Sigma{\ast_\Sigma\alpha'\wedge\tau}=0
\]
by Lemma \ref{lem:Poincare:Duality}(ii).
\endproof
\end{theorem}

\part{ALC \gtmfd s: deformation theory and applications}\label{Part2}

In the second part of the paper we study the deformation theory of ALC \gtwo--manifolds, \ie ALC \gtwo--holonomy metrics. Thanks to pioneering work of Joyce \cite{Joyce:Book}, compact \gtwo--manifolds have a nice deformation theory and form smooth moduli spaces locally immersed in $H^3(M)\oplus H^4(M)$ as a Lagrangian graph.  Nordstr\"om \cite{Nordstrom:ACyl:G2} and Karigiannis--Lotay \cite{Karigiannis:Lotay} have extended this theory to the noncompact ACyl and AC (and conically singular) settings respectively. We will now further extend these results to the ALC setting. Given the analytic set-up developed in Part~\ref{Part1}, this generalisation is fairly straightforward, though there are some subtleties regarding how to specialise the general Definition \ref{def:ALC:n} to the $\gtwo$--holonomy setting. Less routine work is contained in Sections~\ref{sec:Symmetries} and~\ref{sec:Dimension} that discuss, respectively, the extension of symmetries from the asymptotic model at infinity
and the calculation of the dimension of the moduli spaces of ALC \gtwo--metrics for many known examples. 

\section{The asymptotic geometry of ALC \gtmfd s}\label{sec:ALC:G2}

In this preliminary section we collect the necessary definitions and computations that will be used in the rest of Parts \ref{Part2} and \ref{Part3} of the paper.

\subsection{\gtmfd s}

In this brief section we collect well-known facts about \gtmfd s, \ie Riemannian $7$-manifolds $(M^7,g)$ with $\tu{Hol}(g)=\gtwo$.

The exceptional Lie group \gtwo~is the group of automorphisms of the octonions $\mathbb{O}$. In particular, \gtwo~acts on $\R^7=\Imag\mathbb{O}$ preserving the standard inner product and the \emph{cross-product} $u\times v = \Imag\, uv$. In fact $\gtwo$~can be defined as the subgroup of $\tu{GL}(7,\R)$ that preserves the $3$-form
\begin{equation}\label{eq:G2:form}
\varphi_0 (u,v,w)=\langle u\times v,w\rangle, \qquad \forall\,  u,v,w\in\R^7.
\end{equation}

A \gtstr~on a smooth $7$-manifold $M$ is the choice of a $3$-form $\varphi$ such that for all $p\in M$ there exists a frame $u\co \R^7\ra T_p M$ with $u^\ast\varphi = \varphi_0$. Because the $GL(7,\mathbb{R})$-orbit of $\varphi_0$ is open, this is an open condition on $\varphi$. Every \gtstr~$\varphi$ on $M$ defines a volume form $\dvol_\varphi$ and Riemannian metric $g=g_\varphi$ on $M$ via
\[
\tfrac{1}{6} (u\lrcorner \varphi) \wedge (v\lrcorner\varphi) \wedge \varphi = g (u,v)\, \dvol_\varphi, \qquad \forall\,  u,v\in TM.
\]
We will denote by $\psi$ the $4$-form $\psi = \ast_\varphi\varphi$.

There is an alternative definition of $\gtwo$--structures in spinorial terms. Indeed, note that since $\gtwo$ is simply connected, every $7$-manifold endowed with a {\gtstr} is automatically spin. Moreover, since the group $\gtwo$ can also be defined as the stabiliser in $\tu{Spin}(7)$ of a nonzero spinor, a {\gtstr} can also be specified via the choice of a spin structure on $M$, \ie a principal $\tu{Spin}(7)$--bundle $P\ra M$ such that $P\times_{\tu{Spin}(7)}\R^7\simeq TM$ where $\R^7$ is the representation of $\tu{Spin}(7)$ induced by the standard representation of $\tu{SO}(7)$, and a nowhere-vanishing spinor field $\Psi$ on $M$. The 3-form $\varphi$ is determined by the spinor $\Psi$ via the formula
\[
\varphi (u,v,w)=\langle u\cdot v\cdot w\cdot \Psi, \Psi\rangle, \qquad \forall u,v,w\in TM,
\]
see for example \cite[Section 2]{Agricola:Chiossi}. Here $u\cdot \Psi$ denotes Clifford multiplication.

We will now collect important identities for differential forms on $7$-manifolds endowed with a $\gtwo$--structure. We refer the reader mostly to Bryant \cite{Bryant:Rmks:G2} for their proof.

\begin{lemma}\label{lem:decomp:forms}
Let $M$ be a $7$-manifold endowed with a {\gtstr} $\varphi$. Then there are pointwise orthogonal decompositions
\[
\Omega^2(M) = \Omega^2_7(M) \oplus \Omega^2_{14}(M), \qquad \Omega^3(M) = \Omega^3_1(M)\oplus \Omega^3_7(M) \oplus \Omega^3_{27}(M)
\]
of differential forms according to irreducible representations of $\gtwo$, where: 
\[
\begin{gathered}
\Omega^2_7(M) = \{ \gamma^\sharp\lrcorner\varphi=\ast(\gamma\wedge\psi)\, |\, \gamma\in\Omega^1(M)\}, \qquad \Omega^3_7(M) = \{ \gamma^\sharp\lrcorner\psi=-\ast(\gamma\wedge\varphi)\, |\, \gamma\in\Omega^1(M)\},\\
\Omega^2_{14}(M)=\{ \tau\in\Omega^2(M)\, |\, \tau\wedge\psi=0\} = \{ \tau\in\Omega^2(M)\, |\, \ast\tau=-\tau\wedge\varphi\},\\
\Omega^3_1(M)=\{ f\varphi\, |\, f\in\Omega^0(M)\}, \qquad 
\Omega^3_{27}(M) = \{ \rho\in\Omega^3(B)\, |\, \rho\wedge\varphi=0=\rho\wedge\psi\}.
\end{gathered}
\]
\end{lemma}
By acting with the Hodge-star operator $\ast$, Lemma \ref{lem:decomp:forms} implies analogous decompositions of the space of $4$-forms and $5$-forms. The decompositions of $3$-forms and $4$-forms can be used to describe the linearisation of the map $\varphi\mapsto \psi$, see \cite[Proposition 4]{Bryant:Rmks:G2}.

\begin{lemma}\label{lem:Linearisation:Hitchin:dual:3form:7d}
The linearisation of the map $\varphi\mapsto \psi$ is
\[
\rho \longmapsto \ast\left( \tfrac{4}{3}\pi_1\rho + \pi_7 \rho - \pi_{27}\rho\right).
\]
\end{lemma}

A crucial invariant associated to a $\gtwo$--structure on a $7$-manifold $M$ is its intrinsic torsion, which measures the failure of the reduction of the frame bundle of $M$ to $\gtwo$ to be parallel with respect to the Levi-Civita connection of the metric induced by the $\gtwo$--structure.

\begin{lemma}\label{lem:Torsion:G2:structures}
Let $M$ be a $7$-manifold endowed with a {\gtstr} $\varphi$. Then there exists a function $\tau_0$, a $1$-form $\tau_1$, $\tau_2\in \Omega^2_{14}(M)$ and $\tau_3\in\Omega^3_{27}(M)$ such that
\begin{align*}
&d\varphi = \tau_0\psi + 3\tau_1\wedge\varphi + \ast \tau_3,\\
&d\psi = 4\tau_1\wedge\psi + \tau_2\wedge\varphi.
\end{align*}
In terms of the definition of the $\gtwo$--structure via a nowhere vanishing spinor field $\Psi$ on $M$, the forms $(\tau_0,\tau_1,\tau_2,\tau_3)$ can be identified by $\nabla\Psi$ via
\[
\nabla_X\Psi = S(X)\cdot \Psi,
\]
where $S\in C^\infty(M,\tu{End}\, TM)$ is determined by $\tau_0,\tau_1,\tau_2,\tau_3$ via the isomorphism of $\gtwo$--representations
\[
\tu{End}\,\R^7 = \Lambda^2\R^7\oplus \left( \R\,\tu{id}\oplus \tu{Sym}^2_0(\R^7)\right) \simeq (\Lambda^2_7 \oplus \Lambda^2_{14}) \oplus (\Lambda^0\oplus \Lambda^3_{27}).
\]
\end{lemma}
For a proof see Bryant~\cite[Proposition 1]{Bryant:Rmks:G2} and for more details on the spinorial description of the intrinsic torsion see \cite[Section 4]{Agricola:Chiossi}.

If $d\varphi=0=d\psi$ (equivalently, $\nabla\Psi=0$), \ie $\tau_i=0$ for all $i$, we say that $\varphi$ is \emph{torsion-free} and that $(M,\varphi)$ is a \emph{\gtmfd}. If $\varphi$ is torsion-free then linear algebra and the representation theory of $\gtwo$ imply that $\varphi$ is parallel for the Levi-Civita connection of $g_{\varphi}$ and therefore the holonomy of $g_{\varphi}$ reduces from $\sorth{7}$ to (a subgroup of) $\gtwo$. We also note that a crucial consequence of the torsion-free condition is the vanishing of the Ricci curvature, see \cite[\S 4.5.3]{Bryant:Rmks:G2}.

We conclude with some useful identities for functions and $1$-forms on a {\gtmfd} $(M,\varphi)$. In the lemma below the $\tu{curl}$ operator acting on $1$-forms is defined by
\begin{equation}\label{eq:curl}
\tu{curl}\,\gamma =\ast(d\gamma\wedge \psi).
\end{equation}
The $1$-form $\tu{curl}\, \gamma$ can be identified with the 
$\Omega^2_7$-component of the $2$-form $d \gamma$ (up to a proportionality constant). 

\begin{lemma}\label{lem:identities:0:1:forms:torsion:free} Let $(M,\varphi)$ be a $\gtwo$--manifold. For $f\in\Omega^0(M)$ and $\gamma\in\Omega^1(M)$ the following identities hold:
\begin{enumerate}
\item $\pi_1\left( d\ast d (f\varphi)\right) = \tfrac{2}{7}(\triangle f)\psi$.
\item\label{lem:identities:0:1:forms:diff} $\pi_1 d (\gamma^\sharp\lrcorner\varphi) = -\tfrac{3}{7}(d^\ast\gamma)\varphi$ and $\pi_7d (\gamma^\sharp\lrcorner\varphi) = \tfrac{1}{2}\ast\big( \tu{curl}\,\gamma \wedge \varphi\big)$.
\item\label{lem:identities:0:1:forms:7:component} $ d(\gamma^\sharp\lrcorner\varphi) + \ast  d(\gamma^\sharp\lrcorner\psi)=d\ast(\gamma\wedge\psi) - d^\ast (\gamma\wedge\varphi) = \ast(\tu{curl}\,\gamma\wedge\varphi )- (d^\ast\gamma)\varphi\in\Omega^3_{1\oplus 7}(M)$.
\item\label{lem:identities:0:1:forms:Dirac} The Dirac operator $\slashed{D}$ can be identified with either of the following two operators:
\[
\begin{gathered}
\slashed{D}\co \Omega^0(M)\oplus \Omega^1(M)\ra \Omega^0(M)\oplus \Omega^1(M), \qquad (f,\gamma)\mapsto (d^\ast\gamma, d f + \tu{curl}\, \gamma),\\
\slashed{D}\co \Omega^0(M)\oplus \Omega^1(M)\ra \Omega^3_{1\oplus 7}(M), \qquad (f,\gamma)\mapsto \pi_{1\oplus 7}\Big(\ast d (f\varphi) + d \ast(\gamma\wedge\psi)\Big).
\end{gathered}
\]
\item\label{lem:identities:0:1:forms:gauge:fixing} $\pi _7 d^\ast d (\gamma^\sharp\lrcorner\varphi)= \left(dd^\ast \gamma + \tfrac{2}{3} d^\ast d\gamma\right)^\sharp\lrcorner\varphi$.
\end{enumerate}
\proof
The lemma can be deduced from \cite[Proposition 3]{Bryant:Rmks:G2} and the identification of the Dirac operator of a $\gtwo$--manifold $M$ via the isomorphism between the spinor bundle and $\R\oplus TM$, see for example \cite[Equation (6.2)]{Nordstrom:ACyl:G2}. The last identity is \cite[Proposition 2.24]{Karigiannis:Lotay}.
\endproof
\end{lemma}

\begin{remark}\label{rmk:identities:0:1:forms:torsion:free}
Taking $L^2$--inner products with compactly supported forms and using part (iii) of Lemma \ref{lem:identities:0:1:forms:torsion:free} yields the following additional useful identity: if $\rho\in\Omega^3_{27}$ then $\pi_7 d^\ast\rho=0$ if and only if $\pi_7 d\rho=0$. Note that this can be regarded as a linearised version of the torsion constraints of Lemma \ref{lem:Torsion:G2:structures}.
\end{remark}

\subsection{ALC \gtmfd s}

In this section we specialise Definition \ref{def:ALC:n} to the case of ALC \gtmfd s. As we will see, in the \gtwo--setting it is necessary to refine the asymptotic model $\BC (\Sigma)$ and it is natural to prescribe the asymptotic geometry of an ALC {\gtmfd} in terms of this refined model.

The end of an ALC manifold $(M,g)$ admits an almost parallel vector field: indeed, by \eqref{eq:Levi:Civita}, $\nabla\xi = O(r^{-2})$, where $\xi$ is the vector field that generates the circle action (in the dihedral case, defined only up to a sign) on the end. Thus, if $\tu{Hol}(g)\subseteq\gtwo$ then the (double cover of the) tangent cone at infinity $\tu{C}(\Sigma)$ of $(M,g)$ must have (restricted) holonomy contained in $\sunitary{3}=\gtwo\cap \tu{GL}(6,\R)$, \ie it must be a Calabi--Yau cone of complex dimension $3$.

We first define an ALC \gtwo{} end $\BC(\Sigma)$ of \emph{cyclic type}. In this case, the above discussion means we need to take $\Sigma$ be a closed Sasaki--Einstein $5$-manifold, \ie a closed Riemannian $5$-manifold $(\SigmP,g_\SigmP)$ such that the Riemannian cone $\tu{C}(\SigmP)$ over $\SigmP$ is a Calabi--Yau cone. More precisely, a Sasaki--Einstein structure on a $5$-manifold $\Sigma$ is a special type of \sutwostr: this is a $4$-tuple $(\eta, \omega_1,\omega_2,\omega_3)$ of differential forms on $\Sigma$ satisfying the following algebraic constraints. The 1-form $\eta$ vanishes nowhere and therefore defines a codimension $1$ distribution $\ker{\eta}$ that we will denote by $\mathcal{H}$. $(\omega_1,\omega_2,\omega_3)$ is a triple of $2$-forms that at every point span a definite subspace of $\Lambda^2\mathcal{H}^\ast$: $\eta\wedge \omega_1^2\neq 0$ and
\[
\omega_i\wedge\omega_j = \delta_{ij}\,\omega_1^2
\]
for $i,j=1,2,3$. Since for any oriented $4$-dimensional vector space $V$ the $3$-dimensional space $\Lambda^2_+V^\ast$ of self-dual $2$-forms on $V$ is naturally oriented, it makes sense to require  further that $(\omega_1,\omega_2,\omega_3)$ is an oriented basis of the subspace $\Lambda^2_+\mathcal{H}^\ast$ of $\Lambda^2\mathcal{H}^\ast$ that they span. Since $\sunitary{2}\subset\sorth{4}\subset\sorth{5}$ every {\sutwostr} induces a Riemannian metric $g_\SigmP$. The {\sutwostr} $(\eta,\omega_1,\omega_2,\omega_3)$ is called \emph{Sasaki--Einstein} if
\begin{equation}\label{eq:Sasaki:Einstein}
d\eta = 2\omega_1, \qquad d\omega_2 = -3\eta\wedge\omega_3, \qquad d\omega_3 = 3\eta\wedge\omega_2.
\end{equation}
These equations are equivalent to requiring that the conical {\suthreestr} on $\tu{C}(\SigmP)$ defined by
\begin{equation}\label{eq:conical:CY}
\omega_\tu{C} = rdr\wedge\eta + r^2\omega_1, \qquad \Omega_\tu{C} = r^2(dr+ir\eta)\wedge (\omega_2 + i\omega_3)
\end{equation}
is torsion-free, \ie that $\omega_\tu{C}$ and $\Omega_\tu{C}$ are both closed, and therefore defines a conical Calabi--Yau structure on $\tu{C}(\SigmP)$. Since every Calabi--Yau manifold is Ricci-flat, we immediately deduce that every Sasaki--Einstein $5$-manifold is Einstein with positive scalar curvature $\text{Scal}(g_\SigmP)=20$. In particular, complete Sasaki--Einstein $5$-manifolds are compact with finite fundamental group.

Let $\pi\co N\ra \SigmP$ be a principal circle bundle. Fix a connection $\theta$ on $N$. We can think of the pair $(N,\theta)$ as defined on the cone $\tu{C}(\SigmP)$ by radial extension and ask for $\theta$ to be a \emph{Hermitian--Yang--Mills} (HYM) connection with respect to the conical Calabi--Yau structure \eqref{eq:conical:CY}, \ie we ask that $d\theta$ be a primitive $(1,1)$--form on the cone $\tu{C}(\SigmP)$: explicitly,
\[
d\theta\wedge\omega_\tu{C}^2=0=d\theta\wedge\Omega_\tu{C}.
\]
Using the expression \eqref{eq:conical:CY} for the conical Calabi--Yau structure $(\omega_\tu{C},\Omega_{\tu{C}})$, it is easy to see that the radial extension of $\theta$ to $\tu{C}(\SigmP)$ is HYM if and only if the 2-form $d\theta$ on $\SigmP$ satisfies $d\theta\wedge\omega_i =0$ for $i=1,2,3$. In other words, $d\theta$ is a section of $\Lambda^2_-\mathcal{H}^\ast$, or equivalently it is a primitive $(1,1)$--form with respect to the transverse K\"ahler structure on $\mathcal{H}$ induced by the closed non-degenerate $2$-form $\omega_1$ and the metric.   

It is easy to produce such a HYM connection. Every harmonic $2$--form on a Sasaki--Einstein manifold is basic primitive of type $(1,1)$, \ie a section of $\Lambda^2_-\mathcal{H}^\ast$. Indeed, by Boyer--Galicki \cite[Proposition 7.4.13]{Boyer:Galicki} every harmonic $2$-form on a Sasaki manifold is primitive and basic harmonic. On the other hand a Bochner formula due to Goto \cite[Lemma 5.3]{Goto:Crepant} shows that basic harmonic $2$-forms on a Sasaki--Einstein manifold must be of type $(1,1)$. Thus given the $U(1)$--bundle $N\ra\SigmP$, the harmonic representative of $c_1 (N)\in H^2(\SigmP)$ is the curvature of a HYM connection $\theta$ on $N$. Moreover, if we assume that $\SigmP$ is simply connected, then every $\unitary{1}$--bundle $N$ admits a unique HYM connection.

Using the conical Calabi--Yau structure $(\omega_\tu{C}, \Omega_\tu{C})$ and the HYM connection $\theta$,  on $\tu{BC}(\SigmP)$ we now define a $\gtwo$--structure
\begin{subequations}\label{eq:G2:infinity}
\begin{equation}
\varphi_\BC = \theta\wedge\omega_\tu{C} + \Real\Omega_\tu{C}.
\end{equation}
The metric induced by $\varphi_\BC$ coincides with $g_\BC$ from \eqref{eq:FB:N}, with $\ell=1$ (the latter can always be achieved by scaling) and the dual 4-form of $\varphi_\BC$ is given by
\begin{equation}
\psi_\BC = -\theta\wedge\Imag\Omega_\tu{C} + \tfrac{1}{2}\omega_\tu{C}^2.
\end{equation}
\end{subequations}

The \gtwo--structure $\varphi_\BC$ therefore captures the leading-order asymptotic behaviour of any ALC \gtwo--manifold asymptotic to $\BC (\SigmP)$. A drawback is that the $3$-form $\varphi_\BC$ is not closed except when the connection $\theta$ is flat (while $\psi_\BC$ is always closed since $d\theta\wedge\Imag\Omega_\tu{C}=0$
because $\theta$ is HYM). However, we observe that there is a canonical correction to $\varphi_\BC$ to make it closed:
\begin{subequations}\label{eq:G2:infinity:closed}
\begin{equation}
\varphi '_\BC = \theta\wedge\omega_\tu{C} + \Real\Omega_\tu{C} - \tfrac{1}{2}r^2 \eta\wedge d\theta.
\end{equation}
The correction term is of order $O(r^{-1})$ with respect to $g_\BC$ and therefore $\varphi '_\BC$ is still a~\gtstr~for $r$ sufficiently large. Using the fact that $d\theta\wedge\omega_i=0$ for $i=1,2,3$, one can verify that the correction term we added satisfies $(\varphi'_\BC - \varphi_\BC)\wedge\varphi_\BC =0= (\varphi'_\BC - \varphi_\BC)\wedge\psi_\BC$. In other words, $\varphi'_\BC - \varphi_\BC\in \Omega^3_{27}$ with respect to $\varphi_\BC$ and therefore, using Lemma \ref{lem:Linearisation:Hitchin:dual:3form:7d}, we can compute
\begin{equation}
\psi'_\BC = \ast_{\varphi_\BC}\varphi_\BC - \ast_{\varphi_\BC}(\varphi'_\BC-\varphi_\BC) + Q(\varphi'_\BC - \varphi_\BC)= \psi_\BC - \tfrac{1}{2}\theta\wedge rdr \wedge d\theta +O(r^{-2}),
\end{equation}
\end{subequations}
where we also used the fact that $\ast_\SigmP d\theta = -\eta\wedge d\theta$ since $d\theta$ is a basic primitive $(1,1)$-form. Here $Q(\varphi'_\BC - \varphi_\BC)$ contains all the nonlinear terms, hence the estimate $Q(\varphi'_\BC - \varphi_\BC)=O(r^{-2})$ holds with analogous estimates for all derivatives.

\begin{remark}\label{rmk:G2:infinity:pertubed}
In particular, it follows that the torsion of the \gtwo--structure $\varphi'_\BC$ is of order $O(r^{-3})$. Indeed, $d\psi'_\BC = \tfrac{1}{2} r dr\wedge d\theta\wedge d\theta + O(r^{-3})$.
\end{remark}

Now we consider ALC \gtwo{} ends $\BC(\Sigma)$ of \emph{dihedral type}. In this case $\Sigma$ is non-orientable, and the structure we need is a Sasaki--Einstein structure $(\eta,\omega_1,\omega_2,\omega_3)$ on its oriented double cover~$\SigmP$, such that the nontrivial deck transformation $\check \iota$ acts by
\[
\check{\iota}^\ast(\eta,\omega_1,\omega_2,\omega_3)=(-\eta, -\omega_1, \omega_2, -\omega_3).
\]
Using the expression \eqref{eq:conical:CY} for $(\omega_\tu{C},\Omega_\tu{C})$ in terms of the Sasaki--Einstein structure, that is equivalent to the extension of $\check{\iota}$ to $C(\SigmP)$ acting by
\[
\check{\iota}^\ast\omega_\tu{C}=-\omega_\tu{C}, \qquad \check{\iota}^\ast\Omega_\tu{C}=\overline{\Omega}_\tu{C}
\]
(in particular, this makes $\check \iota$ an \emph{antiholomorphic involution} on $C(\SigmP)$) and in turn to the standard involution $\iota$ on $\tsp$ preserving $\varphi_\BC$ from \eqref{eq:G2:infinity}. Taking into account that $\check{\iota}^\ast d\theta=-d\theta$, the standard involution moreover preserves the canonical closed correction $\varphi'_\BC$ defined in \eqref{eq:G2:infinity:closed}.

We are finally able to give the definition of an ALC $\gtwo$--manifold.

\begin{definition}\label{def:ALC:G2}
Let $(M^7,\varphi)$ be a smooth 7-manifold endowed with a $\gtwo$--structure $\varphi$.
\begin{enumerate}
\item $(M,\varphi)$ is an \emph{ALC $\gtwo$--manifold of cyclic type} asymptotic to $\BC (\Sigma)$ with rate $\nu<-1$ if $\Sigma$ is a Sasaki-Einstein 5-manifold, there exists a compact set $K\subset M$, a positive number $R>0$ and a diffeomorphism $f\co \BC(\Sigma) \cap \{ r> R\} \ra M\setminus K$ such that for all $j\geq 0$
\[
|\nabla_{\varphi'_\BC}^j (f^\ast \varphi - \varphi'_\BC) |_{g_{\varphi'_\BC}}=O(r^{\nu-j})
\]
for $\varphi'_\BC$ defined by the Sasaki-Einstein structure as in \eqref{eq:G2:infinity:closed}.
\item $(M,\varphi)$ is an \emph{ALC $\gtwo$--manifold of dihedral type} asymptotic to $\BC (\Sigma)$ with rate $\nu<-1$ if $\Sigma$ is the quotient of a Sasaki-Einstein 5-manifold $\SigmP$ by an anti-holomorphic involution, there exists a compact set $K\subset M$, a positive number $R>0$ and a double cover $f_+\co \BC(\Sigma^+) \cap \{ r> R\} \ra M\setminus K$ such that the group of deck transformations is generated by the standard involution $\iota$ of $\BC(\SigmP)$ and for all $j\geq 0$
\[
|\nabla_{\varphi'_\BC}^j (f_+^\ast \varphi - \varphi'_\BC) |_{g_{\varphi'_\BC}}=O(r^{\nu-j}).
\]
\end{enumerate}
\end{definition}

\begin{remark*}
Since $\varphi'_\BC = \varphi_\BC + O(r^{-1})$ we can equivalently use the metric $g_\BC = g_{\varphi_\BC}$ to compute covariant derivatives and norms.
\end{remark*}

Note that with this definition every ALC $\gtwo$--manifold is an ALC manifold of rate $\tau=-1$ according to Definition \ref{def:ALC:n} and that $\nu$ in Definition \ref{def:ALC:G2} measures the rate of subleading corrections beyond the canonical closed modification $\varphi'_\BC$ of the model $\varphi_\BC$. It is not clear a priori whether every complete noncompact $\gtwo$--holonomy metric which is ALC in the sense of Definition \ref{def:ALC:n} must satisfy the asymptotic conditions of Definition \ref{def:ALC:G2}, but all known examples do. We note that the strong decay assumption $\nu<-1$ does simplify the deformation theory, as gauge-fixing conditions for the action of diffeomorphisms would be more complicated otherwise (see \cite[Theorem 5.6]{Karigiannis:Lotay} for this discussion in the AC setting).

\begin{remark*}
Our definition assumes $\ell=1$, but the general case $\ell>0$ can be achieved by scaling. Since we are interested in the moduli space of ALC $\gtwo$--holonomy metrics, fixing this scaling freedom by setting $\ell=1$ is natural. 
\end{remark*}

\subsection{Indicial roots}\label{sec:Indicial:Roots:CY:cone}

By Theorems \ref{thm:Fredholm}, \ref{thm:Index:jump} and \ref{thm:Fredholm:2nd:order}, in order to understand the Fredholm property and calculate the index of differential operators on ALC \gtmfd s it is necessary to study the induced operators on the tangent cones $\tu{C}(\Sigma)$ at infinity, which is a Calabi--Yau cone in the cyclic case,  and a quotient $\tu{C}(\SigmP)/\check{\iota}$ of a Calabi--Yau cone in the dihedral case. We have given a careful exposition of the indicial roots of natural differential operators on $3$-dimensional Calabi--Yau cones in \cite[Section 4]{FHN:ALC:G2:from:AC:CY3}. Because of the spectral properties of Sasaki--Einstein 5-manifolds, these results improve on the calculation of indicial roots of the Hodge Laplacian and first-order operator $d+d^\ast$ described for an arbitrary $n$-dimensional cone in Proposition \ref{prop:Indicial:Roots:pforms}. For the convenience of the reader we will now state all the results we need in this paper and refer back to \cite{FHN:ALC:G2:from:AC:CY3} for their proofs.

Let $\SigmP$ be a Sasaki--Einstein $5$-manifold and consider the Calabi--Yau cone $\tu{C}(\SigmP)$. We first make a couple of remarks regarding spectral properties of $\SigmP$, as described in \cite[Proposition 4.9]{FHN:ALC:G2:from:AC:CY3}.

First of all, the Lichnerowicz--Obata Theorem states that the first nonzero eigenvalue of the scalar Laplacian on $\SigmP$ is strictly larger than $5$ unless $\SigmP$ has constant curvature, see the first part of \cite[Proposition 4.9]{FHN:ALC:G2:from:AC:CY3} and references therein. For the round $5$-sphere the first nonzero eigenvalue is exactly $5$, with $6$-dimensional eigenspace corresponding to the restriction of linear functions on~$\R^6$. If $\SigmP=\Sph^5/\Gamma$ is a \emph{smooth} nontrivial finite quotient of the round $5$-sphere $\Sph^5$ then $\Gamma$ cannot have any nontrivial fixed point on~$\R^6$ and therefore the first eigenvalue on~$\SigmP$ is strictly larger than $5$. On the other hand, if $\SigmP$ is allowed to be an orbifold $\Sph^5/\Gamma$ (but $N\ra \SigmP$ is a principal orbibundle with smooth total space) then there could be $\Gamma$--invariant eigenfunctions with eigenvalue $5$. We will say that $\SigmP$ (smooth or an orbifold) is \emph{exceptional} if there are nontrivial eigenfunctions with eigenvalue~$5$. Unless specified otherwise, in this section we will assume that the Sasaki--Einstein 5-manifold $\SigmP$ is \emph{not} exceptional. 

\begin{remark*}
In the dihedral case the involution $\check{\iota}$ must necessarily act freely on $\SigmP$, otherwise the standard involution $\iota$ would have fixed points on $\BC (\SigmP)$. Therefore the Laplacian acting on $\check{\iota}$--invariant functions on $\SigmP$ always has first nonzero eigenvalue strictly larger than $5$. However, $\check{\iota}$--anti-invariant functions also play a role, for example when studying indicial roots for $\BC (\SigmP)$ on $1$-forms, which can be of the form $\gamma = u\, \theta$ for a function $u$.
\end{remark*}

Besides the exceptional Sasaki--Einstein spaces, a second distinguished class of Sasaki--Einstein manifolds is that of \emph{regular} Sasaki--Einstein manifolds, \ie Sasaki--Einstein manifolds where the vector field dual to the contact 1-form $\eta$, usually called the Reeb vector field, generates a free circle action.

\begin{lemma}[last part of {\cite[Proposition 4.9]{FHN:ALC:G2:from:AC:CY3}}]
\label{lem:regular_se}
For a regular Sasaki--Einstein 5-manifold $\SigmP$, the first nonzero eigenvalue of the Laplacian acting on coclosed $2$-forms is strictly greater than 4.
\end{lemma}

Building on these spectral properties, one obtains the following refined statements on the excluded indicial roots of the Hodge Laplacian and the first-order operator $d+d^\ast$, improving on the ones stated in Proposition \ref{prop:Indicial:Roots:pforms}.

\begin{prop}\label{prop:Harmonic:Functions:Cone} 
Let $u$ be a harmonic function on $\tu{C}(\SigmP)$ homogeneous of order $\lambda$. Then $u=0$ if $\lambda\in [-5,1]\setminus \{ -4,0\}$ and $u=Kr^\lambda$ for some $K\in\R$ if $\lambda=-4,0$.
\end{prop}

\begin{remark}\label{rmk:Harmonic:Functions:Cone}
If $\SigmP$ is exceptional then replacing the closed interval $[-5,1]$ with the open interval $(-5,1)$ yields a valid conclusion.
\end{remark}

In the following results, we will denote by $\mathcal{K}(\SigmP)$ the space of Killing fields on the Sasaki--Einstein manifold $\SigmP$ and by $\mathcal{K}_0(\SigmP)$ the subspace of those Killing fields that preserve the Sasaki--Einstein structure and not only the metric.

\begin{prop}\label{prop:Harmonic:1Forms:Cone}
Let $\gamma$ be a harmonic $1$-form homogeneous of order $\lambda$. Then
\[
\gamma = \begin{cases}
Krdr + d\bigl( \tfrac{1}{2}r^2\alpha \bigr) + r^2\beta & \text{if } \lambda=1,\\
d\bigl( \tfrac{1}{\lambda+1}r^{\lambda+1}\alpha \bigr), \text{ where }\triangle\alpha=(\lambda +5)(\lambda +1)\alpha & \text{if } \lambda \in (0,1),\\
0 & \text{if } \lambda\in [-4,0],\\
r^\lambda \alpha\, dr -\frac{r^{\lambda+1}}{\lambda+3} d\alpha, \text{ where }\triangle\alpha = (\lambda+3)(\lambda-1)\alpha & \text{if } \lambda\in (-5,-4),\\
Kr^{-5}dr + (r^{-5}\alpha dr + \tfrac{1}{2}r^{-4}d\alpha) + r^{-4}\beta & \text{if } \lambda = -5.
\end{cases}
\]
where $K\in\R,\, \triangle\alpha=12\alpha$ and $\beta^\flat\in\mathcal{K}(\SigmP)$ in the first and last cases.
Moreover, a harmonic $1$-form $\gamma$ homogeneous of order $\lambda\in [-5,1]$ is coclosed if and only if $\lambda\in [-4,1)$, or $\lambda=1$ and $\gamma$ does not contain a term of the form $Krdr$ for $K\in\R$, or $\lambda=-5$ and $\gamma = Kr^{-5}dr + r^{-4}\beta$ for some $K\in\R$ and $\beta^\flat\in\mathcal{K}(\SigmP)$.

Finally, there are no homogeneous $1$-forms $\gamma$ of order $\lambda\in [-4,0]$ such that $\triangle\gamma - \tfrac{1}{3}d^\ast d\gamma=0$.
\end{prop}
The last statement is relevant to the indicial roots of the operator in Lemma \ref{lem:identities:0:1:forms:torsion:free}\ref{lem:identities:0:1:forms:gauge:fixing}. 

\begin{remark}\label{rmk:Harmonic:1Forms:Cone}
If $\SigmP$ is exceptional then the conclusions remain valid provided every interval containing/excluding an endpoint $-4$ or $0$ is replaced by the corresponding open/closed interval. 
\end{remark}

Analogously to Lemma \ref{lem:identities:0:1:forms:torsion:free}\ref{lem:identities:0:1:forms:Dirac}, the spinor representation of $\tu{Spin}(6)$ can be identified with $\R\oplus\R\oplus\R^6$ as an $\sunitary{3}\subset \tu{Spin}(6)$ representation and under this identification the Dirac operator of a Calabi--Yau 3-fold (in particular, the Dirac operator of $\tu{C}(\SigmP)$) can be identified with the operator from $ \Omega^0\oplus\Omega^0\oplus\Omega^1$ into itself defined by 
\[
(f,g,\gamma)\longmapsto \left(d^\ast\gamma, -d^\ast J\gamma , \tu{curl}_\Omega\gamma + df-Jdg\right),
\]
where $\tu{curl}_\Omega\gamma = \ast (d\gamma\wedge\Real\Omega)$. For more details, see \cite[Lemma 2.22]{FHN:ALC:G2:from:AC:CY3}.

\begin{prop}\label{prop:Dirac:cone}
Let $(f,g,\gamma)$ be a harmonic spinor on the Calabi--Yau cone $\tu{C}(\SigmP)$, where $f,g,\gamma$ are homogeneous of order $\lambda$. Then
\[
(f,g,\gamma) = \begin{cases}
\bigl( 0,0,r^2\beta + d\bigl( \tfrac{1}{2}r^2\alpha\bigr) \bigr) \text{ where }\beta^\flat\in\mathcal{K}_0(\SigmP) \text{ and }\triangle\alpha=12\alpha & \text{if } \lambda=1,\\
\bigl( 0,0, d\bigl( \tfrac{1}{\lambda+1}r^{\lambda+1}\alpha\bigr)\bigr) \text{ where }\triangle\alpha=(\lambda+5)(\lambda+1)\alpha & \text{if } \lambda\in (0,1),\\
\left( K_1,K_2,0\right) \text{ where }K_1,K_2\in\R & \text{if } \lambda=0,\\
\left( 0,0,0\right) & \text{if } \lambda \in (-5,0),\\
\left( 0,0,K_1 r^{-5}dr + K_2 r^{-4}\eta \right) \text{ where }K_1,K_2\in\R & \text{if } \lambda=-5.
\end{cases}
\]
\end{prop}

\begin{remark}\label{rmk:Dirac:cone}
If $\SigmP$ is exceptional then there are additional elements in the kernel of the Dirac operator of the form $(0,0,d(r\alpha))$ where $\alpha$ is an eigenfunction of eigenvalue $5$: in terms of spinor fields these are just the additional parallel spinors arising from the fact that $\tu{C}(\SigmP)$ splits off some Euclidean factor when $\SigmP$ is exceptional.
\end{remark}

\begin{prop}\label{prop:Harmonic:2forms:cone}
Let $\gamma$ be a harmonic $2$-form on $\tu{C}(\SigmP)$ homogeneous of order $\lambda$. Decompose $\gamma = \gamma_1+\gamma_2+\gamma_3+\gamma_4$ as in \cite[Theorem A.2]{FHN:ALC:G2:from:AC:CY3}.
\begin{enumerate}
\item $\gamma_1=r^{\lambda+1}dr\wedge df$ where $f$ is a function on $\SigmP$ satisfying $\triangle f=\lambda(\lambda+4)f$. In particular, $\gamma_1=0$ for all $\lambda\in [-5,1]$.
\item $\gamma_2=0$ for all $\lambda\in (-6,0)$. If $\lambda=-6$ or $\lambda=0$ then $\gamma_2 = d\left( \tfrac{1}{\lambda+2}r^{\lambda+2}\alpha\right)$, where $\alpha^\flat\in\mathcal{K}(\SigmP)$.
\item $\gamma_3=0$ for all $\lambda\in (-4,2)$. If $\lambda=-4$ or $\lambda=2$ then $\gamma_3 = r^{\lambda+1}dr\wedge\alpha - \tfrac{1}{\lambda+2}r^{\lambda+2}d\alpha$, where $\alpha^\flat\in\mathcal{K}(\SigmP)$.
\item $\gamma_4= r^{\lambda+2}\beta$ for a coclosed $2$-form $\beta$ on $\SigmP$ satisfying $\triangle\beta = (\lambda+2)^2\beta$. In particular, if $\SigmP$ is regular then $\gamma_4=0$ for all $\lambda\in [-4,0]\setminus\{ -2\}$ by Lemma \ref{lem:regular_se}.
\end{enumerate}
Finally, the only harmonic $2$-forms on $\tu{C}(\SigmP)$ which are polynomials in $\log{r}$ with coefficients given by homogeneous $2$-forms are of the form $\tau_1\,\log{r} + \tau_0$ for harmonic $2$-forms $\tau_0,\tau_1$ on $\SigmP$.
\end{prop}

\begin{remark}\label{rmk:Harmonic:2forms:cone}
When $\SigmP$ is exceptional the only modification needed is in case (i), where the closed interval $[-5,1]$ is replaced by the open interval $(-5,1)$.
\end{remark}

\begin{prop}\label{prop:Harmonic:3forms:Cone}
Let $\gamma$ be a harmonic $3$-form on $\tu{C}(\SigmP)$ homogeneous of order $\lambda$. Decompose $\gamma = \gamma_1+\gamma_2+\gamma_3+\gamma_4$ as in \cite[Theorem A.2]{FHN:ALC:G2:from:AC:CY3}.
\begin{enumerate}
\item $\gamma_1=0$ for all $\lambda\in (-5,1)\setminus\{-3,-1\}$. If $\lambda=-1$ or $\lambda=-3$ then $\gamma_1 = r^{\lambda+2}dr\wedge \alpha$ where $\alpha$ is a harmonic $2$-form. If $\lambda=1$ or $\lambda=-5$ then $\gamma_1 = r^{\lambda+2}dr\wedge d\beta$ for some $\beta^\flat\in\mathcal{K}(\SigmP)$.
\item $\gamma_2 = d\left( \tfrac{1}{\lambda+3}r^{\lambda+3}\alpha\right)$, where $\alpha$ is a coexact $2$-form on $\SigmP$ satisfying $\triangle\alpha = (\lambda+3)^2\alpha$. In particular, if $\SigmP$ is regular then $\gamma_2=0$ for all $\lambda\in [-5,-1]$  by Lemma \ref{lem:regular_se}.
\item $\gamma_3 = r^{\lambda+2}dr\wedge\alpha - \tfrac{1}{\lambda+1}r^{\lambda+3}d\alpha$, where $\alpha$ is a coexact $2$-form on $\SigmP$ satisfying $\triangle\alpha = (\lambda+1)^2\alpha$. In particular, if $\SigmP$ is regular then $\gamma_3=0$ for all $\lambda\in [-3,1]$ by Lemma \ref{lem:regular_se}.
\item $\ast\gamma_4$ satisfies the same conditions as $\gamma_1$.
\end{enumerate}
Finally, if $\SigmP$ is a regular Sasaki--Einstein $5$-manifold then there are no harmonic $3$-forms on $\tu{C}(\SigmP)$ which are nontrivial polynomials in $\log{r}$ with coefficients in the space of homogeneous $3$-forms.
\end{prop}

Finally we state the indicial roots of the first-order operator $d+d^\ast$. We will be particularly interested in understanding closed and coclosed even-degree forms of rate $-2$ and odd-degree forms of rate~$-3$.

\begin{prop}\label{prop:Closed:Even:Forms}
Let $\gamma$ be a closed and coclosed form of even degree on $\tu{C}(\SigmP)$ homogeneous of rate $\lambda=-2$. Then $\gamma$ only has components of pure degree $2$ and $4$, both individually closed and coclosed:
\[
\gamma = \tau_1 + rdr\wedge \eta\wedge\tau_2
\]
for harmonic $2$-forms $\tau_1,\tau_2$ on $\SigmP$. Moreover, there are no closed and coclosed $2$-forms on $\tu{C}(\SigmP)$ homogeneous of rate $\lambda\in (-6,0)\setminus \{ -2\}$ and if $\SigmP$ is a regular Sasaki--Einstein $5$-manifold then there are no closed and coclosed even-degree forms on $\tu{C}(\SigmP)$ homogeneous of rate $\lambda\in (-4,0) \setminus\{ -2\}$.  
\end{prop}
 
\begin{prop}\label{prop:Closed:Odd:Forms}
Let $\gamma$ be a closed and coclosed form of odd degree on $\tu{C}(\SigmP)$ homogeneous of rate $\lambda\in [-4,0]$. Then $\gamma$ is of pure degree $3$. Moreover, if $\lambda=-3$ then
\[
\gamma = \eta\wedge\tau_1 + \frac{dr}{r}\wedge\tau_2
\]
for harmonic $2$-forms $\tau_1,\tau_2$ on $\SigmP$ and if $\lambda\in [-4,0]\setminus \{ -3\}$ then $\gamma = d\bigl( \frac{r^{\lambda+3}}{\lambda+3}\alpha\bigr)$, where $\alpha$ is a coclosed $2$-form on $\SigmP$ satisfying $\triangle\alpha=(\lambda+3)^2\alpha$. In particular, if $\SigmP$ is a regular Sasaki--Einstein $5$-manifold then there are no closed and coclosed odd-degree forms on $\tu{C}(\SigmP)$ homogeneous of rate $\lambda\in [-4,-1]\setminus\{ -3\}$.
\end{prop}

Propositions \ref{prop:Harmonic:3forms:Cone}, \ref{prop:Closed:Even:Forms} and \ref{prop:Closed:Odd:Forms} hold unchanged when $\SigmP$ is exceptional.

\section{Deformation theory}
\label{sec:def_theory}

This section has two goals. The first is to develop the deformation theory of ALC $\gtwo$--manifolds. The second goal is to then consider torsion-free ALC \gtwo--structures on an exterior domain in $\BC(\Sigma)$ and prove that they are all exponentially decaying to circle-invariant $\gtwo$--structures (as stated in Theorem \ref{mthm:Asymptotics}). The latter result will have important applications in Section \ref{sec:Symmetries}; we have included this result in this section as it is a consequence of a suitable reformulation of the torsion-free condition for small perturbations of the closed $\gtwo$--structure $\varphi'_\BC$, in a similar way as in the first part of the section where we consider (global) small perturbations of a given ALC torsion-free $\gtwo$--structure.   

\subsection{Moduli spaces of ALC \gtwo--manifolds}

In this subsection we apply the weighted Fredholm analysis for ALC spaces developed earlier in the paper to study the local structure of the moduli space of ALC \gtmetric s on a fixed smooth $7$-manifold $M$. We will follow closely the strategy employed in \cite[\S\S 3.2 and 4.2]{Nordstrom:Thesis} to study deformations of compact and ACyl \gtmfd s respectively (for the latter, see also \cite{Nordstrom:ACyl:G2}). The same strategy has been used to study deformations of AC and conically singular \gtmfd s in \cite{Karigiannis:Lotay}. We will therefore concentrate our attention on describing the adaptations of the linear-analytic arguments needed to accommodate the ALC setting, referring the reader back to \cite{Nordstrom:Thesis}, \cite{Nordstrom:ACyl:G2} and \cite{Karigiannis:Lotay} for further details of the argument
as needed.

We wish to define a moduli space of torsion-free ALC \gtwo-structures asymptotic to a \emph{fixed} $\BC(\Sigma)$, quotiented by an appropriate group of diffeomorphisms \emph{asymptotic to the identity} on the ALC end.
We fix a torsion-free ALC {\gtstr} $\varphi$ on $M$ asymptotic to $\tu{BC}(\Sigma)$. Fix $\mu<-1$ and let $\mathcal{Z}^3_\mu$ be the space of smooth closed $3$-forms on $M$ of the form $\varphi + \rho$ with $\rho\in C^\infty_\mu$. We denote by $\mathcal{C}_\mu$ and~$\mathcal{X}_\mu$ the subsets consisting of closed \gtstr s and of torsion-free ones respectively: $\mathcal{C}_\mu$ is open in~$\mathcal{Z}^3_\nu$ and $\mathcal{X}_\mu$ is closed in $\mathcal{C}_\mu$.

\begin{remark}
\label{rmk:mu_vs_nu}
    If the torsion-free ALC \gtwo-structure $\varphi$ we started with is ALC with rate $\nu$, then $\mathcal{X}_\nu$ can simply be described as the set of torsion-free \gtwo-structures that are ALC with rate $\nu$ to the same model $\BC(\Sigma)$ as $\varphi$.
    When we consider $\mu$ distinct from the rate $\nu$ of $\varphi$ itself then $\mathcal{X}_\mu$ need not be ALC with rate $\nu$ in the sense of Definition \ref{def:ALC:G2}, but in any case:
    \begin{itemize}
    \item %
    elements of $\mathcal{X}_\mu$ are ALC with rate $\tu{max}(\mu, \nu)$ to the same $\BC(\Sigma)$
    \item replacing the given $\varphi$ with any other element of $\mathcal{X}_\mu$ would yield the same set $\mathcal{X}_\mu$.
    \end{itemize}
\end{remark}

Let $\mathcal{D}_{\mu+1}$ be the set of diffeomorphisms of $M$ generated by vector fields $X\in C^\infty_{\mu+1}$. $\mathcal{D}_{\mu+1}$ acts on~$\mathcal{Z}^3_\mu$ by pull-back preserving $\mathcal{C}_\mu$ and $\mathcal{X}_\mu$. The moduli space we want to study is $\mathcal{M}_\mu = \mathcal{X}_\mu/\mathcal{D}_{\mu+1}$. The main result of the subsection, Theorem \ref{thm:Moduli:Spaces}, is that $\mathcal{M}_\mu$ is a smooth finite-dimensional manifold near $\varphi$ for generic $\mu \in (-3-\delta,-1)$ which is locally diffeomorphic to the space $\mathcal{H}^3_\mu(M)$ of closed and coclosed $3$-forms in $C^\infty_\mu$. Here $\delta>0$ is a constant that depends only on $\Sigma$; with more work one could study the deformation problem in the range $\mu\leq -3-\delta$, but note that $\mathcal{X}_{\mu'}\subseteq\mathcal{X}_\mu$ if $\mu'\leq \mu$: restricting to a moduli space consisting entirely of elements with very fast decay seems unlikely to arise in applications and therefore we do not pursue these refinements here.

As usual in deformation theory problems, it will be necessary to work with weighted Banach spaces instead of $C^\infty_\mu$. We will work with weighted $L^2$--spaces but we could have equally chosen to work with weighted H\"older spaces. We will therefore fix $l\geq 1$ sufficiently large and consider spaces $\mathcal{X}_{l,\mu}\subset \mathcal{C}_{l,\mu}\subset\mathcal{Z}^3_{l,\mu}$ of 3-forms of class $L^2_{l,\mu
}$ acted upon by the space $\mathcal{D}_{l+1,\mu+1}$ of diffeomorphisms generated by vector fields of class $L^2_{l+1,\mu+1}$. Here we must choose $l$ sufficiently large so that elements in $L^2_{l,\mu}$ are in $C^0_\mu$: by Theorem \ref{thm:Sobolev}, any $l$ with $2l>7$ satisfies this requirement.

\subsubsection{Gauge-fixing}

The first step of the proof is to construct a slice for the action of the diffeomorphism group. As in \cite{Nordstrom:Thesis}, the main tool for analysing the gauge-fixing problem is the Dirac operator and its identification with the operator $\slashed{D}\co \Omega^0\oplus\Omega^1\ra \Omega^3_1\oplus\Omega^3_7$ of Lemma \ref{lem:identities:0:1:forms:torsion:free}(iv).

In order to study the mapping properties of $\slashed{D}$ we need the following preliminary vanishing result for decaying harmonic functions and $1$-forms on ALC \gtmfd s.

\begin{lemma}\label{lem:Vanishing:harmonic}
There are no nontrivial harmonic functions or $1$-forms in $L^2_{\delta}$ for $\delta <0$. As a consequence the Laplacian $\triangle\co L^2_{2,2+\delta}\ra L^2_\delta$ acting on functions and $1$-forms is an isomorphism for all $\delta\in (-4,0)$.
\proof
If $\delta<-2$ then we can integrate by parts (and use the Ricci-flatness of $M$ in the Weitzenb\"ock formula for $1$-forms) to obtain
\[
0=\langle \triangle u, u\rangle_{L^2} = \| \nabla u\|_{L^2}^2, \qquad 0=\langle \triangle \gamma, \gamma\rangle_{L^2} = \| \nabla \gamma\|_{L^2}^2.
\] 
Hence every harmonic function (respectively, $1$-form) in $L^2_{2,\delta}$ is constant (parallel) and therefore vanishes (since $f$ and $|\gamma|$ must decay at infinity). On the other hand, by Propositions \ref{prop:Harmonic:Functions:Cone} and~\ref{prop:Harmonic:1Forms:Cone} there are no indicial roots for the Laplacian acting on functions and $1$-forms in $(-4,0)$ and therefore the space of harmonic functions and $1$-forms in $L^2_{2,\delta}$ remains trivial for all $\delta < 0$.

The final statement is then a consequence of the duality statement of Lemma \ref{lem:Dual:weighted:Sobolev}.
\endproof
\end{lemma}

\begin{remark}\label{rmk:Vanishing:Harmonic}
Propositions \ref{prop:Harmonic:Functions:Cone} and \ref{prop:Harmonic:1Forms:Cone} and their modifications in Remark \ref{rmk:Harmonic:Functions:Cone} and \ref{rmk:Harmonic:1Forms:Cone} when $\Sigma$ is exceptional allow one to extend this result. For example, the scalar Laplacian $\triangle\co L^2_{2,2+\delta}\ra L^2_\delta$ remains surjective with kernel given by constant functions for all $\delta\in (-4,1]$ unless $M$ is ALC of cyclic type asymptotic to $\tu{C}(\Sigma)$ with $\Sigma$ exceptional.
\end{remark}

\begin{remark*}
On a Riemannian $7$-manifold $(M,g)$ with holonomy contained in $\gtwo$ the decomposition of $\Lambda^k T^\ast M$ into irreducible representations for $\gtwo$ is parallel. In particular, the Hodge Laplacian preserves this decomposition. Moreover, if $\Lambda^k T^\ast M$ contains a copy of the trivial or standard representation of $\gtwo$ then the restriction of the Hodge Laplacian to that factor coincides with the scalar Laplacian or, respectively, the Hodge Laplacian acting on $1$-forms. In particular, Lemma \ref{lem:Vanishing:harmonic} implies that on an ALC $\gtwo$--manifold every decaying harmonic $2$-form is of type $14$ and every decaying harmonic $3$-form is of type $27$.
\end{remark*}

\begin{prop}\label{prop:Mapping:Dirac}
$\slashed{D}\co L^2_{1,\delta+1}\ra L^2_{\delta}$ is injective for $\delta <-1$ and surjective for $\delta >-6$.
\proof
If $(f,\gamma)$ is in the kernel of $\slashed{D}$ then it is also in the kernel 
of ${\slashed{D}}\,^2$ and since $M$ is Ricci-flat (and hence scalar-flat)
$f$ and $\gamma$ are therefore harmonic. By Lemma \ref{lem:Vanishing:harmonic}, if $\delta +1<0$ then $f=0=\gamma$.
For $-6-\delta <0$, applying Lemma \ref{lem:Vanishing:harmonic} similarly proves the surjectivity statement by duality.
\endproof
\end{prop}

Now, the tangent space $T_\varphi\mathcal{Z}^3_{l,\mu}$ is the vector space
\[
T_\varphi\mathcal{Z}^3_{l,\mu} = \{ \rho\in\Omega^3_{l,\mu}\, |\, d\rho=0\}
\]
and by Lemma \ref{lem:decomp:forms} the tangent space at $\varphi$ of its $\mathcal{D}_{l+1,\mu+1}$--orbit is naturally identified with the vector space $d\left( \Omega^2_{7,l+1, \mu+1}\right)$, where $\Omega^2_{7,l+1,\mu+1}$ is the space of $L^2_{l+1,\mu+1}$--forms of type $\Omega^2_7$. Define
\[
\mathcal{K}_{l,\mu} = \{ \rho\in \Omega^3_{27,l,\mu}\, |\, d\rho=0\}.
\]
We show that $\mathcal{K}_{l,\mu}$ is a direct complement of $d(\Omega^2_{7,l+1,\mu+1})$ in $T_\varphi\mathcal{Z}^3_{l,\mu}$.

\begin{prop}\label{prop:Slice:Thm}
For all $\mu\in (-6,-1)$
\[
T_\varphi\mathcal{Z}^3_{l,\mu} = \mathcal{K}_{l,\mu} \oplus d(\Omega^2_{7,l+1,\mu+1})
\]
is a direct sum decomposition.
\proof
By Proposition \ref{prop:Mapping:Dirac} the Dirac operator $\slashed{D}\co L^2_{l+1,\mu+1}\ra L^2_{l,\mu}$ is an isomorphism in the range specified. Thus every $3$-form $\rho\in \Omega^3_{l,\mu}$ can be written uniquely as
\[
\rho = d(\gamma^\sharp\lrcorner\varphi) + \ast d(f\varphi) + \rho_0,
\]
with $\gamma,f\in L^2_{l+1,\mu+1}$ and $\rho_0$ in $\Omega^3_{27}$.

Note in particular that $\pi_1 (d\rho_0)=0$: indeed, $\rho_0 \wedge \varphi=0$ (since $\rho_0$ has type 27) and so by differentiation $d\rho_0\wedge\varphi=0$. Now, if $\rho$ is closed, then by Lemma \ref{lem:identities:0:1:forms:torsion:free} (i)
\[
\tfrac{2}{7}(\triangle f)\psi = \pi_1 d\left( d(\gamma^\sharp\lrcorner\varphi) + \ast d(f\varphi)\right) = -\pi_{1}d\rho_0 =0.
\]
Since $\mu +1 <0$ then $f=0$ by Lemma \ref{lem:Vanishing:harmonic}. Therefore $\rho = d(\gamma^\sharp\lrcorner\varphi)+\rho_0$ and thus $d\rho_0=d \rho=0$.
\endproof
\end{prop}

\subsubsection{The deformation problem}

Let $\varphi$ be a torsion-free ALC~\gtstr. Given a closed form $\rho\in\Omega^3_{l,\mu}$ with $\| \rho \| _{L^2_{l,\mu}}$ and $\|\rho\| _{C^0}$ sufficiently small, consider the closed $G_2$--structure $\varphi_{\rho}=\varphi+\rho$. Denote by $\psi_\rho$ the $4$--form $\ast_{\varphi_\rho}\varphi_\rho$. We want to study the equation $d\psi_\rho=0$.

The crucial observation is that since $d\varphi_\rho=0$, the $5$-form $d\psi_\rho$ has no $7$ component with respect to $\varphi_\rho$, \ie $d\psi_\rho\wedge (X\lrcorner\varphi_\rho)=0$ for every tangent vector $X\in T_x M$ at any point of $M$, see Lemma \ref{lem:Torsion:G2:structures}. This fact is at the core of any reformulation of the equation $d\psi_\rho=0$. In our context we find the following formulation to be the most convenient.

\begin{prop}\label{prop:Deformation:Problem}
For $\mu<-1$, there exists $\epsilon>0$ such that if $\| \rho\| _{C^0} <\epsilon$ then the following holds. Suppose that there exists $\gamma\in\Omega^1_{l+1,\mu+1}$ such that
\[
d\psi_\rho = d\ast_\varphi d(\gamma^\sharp\lrcorner\varphi).
\]
Then $d\psi_\rho=0$.
\proof
Since $\varphi_\rho$ is closed, the torsion constraints of Lemma \ref{lem:Torsion:G2:structures} imply that for every $\eta\in\Omega^1_{1,-5-\mu}$ we have
\[
\langle d^{\ast_\varphi} d(\gamma^\sharp\lrcorner\varphi), \eta^\sharp\lrcorner\varphi_\rho \rangle _{L^2} =  -\int{d\psi_\rho\wedge (\eta^\sharp\lrcorner\varphi_\rho)}=0.
\]
In other words, $\gamma$ lies in the kernel of $\pi_7^{\varphi_\rho}d^{\ast_\varphi}d(\,\cdot\,^\sharp\lrcorner\varphi)\co \Omega^1_{2,\mu+1}\ra \Omega^2_7\cap L^2_{\mu-1} \simeq_{\varphi_\rho}\Omega^1_{\mu-1}$.

We have to show that this operator is injective. Since $\| \rho\| _{C^0} <\epsilon$ the operator $\pi^{\varphi_\rho}_7\left( d^\ast d (\, \cdot\,\lrcorner\varphi) \right)$ differs from $\pi _7^\varphi\left( d^\ast d (\,\cdot\,\lrcorner\varphi) \right)$ by a bounded operator with norm $O(\epsilon)$. Since $\varphi$ is torsion-free we can identify $\pi _7 d^\ast d (\gamma^\sharp\lrcorner\varphi)$ with $dd^\ast \gamma + \tfrac{2}{3} d^\ast d\gamma$, see Lemma \ref{lem:identities:0:1:forms:torsion:free}\ref{lem:identities:0:1:forms:gauge:fixing}. It is therefore enough to show that $dd^\ast  + \tfrac{2}{3} d^\ast d$ is bounded below for the range of $\mu$ specified and then choose $\epsilon$ small enough. Since $dd^\ast + \tfrac{2}{3} d^\ast d$ is a Fredholm operator,  existence of a positive lower bound for its operator norm is equivalent to it being injective.

In order to show that the kernel of $dd^\ast + \tfrac{2}{3} d^\ast d$ acting on $\Omega^1_{\mu+1}$ is trivial for $\mu<-1$, first observe that if $\mu\leq -3$ then by Lemma \ref{lem:integration:parts} we can integrate by parts and conclude that $d\gamma=0=d^\ast \gamma$ and thus $\gamma=0$ by Lemma \ref {lem:Vanishing:harmonic}. It is therefore enough to ensure that there are no indicial roots for $dd^\ast + \tfrac{2}{3} d^\ast d$ in the range $(-2,0)$. If we write a 1-form $\gamma$ on $\BC (\Sigma)$ as $\gamma = u\,\theta + \gamma_H$ and replace $d$ and $d^\ast$ with the covariant differential and codifferential associated with the adapted connection $\overline{\nabla}$ (for which $\theta$ is parallel), then it is clear that the indicial roots for $dd^\ast + \tfrac{2}{3} d^\ast d$ are given by the indicial roots for $\tfrac{2}{3} d^\ast d$ acting on functions and $dd^\ast + \tfrac{2}{3} d^\ast d$ acting on $1$-forms on $\tu{C}(\Sigma)$. The fact that there are none of these in the interval $(-2,0)$ is guaranteed by Proposition \ref{prop:Harmonic:Functions:Cone} and, respectively, the last statement of Proposition \ref{prop:Harmonic:1Forms:Cone}. 
\endproof
\end{prop}

\subsubsection{Application of the Implicit Function Theorem}

By Lemma \ref{lem:Linearisation:Hitchin:dual:3form:7d} the linearisation of the map $\rho\mapsto\psi_\rho$ is given by $\rho\mapsto \ast \left( \tfrac{4}{3}\pi_1(\rho) + \pi_7 (\rho) - \pi_{27}(\rho)\right)$. The next lemma provides control of the nonlinear terms of the equation $d\psi_\rho=0$.

\begin{lemma}\label{lem:Hitchin:dual}
For $2l > 9$ we have
\[
\psi_\rho = \psi + \ast \left( \tfrac{4}{3}\pi_1(\rho) + \pi_7 (\rho) - \pi_{27}(\rho)\right) + \ast Q_\varphi (\rho),
\]
where $Q_\varphi (\rho)\in\Omega^3_{l,\mu}$ and $\| Q_\varphi (\rho) \|_{L^2_{l,\mu}}\leq C \| \rho \|^2 _{L^2_{l,\mu}}$.
\proof
The proof is analogous to \cite[Lemma 5.16]{Karigiannis:Lotay} and uses the ALC weighted Sobolev Embedding Theorem \ref{thm:Sobolev}.
\endproof
\end{lemma}

Exploiting Propositions \ref{prop:Slice:Thm} and \ref{prop:Deformation:Problem}, modulo the action of $\mathcal{D}_{l+1,\mu+1}$ we can formulate the deformation problem for the ALC torsion-free $\gtwo$--structure $\varphi$ as the vanishing of the map
\[ F : \mathcal{K}_{l,\mu} \oplus \Omega^1_{l+1,\mu+1} \to d^\ast\Omega^3_{l,\mu}, \;
(\rho,\gamma)\longmapsto d^\ast\rho - d^\ast d(\gamma^\sharp\lrcorner\varphi) - d^\ast Q_\varphi (\rho).
\]
The linearised operator is
\begin{equation}\label{eq:Linearised:Def:ALC:G2}
\mathcal{L}_\varphi (\rho,\gamma) = d^\ast\rho - d^\ast d (\gamma^\sharp\lrcorner \varphi),
\end{equation}
while in view of Lemma \ref{lem:Hitchin:dual} the nonlinearities are contained in the operator
\[
\mathcal{N}_\varphi (\rho,\gamma) =  -d^\ast Q_\varphi (\rho),
\]
which has range contained in $d^\ast\Omega^3_{l,\mu}$.

\begin{lemma}\label{lem:Linearised:Def:ALC:G2}
There exists $\delta>0$ such that, for every $\mu\in(-3-\delta,-1)$ away from a finite set of indicial roots, the operator $\mathcal{L}_\varphi\co \mathcal{K}_{l,\mu} \oplus \Omega^1_{l+1,\mu+1}\ra d^\ast \Omega^3_{l,\mu}$ of \eqref{eq:Linearised:Def:ALC:G2} is a surjective operator between Banach spaces, with kernel $\mathcal{H}^3_\mu (M) \oplus \{ 0\}$.
\proof
The claim about the kernel of $\mathcal{L}_\varphi$ is proved as follows.  $(\rho,\gamma)\in \ker\mathcal{L}_\varphi$ if and only if it solves
the system $d \rho=0$, $d^\ast \rho = d^\ast d (\gamma^\sharp\lrcorner\varphi)
.$
Since $\rho\in\Omega^3_{27}$, Remark \ref{rmk:identities:0:1:forms:torsion:free} then implies that $0=\pi_7d\rho = \pi_7 d^\ast \rho=\pi_7d^\ast d(\gamma^\sharp\lrcorner\varphi)$. We have already argued in the proof of Proposition \ref{prop:Deformation:Problem} that 
$\pi_7d^\ast d(\gamma^\sharp\lrcorner\varphi)=0$ forces $\gamma=0$ provided $\mu <-1$. Hence $\rho$ is a closed and coclosed
$3$-form of type 27. Conversely, note that every closed and coclosed 3-form of rate $\mu<-1$ must lie in $\Omega^3_{27}$ because of Lemma \ref{lem:Vanishing:harmonic}. 
Note that so far we have needed only to assume that $\mu< -1$.

For the surjectivity statement, given an element $d^\ast\xi$ with $\xi\in \Omega^3_{l,\mu}$ we first consider the following Poisson equation
for a $2$-form $\sigma\in \Omega^2_{l+1,\mu+1}$
\[
\triangle\sigma = d^\ast\xi.
\]
When $\mu \in (-3,-1)$ we can appeal to part (ii) of Proposition \ref{prop:Closed:to:Coclosed:Laplacian} to find a solution of this equation with $\sigma \in\Omega_{l+1,\mu+1}$. Moreover, this solution $\sigma$ must in fact be coclosed, because (as observed in the proof of Proposition \ref{prop:Closed:to:Coclosed:Laplacian}) $d^\ast \sigma$ is a harmonic $1$-form of rate $\mu<-1$, and by Lemma~\ref{lem:Vanishing:harmonic} this
must vanish. Since $\sigma$ is coclosed, we have $d^\ast\rho=d^\ast\xi$ with $\rho=d\sigma\in L^2_{l,\mu}$. Working $L^2_{\mu+1}$--orthogonally to harmonic $2$-forms to make the solution $\sigma$ unique, Fredholmness of the Hodge Laplacian (guaranteed by our assumption that $\mu$ is not an indicial root) implies the existence of uniform estimates $\| \rho\|_{L^2_{l,\mu}}\leq C\| d^\ast\xi\|_{L^2_{l-1,\mu-1}}$ for a constant $C>0$ independent of $\rho,\xi$. In particular, these estimates show that $d^\ast\Omega^3_{l,\mu}$ is closed in $\Omega^2_{l-1,\mu-1}$ and therefore a Banach space as claimed. The result is therefore proven for $\mu\in (-3,-1)$.

In order to work at rates below $\mu=-3$ we need to use (case~(i) of) Proposition \ref{prop:Closed:to:Coclosed:Laplacian:Mid:deg} instead of Proposition \ref{prop:Closed:to:Coclosed:Laplacian} to solve the Poisson equation. This implies that for $\delta>0$ sufficiently small and $\mu \in (-3-\delta,-3)$ we can still solve $\triangle\sigma=d^\ast\xi$ (uniquely by working orthogonally to harmonic forms) but where the solution is now a coclosed $2$-form $\sigma$ in $\Omega^2_{l+1,\mu+1}\oplus V$, where $V$ is the finite-dimensional space of 2-forms that can be written as $\chi \tau$, for $\chi$ a (fixed) cutoff function with $\chi\equiv 1$ on the ALC end and $\tau\in\mathcal{H}^2(\Sigma)$. Nevertheless, we still have that $d\sigma\in L^2_{l,\mu}$ (because as noted in Remark~\ref{rmk:Closed:to:Coclosed:Laplacian:Mid:deg} $d(\chi\tau)$ is compactly supported), and so the rest of the proof in the case $\mu\in (-3,-1)$ goes through unchanged.
\endproof
\end{lemma}

Applying the Implicit Function Theorem in Banach spaces to $F$ now yields the following result.

\begin{theorem}\label{thm:Moduli:Spaces}
Let $(M,\varphi)$ be an ALC $\gtwo$--manifold asymptotic to $\tu{BC}(\Sigma)$, and let
$\mu\in (-3-\delta,-1)$ for $\delta>0$ given in Lemma \ref{lem:Linearised:Def:ALC:G2} be outside a finite set of indicial roots. Then the moduli space $\mathcal{M}_\mu$ of ALC torsion-free $\gtwo$--structures on~$M$ with the same asymptotics as $\varphi$ up to rate $\mu$ is a smooth manifold with tangent space at the class of $\varphi$ identified with  $\mathcal{H}^3_\mu (M,g_{\varphi})$.
\end{theorem}

As in Remark \ref{rmk:mu_vs_nu}, if we take $\mu$ to be the asymptotic rate $\nu$ of $\varphi$ itself then we can think of $\mathcal{M}_\nu$ simply as the moduli space of ALC torsion-free \gtstr s asymptotic to a fixed $\BC(\Sigma)$ with rate $\nu$, and we obtain Theorem \ref{mthm:Moduli}. (If $\nu$ itself is an indicial root, then we can use $\mu$ just larger than~$\nu$, and $\mathcal{M}_\mu$ is the same set as $\mathcal{M}_\nu$.)

\subsection{Asymptotic behaviour on the end}

Let us now consider an arbitrary torsion-free $\gtwo$--structure $\varphi$ defined on an exterior domain in $\BC (\Sigma)$ of the form
\[
\varphi = \varphi'_\BC + \rho 
\]
for some $\rho\in\Omega^3_\nu$ for $\nu <-1$. We want to exploit the diffeomorphism freedom to put $\varphi$ in suitable gauge with respect to $\varphi'_\BC$, so that the torsion-free condition can be reformulated as an elliptic PDE problem. As a consequence, we will derive the decay conditions of part (i) of Theorem \ref{mthm:Asymptotics}.

We begin with a suitable slice theorem along an ALC end. It is analogous to the gauge-fixing result on the end of an AC $\gtwo$--manifold discussed in \cite[\S 6.6]{Karigiannis:Lotay}. The proof of our result differs from the argument used by 
Karigiannis--Lotay, in that we employ a suitable boundary value problem for the Laplacian on differential forms. This gives what seems to us the most streamlined proof of the result and allows us to work entirely on an exterior domain rather than on a complete ALC (or AC) manifold.

\begin{prop}\label{prop:Slice:End}
Fix $\nu\in (-5,-1)$ and let $\varphi=\varphi'_{\tu{BC}}+\rho$ be an ALC torsion-free $\gtwo$--structure defined on an exterior domain in $\BC (\Sigma)$ with $\rho\in\Omega^3_\nu$. For $R_0>0$ sufficiently large, there exists a diffeomorphism $f$ of $\{ r\geq R_0\}\subset \BC (\Sigma)$ of class $C^\infty_{\nu+1}$ such that
\[
\left( f^\ast\varphi -\varphi'_\BC\right)\wedge\varphi'_\BC =0=\left( f^\ast\varphi -\varphi'_\BC\right)\wedge\psi'_\BC, 
\]
\ie $f^\ast\varphi -\varphi'_\BC$ lies in the space of $3$-forms of type $27$ with respect to the model closed \gtwo--structure~$\varphi'_\BC$.
\end{prop}

We can formulate the statement as a nonlinear PDE problem for a vector field $X$ on $\{ r\geq R_0\}$ of class $C^{l+1,\alpha}_{\nu+1}$.
The linearisation of this problem is
\[
X\longmapsto \pi'_{1\oplus 7}\mathcal{L}_X\varphi
\]
(where $\pi'_{1\oplus7}$ means projection to the type 1+7 components defined in terms of $\varphi'_\BC$)
and we have to show that this map surjects onto the space $\pi'_{1\oplus 7}\mathcal{Z}_\nu (r\geq R_0)$, where $\mathcal{Z}_\nu (r\geq R_0)$ is the set of closed $3$-forms on $\{ r\geq R_0\}$ of class $C^{l,\alpha}_{\nu}$. Once this is established, the result follows from a straight-forward application of the contraction mapping theorem.

We follow the same strategy as in the proof of Proposition \ref{prop:Slice:Thm}.  However some adaptations are required because (a) we now work on an exterior domain in $\BC (\Sigma)$ (rather than on a complete ALC manifold $M$) and (b) $\varphi'_\BC$ is closed but not torsion-free.

\begin{lemma}
For $R_0$ sufficiently large, there is a bounded (with bound independent of~$R_0$) linear map $Y$ from $\Omega^3_{%
\nu}(r \geq R_0)$ to vector fields of class $C^{l+1,\alpha}_{\nu+1}$ on $\{r \geq R_0\}$ such that $\xi -\mathcal{L}_{Y(\xi)}\varphi$ has type $27$ with respect to $\varphi'_\BC$ for any closed $\xi \in \Omega^3_\nu(r \geq R_0)$.
\end{lemma}

\begin{proof}
 Consider first the operator %
 $D \co \Omega^0 \oplus \Omega^1 \ra \Omega^3_{1\oplus 7}, \; (f, X^\sharp) \mapsto \pi_{1\oplus 7}(d^\ast(f\psi) + d(X \lrcorner \varphi))$ (projecting to the type 1+7 part with respect to the torsion-free \gtstr{} $\varphi$).

Since $D$ can be identified with the Dirac operator (using Lemma \ref{lem:identities:0:1:forms:torsion:free} which applies because $\varphi$ is torsion-free), up to constant multiplicative factors
\begin{align*}
D^\ast(u \varphi + \gamma^\sharp\lrcorner\varphi) &= (d^\ast\gamma, (\tu{curl}\,\gamma + du)^\flat) \\
DD^\ast(u \varphi + \gamma^\sharp\lrcorner\varphi) &= (\triangle u, \triangle \gamma)
\end{align*}
We apply the Hodge theory for manifolds with boundary described in \cite{Hodge:BVP} to the boundary value problem
\[
(\triangle u, \triangle \gamma) = (\xi_0, \xi_1), \qquad u|_{r=R_0}=0, \qquad \gamma|_{r=R_0}=0=d^\ast\gamma|_{r=R_0},
\]
to show that $DD^\ast$ is an isomorphism on the domain defined by these boundary conditions.

Schwarz~\cite[Theorem 2.5.4]{Hodge:BVP} establishes the main existence and uniqueness result in weighted $L^2$ spaces for exterior domains in $\R^n$. Given the materials of Part \ref{Part1}, his argument immediately generalises to the exterior domain $\{ r\geq R_0\}$ of $\BC (\Sigma)$. For $\nu<-3$ an integration by parts shows that the kernel of the boundary value problem
\begin{equation}\label{eq:BVP:Exterior:Domain:kform}
\triangle \zeta = \zeta', \qquad \zeta|_{r=R_0}=0=d^\ast\zeta|_{r=R_0}
\end{equation}
for a $k$-form $\zeta$ in $C^{l+1,\alpha}_{\nu+1}$ is identified with the space $\mathcal{H}^k_{\nu+1,D}$ of closed and coclosed $k$-forms of the given rate satisfying the Dirichlet boundary condition $\zeta|_{r=R_0}=0$. By duality, for $\nu>-3$ the cokernel of the boundary value problem \eqref{eq:BVP:Exterior:Domain:kform} is identified with $\mathcal{H}^k_{-5-\nu,D}$. The indicial roots computations of Propositions \ref{prop:Harmonic:1Forms:Cone} and \ref{prop:Harmonic:1Forms:Cone} imply that for $k=0,1$ both the kernel and cokernel of the problem \eqref{eq:BVP:Exterior:Domain:kform} remain constant for $\nu\in (-5,-1)$.

When $k=0$ it is immediate that $\mathcal{H}^0_{\nu+1,D}$ for $\nu<-3$ and $\mathcal{H}^0_{-5-\nu,D}$ for $\nu>-3$ both vanish. A Bochner argument using the Dirichlet boundary conditions at $\{r=R_0\}$ and the decay at infinity shows that the same vanishing results hold for $\mathcal{H}^1_{\nu+1,D}$ for $\nu<-3$ and $\mathcal{H}^1_{-5-\nu,D}$ for $\nu>-3$. Here, following \cite[Theorem 2.6.4]{Hodge:BVP}, the Bochner technique yields a vanishing statement since the Ricci curvature of $g_\varphi$ vanishes and by taking $R_0$ larger if necessary we can assume that the hypersurface $\{ r=R_0\}$ has positive mean curvature. Indeed, the mean curvature $H_{g_{\varphi'_\BC}}$ of $\{ r=R_0\}$ with respect to $g_{\varphi'_\BC}$ is strictly positive and of order $O(R_0^{-1})$, while the difference $H_{g_\varphi}-H_{g_{\varphi'_\BC}}$ is of order $O(R_0^{\nu-1})$ since $\varphi-\varphi'_\tu{BC}\in C^\infty_\nu$.

Thus $DD^\ast$ is an isomorphism.
To see that its inverse is uniformly bounded as $R_0$ increases, note that by rescaling we can think of the operators as being defined on a fixed region $\{r \geq 1\}$, but with the circle fibration shrinking having circumference $\ell = O(R_0^{-1})$. On the higher Fourier modes, the fibrewise Poincar\'e inequality from Lemma \ref{lem:Control:Oscillatory}
gives a lower bound for $DD^\ast$ proportional to $\ell^{-1}$, while for sections in $\ker \Pi_\perp$ the operator becomes arbitrarily close to the model operator on the cone $\tu{C}(\Sigma)$.

Now consider $D' \co \Omega^0 \oplus \Omega^1 \ra \Omega^3_{1\oplus 7}, \; (f, X^\sharp) \mapsto \pi'_{1\oplus 7}(d^{\ast'}(f\psi') + d(X \lrcorner \varphi)$. Here $\ast'$ denotes the Hodge-star operator with respect to $g_{\varphi'_\tu{BC}}$ and $\psi'=\ast_{\varphi'_\tu{BC}}\varphi'_\tu{BC}$. By increasing $R_0$, we can make the operator norm of the difference between $D'$ and $D$ arbitrarily small, ensuring that $D'D^\ast$ too is an isomorphism (with uniformly bounded inverse)

Given $\xi \in \Omega^3_\nu(r \geq R_0)$, we can thus let $Y(\xi)^\flat$ be the $\Omega^1$-component of $D^\ast(D'D^\ast)^{-1}(\pi'_{1\oplus 7}\xi)$. Then
\[ \xi - \mathcal{L}_{Y(\xi)}\varphi - d^{\ast'}(f\psi') \]
has type 27 with respect to $\varphi'_\BC$ for some decaying function $f$ on ${r \geq R_0}$ that vanishes on the boundary $r = R_0$ (specifically $f = d^\ast \gamma$, where $\gamma^\sharp \lrcorner \varphi$ is the type 7 component of $(D'D^\ast)^{-1}(\xi)$). Because $d\varphi'_\BC = 0$, it follows that $\varphi'_\BC \wedge d(\xi - \mathcal{L}_{Y(\xi)}\varphi - d^{\ast'}(f\psi')) = 0$. Therefore, if $d\xi = 0$ too then, using again that $\varphi'_\tu{BC}$ is closed, $\Delta'f = 0$, so $f = 0$ by the maximum principle, and $\xi - \mathcal{L}_{Y(\xi)}\varphi$ has type 27 as required.
\end{proof}

\begin{proof}[Proof of Proposition \ref{prop:Slice:End}]
For $f = \exp X$, we are trying to solve
\begin{align*}
   \pi'_{1\oplus7}(f^*\varphi - \varphi'_\BC) &= 0 \\
   \Leftrightarrow \pi_{1\oplus7}\mathcal{L}_X \varphi &= -\pi'_{1\oplus7}f^*\varphi + \pi_{1\oplus7}\mathcal{L}_X \varphi + \varphi'_\BC\\
   \Leftrightarrow X &= Y\big( \varphi'_\BC - \pi'_{1\oplus7} \varphi - \pi'_{1\oplus7}((f^* - \mathrm{Id}^* - \mathcal{L}_X)\varphi)\big)
\end{align*}
The last term is quadratic in $X$, with bounds independent of $R_0$, while the norm of the constant (in $X$) term $\varphi'_\BC - \pi'_{1\oplus7}\varphi$ can be made arbitrarily small by increasing $R_0$. Thus for sufficiently large $R_0$ we can solve by the contraction mapping theorem.
\end{proof}
The final main result of this section is the following improved asymptotic control of torsion-free ALC $\gtwo$--structures: they are all exponentially close to an $S^1$--invariant $\gtwo$--structure. 

\begin{theorem}\label{thm:Exp:Decay:ALC:G2}
Let $\varphi$ be a torsion-free $\gtwo$--structure defined on the exterior domain $\{r\geq R_0\}$ in $\BC (\Sigma)$ such that
\[
\varphi = \varphi'_\BC + \rho
\]
for a 3-form $\rho$ in $C^\infty_\nu$ for some $\nu<-1$ and such that $\rho\wedge\varphi'_\BC =0=\rho\wedge\psi'_\BC$. Write $\rho = \rho_0 + \rho_\perp$, where $\rho_0$ is the $S^1$--invariant component of $\rho$. Then taking $R_0$ larger if necessary there exist constants $C_k,c>0$ such that
\[
\sup_{r=R}|\nabla^k \rho_\perp| \leq C_k\, e^{-cR}.
\]
for all $R>R_0$ and $k\geq 0$.
\proof
Since $\rho$ lies in $\Omega^3_{27}$ with respect to $\varphi'_\BC$, we write the torsion-free condition as the elliptic system
\[
d\rho=0, \qquad \ast d\psi'_\BC + d^{\ast}\rho + d^\ast Q(\rho)=0,
\]
where the Hodge-$\ast$ operator and the nonlinear term $Q$ are both computed using the closed model $\gtwo$--structure $\varphi'_\BC$. Since the latter is $S^1$--invariant, note that $d\psi'_\BC$ and $Q(\rho_0)$ are also $S^1$--invariant, and $d^\ast$ preserves the decomposition into $S^1$--invariant and oscillatory components. Hence $\rho_0$ and $\rho_\perp$ satisfy the system
\[
d\rho_0=0=\ast d\psi'_\BC + d^\ast\rho_0 + d^\ast Q(\rho_0+\rho_\perp)_0, \qquad d\rho_\perp = 0 = d^\ast\rho_\perp + d^\ast\left( Q(\rho_\perp + \rho_0)- Q(\rho_0)\right)_\perp.
\]

The key consequence of the equation that $\rho_\perp$ satisfies is that, similarly to Lemma \ref{lem:Hitchin:dual} (which we can apply even if $\varphi'_\BC$ is not torsion-free since the torsion and all its derivatives decay), for every $\delta>0$ we can find $R_0$ sufficiently large so that we have a pointwise estimate
\[
|d^\ast\left( Q(\rho)-Q(\rho_0)\right)_\perp| \leq C \left( | \nabla \rho| \, | \rho_\perp| + |\rho|\, |\nabla\rho_\perp| \right) \leq \delta \left( r^{-1}| \rho_\perp| + |\nabla\rho_\perp|\right) 
\]
for $r>R_0$. Here we used the fact that $| \rho| + r|\nabla \rho| = O(r^{\nu})$ is arbitrarily small for $r$ large enough.

An application of the estimates of Proposition \ref{prop:Estimate:Oscillatory} then implies
\[
\| \rho_\perp\|_{L^2_{1,\mu}} \leq C\delta\, \|\rho_\perp\|_{L^2_{1,\mu}} + C\|\rho_\perp\|_{L^2_\nu}
\]
so that $\rho_\perp\in L^2_{1,\mu}$ for any $\mu\in \R$. The same estimate in H\"older spaces already implies that $\rho_\perp$ decays faster than any polynomial.

In order to show exponential decay, fix a smooth cutoff function $\chi$ with $\chi\equiv 1$ for $r$ large. We calculate
\[
|(d+ d^\ast)(\chi\rho_\perp)|^2 \leq C_1\, |\nabla\chi|^2\, |\rho_\perp|^2 + C_2\, \delta^2 \left( |\chi\rho_\perp|^2 + |\nabla (\chi\rho_\perp)|^2\right).
\]
On the other hand, an integration by parts (allowed because of the fast decay of $\rho_\perp$ at infinity) using the Weitzenb\"ock formula, the decay of the curvature term and Lemma \ref{lem:Control:Oscillatory} implies
\[
\|(d+ d^\ast)(\chi\rho_\perp)\|_{L^2} \geq C \|\chi\rho_\perp\|_{L^2}.
\]
For $\delta$ small enough, by choosing a cutoff function with uniformly bounded gradient such that $\chi\equiv 0$ for $r\leq R$ and $\chi\equiv 1$ for $r\geq R+1$, we deduce that
\[
\| \rho_\perp\|_{L^2 (r>R+1)} \leq C \|\rho_\perp\|_{L^2(R<r<R+1)}.
\]
The classical hole-filling trick than shows that $F(R)=\|\rho_\perp\|^2_{L^2 (r>R)}$ satisfies the inequality
\[
F(R+1)\leq e^{-c}F(R)
\]
for some $c>0$. Hence we have exponential decay of the $L^2$--norm of $\rho_\perp$. The elliptic regularity of Theorem \ref{thm:Weighted:Regularity} finally implies pointwise exponential decay of $\rho_\perp$ and all its derivatives.
\endproof
\end{theorem}

\begin{remark}\label{rmk:Exp:Decay:ALC:G2}
Following the notation in the statement of the theorem, set $\varphi_0=\varphi'_\BC + \rho_0$. Since
\[
\ast d\psi_0 = \ast d\psi'_\BC +d^\ast\rho_0 + d^\ast Q(\rho_0) = d^\ast \left( Q(\rho_0)-Q(\rho)_0\right),
\]
$\varphi_0$ is a closed $S^1$-invariant $\gtwo$--structure with exponentially decaying torsion.
\end{remark}

\section{Extension of symmetries from the asymptotic model}\label{sec:Symmetries}

A natural question when studying complete noncompact Ricci-flat manifolds (or manifolds with other curvature constraints) with tame asymptotic geometry, is whether symmetries of the asymptotic model at infinity extend to symmetries of the manifold itself. In this section we address this question for ALC $\gtwo$--manifolds.

For cyclic (but not dihedral) ALC $\gtwo$--manifolds there is a distinguished subgroup of the symmetry group of $\BC(\Sigma)$: 
the circle action arising from the principal bundle structure of $\BC(\Sigma)$. We prove that this circle subgroup of the end symmetries always extends as a global circle symmetry. 

\begin{theorem}\label{thm:Circle:Symmetry:Cyclic:ALC}
Let $(M,\varphi)$ be a cyclic ALC $\gtwo$--manifold asymptotic to $\BC(\Sigma)$ with rate $\nu\in (-3,-1)$. Then the circle symmetry of $\BC (\Sigma)$ extends to a circle action on $M$ by ALC diffeomorphisms of rate $\nu$ that preserve $\varphi$.
\end{theorem}

In the proof of this theorem we will distinguish two cases according to whether or not the Sasaki--Einstein structure on $\Sigma$ is exceptional (in the sense of Section \ref{sec:Indicial:Roots:CY:cone}). In the exceptional case we prove that in fact $M$ must be a finite quotient of the flat ALC~\gtmfd~$\Sph^1 \times \C^3$. This turns out to be a consequence of a more general result of independent interest, Theorem \ref{thm:ALC:G2:Positive:Mass}: 
a cyclic ALC $\gtwo$--manifold asymptotic to a flat circle bundle $\BC(\Sigma)\ra \tu{C}(\Sigma)$ cannot have full holonomy $\gtwo$.

These two theorems are motivated by analogous results in the 4-dimensional ALF hyperk\"ahler setting, both due to Minerbe. The fact that any cyclic ALF hyperk\"ahler metric admits a triholomorphic circle action is the key  step in Minerbe's classification of such metrics in \cite{Minerbe:Ak}.  The main result of \cite{Minerbe:Mass} implies that there are no non-flat ALF hyperk\"ahler manifolds asymptotic to $\R^3\times \Sph^1$.

The problem of extending more general symmetries of $\BC (\Sigma)$ has a different character from the circle symmetry extension problem. Even at the linearised level, there is no reason to expect a  symmetry of $\Sigma$ that lifts to $\BC (\Sigma)$ to preserve all homogeneous $\overline{\nabla}$--covariantly closed and coclosed 3-forms on $\BC (\Sigma)$ of rate $\nu$, unless such forms are all accounted for by topology. In this direction, 
we can prove a symmetry inheritance result
provided we make an injectivity assumption 
about the relevant local Torelli-type data. 
The following result is motivated by an analogous statement for AC $\gtwo$--manifolds due to Karigiannis--Lotay \cite[Proposition 6.8]{Karigiannis:Lotay}.

\begin{theorem}\label{thm:Symmetries}
Let $(M,\varphi)$ be an ALC $\gtwo$--manifold (of cyclic or dihedral type) asymptotic to $\BC(\Sigma)$ with rate $\nu\in (-3,-1)$. Assume that
\begin{equation}\label{eq:Symmetries:Assumption}
\ker\left( \mathcal{H}^3_{\nu}(M) \longrightarrow H^3(M)\oplus H^4(M)\right)=\{ 0\}
\end{equation}
for the natural map $\rho\mapsto ([\rho],[\ast\rho])$, $\rho\in \mathcal{H}^3_\nu(M)$. Then every continuous group of automorphisms of~$\BC(\Sigma)$ (commuting with the standard involution in the dihedral case) extends to a group of ALC automorphisms of $M$ of rate $\nu$.
\end{theorem}

\begin{remark}
\label{rmk:regular:SE} Thanks to Propositions \ref{prop:Closed:Even:Forms} and \ref{prop:Closed:Odd:Forms}, the assumption \eqref{eq:Symmetries:Assumption} is always satisfied when $\Sigma$ is a regular Sasaki--Einstein manifold, see Remark \ref{rmk:inj_across_L2} stated in the next section as an application of the Hodge theory of Section \ref{sec:Hodge:ALC} to ALC $\gtwo$--manifolds.
\end{remark}

The proofs of both Theorems~\ref{thm:Circle:Symmetry:Cyclic:ALC} and ~\ref{thm:Symmetries} have two steps. The first step uses the ALC deformation theory developed in Section~\ref{sec:def_theory} to prove the results about
extending continuous symmetries of $\BC(\Sigma)$ to global symmetries
(ALC automorphisms) of $M$; 
the second step shows that such a global symmetry extension is unique. 
For the extension of the principal circle action in the cyclic setting, Theorem~\ref{thm:Exp:Decay:ALC:G2} plays a crucial role in the proof of the first step. The assumption~\eqref{eq:Symmetries:Assumption} is used only in the first step
for other types of end symmetry.
The second step proceeds differently according to whether or not 
$\Sigma$ is exceptional. In the exceptional case we argue directly based on the fact that $M$ must be flat in this case. In the non-exceptional cases we prove the existence of a canonical foliation of the end of $M$ by gradient flow lines of a distinguished (unique up to constants) function asymptotic to~$\frac{1}{2}r^2$. 
The key point is that the canonical nature of this foliation implies that any automorphism of~$M$ must therefore act on these flow lines.

The rest of this section is structured as follows. In Section~\ref{ss:symmetries} we discuss symmetries of the model end $\BC(\Sigma)$ in both the cyclic and dihedral cases and explain what we mean by an ALC automorphism of rate $\nu$. Section~\ref{ss:sym:extension} proves the results about the existence of ALC automorphisms extending end symmetries. Section~\ref{ss:canonical:foliation} proves the existence of a canonical foliation on the end of any non-exceptional ALC~\gtmfd~and uses this to prove the uniqueness of these symmetry extensions. Section~\ref{ss:uniqueness:symmetry}  concludes this section with a quick application of these symmetry extension results to prove strong uniqueness results for ALC \gtwo--manifolds with certain asymptotic models.

\subsection{Symmetries}
\label{ss:symmetries}

We begin by explaining what we mean by automorphisms of $\BC(\Sigma)$ and ALC automorphisms of $M$ of rate~$\nu$ for an ALC \gtwo-structure. 
We do this for the cyclic case first and then explain the
modifications needed in the dihedral setting. 

In the cyclic case the group of automorphisms of $\BC(\SigmP)$, which we will denote by $\tu{Aut}\,\BC(\SigmP)$, is the group of bundle automorphisms of $\pi\co N\ra \SigmP$ that preserve the 1-form $\theta$ and the Sasaki--Einstein structure $(\eta,\omega_1,\omega_2,\omega_3)$ on $\SigmP$. In particular, the radial extension of any such diffeomorphism to $\BC(\SigmP)$ preserves the model closed $\gtwo$--structure $\varphi'_\BC$.

As a closed subgroup of the isometry group of a closed manifold, $\tu{Aut}\,\BC(\SigmP)$ is a compact Lie group. An element in the Lie algebra $\Lie{aut}\,\BC(\SigmP)$ must have the form
\[
\mu\, \xi + X_0,
\]
where $c\in \R$,  $X_0$ is the horizontal lift of a vector field on $\SigmP$ preserving both the whole Sasaki--Einstein $\sunitary{2}$--structure and the 2-form $d\theta$, and $\mu$ is a function on $\SigmP$ such that $d\mu = -X_0 \lrcorner d\theta$ (such a function $\mu$ always exists because, since $\Sigma$ has finite fundamental group, the condition that $X_0$ preserve $d \theta$ implies that $X_0 \lrcorner d\theta$ is an exact $1$-form). When $X_0$ is identically zero, we get a multiple of the vertical vector field $\xi$, the generator of the principal circle action on $\BC(\Sigma)$. The latter gives a distinguished circle subgroup of $\tu{Aut}\,\BC(\Sigma)$ in the cyclic case. 

In the dihedral case,  $\tu{Aut}\,\BC(\Sigma)$ means the  subgroup of  $\tu{Aut}\,\BC(\SigmP)$  that commutes with the involution $\iota$ of the oriented double cover $\BC(\SigmP) \to \BC(\Sigma)$. In particular, in the dihedral case the distinguished circle subgroup of $\tu{Aut}\,\BC(\Sigma)$ present in the cyclic case no longer exists.

We are now ready to define the group of ALC automorphisms of an ALC $\gtwo$--manifold.

\begin{definition}\label{def:ALC:automorphism}
  Let $(M,\varphi)$ be an ALC $\gtwo$--manifold asymptotic to $\BC (\Sigma)$ with rate $\nu<-1$ and let $f\co \BC (\Sigma)\cap \{ r>R\}\ra M\setminus K$ be the ALC chart at infinity. A diffeomorphism $F\co M\ra M$ is said to be an ALC automorphism of $(M,\varphi)$ with rate $\nu$ if $F^\ast\varphi=\varphi$ and there exist an element $F_\BC\in \tu{Aut}\,\BC(\Sigma)$ and a vector field $V\in C^\infty_{\nu+1}$ on $\BC (\Sigma)\cap \{ r>R\}$ such that
  \[
  f^{-1}\circ F \circ f = \exp(V)\circ F_\BC.
  \]
\end{definition}

In other words, the ALC automorphisms of $M$ are those automorphisms that are asymptotic to an element of $\tu{Aut}\,\BC(\Sigma)$ with polynomial rate. We will denote the group of ALC automorphisms of $(M,\varphi)$ of rate $\nu$ by $\tu{Aut}_\nu (M,\varphi)$.
 
\begin{remark}
\label{rmk:aut_flat}
    In the trivial case of the flat ALC~\gtmfd~$\Sph^1 \times \C^3$, translations of the $\C^3$ factor are not ALC automorphisms in the sense above, because the corresponding vector field $V$ does not decay along the end. Thus $\tu{Aut}_\nu(\Sph^1 \times \C^3) = \unitary{1} \times \sunitary{3}$ for every $\nu < -1$.
\end{remark}
\subsection{Global extensions of ALC end symmetries}
\label{ss:sym:extension}
In this subsection we prove the first step of both the symmetry extension 
theorems described in the opening part of this section, \ie we prove results about end symmetries extending to global ALC automorphisms of $(M,\varphi)$.

\begin{lemma}
\label{lem:sextension_merge}
    Let $(M,\varphi)$ be an ALC \gtwo-manifold asymptotic to $\BC(\Sigma)$ with rate $\nu < -1$, and let $\mu \in (-3, \nu]$ such that the natural map $\mathcal{H}^3_\mu (M) \to H^3(M) \oplus H^4(M)$ is injective.
    Let $F_t$ be a path in $\tu{Aut}\,\BC(\Sigma)$ with $F_0 = \tu{Id}$ such that $F^*_t\varphi - \varphi$ has rate $\mu$ for every $t$. Then there is $F \in \tu{Aut}_\mu(M, \varphi)$ with asymptotic limit $F_1$.
\end{lemma}

\begin{proof}
    For each $t \in [0,1]$, define a diffeomorphism $G_t : M \to M$ as the identity on the compact piece, and mapping $(r, n) \mapsto \left(r, F_{t\rho(r)}\right)$ on the end $\BC(\Sigma) = \R^+ \times N$, for a cutoff function $\rho$ that is 0 for small $r$ and 1 for large $r$.

    Consider the path of torsion-free \gtwo-structures $\varphi_t := G_t^\ast \varphi$ on $M$. Because $G_t$ is exactly $F_t$ outside a compact set, $\varphi_t$ has rate $\mu$ to the same $\BC(\Sigma)$. Thus $\varphi_t$ defines a path in the moduli space~$\mathcal{M}_\mu$ of ALC torsion-free \gtstr s with the same asymptotics as $\varphi$ up to rate $\mu$.

    Theorem \ref{thm:Moduli:Spaces} together with the hypothesis on $\mathcal{H}^3_\mu(M)$ implies that $\mathcal{M}_\mu \to H^3(M) \oplus H^4(M)$ is an immersion on a neighbourhood of the class of $\varphi$. Because $([\varphi_t], [*\varphi_t]) \in H^3(M) \oplus H^4(M)$ is constant in $t$, that means all $\varphi_t$ represent the same element of $\mathcal{M}_\mu$. In particular, there exists an $H \in \mathcal{D}_{\mu+1}$ (\ie a diffeomorphism $H : M \to M$ asymptotic with rate $\mu+1$ to the identity) such that $H^\ast \varphi_1 = \varphi_0$.
    Then $F := G_1 \circ H : M \to M$ satisfies $F^\ast \varphi = H^\ast G_1^\ast \varphi = H^\ast \varphi_1 = \varphi$. Hence $F$ is an automorphism of $\varphi$, and it is moreover ALC with the same asymptotics as $G_1$, \ie it is asymptotic to the given~$F_1$.
\end{proof}
\begin{prop}\label{prop2:Decaying:symmetries}
Let $(M,\varphi)$ be an ALC \gtwo-manifold asymptotic to $\BC(\Sigma)$ with rate $\nu < -1$. If \eqref{eq:Symmetries:Assumption} holds, then the image of the restriction map
\[
\tu{Aut}_\nu (M,\varphi) \longrightarrow \tu{Aut}\,\BC(\Sigma)
\]
contains the identity component of $\tu{Aut}\,\BC(\Sigma)$.
\end{prop}

\begin{proof}
    Given an element $F_\BC : N \to N$ of the identity component of $\tu{Aut}\,\BC(\Sigma)$, consider a path $F_t$ in  $\tu{Aut}\,\BC(\Sigma)$ from $F_0 = \tu{Id}$ to $F_1 = F_\BC$.
    Hypothesis \eqref{eq:Symmetries:Assumption} means that we can apply Lemma \ref{lem:sextension_merge} with $\mu = \nu$.
\end{proof}

In the special case of the principal $S^1$-action on an end of cyclic type, we do not need to assume \eqref{eq:Symmetries:Assumption} thanks to the exponentially asymptotic $S^1$-invariance of the end assured by Theorem \ref{thm:Exp:Decay:ALC:G2}.

\begin{prop}\label{prop:S1_surj}
Let $(M, \varphi)$ be an ALC \gtwo-manifold asymptotic to a cyclic $\BC(\Sigma)$.
For any $\phi \in \R$ and every $\nu \in (-3,-1)$ there exists an ALC automorphism $F_\phi \in \tu{Aut}_\nu(M,\varphi)$ asymptotic to the principal action of $e^{i\phi}$ on $\BC(\Sigma)$.
\end{prop}

\begin{proof}
    Let $F_t \in \tu{Aut}\,\BC(\Sigma)$ be the principal action of  $e^{it\phi}$. Then Theorem~\ref{thm:Exp:Decay:ALC:G2} implies that each $F_t^\ast \varphi - \varphi$ is exponentially decaying, so it is of polynomial decay rate $\mu$ for any $\mu$, in particular for $\mu = -3 + \epsilon$. By Remark \ref{rmk:inj_across_L2} in the next section (as an application of the Hodge theory proved in Section \ref{sec:Hodge:ALC} to ALC $\gtwo$--manifolds) the hypotheses of Lemma \ref{lem:sextension_merge} are satisfied.
\end{proof}

\begin{corollary}\label{cor:Circle:Symmetry:Cyclic:ALC:Killing:Field}
For any $\nu > -2$, there exists a $1$-form $\gamma$ on $M$ with $\gamma = \theta + O(r^{\nu})$ such that $\mathcal{L}_{\gamma^\sharp}\varphi=0$. In particular, $\gamma$ is coclosed and harmonic with $\pi_7 d \gamma=0$.
\end{corollary}

\begin{proof}
    Proposition \ref{prop:S1_surj} gives a 1-parameter family $F_\phi \in \tu{Aut}_{\nu-1}(M,\varphi)$. Defining the vector field $\gamma^\sharp$ as $\frac{dF_\phi}{d\phi}$ at $\phi = 0$ ensures that $\mathcal{L}_{\gamma^\sharp}\varphi = 0$ because $F_\phi$ are automorphisms and $\gamma = \theta + O(r^\nu)$ because $F_\phi$ are asymptotic with rate $\nu-1$ to the principal action of $e^{i\phi}$.

    Now $d^\ast \gamma = \pi_7d\gamma = 0$ follows from Lemma \ref{lem:identities:0:1:forms:torsion:free}(iii). Since $d\gamma$ has type 14, $*d\gamma = -\varphi \wedge d\gamma$, so $\Delta \gamma = d^\ast d\gamma = 0$.
\end{proof}

\begin{remark}
We can also reverse the logic of Proposition \ref{prop:S1_surj} and Corollary \ref{cor:Circle:Symmetry:Cyclic:ALC:Killing:Field}: First finding a harmonic 1-form $\gamma = \theta + O(r^\nu)$, deducing from Lemma~\ref{lem:Vanishing:harmonic} that $d^\ast \gamma$ and $\pi_7^\varphi d\gamma$ both vanish, and hence from Lemma \ref{lem:identities:0:1:forms:torsion:free} that $\gamma^\sharp$ is automorphic, thus proving Corollary \ref{cor:Circle:Symmetry:Cyclic:ALC:Killing:Field}. Proposition \ref{prop:S1_surj} then follows by exponentiation.

In a similar way, Proposition \ref{prop2:Decaying:symmetries} can alternatively be proved by first arguing that any infinitesimal symmetry of $\BC(\Sigma)$ can be extended to the dual of a harmonic 1-form on $M$, which one shows must then be an infinitesimal automorphism of $(M,\varphi)$.
\end{remark}

Whichever way we go about arguing that certain elements of $\tu{Aut}\,\BC(\Sigma)$ can be extended to global symmetries of~$M$, none of the arguments we have presented address whether that extension is unique; thus if $G$ is not simply connected (in particular, if we consider the principal $S^1$--action on an ALC end of cyclic type) we have not yet ruled out that we merely end up with an action on~$M$ by the universal cover of $G$ rather than by $G$ itself.

To address this question it turns out to be helpful to deal separately with the case
where the Sasaki--Einstein 5-manifold $\Sigma$ associated with the ALC~\gtmfd~$M$ is exceptional. To this end it proves to be useful to build on Corollary \ref{cor:Circle:Symmetry:Cyclic:ALC:Killing:Field} to deduce that $M$ is in fact flat in this case. As a first step, we argue more generally that whenever the HYM connection $\theta$ associated to the end of a cyclic ALC \gtmfd~$M$ is flat, then $M$ must be reducible, Theorem \ref{mthm:PositiveMass} from the Introduction.

\begin{theorem}\label{thm:ALC:G2:Positive:Mass}
Let $(M,\varphi)$ be an ALC $\gtwo$--manifold of cyclic type asymptotic to $\BC (\Sigma)$ with rate $\nu<-1$. Furthermore assume that the HYM connection $\theta$ on $\BC(\Sigma)\ra\tu{C}(\Sigma)$ is flat. Then the automorphic vector field $\gamma^\sharp$ from Corollary \ref{cor:Circle:Symmetry:Cyclic:ALC:Killing:Field} is parallel, and $M$ is a finite quotient of 
the isometric product $\Sph^1 \times B$ where $B$ is an AC CY $3$-fold.
\end{theorem}

\begin{proof}
    Consider the automorphic 1-form $\gamma = \theta + O(r^{\nu})$ constructed in Corollary~\ref{cor:Circle:Symmetry:Cyclic:ALC:Killing:Field}. We want to show that our hypothesis that $\theta$ is flat forces $\gamma$ to be parallel.
Since $\gamma^\sharp$ is an automorphic vector field it is certainly a Killing field and so it suffices to show that $d\gamma=0$.

Because $\gamma^\sharp$ is automorphic with respect to the torsion-free
structure $\varphi$, $\gamma$ is harmonic and $d^*\gamma$ and $\pi_7 d \gamma$ both vanish.
Since $\gamma$ is harmonic and $d^\ast \gamma=0$, the $2$-form $d\gamma$ is closed and coclosed.
Now, recall that for closed and coclosed $2$-forms there are no indicial roots in the interval $(-6,-2)$
(see  Proposition \ref{prop:Closed:Even:Forms}). Hence because we assumed $\nu-1<-2$ and $d\theta=0$, then $d\gamma = d\theta + O(r^{\nu-1}) = O(r^{\nu-1}) $,  can in fact be improved to $d\gamma = O(r^{-6+\delta})$ for any $\delta>0$. Now this fast decay of $d \gamma$ allows us to integrate by parts to conclude that
\[
\| d\gamma\|_{L^2}^2 = -\int_{M}{d\gamma\wedge d\gamma\wedge\varphi}=0,
\]
where we have also used the fact that $d\gamma\in\Omega^2_{14}$ to write $|d\gamma|^2\,\dvol_{g_\varphi} = -d\gamma\wedge d\gamma\wedge\varphi$.

Because $\gamma$ is parallel, by the de Rham splitting theorem the universal cover of $M$ splits as an isometric product $\R \times B$ for a Calabi--Yau 3-fold $B$, with the lift of $\gamma^\sharp$ tangent to the $\R$ factor (in other words, $p^\ast\gamma = dt$, where $t$ is the $\R$ coordinate). Since the covering group of $p : \R \times B \to M$ preserves $\gamma$, it must act by products of translations on $\R$ and isometries of $B$.

For a sufficiently large $R$, define $\wt C \subset B$ as the intersection of $\{0\} \times B$ with the pre-image of $\BC_R(\Sigma) \subset M$ (the region where $r > R$). Provided that we take $R$ large enough that $\gamma$ is never 0 on the circle fibres, the composition $q : \wt C \to C_R(\Sigma)$ of the restriction of $p$ and the circle bundle map $\pi: \BC_R(\Sigma) \to C_R(\Sigma)$ is a covering map. It must be a finite cover, because $\wt C$ has at most two components by the line splitting theorem, and each component is a finite cover because $\Sigma$ has positive Ricci curvature.

Therefore the quotient of $\R \times B$ by the subgroup of the covering group that acts trivially on the $B$ factor yields a finite isometric cover $\Sph^1 \times B$.
Finally, $p$ being a local isometry means that we can view the Calabi--Yau metric $g_B$ as the restriction of $p^\ast g_M$ to the submanifold $\{0\} \times B$.
On the end (noting that $g_\BC = g'_\BC$ because $\theta$ is flat), $g_M = g_\BC + O(r^\nu) = \theta^2 + \pi^\ast g_C + O(r^\nu)$, so $p^\ast g_M = (p^\ast \theta)^2 + (\pi\circ p)^* g_C + O(r^\nu)$. Restricting to $\{0\} \times B$, the first term is $O(r^\nu)$ because $\theta = \gamma + O(r^\nu)$ and $p^\ast \gamma$ restricts to 0, while the second term equals $q^\ast g_C$ by the definition of $q$.
Thus $g_B = q^\ast g_C + O(r^\nu)$ on $\wt C$, implying that $B$ is AC.
\end{proof}

\begin{corollary}
\label{cor:flat}
    Let $(M,\varphi)$ be an ALC \gtwo-manifold of cyclic type asymptotic to $\BC(\Sigma)$ with $\Sigma$ exceptional.
    Then $M$ is a finite quotient of the flat ALC~\gtmfd~$\Sph^1 \times \C^3$.
\end{corollary}

\begin{proof}
If $\Sigma$ is exceptional then $H^2(\Sigma)=(0)$ and so $\theta$ must be flat. Hence Theorem \ref{thm:ALC:G2:Positive:Mass} applies to give a finite cover $\Sph^1 \times B$ with $B$ an AC Calabi--Yau 3-fold. 
Moreover, the link of $B$ must be the round $\Sph^5$, so $B$ is $\C^3$ with the Euclidean metric.
\end{proof}

\subsection{Extension of group actions}
\label{ss:canonical:foliation}

In order to conclude the proof of part (ii) of Theorem \ref{mthm:Asymptotics} from the Introduction, \ie the extension of the circle action from infinity to any cyclic ALC $\gtwo$--manifold, we need to address the question of why we have a circle rather than an $\R$--action from Proposition \ref{prop:S1_surj}. Consider more generally a Lie subgroup $G$ of the compact Lie group $\tu{Aut}\,\BC(\Sigma)^\circ$. By passing to a cover allowing non-effective actions if necessary, we can always assume that $G=G_0\times G_1$, where $G_0$ is a torus and $G_1$ is a simply connected semisimple Lie group. Proposition~\ref{prop2:Decaying:symmetries} and Corollary~\ref{cor:Circle:Symmetry:Cyclic:ALC:Killing:Field} yield an infinitesimal action of $G$ on $M$ by ALC automorphisms, or equivalently an automorphic action on $M$ of the universal cover of $G$. This is enough to deduce that the semisimple part $G_1$ of $G$ lifts to $\tu{Aut}_\nu (M,\varphi)$. Write $G_0=\Lie{g}_0/\Lambda$ for a full rank lattice $\Lambda$ in $\Lie{g}_0$. 
Propositions~\ref{prop2:Decaying:symmetries} and Corollary~\ref{cor:Circle:Symmetry:Cyclic:ALC:Killing:Field} yield an automorphic 
$\Lie{g}_0$--action on $M$ and elements in $\Lambda$ correspond to elements in the kernel of $\tu{Aut}_\nu (M,\varphi)\ra \tu{Aut}(\BC (\Sigma))$.

The following proposition shows that this kernel is always trivial. 
Note that assumption~\eqref{eq:Symmetries:Assumption} is not used here. 

\begin{prop}\label{prop:Decaying:symmetries}
For any ALC~\gtmfd~($M,\varphi)$ the map
\[
\tu{Aut}_\nu (M,\varphi) \longrightarrow \tu{Aut}\,\BC(\Sigma)
\]
is injective for any rate $\nu < -1$.
\proof
If $M$ is ALC of cyclic type asymptotic to $\BC (\Sigma)$ with exceptional $\Sigma$, then $M$ is flat by Corollary \ref{cor:flat}, and the claim now follows from Remark \ref{rmk:aut_flat}. We can therefore assume that $M$ is not of cyclic type with exceptional $\Sigma$.

The idea of the proof is to identify  the circle bundle $N \to \Sigma$, thought of as the boundary at infinity of $M$, as the parameter space for a canonical family of semi-infinite curves that foliate the (double cover of the) end of $M$. The fact that the family is canonical implies that any ALC automorphism $F$ of $(M,\varphi)$ acts naturally on this space. Then if $F$ has limit $F_\BC = \tu{id}$, this action is trivial and $F\equiv \tu{id}$ on the end of $M$ and therefore everywhere. Since we will work only on the end of $M$, in the dihedral case we pass once and for all to a double cover of the end.

One natural choice of canonical family of curves is provided by deformations of the family of radial geodesics on the model $\BC(\Sigma)$. While one can easily construct this family by studying the Jacobi operator of the radial geodesics, we find it quicker to replace geodesics with gradient flow lines of a certain essentially canonical function $u$ on $M$ with quadratic growth at infinity, namely the unique (up to an additive constant) solution to $\triangle u = 6$ with $u = \frac{1}{2}r^2 + O(r^{1+\delta})$.

 We now explain how to construct such a function $u$. An explicit calculation shows that the function $u_\BC =\tfrac{1}{2}r^2$ satisfies $\triangle_{g_\BC}u_\BC = 6 +O(r^{-1})$. Since $g_\varphi$ is asymptotic to $g_\BC$ with rate $-1$, the same is true if we replace $\triangle_{g_\BC}$ with $\triangle_{g_\varphi}$. Under our assumption, the indicial root calculations of Proposition \ref{prop:Harmonic:Functions:Cone} 
 imply that, for any $\delta>0$ sufficiently small, the scalar Laplacian $\triangle \co C^{2,\alpha}_{1+\delta}\ra C^{0,\alpha}_{-1+\delta}$ on $M$ is surjective with kernel consisting only of constant functions. Indeed, by the Lichnerowicz--Obata Theorem our assumption implies that there are no indicial roots for the scalar Laplacian (restricted to $\check{\iota}$--invariant functions in the dihedral case) on $\tu{C}(\Sigma)$ in $(0,1+\delta_0)$ for some $\delta_0>0$. The existence of $u$ follows.

Since $u$ is unique up to an additive constant, its gradient flow lines are canonical, and therefore preserved by any isometry (and in particular any automorphism) of $M$. Moreover, the asymptotic behaviour $u=u_\BC + O(r^{1+\delta})$ implies that for every $p\in N$ there is a unique gradient flow line asymptotic to $(0,\infty)\times \{ p\} \subset \BC(\Sigma)$, the radial geodesic on $\BC (\Sigma)$ through $p$, as $r\ra\infty$. If $F$ is an isometry that decays to the identity at infinity, the action of $F$ on the set of these integral curves is trivial and since these curves foliate the end of $M$ we deduce $F\equiv\tu{id}$ on the end.
\endproof
\end{prop}

Proposition \ref{prop:Decaying:symmetries} concludes the proof of part (ii) of Theorem \ref{mthm:Asymptotics} as well as of Theorem \ref{thm:Symmetries}.

\subsection{Uniqueness via symmetry}
\label{ss:uniqueness:symmetry}
We conclude this section by observing that our symmetry extension results, Theorems \ref{thm:Circle:Symmetry:Cyclic:ALC} and \ref{thm:Symmetries}, can be used to establish uniqueness results. For analogous arguments in the AC setting see \cite[Corollary 6.10]{Karigiannis:Lotay}. 

In \cite{FHN:Coho1:ALC} we constructed infinitely many families of cohomogeneity-one ALC $\gtwo$--manifolds, all with symmetry group $\sunitary{2}^2\times\unitary{1}$.

\begin{theorem}\label{thm:Uniqueness:Coho1:ALC}
Let $\tu{BC}(\Sigma)$ be the asymptotic model of one of the cohomogeneity one examples constructed in \cite{FHN:Coho1:ALC}.
Then the cohomogeneity-one ALC $\gtwo$--manifolds asymptotic to $\tu{BC}(\Sigma)$ constructed in \cite{FHN:Coho1:ALC} are the unique ALC $\gtwo$--manifolds asymptotic to it.
\proof
In \cite[Theorem E]{FHN:Coho1:ALC} we proved that the examples constructed in that paper provide an exhaustive list of the  $\sunitary{2}^2\times\unitary{1}$--invariant complete $\gtwo$--holonomy metrics. 
So it suffices to prove that any complete ALC $\gtwo$--manifold 
with the given asymptotic model is in fact $\sunitary{2}^2\times\unitary{1}$-invariant. 
The asymptotic model for the ALC metrics of \cite{FHN:Coho1:ALC} is $\BC(\Sigma)$ where $\Sigma$ is the Sasaki--Einstein manifold $\Sigma_{\tu{hom}}=\sunitary{2}^2/\triangle\unitary{1}\simeq S^2\times S^3$ endowed with its unique homogeneous Sasaki--Einstein structure or a finite quotient  $\Sigma_{\tu{hom}}/\Gamma$ for $\Gamma$ a finite subgroup of the center $\sunitary{2}^2\times\unitary{1}$ of the identity component $\sunitary{2}^2\times\unitary{1}$ of $\tu{Aut}\, \tu{BC}(\Sigma_{\tu{hom}})$. Since $\Sigma$ is a regular Sasaki--Einstein $5$-manifold, by Remark~\ref{rmk:regular:SE}
the assumption \eqref{eq:Symmetries:Assumption} of Theorem \ref{thm:Symmetries}
is satisfied. Therefore any complete ALC \gtwo--manifold asymptotic to $\BC (\Sigma)$ must be $\sunitary{2}^2\times\unitary{1}$--invariant by Theorems \ref{thm:Circle:Symmetry:Cyclic:ALC} and \ref{thm:Symmetries}.  
\endproof
\end{theorem}

\begin{remark*}
Despite we do not use cohomogeneity one methods in the proof of Theorem \ref{mthm:Dihedral}, the examples of dihedral ALC $\gtwo$--manifolds that we construct in this paper admit a cohomogeneity one action of $\sunitary{2}^2$. The model $\tu{BC}(\Sigma)$ for these examples has $\tu{Aut}\, \tu{BC}(\Sigma)^\circ =\sunitary{2}^2$ and $\Sigma$ is a finite quotient of $\Sigma_{\tu{hom}}$, so that Theorem \ref{thm:Symmetries} still implies that any ALC \gtwo--manifold sharing the same ALC model must be also $\sunitary{2}^2$--invariant. However, the classification of cohomogeneity one $\gtwo$--manifolds in \cite{FHN:Coho1:ALC}  does not apply since the symmetry group of the examples of Theorem \ref{mthm:Dihedral} is now only $\sunitary{2}^2$ and not $\sunitary{2}^2\times\unitary{1}$.
\end{remark*}

\section{Examples of moduli spaces}\label{sec:Dimension}

In prior work we have given two different methods to produce many ALC $\gtwo$--manifolds of cyclic type. In \cite{FHN:ALC:G2:from:AC:CY3} we constructed infinitely many examples of cyclic type by considering principal circle bundles over AC Calabi--Yau 3-folds (the construction was extended to the case of principal Seifert circle bundle over AC Calabi--Yau orbifolds in \cite{Foscolo:ALC:Spin7}). In \cite{FHN:Coho1:ALC} we instead constructed (and classified) all complete $\gtwo$--manifolds with a cohomogeneity-one action of $\sunitary{2}^2\times\unitary{1}$, producing as a byproduct infinitely many 1-parameter families (up to scale) of ALC $\gtwo$--manifolds of cyclic type. In this section we study the moduli spaces of ALC \gtwo--holonomy metrics arising from these constructions.

In the first part of the section we specialise the Hodge theory results of Section \ref{sec:Hodge:ALC} to $7$-dimensional ALC manifolds with tangent cone at infinity a Calabi--Yau cone, thus determining contributions to the dimension of the moduli space of ALC \gtwo--holonomy metrics that only depend on the topology of the underlying smooth manifold. Thanks to Propositions \ref{prop:Closed:Even:Forms} and \ref{prop:Closed:Odd:Forms}, when the asymptotic cone $\tu{C}(\Sigma)$ is a regular Calabi--Yau cone, \ie $\Sigma$ is a regular Sasaki--Einstein manifold,  all infinitesimal deformations of rate $\nu<-1$ are in fact accounted for by these topological contributions. We illustrate these computations with the cohomogeneity-one ALC $\gtwo$--manifolds of \cite{FHN:Coho1:ALC} and some related examples from \cite{FHN:ALC:G2:from:AC:CY3} and \cite[Section 4.3]{Foscolo:ALC:Spin7}. While these calculations are somewhat redundant in the cohomogeneity-one case because of the uniqueness results of Theorem \ref{thm:Uniqueness:Coho1:ALC}, the computations in this section allow one to interpret the parameters of these \gtwo--holonomy metrics in terms of cohomological data as for compact $\gtwo$--manifolds \cite{Joyce:Book}.

In the second part of the section we explain how to relate ALC $\gtwo$--deformations to AC Calabi--Yau deformations in the context of our construction of ALC $\gtwo$--holonomy metrics on principal circle bundles over AC Calabi--Yau 3-folds \cite{FHN:ALC:G2:from:AC:CY3}. Thanks to the recent classification of AC Calabi--Yau metrics by Conlon--Hein \cite{Conlon:Hein:Classification}, we are therefore able to compute the dimension of the moduli space of the ALC $\gtwo$--metrics produced in \cite{FHN:ALC:G2:from:AC:CY3} at rates beyond those accounted for purely by topological contributions. 

\subsection{Cohomogeneity-one ALC \gtwo--metrics and related examples}

Recall that we concluded the first part of the paper with a discussion of weighted Hodge theory on ALC manifolds. In particular, Theorems \ref{thm:ALC_hodge_fastdecay} and \ref{thm:ALC_hodge_jump} describe topological contributions to the spaces of closed and coclosed $k$-forms on an arbitrary ALC manifolds. In the following theorem we specialise these results to the case of $3$-forms on an ALC $7$-manifold $M$ with tangent cone at infinity a Calabi--Yau cone $\tu{C}(\Sigma)$ of complex dimension $3$ (or its quotient by an antiholomorphic involution in the case of an ALC manifold of dihedral type). By Theorem \ref{thm:Moduli:Spaces}, for ALC \gtwo--manifolds these computations provide at least a topological lower-bound for the dimension of the moduli space of ALC $\gtwo$--metrics. 

Recall the notation introduced in Section \ref{sec:Hodge:ALC} before the statement of Theorem \ref{thm:ALC_hodge_jump}. The compact manifold (orbifold) $X$ is the compactification of the ALC manifold $M$ obtained by collapsing the circle fibres of $\tu{BC}(\Sigma)$ at infinity. The cohomological maps $\Phi,\overline{\Phi}$ and $\Psi$ are natural maps that associate to a closed $k$-form its cohomology class (when defined) in, respectively, $H^k_c(M), H^k(X)$ and $H^k(M)$. We use the notation $\triangle \mathcal{H}^k_\lambda(M)$ to denote the quotient $\mathcal{H}^k_{\lambda+\epsilon}(M)/\mathcal{H}^k_{\lambda-\epsilon}(M)$ for $\epsilon>0$ small enough. Finally, for $\lambda$ an indicial root, $f(\mathcal{H}^k_\lambda(\Sigma))$ is the subspace of $\mathcal{H}^k_\lambda(M)$ which accounts for all the non-topological contributions to $\mathcal{H}^k_\lambda(M)$ for $\lambda\in (-k,-k+1)$ and $k\leq\frac{n}{2}$ by Proposition \ref{prop:ALC:closed:coclosed:exact}/Remark \ref{rmk:accounted}.

\begin{theorem}\label{thm:Topological:Hodge:G2}
Let $M^7$ be a complete ALC manifold with tangent cone at infinity $\tu{C}(\Sigma)$ a Calabi--Yau cone, or a quotient of one by an antiholomorphic involution.
\begin{enumerate}
\item There exists $\delta>0$ such that $\mathcal{H}^3_{\lambda}(M) \simeq \tu{im}\, H^3_c(M)\ra H^3(X)$
via $\overline{\Phi}$ for all $\lambda\in (-3-\delta, -3)$.
\item $\triangle\mathcal{H}^3_{-3}(M) \simeq \textup{im}\left( H^3(X)\ra H^3(\Sigma)\right) \oplus \textup{im}\left( H^4(M)\ra H^3_-(\Sigma)\right)$ via $\overline{\Phi}\oplus (\Psi\circ\ast)$.
\item $\triangle\mathcal{H}^3_\lambda (M)= f(\mathcal{H}^3_\lambda(\Sigma))$ for all $\lambda \in (-3,-2)$;
\item $\triangle\mathcal{H}^3_{-2}(M) /f(\mathcal{H}^3_{-2}(\Sigma)) \simeq \textup{im}\left( H^3(M)\ra H^2_-(\Sigma)\right)$ via $\Psi$.
\end{enumerate}
Moreover, if $\Sigma$ or its oriented double cover $\Sigma^+$ is regular as a Sasaki--Einstein 5-manifold then $f(\mathcal{H}^3_{-2}(\Sigma))=0$ and $\triangle\mathcal{H}^3_{\lambda}(M)=0$ for all $\lambda\in (-3,-1]\setminus\{ -2\}$.
\proof
Item (i) is a specialisation of Theorem \ref{thm:ALC_hodge_fastdecay} and items (ii), (iii) are specialisations of Theorem \ref{thm:ALC_hodge_jump}. The final statement follows from the last statements of Propositions \ref{prop:Closed:Even:Forms} and \ref{prop:Closed:Odd:Forms}.
\endproof
\end{theorem}

\begin{remark}\label{rmk:inj_across_L2}
Since $\Sigma$ has positive Ricci curvature we have $H^1_-(\Sigma)=0$ and therefore $H^3(X)\ra H^3(M)$ is injective by \eqref{eq:thom_exact}. Thus under the hypotheses of Theorem \ref{thm:Topological:Hodge:G2}, (i) means that there exists $\delta>0$ such that the natural map $\Psi \co \mathcal{H}^3_{\lambda}(M) \ra 
H^3(M)$ is injective (with image equal to that of $H^3_c(M)$) %
for every $\lambda<-3-\delta$. 
Further, (ii) implies that $\Psi\oplus (\Psi\circ\ast)\co \mathcal{H}^3_\lambda (M)\ra H^3(M)\oplus H^4(M)$ induced by $\rho\mapsto ([\rho],[\ast\rho])$ is injective for $\lambda<-3+\delta$; if $\Sigma$ or $\Sigma^+$ is a regular Sasaki--Einstein $5$-manifold then it remains injective for all $\lambda\leq -1$.
\end{remark}

In order to illustrate the use of these results in studying moduli spaces of ALC $\gtwo$--manifolds, we now apply Theorem \ref{thm:Topological:Hodge:G2} to calculate the dimension of $\mathcal{H}^3_{\nu}(M)$ for the cohomogeneity-one $\sunitary{2}^2\times\unitary{1}$--invariant ALC $\gtwo$--manifolds of \cite{FHN:Coho1:ALC} (and some related examples). Restricting to the case $\nu>-3-\delta$ for $\delta>0$ sufficiently small, by Theorem \ref{thm:Moduli:Spaces} the dimension of $\mathcal{H}^3_{\nu}(M)$ coincides with the dimension of the moduli space of these ALC $\gtwo$--manifolds modulo scaling.

We begin with some general remarks. In all the cohomogeneity-one cases considered here, the subgroup $\sunitary{2}^2\subset\sunitary{2}^2\times\unitary{1}$ already acts with cohomogeneity one. The noncompact $7$-manifold $M$ is therefore described by a group diagram of the form
\[
K\subset K_0\subset \sunitary{2}^2,
\]
meaning that $M$ contains a submanifold (the unique singular orbit) $\Sigma_0=\sunitary{2}^2/K_0$ and $M\setminus\Sigma_0 = \R_+\times N$, where $N=\sunitary{2}^2/K$ is the principal orbit. Moreover, since $M$ is smooth the quotient $K_0/K$ must be diffeomorphic to a sphere. Note that for dimensional reasons $K$ is a finite group. The Sasaki--Einstein $5$-manifold link of the asymptotic cone of $M$ is instead $\Sigma = \sunitary{2}^2/K_\infty$, where $1\ra \triangle\unitary{1}\ra K_\infty\ra K\ra 1$ for $\triangle\unitary{1}$ the diagonal circle in the maximal torus of $
\sunitary{2}^2$. It follows that the compactification $X$ of $M$ obtained by adding to $M$ a copy of $\Sigma$ by collapsing the circle fibres of the ALC end is also a cohomogeneity-one manifold with group diagram
\[
K\subset K_0, K_\infty\subset\sunitary{2}^2,
\]
\ie the compactification $X$ contains two singular orbits $\Sigma_0=\sunitary{2}^2/K_0$ and $\Sigma=\sunitary{2}^2/K_\infty$ and $X\setminus(\Sigma \sqcup\Sigma_0)$ is $\sunitary{2}^2$--equivariantly diffeomorphic to $\R_+\times N$, where $N=\sunitary{2}^2/K$.

\begin{example}
Two of the families of cohomogeneity-one ALC $\gtwo$--metrics constructed in \cite{FHN:Coho1:ALC} (known as the $\mathbb{B}_7$ and $\mathbb{D}_7$ families of ALC $\gtwo$--metrics in the physics literature) are both defined on the same smooth manifold $S^3\times\R^4$, but viewed as a cohomogeneity-one manifold in two different ways. (Some of these metrics were known before \cite{FHN:Coho1:ALC} and we refer to the introduction of that paper for a precise history.) In both cases $K, K_0,K_\infty$ are contained in a subgroup $K'$ of $\sunitary{2}^2$, so that $M$ and its compactification $X$ can be described as the bundles over $\sunitary{2}^2/K'$ associated to the principal $K'$--bundle $\sunitary{2}^2\ra \sunitary{2}^2/K'$ and fibres the cohomogeneity-one manifolds with group diagrams, respectively,
\[
K \subset K_0\subset K', \qquad K \subset K_0,K_\infty\subset K'.
\]
\begin{enumerate}[left=0pt]
\item For the so-called $\mathbb{B}_7$ family \cite[Theorem B (i)]{FHN:Coho1:ALC} we have $K=\{1\}$ and $K_0=\triangle\sunitary{2}$ so that we can take $K'=K_0$. Then $\sunitary{2}^2\ra \sunitary{2}^2/K'\simeq S^3$ is a (necessarily trivial) $\sunitary{2}$--bundle and $M$ and $X$ are bundles over $S^3$ with fibres $\R^4$ and $\C\PP^2$ respectively. Thus $M=S^3\times \R^4$, $X=S^3\times\CP^2$, $N=S^3\times S^3$ and $\Sigma=S^2\times S^3$ and so according to Theorem \ref{thm:Topological:Hodge:G2} we have
\[
\dim \mathcal{H}^3_{\nu} = \begin{cases}
0 & \mbox{if } \nu<-3,\\
1 & \mbox{if } -3<\nu<-1.
\end{cases}
\]
\item For the so-called $\mathbb{D}_7$ family \cite[Theorem B (ii)]{FHN:Coho1:ALC} we have $K=\{1\}$ and $K_0=\unitary{1}\times\sunitary{2}$ so that we can take $K'=\unitary{1}\times\sunitary{2}$. Then we have $\sunitary{2}^2/K'\simeq S^2$ and $M$ and $X$ are fibre bundles over $S^2$ of the form $\sunitary{2}\times_{\unitary{1}}F$. The fibres $F$ are, respectively, $S^1\times \R^4$ and $S^5$ with $S^1$--actions given by the circle action on the first factor of $S^1\times \R^4$ and the action on $S^5$ induced by the diagonal action on $\C^3$. Hence $M=S^3\times \R^4$ and $X$ is the unit sphere bundle in $\mathcal{O}(1)\oplus\mathcal{O}(1)\oplus\mathcal{O}(1)\ra S^2$. Note that as a smooth sphere bundle the latter is trivial. Hence 
\[
\dim \mathcal{H}^3_{\nu} = \begin{cases}
0 & \mbox{if } \nu<-2,\\
1 & \mbox{if } -2<\nu<-1.
\end{cases}
\]
\end{enumerate}
\end{example}

\begin{example}
\label{ex:toric:Seifert}
A similar computation can be carried out for the infinitely many families of (Seifert) ALC $\gtwo$--holonomy metrics on $S^3\times\R^4$ constructed in \cite[Theorem 4.12]{Foscolo:ALC:Spin7}. Each family is labelled by a choice of positive integers $(p_1,p_2,q_1,q_2)$ such that $p_1+p_2=q_1+q_2$ and $(p_i,q_j)$ are coprime. This choice determines the tangent cone at infinity and because these cones are distinct for different choices of $(p_1,p_2,q_1,q_2)$, these families are all mutually nonisometric. The case $p_1=p_2=q_1=q_2=1$ corresponds to the~$\mathbb{D}_7$ family discussed above, but the metric is not of cohomogeneity-one in any other case (since for example the Sasaki--Einstein cross-section of the tangent cone at infinity will not be homogeneous): generically the metrics have only a \mbox{$T^3$--symmetry}, while in the special case ${p_1=p_2=p}$ there is a larger cohomogeneity-two symmetry group $\sunitary{2}\times T^2$. 
Nevertheless, in all cases we still have ${M=S^3\times \R^4}$, $N=S^3\times S^3$ and $\Sigma$ is diffeomorphic to $S^2\times S^3$. The compact manifold $X$ can be described as the unit sphere bundle in the orbibundle $\mathcal{O}(1)\oplus\mathcal{O}(q_1)\oplus\mathcal{O}(q_2)$ over the weighted projective line $W\C\PP^1[p_1,p_2]$. In particular, $H^3(X)=0$ and $\dim \mathcal{H}^3_{\nu}$ is as for the $\mathbb{D}_7$ family. 

Note however that in these examples (except when it reduces to the $\mathbb{D}_7$ family)
the Sasaki--Einstein cross-section $\Sigma$ is not regular. Therefore we cannot immediately  exclude the existence  of  further non-topological contributions to the moduli space of ALC $\gtwo$--holonomy metrics. See, however, the discussion in the next subsection. 
\end{example}

The infinitely many cohomogeneity-one ALC $\gtwo$--manifolds constructed in \cite[Theorem D (i)]{FHN:Coho1:ALC} are perhaps best understood as a special instance of the construction of \cite{FHN:ALC:G2:from:AC:CY3}. 
In the latter paper we produced families of ALC $\gtwo$--holonomy metrics on nontrivial principal circle bundles $M\ra B$ over an AC Calabi--Yau 3-fold $(B,\omega_0,\Omega_0)$ subject only to the constraint that $c_1(M)\cup[\omega_0]=0\in H^4(B)$. This construction yields cohomogeneity-one ALC $\gtwo$--metrics when the AC Calabi--Yau structure 
on $B$ is of cohomogeneity one. The simplest instance of this construction is provided by taking~$B$ to be the small resolution of the conifold $\{z_1^2 + z_2^2 + z_3^2 + z_4^2 = 0\} \subset \mathbb{C}^4$ (in this case there is no constraint on $c_1(M)$ because $H^4(B)=0$ as a small resolution);  this yields a different construction of metrics belonging to the~$\mathbb{D}_7$ family above. Another simple case to analyse is when $B=K_D$ is the canonical line bundle over a del Pezzo surface $D$. The cohomogeneity-one examples in\mbox{\cite[Theorem D (i)]{FHN:Coho1:ALC}} correspond to the choice $B=K_{\C\PP^1\times\C\PP^1}$, the canonical line bundle of  $\C\PP^1\times\C\PP^1$.

\begin{example}
Let $D$ be $\C\PP^1\times\C\PP^1$ or the blow-up of $\C\PP^2$ at $b\in \{ 1,\dots, 8\}$ points in general position. (Here we can safely exclude the case $D=\C\PP^2$ because of Corollary \ref{cor:flat}.) A K\"ahler class $[\omega_0]$ on $K_D$ can be identified with a K\"ahler class in $H^2(D)$. Consider now a nontrivial line bundle $L\ra D$ satisfying $c_1(L)\cup [\omega_0]=0$, 
let $S\subset L$ be the unit circle bundle in~$L$ and set $M=\pi^\ast S$. In the cohomogeneity-one case $D=\C\PP^1\times\C\PP^1$, the choice of $L$ is labelled by two positive integers $m,n>0$. The noncompact $7$-manifold $M$ retracts to the $5$-manifold $S$;  $S$ is finitely covered by the connected sum of $b$ copies of $S^2\times S^3$, where we set $b=1$ if $D=\C\PP^1\times \C\PP^1$. The Sasaki--Einstein 5-manifold $\Sigma$ is also covered by the connected sum of $b$ copies of $S^2\times S^3$ (in fact, it is simply connected unless $D=\C\PP^1\times\C\PP^1$, in which case the fundamental group of $\Sigma$ is~$\Z_2$). Since $H^3_c(M)\simeq H^4(M)\simeq H^4(S)=0$, a portion of the exact sequence \eqref{eq:pair_exact} reads
\[
H^3_c(M)=0\ra H^3(X)\ra H^3(\Sigma)\ra H^4_c(M)\ra H^4(X)\ra 0=H^4(M)
\] 
while the exact sequence \eqref{eq:thom_exact} is
\[
H^1(\Sigma)=0\ra H^3(X)\ra H^3(M)\ra H^2(\Sigma) \ra H^4(X) \ra 0=H^4(M).
\]
Now, the compactification $X$ of $M$ obtained by collapsing the circle fibres in $S$ can be identified with the unit $3$-sphere bundle in $L\oplus K_D\ra D$. We distinguish two cases.
\begin{enumerate}
\item If $c_1(L)\cup c_1(K_D)=0$ (for example if $[\omega_0]$ is a multiple of $c_1(K_D)$, \ie the K\"ahler class of $K_D$ is compactly supported) then the Euler class of the bundle $L\oplus K_D$ is trivial. The Gysin sequence then yields isomorphisms $H^3(X)\simeq H^0(D)$ (by push-forward) and $H^4(D)\simeq H^4(X)$ (by pull-back). In particular, $H^3(X)$ is one-dimensional. We conclude that
\[
\dim \mathcal{H}^3_{\nu} = \begin{cases}
0 & \mbox{if } \nu<-3,\\
1 & \mbox{if } -3<\nu<-2,\\
b & \mbox{if } -2<\nu<-1.
\end{cases}
\]
\item Otherwise the Euler class of the bundle $L\oplus K_D$ is not trivial and the Gysin sequence forces $H^3(X)=0$. We conclude that
\[
\dim \mathcal{H}^3_{\nu} = \begin{cases}
0 & \mbox{if } \nu<-2,\\
b & \mbox{if } -2<\nu<-1.
\end{cases}
\]
\end{enumerate}
\end{example}

Finally, the existence of a further $1$-parameter family of cohomogeneity-one ALC $\gtwo$--metrics of dihedral type, termed the $\mathbb{A}_7$ family, has been conjectured in the physics literature~\cite{ALC:A7}. These conjectural metrics are invariant under $\sunitary{2}^2$ but not $\sunitary{2}^2\times\unitary{1}$ so they are not covered by our work \cite{FHN:Coho1:ALC}, 
where the extra $\unitary{1}$ symmetry is used crucially to obtain 
a more tractable ODE system. 
Our main Theorem \ref{mthm:Dihedral} establishes the existence of (some of) these metrics by analytic methods, but uniqueness in the construction and invariance of the building blocks can be used to prove that these metrics are indeed $\sunitary{2}^2$--invariant. However, we still lack an ODE-based proof of the existence 
of the $\mathbb{A}_7$ family.
The following computation shows that the construction of Theorem \ref{mthm:Dihedral} is locally surjective.

\begin{example}
The description of the $\sunitary{2}^2$--action implies that $M$ is $\sunitary{2}^2$--equivariantly diffeo\-morphic to the (necessarily trivial) fibration $\sunitary{2}^2\times_{\triangle\sunitary{2}}F$, with $F$ the double cover of the Atiyah--Hitchin manifold. Hence $M=S^3\times (\CP^2\setminus\R\PP^2)$, $X=S^3\times\CP^2$, $N=(S^3\times S^3)/\triangle\Z_4$ and $\Sigma=(S^2\times S^3)/\Z_2$, where $\Z_2$ acts as the antipodal map on the first factor. This is an ALC manifold of dihedral type and applying Theorem \ref{thm:Topological:Hodge:G2} we obtain
\[
\dim \mathcal{H}^3_{\nu} = \begin{cases}
0 & \mbox{if } \nu<-3,\\
1 & \mbox{if } -3<\nu<-1
\end{cases}
\]
since $H^3(M)\ra H^3(\Sigma)$ is an isomorphism of $1$-dimensional spaces, the map $H^3(M)\ra H^2_-(\Sigma)$ vanishes and $H^4(M)=0=H^2(\Sigma)$. 
\end{example}

\subsection{ALC \gtwo--metrics on principal circle bundles over AC Calabi--Yau 3-folds} In \cite{FHN:ALC:G2:from:AC:CY3} (extended to the Seifert case in \cite[Section 4.3]{Foscolo:ALC:Spin7}) we established a general existence result for (highly collapsed) ALC \gtwo--metrics on the total space of any suitable nontrivial principal circle bundle $M\ra B$ over an AC Calabi--Yau 3-fold $B$. In this context, it is natural to compare the moduli space of ALC $\gtwo$--metrics on $M$ with the moduli space of AC Calabi--Yau structures (satisfying appropriate constraints) on the underlying $6$-manifold $B$. As a coda to the discussion in this section, we will now explain why these two different moduli spaces are locally diffeomorphic. In particular, exploiting the recent classification of AC Calabi--Yau manifolds by Conlon--Hein \cite{Conlon:Hein:Classification}, for all the examples of ALC $\gtwo$--manifolds arising from \cite{FHN:ALC:G2:from:AC:CY3} we are able to determine the dimension of the moduli space of ALC $\gtwo$--metrics even for decay rates beyond the range where all moduli are accounted for by cohomology. 
(The cohomological contributions accounted for by Theorem \ref{thm:Topological:Hodge:G2} can be directly compared to topological contributions to the moduli space of AC Calabi--Yau structures via the Gysin sequence). 

Let $M\ra B$ be a principal circle bundle over a noncompact 1-ended $6$-manifold $B$. The ALC $\gtwo$--metrics on $M$ constructed in \cite{FHN:ALC:G2:from:AC:CY3} are $S^1$--invariant (a posteriori necessarily so, given Theorem \ref{thm:Circle:Symmetry:Cyclic:ALC}) and take the form
\[
\varphi = \epsilon\,\theta\wedge\omega + h^{\frac{3}{4}}\Real\Omega, \qquad g_\varphi = \sqrt{h}\, g_{\omega,\Omega} + \epsilon^2\, h^{-1}\theta^2,
\]
where $(\omega,\Omega)$ is an $\sunitary{3}$--structure on $B$, $h\co B\ra \R_{>0}$, $\theta$ a connection on $M\ra B$ and $\epsilon>0$ a small parameter. The torsion-free condition for $\varphi$ becomes a nonlinear PDE system on $B$ for $(\omega,\Omega,h,\theta)$ that in \cite{FHN:ALC:G2:from:AC:CY3} we called the \emph{Apostolov--Salamon equations} (since it was first considered in the mathematics literature by Apostolov--Salamon in~\cite{Apostolov:Salamon}). The Apostolov--Salamon equations depend real analytically on the small parameter $\epsilon\geq 0$ and when $\epsilon= 0$ they reduce to the Calabi--Yau condition for the $\sunitary{3}$--structure $(\omega,\Omega)$ together with the (abelian) Calabi--Yau monopole equation for $(h,\theta)$. The strategy of \cite{FHN:ALC:G2:from:AC:CY3} is to construct a real analytic family of solutions 
to the Apostolov--Salamon equations for $\epsilon\geq 0$ sufficiently small starting from an AC Calabi--Yau structure on $B$ and the unique Hermitian Yang--Mills connection $\theta_0$ on $B$ with curvature in $c_1(M)$ (which gives rise to the solution $(1,\theta_0)$ of the Calabi--Yau monopole equation). This curve of solutions exists for any AC Calabi--Yau structure $(\omega_0,\Omega_0)$ on $B$ subject only to the constraint that
$c_1(M)\cup [\omega_0]=0\in H^4(B)$ for the chosen nonzero cohomology class $c_1(M)$.

\begin{theorem}\label{thm:Moduli:AC:CY}
Let $(B,\omega_0,\Omega_0)$ be an AC Calabi--Yau 3-fold asymptotic to the Calabi--Yau cone $\tu{C}(\Sigma)$ and let $M\ra B$ be a principal circle bundle with $c_1(M)\cup [\omega_0]=0\in H^4(B)$. Fix $\epsilon_0>0$ sufficiently small so that for all $\epsilon\in (0,\epsilon_0)$ \cite[Theorem 1.1]{FHN:ALC:G2:from:AC:CY3} guarantees the existence of an $S^1$--invariant ALC torsion-free $\gtwo$--structure $\varphi_\epsilon$ on $M$ with length of the circle fibre of the ALC end equal to $2\pi\epsilon$.
\begin{enumerate}[left=0pt]
\item For $\nu\in (-3,-1)$ generic there exists a smooth moduli space $\mathcal{M}_\nu^{\tu{CY}}$ of AC Calabi--Yau deformations $(\omega'_0,\Omega'_0)$ of $(\omega_0,\Omega_0)$ on $B$ asymptotic with rate $\nu$ to the same asymptotic cone $C(\Sigma)$ and satisfying $c_1(M)\cup [\omega'_0]=0\in H^4(B)$.
\item For $\nu$ as above, there exists $\epsilon'_0\leq \epsilon_0$ such that for all $\epsilon\in (0,\epsilon'_0)$ the moduli space $\mathcal{M}_\nu^{\tu{CY}}$ is locally diffeomorphic to the moduli space $\mathcal{M}_\nu=\mathcal{M}_\nu^\epsilon$ of ALC $\gtwo$--deformations of $\varphi_\epsilon$ of rate $\nu$ constructed in Theorem \ref{thm:Moduli:Spaces}.
\end{enumerate}
\end{theorem}
For the precise definition of an AC Calabi--Yau 3-fold we refer the reader to \cite[Definition 3.10]{FHN:ALC:G2:from:AC:CY3}. 
We point out that this definition is equivalent to requiring that the isometric product $B\times S^1$ be an ALC $\gtwo$--manifold as in Definition \ref{def:ALC:G2} asymptotic to the trivial circle bundle $\tu{C}(\Sigma)\times S^1$.

The notation $\mathcal{M}_\nu^\epsilon$ emphasises the fact that we are gauging away scaling by considering a moduli space of ALC $\gtwo$--metrics with fixed asymptotic length $2\pi\epsilon$ of the circle fibres at infinity. The moduli space of Calabi--Yau metrics $\mathcal{M}_\nu^{\tu{CY}}$ does instead include deformations arising from scaling.

\proof
We provide a detailed sketch of the proof. Fix $\epsilon>0$ as small as necessary and set $\varphi=\varphi_\epsilon$. In the limit $\epsilon\ra 0$, we will first explain how to interpret the tangent space $\mathcal{H}^3_\nu(M,g_\varphi)$ to $\mathcal{M}_\nu$ at the diffeomorphism-orbit of $\varphi$ in terms of limiting data purely on the AC Calabi--Yau 3-fold $(B,\omega_0,\Omega_0)$. We will then explain how to use existing results about AC Calabi--Yau 3-folds to identify the limiting data as the tangent space to the moduli space $\mathcal{M}_\nu^{\tu{CY}}$ at the diffeomorphism-orbit of $(\omega_0,\Omega_0)$. 

Consider a $3$-form $\rho\in \mathcal{H}^3_\nu(M,g_\varphi)$ for some $\nu<-1$. Together with the invariance of the slice of Proposition \ref{prop:Slice:Thm} under $\tu{Aut}_\nu (M,\varphi)$, a linearisation of the proof of Theorem \ref{thm:Circle:Symmetry:Cyclic:ALC} implies that $\rho$ must be $S^1$--invariant. Therefore $\rho$ can be written in the form
\[
\rho = \epsilon\,\theta\wedge\rho_2 + h^{\frac{1}{2}}\ast_{\omega,\Omega}\rho_3, \qquad \ast_{\varphi}\rho = \epsilon\,\theta\wedge \rho_3 + h \ast_{\omega,\Omega}\rho_2,
\]
for a pair $(\rho_2,\rho_3)\in \Omega^2(B)\oplus\Omega^3(B)$ of rate $O(r^\nu)$. The pair $(\rho_2,\rho_3)$ satisfies a PDE system on $B$ depending on $(\omega,\Omega,h,\theta)$ and $\epsilon$:
\begin{equation}\label{eq:AS:linearised}
d\rho_2=0= d\left( h^{\frac{1}{2}}\ast_{\omega,\Omega}\rho_3\right) + \epsilon d\theta\wedge\rho_2, \qquad d\rho_3=0= d\left( h \ast_{\omega,\Omega}\rho_2\right) +\epsilon d\theta\wedge \rho_3.
\end{equation}
This system is real analytic in $\epsilon$ all the way to $\epsilon=0$. Since $(\omega,\Omega,h)$ converges smoothly to $(\omega_0,\Omega_0,1)$ as $\epsilon\ra 0$, solutions to the limiting system for $\epsilon=0$ can be identified with the subspace of $\mathcal{H}_\nu^2(B,g_0)\oplus\mathcal{H}^3_\nu(B,g_0)$ cut out by the constraint $c_1(M)\cup [\rho_2]=0\in H^4(B)$. Here $\mathcal{H}^k_\nu(B,g_0)$ denotes the space of $k$-forms on $B$ of rate $\nu$ that are closed and coclosed with respect to the metric $g_0$ induced by the AC Calabi--Yau structure $(\omega_0,
\Omega_0)$. The cohomological constraint $c_1(M)\cup [\rho_2]=0\in H^4(B)$ arises from the fact that~\eqref{eq:AS:linearised} implies that $d\theta\wedge\rho_2$ is exact for $\epsilon>0$. The analogous constraint $c_1(M)\cup \left[ \rho_3\right]=0$ is automatically satisfied since one can show that $H^5(B)=0$ for any AC Calabi--Yau 3-fold. 

The linear analysis on AC Calabi--Yau 3-folds developed in \cite[Section 5]{FHN:ALC:G2:from:AC:CY3} can now be used to identify the space of solutions $(\rho_2,\rho_3)$ to \eqref{eq:AS:linearised} for all $\epsilon>0$ sufficiently small with the graph of a bounded linear map on the space of solutions of the limiting $\epsilon=0$ system. Indeed, we can write \eqref{eq:AS:linearised} as the system
\[
d\rho=0, \qquad d^{\ast_{g_0}}\rho = \ast_{g_0}R(\rho)
\]
for the pair of forms $\rho=(\rho_2,\rho_3)$ and  an appropriate remainder term $R(\rho)$. The two components of this remainder term $R(\rho)$ are given by
\begin{equation}\label{eq:AS:linearised:remainder}
d\left( (h\ast_{\omega,\Omega}-\ast_{g_0})\rho_2\right)+\epsilon d\theta\wedge\rho_3, \qquad d\left( (h^{\frac{1}{2}}\ast_{\omega,\Omega}-\ast_{g_0})\rho_3\right)+\epsilon d\theta\wedge\rho_2,
\end{equation}
for the pair $(\rho_2,\rho_3)$ of closed forms. By \cite[Theorem 8.1]{FHN:ALC:G2:from:AC:CY3}, $g_{\omega,\Omega}-g_0=O(\epsilon r^{-1})$, $h-1=O(\epsilon r^{-2+\delta})$ for any $\delta>0$ and $d\theta=O(r^{-2})$. We will use the following observation to rewrite the remainder terms.
\begin{lemma}
Let $\gamma$ be an exact $k$-form of rate $\mu$ on the AC manifold $(B,g_0)$. Then $\gamma=d\xi$ where $\xi=O(r^{\max\{-k+1,\mu+1\}})$.
\proof
We will first show that we can find a primitive $\xi_\infty$ of the restriction of $\gamma$ to the AC end with $\xi_\infty=O(r^{\mu+1})$. If $\mu<-k$, this is just the first part of Lemma \ref{lem:Exact:forms:cone}. We can therefore assume that $\mu\geq -k$. Write $\gamma = dr\wedge\alpha + \beta$. The fact that $\gamma$ is closed implies that $d_\Sigma\alpha - \partial_r\beta = 0= d_\Sigma\beta$. If $\gamma$ is moreover assumed to be exact, then $[\beta|_{\{ r=\tu{const}\}}]=0\in H^{k}(\Sigma)$. In particular, given $r_0$ large, we can write $\beta|_{\{ r=r_0\}}=d_\Sigma\xi_{r_0}$ for some $(k-1)$-form $\xi_{r_0}$ on $\Sigma$. Then
\[
\gamma = d\xi_\infty, \qquad \xi_\infty=  \int_{r_0}^r{\alpha\, dr} + \xi_{r_0}
\]
with $\xi_\infty = O(r^{\mu+1})$ as claimed since $\xi_\infty=O(r^{-k+1})$.

Now, fix a cutoff function $\chi$ with $\chi\equiv 1$ on the AC end. The form $\xi-d(\chi\xi_\infty)$ is a closed compactly supported form representing a cohomology class in $\ker H^k_c(B)\ra H^k(B)$. We can therefore write $\gamma - d(\chi\xi_\infty) = d\xi' + d(\chi \tau)$, with $\xi'$ compactly supported and $\tau\in\mathcal{H}^{k-1}(\Sigma)$ such that $[\tau]\in H^{k-1}(M)$ maps to $[\xi-d(\chi\xi_\infty)]\in H^k_c(B)$. We can then set $\xi = \chi(\xi_\infty + \tau)+\xi'$.
\endproof
\end{lemma}

Using this lemma, we can rewrite \eqref{eq:AS:linearised} as the system
\[
d\rho_2=0=d\rho_3, \qquad d^{\ast_{g_0}}\rho_2 = \ast_{g_0} dQ_4(\rho_2,\rho_3),\qquad d^{\ast_{g_0}}\rho_3 = \ast_{g_0} dQ_3(\rho_2,\rho_3)
\]
where if $\rho_2,\rho_3\in C^{1,\alpha}_\mu$ then $Q_4(\rho_2,\rho_3)$ is of class $C^{1,\alpha}$ and has rate $\max{ \{ -4,\mu-1\}}$ and $Q_3(\rho_2,\rho_3)$ is of class $C^{1,\alpha}$ and has rate $\max{ \{ -3,\mu-1\}}$. In particular, since we assume $\nu>-3$, $Q$ defines an operator $Q\co C^{1,\alpha}_\nu\ra C^{1,\alpha}_\nu$ with norm that can be taken as small as necessary as $\epsilon\ra 0$. The claimed isomorphism between the subspace of $\mathcal{H}^2_\nu(B,g_0)\oplus\mathcal{H}^3_\nu(B,g_0)$ cut out by $c_1(M)\cup [\rho_2]=0$ and solutions of \eqref{eq:AS:linearised} is therefore established by showing that the equations
\[
d\rho_2=0=d\rho_3, \qquad d^{\ast_{g_0}}\rho_2 = \ast_{g_0} dQ_4,\qquad d^{\ast_{g_0}}\rho_3 = \ast_{g_0} dQ_3
\]
have a solution $(\rho_2,\rho_3)\in C^{1,\alpha}_\nu$ for every $Q_4,Q_3\in C^{1,\alpha}_\nu$ (with corresponding estimates). This follows from the analysis developed in \cite[Section 5]{FHN:ALC:G2:from:AC:CY3} by looking for a solution of the form $\rho_2=d\gamma_1$ and $\rho_3=d\gamma_2$ for $\gamma_1,\gamma_2\in C^{2,\alpha}_{\nu+1}$ satisfying
\[
\triangle\gamma_1 = \ast dQ_4, \qquad \triangle \gamma_2=\ast dQ_3.
\]
By \cite[Lemma 5.6]{FHN:ALC:G2:from:AC:CY3} any solution of this second-order system must satisfy $d^\ast\gamma_i=0$ (since $d^\ast\gamma_1$ is a decaying harmonic function and $d^\ast\gamma_2$ is a decaying harmonic $1$-form), so that $(d\gamma_1,d\gamma_2)$ is a solution of the original first-order system. The existence (with estimates) of $\gamma_1$ and $\gamma_2$ follows from \cite[Lemma 5.6]{FHN:ALC:G2:from:AC:CY3} (following the same argument given in \cite[Proposition 5.8]{FHN:ALC:G2:from:AC:CY3}) and from \cite[Proposition 5.16]{FHN:ALC:G2:from:AC:CY3} (since we assume $\nu>-3$, only the initial easier part of the proof is necessary).

It remains to identify, for $\nu\in (-3,-1)$ generic, the space $\mathcal{H}_\nu^2(B,g_0)\oplus\mathcal{H}^3_\nu(B,g_0)$ with the tangent space of an unobstructed moduli space $\widetilde{\mathcal{M}}_\nu^\tu{CY}$ of AC Calabi--Yau deformations of rate $\nu$ at $(\omega_0,\Omega_0)$, with $\mathcal{M}_\nu^\tu{CY}\subseteq\widetilde{\mathcal{M}}_\nu^\tu{CY}$ the submanifold cut out by the constraint $[\omega]\cup c_1(M)=0$. As far as we know, the construction of $\widetilde{\mathcal{M}}_\nu^\tu{CY}$ is not explicitly stated in the literature. However, it can be deduced easily from the analysis we developed in \cite{FHN:ALC:G2:from:AC:CY3} and the existing literature on AC Calabi--Yau metrics. We sketch the main steps.
\begin{enumerate}[left=0pt]
\item Firstly, by proving that closed and coclosed forms of rate $\nu$ are of particular $(p,q)$--types one shows that they coincide with infinitesimal Calabi--Yau deformations. More precisely, one shows that any form $\rho_2\in\mathcal{H}^2_\nu(B,g_0)$ satisfies $\rho_2\wedge\omega_0^2=0=\rho_2\wedge\Omega_0$ (\ie $\rho_2$ is a primitive $(1,1)$--form) and any $\rho_3\in\mathcal{H}^3_\nu(B,g_0)$ satisfies $\rho_3\wedge\omega_0=0=\rho_3\wedge\Omega_0$ (\ie $\rho_3$ is a primitive form of type $(1,2)+(2,1)$). These results hold under the assumption $\nu<0$ and use the fact that $g_0$ has holonomy contained in $\sunitary{3}$, so that, for example, $\ast_{g_0} (\rho_2\wedge\omega_0^2)$ is a decaying harmonic function, and a vanishing result for decaying harmonic functions and 1-forms on an AC Calabi--Yau 3-fold 
\cite[Lemma 5.6]{FHN:ALC:G2:from:AC:CY3}
(analogous to Lemma \ref{lem:Vanishing:harmonic}). It follows that, for any $(\rho_2,\rho_3)\in \mathcal{H}_\nu^2(B,g_0)\oplus\mathcal{H}^3_\nu(B,g_0)$, the pair of closed forms $(\omega,\Omega)=(\omega_0,\Omega_0)+s(\rho_2,\rho_3 -i\ast_{g_0}\rho_3)$ satisfies the $\sunitary{3}$--structure constraints $\omega\wedge\Omega=0=2\omega^3-3\Real\Omega\wedge\Imag\Omega$ up to terms of order $O(s^2)$. Thus $(\rho_2,\rho_3)$ is an infinitesimal Calabi--Yau deformation since the Calabi--Yau condition is precisely that $(\omega,
\Omega)$ is an $\sunitary{3}$--structure satisfying $d\omega=0=d\Omega$.
\item Exploiting results about existence and uniqueness of AC K\"ahler Ricci-flat metrics with respect to a fixed complex structure as systematised by Conlon--Hein \cite{Conlon:Hein:I}, one shows that for any  $\rho_2\in\mathcal{H}^2_\nu(B,g_0)$ and for $s$ sufficiently small there exists an AC Calabi--Yau structure $(\omega,\Omega)$ with $\omega = \omega_0 + s\rho_2 + i\partial\overline{\partial} u$ and $\Omega=\Omega_0$. This is done by solving a complex Monge--Amp\`ere equation for $u$. Indeed, the map $\rho_2\mapsto [\rho_2]\in H^2(B)$ is injective due to \cite[Theorem 5.12 (ii)]{FHN:ALC:G2:from:AC:CY3} and $\omega$ is the unique AC K\"ahler metric in the K\"ahler class $[\omega_0]+s[\rho_2]$ satisfying $\omega^3 = \omega_0^3$, which exists by \cite{Conlon:Hein:I}.
\item Finally, one can establish a version of the Tian--Todorov unobstructedness of deformations of the complex structure of compact Calabi--Yau 3-folds in the AC setting as a byproduct of the proof of the main theorem of \cite{FHN:ALC:G2:from:AC:CY3}. More precisely, the last remark of \cite[Section 8]{FHN:ALC:G2:from:AC:CY3} explains how to use \cite[Theorem 8.1]{FHN:ALC:G2:from:AC:CY3} to show that for any $\rho_3\in\mathcal{H}^3_\nu(B,g_0)$ the pair
$(\omega_0,\Omega_0) + s(0,\rho_3-i\ast_{g_0}\rho_3)$ can be deformed into an AC Calabi--Yau structure by adding lower-order terms. The analogy with the Tian--Todorov unobstructedness of deformations of the complex structure is due to the fact that on a Calabi--Yau 3-fold contraction with the holomorphic volume form $\Omega_0$ allows one to identify infinitesimal deformations of the complex structure $J$ for which the fixed K\"ahler form $\omega$ remains $J$--invariant with $3$-forms $\rho_3$ satisfying $\rho_3\wedge\omega=0=\rho_3\wedge\Omega_0$.
\end{enumerate}
These ingredients can be used to define an unobstructed moduli space $\widetilde{\mathcal{M}}_\nu^\tu{CY}$ of AC Calabi--Yau 3-folds with tangent space at the orbit of $(\omega_0,\Omega_0)$ given by $\mathcal{H}_\nu^2(B,g_0)\oplus\mathcal{H}^3_\nu(B,g_0)$.  For a more detailed description of this purely differential-geometric approach 
to the Calabi--Yau moduli space in the case of smooth compact $3$-folds we refer the reader to \cite{Nordstrom:Thesis}, which builds on Hitchin's seminal work on stable forms \cite{Hitchin:3forms}; see also \cite{Goto:Deformations} for a similar differential-geometric construction of the moduli space of compact Calabi--Yau manifolds in any dimension. 

Finally, the constraint $c_1(M)\cup [\omega]=0$ cuts out a smooth submanifold $\mathcal{M}_\nu^\tu{CY}$ in $\widetilde{\mathcal{M}}_\nu^\tu{CY}$ since $\mathcal{H}^2_\nu (B,g_0)\simeq H^2_c(B)\subseteq H^2(B)$ if $\nu \in (-3,-2)$ and $\mathcal{H}^2_\nu (B,g_0)\simeq H^2(B)$ if $\nu \in (-2,-1)$ by \cite[Theorem 5.12 (ii)]{FHN:ALC:G2:from:AC:CY3}.
 \endproof

We observe that Theorem~\ref{thm:Moduli:AC:CY}
allows us to determine the dimension of the moduli space  of ALC $\gtwo$--structures for all the cyclic ALC \gtwo-manifolds produced by \cite{FHN:ALC:G2:from:AC:CY3}, 
even at rates beyond the range 
that consists entirely of topological contributions. Indeed, the classification of AC Calabi--Yau metrics in \cite{Conlon:Hein:Classification} implies that the moduli space of AC Calabi--Yau structures can be locally identified with the product of $H^2(B)$ with the versal deformations of negative weight of the underlying complex manifold $(B,J_{\Omega_0})$.
For instance, we apply this methodology to compute the
dimension of the ALC moduli spaces for the following infinite class of cyclic ALC \gtwo--manifolds 
we constructed previously. 
\begin{example}
In \cite[Corollary 9.5]{FHN:ALC:G2:from:AC:CY3} we constructed infinitely many diffeomorphism types of simply connected cyclic ALC $\gtwo$--manifolds by considering a principal circle bundle over small resolutions $B$ of affine hypersurface singularities of the form $X_p=\{ x^2+y^2 +z^{p+1} -w^{p+1}=0\}\subset \C^4$. According to Friedman \cite[\S 2]{Friedman} the affine variety $X_p$ admits a complex $(p-1)$-parameter family of deformations with negative weight that lift to deformations of the small resolution $B$. Taking into account the choice of K\"ahler class, each of these deformations carries a $p$-parameter family of AC Calabi--Yau structures asymptotic to the Calabi--Yau cone $X_p$ with rate $-\frac{3}{2}$ (as can be checked by an explicit calculation of the weights of the natural $\C^\ast$--action on the versal deformation space of $X_p$). Bearing in mind that here the constraint $c_1(M)\cup [\omega_0]=0$ is vacuous since $H^4(B)=0$ (because $B$ is a small resolution), we conclude that the ALC $\gtwo$--manifolds of \cite[Corollary 9.5]{FHN:ALC:G2:from:AC:CY3} belong to a moduli space of dimension $3p-2 = 2(p-1)+p$ modulo scalings.
\end{example}

\begin{remark*}
The framework described in the first part of \cite{Foscolo:ALC:Spin7} allows one to extend without any change the statement and proof of Theorem \ref{thm:Moduli:AC:CY} to the setting of highly-collapsed ALC \gtwo--manifolds with ALC Seifert ends arising from principal Seifert circle bundles over AC Calabi--Yau orbifolds. 
Many such ALC \gtwo--manifolds were constructed in \cite{Foscolo:ALC:Spin7}, 
for instance the infinitely many families described in Example~\ref{ex:toric:Seifert}.
It also seems likely to the authors that Conlon--Hein's classification result for AC Calabi--Yau metrics \cite{Conlon:Hein:Classification} can be extended to the orbifold case.  However we are not currently aware of an explicit reference proving this. 
 \end{remark*}

\part{Construction of ALC \gtwo--manifolds of dihedral type}\label{Part3}

The final part of the paper is devoted to a general analytic desingularisation construction for ALC $\gtwo$--holonomy spaces with isolated conical singularities. As a nontrivial  application of this construction we will establish the existence of the first known examples of ALC $\gtwo$--manifolds of dihedral type.

\section{Conically singular ALC \gtwo--manifolds}\label{sec:CS:ALC:Def}

In this preliminary section we introduce conically singular (CS) ALC $\gtwo$--spaces and derive the refined Hodge-theoretic statements we will need in our construction.

By an isolated conical singularity of a Riemannian manifold we will mean an isolated singularity modelled in a strong sense (\ie with polynomial decay) on a metric cone with a smooth cross-section. In the \gthol setting the possible tangent cones must have holonomy contained in~\gtwo.

\begin{definition}\label{def:nK}
Let $N^6$ be a smooth connected compact $6$-manifold endowed with an $\sunitary{3}$--structure $(\omega,\Omega)$. We say that $(N,\omega,\Omega)$ is a \emph{nearly K\"ahler} manifold if
\[
d\omega = 3\Real\Omega, \qquad d\Imag\Omega = -2\omega^2.
\]
Equivalently, $(N,\omega,\Omega)$ is a nearly K\"ahler manifold if and only if the \gtstr
\[
\varphi_\tu{C}= r^2dr\wedge\omega + r^3\Real\Omega
\]
on the Riemannian cone $\tu{C}(N)$ over $N$ is torsion-free. Here $r$ is a radial coordinate on $\tu{C}(N)$.
\end{definition}
 
Since every \gtmfd~is Ricci-flat, nearly K\"ahler manifolds are Einstein with scalar curvature $30$. In particular, nearly K\"ahler manifolds have finite fundamental group. 
Only five simply connected nearly K\"ahler manifolds are currently known. 
The homogeneous ones are classified \cite{Butruille}: 
besides the $6$-sphere with the homogeneous almost complex structure induced by octonionic multiplication, there are unique homogeneous nearly K\"ahler structures 
on $\CP^3$, the flag manifold $\F_3$ and $S^3\times S^3$. 
Two inhomogeneous nearly K\"ahler structures are known (both of cohomogeneity one):  one on $S^6$ and another on $S^3\times S^3$ \cite{Foscolo:Haskins}. Note however that infinitely many singular nearly K\"ahler spaces are known, for example metric suspensions over Sasaki--Einstein 5-manifolds and twistor spaces of self-dual Einstein 4-orbifolds.

\begin{example}\label{eq:nK:S3xS3}
The homogeneous space $\sunitary{2}^3/\triangle \sunitary{2}\simeq S^3\times S^3$ carries
a homogeneous nearly K\"ahler structure. By \cite[Proposition 1.1]{Cortes:Vasquez} $S^3\times S^3$ is the only homogeneous nearly K\"ahler $6$-manifold to admit smooth finite quotients. There are many freely-acting finite groups of isomorphisms of~${S^3\times S^3}$, for instance any product of three cyclic groups with coprime orders in the three $\sunitary{2}$--factors of the connected component of the identity in the isometry group $\sunitary{2}^3/\triangle\Z_2$ \cite[Table 13]{Cortes:Vasquez}. Here the quotient by the diagonal $\Z_2$--subgroup is necessary for an effective action.
\end{example}

\begin{definition}\label{def:CS:ALC}
Let $M_0$ be a complete metric space. We say that $M_0$ is an \emph{ALC \gtmfd~with isolated conical singularities} at $p_1,\dots, p_n\in M_0$ if the following properties hold.
\begin{enumerate}
\item $M'_0=M_0\setminus \{ p_1,\dots, p_n\}$ is a smooth $7$-manifold endowed with a torsion-free \gtstr~$\varphi_0$. The distance on $M'_0$ is induced by the Riemannian metric $g_0=g_{\varphi_0}$ and $M_0$ is the metric completion of $M'_0$.
\item There exists a compact set $K\subset M_0$ such that $p_1,\dots,p_n\in K$ and $(M_0\setminus K,\varphi_0)$ is an ALC \gtmfd~(of cyclic or dihedral type) asymptotic to $\BC (\Sigma)$ in the sense of Definition \ref{def:ALC:G2}.
\item For each $i=1,\dots, n$ there exists an open subset $U_i\subset M_0$ with $U_i\cap \{p_1,\dots, p_n\} =\{ p_i\}$, positive constants $\varepsilon_i,\mu_i>0$, a \gtwo--cone $\tu{C}_i=\tu{C}(N_i)$ over a nearly K\"ahler manifold $(N_i,\omega_i,\Omega_i)$ non-isometric to the round $6$-sphere and a diffeomorphism $f_i \co (0,\varepsilon_i)\times N_i \ra U_i$ such that, for all $j\geq 0$,
\[
\left|\nabla^j_{\tu{C}_i}\left( f_i^\ast\varphi_0-\varphi_{\tu{C}(N_i)}\right)\right|_{g_{\tu{C}_i}} = O(r^{\mu_i-j}).
\]
The cone $\tu{C}_i$ and the constant $\mu_i$ are called the \emph{tangent cone} and the \emph{rate} of the isolated conical singularity $p_i$.
\end{enumerate}
When there is no need to specify the singularities $p_1,\dots, p_n$ we will say that $M_0$ is a \emph{conically singular} (CS) ALC \gtmfd.
\end{definition}

\begin{remark}\label{rmk:CS:ALC}
A similar definition can be given in the more general Riemannian context, where the metric is not assumed to arise from a torsion-free $\gtwo$--structure. In this more general context, we will say that $(M_0,g_0)$ is a \emph{CS ALC manifold}.
\end{remark}

Up to finite quotients, only one nontrivial example of a CS ALC $\gtwo$--manifold is currently known.

\begin{example}\label{eg:CS:ALC}
In \cite{FHN:Coho1:ALC} we constructed a CS ALC \gtmfd~$M_0$ of cohomogeneity-one with one isolated conical singularity $p$ with tangent cone $\tu{C}(S^3\times S^3)$ over the homogeneous nearly K\"ahler structure on $S^3\times S^3$ and rate $\mu = \frac{\sqrt{145}-7}{2}\approx 2.5$. $M_0'=M_0\setminus \{ p\}$ is diffeomorphic to $\R^+ \times S^3\times S^3$. The torsion-free \gtstr~$\varphi_0$ is invariant under the group $\sunitary{2}^2\times\tu{N}$, where the first two factors act on the left on $S^3\times S^3\simeq \sunitary{2}^2$, and $\tu{N}$ is the normaliser of $\unitary{1}$ in the diagonal right-acting $\sunitary{2}$. The $\unitary{1} \subset \tu{N}$ corresponds to the principal circle action of the circle bundle on the ALC end. Note that $\tu{N}/\unitary{1} \cong \Z_2$. Elements of $\tu{N} \setminus \unitary{1}$ have order 4 (think of them as imaginary unit quaternions orthogonal to $\C \subset \mathbb{H}$) and map $\unitary{1}$-orbits to $\unitary{1}$-orbits, but reversing the orientation of the circles.

The diagonal copy of $\Z_2\subset\sunitary{2}^2\times\unitary{1}\subset\sunitary{2}^2\times \tu{N}$ acts trivially, so really the symmetry group is $(\sunitary{2}^2\times\tu{N})/\Z_2$. 
In particular every finite group $\Gamma \subset (\sunitary{2}^2\times\tu{N})/\Z_2$ acting freely on $S^3\times S^3$ acts freely on $M_0'$ and $M_0/\Gamma$ is a CS ALC \gtmfd~with one singularity modelled on the \gtwo--cone over $(S^3\times S^3)/\Gamma$.
\end{example}

We also introduce the definition of an AC \gtmfd.
\begin{definition}\label{def:AC:G2}
Let $(M,\varphi)$ be a complete \gtmfd~and let $(\tu{C},\varphi_{\tu{C}})$ be a \gtwo--cone over a nearly K\"ahler manifold $N$. We say that $(M,\varphi)$ is \emph{asymptotically conical} (AC) asymptotic to $(\tu{C},\varphi_{\tu{C}})$ with rate $\mu<0$ if there exists a compact set $K\subset M$, $R>0$ and a diffeomorphism $f \co (R,\infty)\times N \ra M\setminus K$ such that
\[
\left|\nabla^j_{\tu{C}}\left( f^\ast\varphi-\varphi_{\tu{C}}\right)\right|_{g_{\tu{C}}} = O(r^{\mu-j})
\]
for all $j\geq 0$.
\end{definition}

\begin{remark}\label{rmk:AC:G2}
	We will make use of the following normal form for an AC $\gtwo$--manifold $(M,\varphi)$ asymptotic to $\tu{C}(N)$ with rate $\mu\leq -3$, which is established in \cite[Theorem 3.6]{Karigiannis} under an assumption later shown always to hold in \cite[Proposition 6.25]{Karigiannis:Lotay} (an analogue in the AC setting of our Proposition \ref{prop:Slice:End}). We can choose the diffeomorphism $f$ such that outside the compact set $K$ we have
	\[
\varphi = \varphi_\tu{C}+ \xi + d\eta, \qquad \ast_{\varphi}\varphi =  \ast_{\varphi_\tu{C}}\varphi_\tu{C} + \zeta + rdr\wedge \ast_{N}\xi + d\theta
\]
for some $(\xi,\zeta)\in \mathcal{H}^3(N)\oplus \mathcal{H}^4(N)$ and $(\eta,\theta)=O(r^{-2-\delta})$ for some $\delta>0$. (Note that if $\mu<-3$ then $\xi=0$ and if $\mu<-4$ then $\zeta = 0$ too; in these cases the decay of the primitives $\eta,\theta$ is not optimal: for example, if $\mu<-4$ then $(\eta,\theta)=O(r^{\mu+1})$.)
\end{remark}

\begin{example}\label{eg:AC:G2:Bryant:Salamon}
For a long time the only known AC \gtmfd s were those constructed by Bryant--Salamon \cite{Bryant:Salamon}. Of particular relevance for the concrete examples we will analyse later is the Bryant--Salamon AC \gtmetric~on $S^3\times \R^4$. This metric is asymptotic with rate $\mu=-3$ to the cone over the homogeneous nearly K\"ahler structure on $S^3\times S^3$ described in Example \ref{eq:nK:S3xS3}. Karigiannis--Lotay \cite[Corollary 6.10]{Karigiannis:Lotay} have shown that the Bryant--Salamon metric on $S^3\times \R^4$ is the unique AC $\gtwo$--manifold asymptotic to this cone.\end{example}

\begin{example}\label{eg:FHN:AC:G2}
Recently the authors found an infinite family of new simply connected AC \gtmfd s \cite[Theorem C]{FHN:Coho1:ALC}: for every pair of coprime positive integers $m,n$ there exists an AC \gtmetric, unique up to scale, on the total space $M_{m,n}$ of the principal circle bundle over the canonical line bundle $K_{\CP^1 \times \CP^1}$ of $\C\PP^1\times\C\PP^1$ with first Chern class $(m,-n)$. Topologically, $M_{m,n}$ can alternatively be described as the pull-back of $K_{\CP^1 \times \CP^1}$ to the total space  of the circle bundle over $\CP^1 \times \CP^1$ of Chern class $(m,-n)$, so it is diffeomorphic to $S^3 \times \mathcal{O}_{\CP^1}(-2(m+n))$.

The AC \gtmetric~is asymptotic with rate $\mu=-3$ to a $\Z_{2(n+m)}$--quotient of the cone over the homogeneous nearly K\"ahler structure on $S^3\times S^3$.
More specifically, the group of $2(m+n)$th roots of unity acts on $S^3 \times S^3$ by the embedding $\zeta \mapsto (1, \zeta^{n+m}, \zeta^m) \in \unitary{1}^3/\Delta \Z_2 \subset \sunitary{2}^3/\Delta \sunitary{2} \subset \tu{Aut}(S^3 \times S^3)$. This action commutes with that of $\sunitary{2}^2 \times \unitary{1} \subset \tu{Aut}(S^3 \times S^3)$, and indeed the whole manifold $M_{m,n}$ has a cohomogeneity-one action by that group. $\sunitary{2} \times \sunitary{2}$ acts transitively on each orbit, and the stabiliser of a point in a principal orbit (in particular, the link of the asymptotic cone) is the image of $\zeta \mapsto (\zeta^n, \zeta^{-m})$.

In view of the examples we will discuss later, the AC \gtmetric~on $M_{1,1}$ is of particular interest. In this special case we denote the group $\Gamma_{1,1}\simeq \Z_4$ as $\Gamma_{cyc}$ for reasons that will be clear later. The embedding of $\Gamma_{cyc}$ in $\sunitary{2}^3/\triangle\Z_2$ is induced by $\zeta\mapsto (1,\zeta^2,\zeta)\in \sunitary{2}\times\sunitary{2}\times\unitary{1}\subset\sunitary{2}^3$.
\end{example}

\subsection{Hodge theory}

Let $(M_0,g_0)$ be a $7$-dimensional CS ALC manifold in the sense of Remark \ref{rmk:CS:ALC}, with singularities $p_1,\dots, p_n$. We assume that the tangent cone $\tu{C}(N_i)$ at $p_i$ is a \gtwo--cone and the asymptotic model $\BC(\Sigma)$ at infinity is the same as the one for an ALC \gtmfd. We do not impose any restriction on the holonomy or curvature of $g_0$. In this section we collect some preliminary results about the Laplacian acting on $2$-forms on $M_0$.

First of all, we define weighted Sobolev spaces on $M_0$ adapted to the ALC end \emph{and} the conical singularities. Let $K$ and $U_1,\dots, U_n$ be the subsets given by Definition \ref{def:CS:ALC}. We have to consider multiple weight functions. Let $r$ be a smooth function on $M_0$ which outside of $K$ agrees (up to a double cover in the case of an ALC end of dihedral type) with the radial function on the asymptotic model at infinity $\BC (\Sigma)$ and satisfies $r\equiv 1$ on $K$. For each $i=1,\dots, n$ let $r_i$ be a smooth function on $M_0$ which agrees (by composition with the diffeomorphism $f_i$) with the radial function on the \gtwo--cone $\tu{C}_i$ on $U_i$ and satisfies $r_i\equiv 1$ outside of a small tubular neighbourhood of $U_i$.

\begin{definition}\label{def:Weighted:Sobolev:CS:ALC}
Let $\bm{\nu} = (\nu_1,\dots, \nu_n,\nu_\infty)$ be an $(n+1)$--tuple of real numbers. Given a differential form $\psi$ on $M'_0$ we define its $L^p_{l,\bm{\nu}}$--norm by
\begin{equation}\label{eq:Weighted:Sobolev:CS:ALC}
\begin{aligned}
\| \psi \| ^p_{L^p_{l,\bm{\nu}}} &= \sum_{j=0}^{l}
\int_{K}{
\left( \prod_{i=1}^{n}{
r_i^{-\nu_i p+jp-7}}\right) |\nabla^j\psi|^p \dvol_{M_0}}\\ &+\int_{M_0\setminus K}{\left( |r^{-\nu_\infty+j}\nabla^j \Pi_0\psi |^p + |r^{-\nu_\infty+l}\nabla^j \Pi_\perp\psi |^p \right) r^{-6}\dvol_{M_0}}.
\end{aligned}
\end{equation}
Define the weighted Sobolev space of differential forms on $M_0$ of class $L^p_{l,\bm{\nu}}$ as the closure of the space of smooth compactly supported differential forms on $M'_0$ with respect to the norm \eqref{eq:Weighted:Sobolev:CS:ALC}.
\end{definition}

\begin{remark*}
Note that $\prod_{i=1}^{n}{r_i^{-\nu_i p+jp-7}} = r_i^{-\nu_i p+jp-7}$ in $U_i$ and $\prod_{i=1}^{n}{r_i^{-\nu_i p+jp-7}} =1$ outside of a neighbourhood of the singularities.
\end{remark*}

There is a Fredholm theory for differential operators such as $d+d^\ast$ and the Laplacian on $M_0$ acting between weighted Sobolev spaces analogous to Theorems \ref{thm:Fredholm}, \ref{thm:Index:jump} and \ref{thm:Fredholm:2nd:order}. We summarise this theory in the following statements without proof.
\begin{enumerate}
\item There is a discrete set of codimension $1$ linear subspaces of $\R^{n+1}$ such that if $\bm{\nu}$ lies in the complement of these ``indicial walls'' then an elliptic operator $P\co L^2_{l,\bm{\nu}}\ra L^2_{l-k,\bm{\nu}-\bm{k}}$ is Fredholm. Here $\bm{\nu}-\bm{k}$ denotes the multi-index where $k$ is subtracted from each entry of $\bm{\nu}$, where $k$ is the order of $P$.
\item The growth of elements in the kernel of $P$ is controlled by the behaviour of elements in the kernel of the analogous operators on the models $\tu{C}_i$ and $\BC(\Sigma)$ for the ends.
\item The index of $P$ can only jump when $\bm{\nu}$ crosses an indicial wall and the jump of the index is determined by the number of homogeneous elements in the kernel of the model operators on the ends of $M_0$.
\item Set $\textbf{\tu{dim}}=(7,\dots, 7, 6)$, \ie $\textbf{\tu{dim}}$ collects the dimensions of the tangent cones at the singularities and at infinity of $M_0$ respectively, and note that the $L^2$--inner product defines a non-degenerate pairing between $L^2_{\bm{\nu}}$ and $L^2_{-\textbf{dim}-\bm{\nu}}$. Under this pairing, the cokernel of $P\co L^2_{l,\bm{\nu}}\ra L^2_{l-k,\bm{\nu}-\bm{k}}$ is identified with the kernel of $P^\ast$ in $L^2_{-\textbf{\tu{dim}}-\bm{\nu}+\bm{k}}$. 
\end{enumerate}

The latter statement uses the following integration by parts formula: if $P$ is a first-order operator, $u\in L^2_{1,\bm{\nu}}$ and $v\in L^2_{1,\bm{\mu}}$ with $\nu_i +\mu_i \geq -6$ and $\nu_\infty + \mu_\infty \leq -5$, then $\langle Pu,v\rangle_{L^2} = \langle u,P^\ast v\rangle_{L^2}$.

We now want to determine the indicial walls for the operators we will use. The results of Section \ref{sec:Indicial:Roots:CY:cone} describe the indicial roots for the Laplacian and $d+d^\ast$ on the ALC end of $M_0$. The following lemma collects the results we need about indicial roots for the isolated conical singularities of $M_0$.

\begin{lemma}\label{lem:Indicial:Roots:G2:cone}
Let $\tu{C}(N)$ be a \gtwo--cone and let $\gamma$ be a harmonic $k$-form on $\tu{C}(N)$ homogeneous of order $\lambda$.
\begin{enumerate}
\item If $k=0$ then $\gamma=0$ whenever $\lambda\in (-6,1)\setminus\{ -5,0\}$. Moreover, if $\lambda=0$ (respectively, $\lambda=-5$) then $\gamma = c$ (respectively, $\gamma = r^{-5}c$) for some $c\in\R$.
\item If $k=1$ then $\gamma=0$ whenever $\lambda\in (-5,0)$.
\item If $k=2$ then $\gamma=0$ whenever $\lambda\in (-3,-2)$. Moreover, if $\lambda=-2$ (respectively, $\lambda=-3$) then $\gamma = \beta$ (respectively, $\gamma = r^{-1}\beta$) for some $\beta\in\mathcal{H}^2(N)$.
\item If $k=3$ and $\lambda \in (-4,-3)$ then $\gamma$ is never closed. Moreover, if $\lambda=-3$ then $\gamma = \beta$ for some $\beta\in\mathcal{H}^3(N)$.
\end{enumerate}
Similar statements for $k=4,5,6,7$ are obtained by applying the Hodge star operator. Moreover, for any $k\neq 3,4$ and any $\lambda$ there are no harmonic $k$-forms on $\tu{C}(N)$ that can be expressed as a polynomial in $\log{r}$ with coefficients in the space of harmonic $k$-forms on $\tu{C}(N)$ homogeneous of order $\lambda$.
\proof
Since $N$ is Einstein with scalar curvature $30$, statements (i)--(iii) follow from \cite[Theorem A.2]{FHN:ALC:G2:from:AC:CY3} using the fact that the first nonzero eigenvalue of the Laplacian on functions on $N$ is at least $6$ and the first eigenvalue of the Laplacian on coclosed $1$-forms on $N$ is at least $10$. Part (iv) follows from \cite[Theorem A.2 and Remark A.5]{FHN:ALC:G2:from:AC:CY3}. The last statement is a consequence (i)--(iii) since by \cite[Proposition A.6]{FHN:ALC:G2:from:AC:CY3} harmonic polynomials in $\log{r}$ only occur when $\lambda = -\frac{5}{2}$.
\endproof
\end{lemma}

The following two lemmas are immediate consequences of integration by parts and the indicial roots computations.

\begin{lemma}\label{lem:Harmonic:1:forms:CS:ALC}
Assume that $\nu_i \geq-5$ for all $i=1,\dots, n$ and $\nu_\infty \leq 0$.
\begin{enumerate}
\item There are no harmonic functions on $M_0$ in $L^2_{\bm{\nu}}$.
\item Every harmonic $1$-form on $M_0$ in $L^2_{\bm{\nu}}$ is closed and coclosed.
\end{enumerate} 
\proof
Because of the indicial root computations of Lemma \ref{lem:Indicial:Roots:G2:cone} (i) and (ii) and Propositions \ref{prop:Harmonic:Functions:Cone} and \ref{prop:Harmonic:1Forms:Cone}, we can in fact assume that $\nu_i = -\epsilon$ and $\nu_\infty = -4+\epsilon$ for any $\epsilon>0$ and integrate by parts.
\endproof
\end{lemma} 

\begin{remark}\label{rmk:Harmonic:1:forms:CS:ALC}
Under the same assumptions on $\bm{\nu}$, if $g_0$ is Ricci-flat then there are no harmonic $1$-forms on $M_0$ in $L^2_{\bm{\nu}}$. Indeed, Lemma \ref{lem:Indicial:Roots:G2:cone} (ii) and Proposition \ref{prop:Harmonic:1Forms:Cone} and an integration by parts guarantee that any such form is parallel, and the decay at infinity forces its vanishing. More generally, Lemma \ref{lem:Harmonic:1:forms:CS:ALC} implies that harmonic 1-forms in $L^2_{\bm{\nu}}$ are in fact closed and coclosed and $L^2$--integrable and, combining \cite[Example 0.16]{Lockhart} and Theorem \ref{thm:ALC_hodge_fastdecay}, we deduce that the latter subspace is topological and that it can be identified with the subspace of $H^1(M_0)$ consisting of classes that restrict to zero in $H^1(N)$. In particular, this subspace vanishes if $M_0$ is diffeomorphic to a Ricci-flat CS ALC manifold.
\end{remark}

\begin{lemma}\label{lem:Harmonic:2:forms:CS:ALC}
Every harmonic $2$-form $\sigma$ on $M_0$ in $L^2_{\bm{\nu}}$ is
\begin{enumerate}
\item coclosed if $\nu_i \geq -4$ for all $i=1,\dots, n$ and $\nu_\infty \leq 1$;
\item closed if $\nu_i \geq -3$ for all $i=1,\dots, n$ and $\nu_\infty \leq -2$.
\end{enumerate}
\proof
Since $d^\ast\sigma$ is a harmonic $1$-form in $L^2_{\bm{\nu}-\bm{1}}$, Lemma \ref{lem:Harmonic:1:forms:CS:ALC} implies that it is closed (and coclosed) and hence
$dd^\ast\sigma=0$. Moreover, as in the proof of Lemma \ref{lem:Harmonic:1:forms:CS:ALC}, $d^\ast\sigma$ decays sufficiently fast so that we can integrate by parts:
\[
0=\langle dd^\ast\sigma,\sigma\rangle_{L^2} = \| d^\ast \sigma\|^2_{L^2}.
\]
The second statement follows immediately from an integration by parts whenever $\nu_i \geq -\frac{5}{2}$ and $\nu_\infty \leq -2$. The improvement $\nu_i\geq -3$ follows from Lemma \ref{lem:Indicial:Roots:G2:cone} (iii).
\endproof
\end{lemma}

Our main goal is to study the solvability of the equation $\triangle\sigma = d^\ast\xi$ for a $2$-form $\sigma$ in weighted Sobolev spaces for certain ranges of weight. 

\begin{prop}\label{prop:Laplacian:2Forms:ALC}
Suppose that $\nu_i\in [-5,-3]$ for all $i=1,\dots, n$ and $\nu_\infty \in [-3,0]$ and assume that $\bm{\nu}+\bm{1}$ does not lie on an indicial wall for the Laplacian acting on $2$-forms. Then for every $3$-form $\xi\in L^2_{l,\bm{\nu}}$ the equation $\triangle\sigma = d^\ast\xi$ has a solution $\sigma$ with $dd^\ast\sigma=0$ and $d\sigma\in L^2_{l,\bm{\nu}}$.

Furthermore, there exists $\delta>0$ such that we have the following improved decay for $d\sigma$:
\begin{enumerate}
\item if $\nu_i\in [-3-\delta,-3]$ but in fact $d^\ast\xi \in L^2_{l-1,-4+\delta}$ as $r_i\ra 0$ for some $i$ then $d\sigma \in L^2_{l,-3+\delta}$ near that $p_i$;
\item if $\nu_\infty\in [-3,-3+\delta]$ but in fact $d^\ast\xi \in L^2_{l-1,-4-\delta}$ as $r\ra \infty$ then $d\sigma \in L^2_{l,-3-\delta}$ on the ALC end.
\end{enumerate}
\end{prop}
\proof
We begin with the first part of the Proposition. 
First note that any solution satisfies ${dd^\ast\sigma=0}$: indeed, $d^\ast\sigma$ is a harmonic $1$-form in $L^2_{\bm{\nu}}$ and so $dd^\ast\sigma$ vanishes by Lemma \ref{lem:Harmonic:1:forms:CS:ALC} when $\nu_i\geq -5$ and $\nu_\infty\leq 0$ (in fact $d^\ast\sigma=0$ under the additional assumptions of Remark \ref{rmk:Harmonic:1:forms:CS:ALC}). Next, the obstructions to solve $\triangle\sigma = d^\ast\xi$ in $L^2_{2,\bm{\nu}+\bm{1}}$ lie in the space of harmonic $2$-forms in $L^2_{-\bm{dim}-\bm{\nu}+\bm{1}}$ which are \emph{not} closed. By Lemma \ref{lem:Harmonic:2:forms:CS:ALC} (ii), if $\nu_i\leq -3$ and $\nu_\infty \geq -3$ then there are none of these and thus the first part of the Proposition is proved.

It remains to prove the improved decay for $d\sigma$ claimed in the second part of the Proposition.
Choose $\delta>0$ sufficiently small so that there are no indicial roots for the Laplacian on $2$-forms on $\tu{C}_i$ in $(-2,-2+\delta)$ and on $\BC (\Sigma)$ in $(-2-\delta,-2)$. Consider first case (i). By the first part of the Proposition we can always solve $\triangle\sigma = d^\ast\xi$ for a coclosed $2$-form with $\sigma = O(r_i^{-2-\delta})$ as $r_i\ra 0$ and analogous estimates for the derivatives. However, since $\triangle\sigma = O(r_i^{-4+\delta})$ as $r_i\ra 0$, by Lemma \ref{lem:Indicial:Roots:G2:cone} (iii) we must have
\[
\sigma = \beta + O(r_i^{-2+\delta})
\]
as $r_i\ra 0$ for some $\beta\in\mathcal{H}^2(\Sigma)$ which in particular is closed. Hence $d\sigma$ indeed decays faster than expected.

The proof of the improved decay for $d\sigma$ in case (ii) is analogous to those from \cite[Proposition 5.16]{FHN:ALC:G2:from:AC:CY3} and \cite[Proposition 2.35]{Foscolo:ALC:Spin7} and follows from Proposition \ref{prop:Closed:to:Coclosed:Laplacian:Mid:deg}: the obstructions to solve $\triangle \sigma = d^\ast\xi$ with $\sigma\in L^2_{l+1,-2-\delta}$ as $r\ra\infty$ are compensated by adding to  $\sigma\in L^2_{l+1,-2-\delta}$ a term of the form $\chi \tau$ with $\tau\in \mathcal{H}^2(\Sigma)$ and $\chi$ a cutoff function with $\chi\equiv 1$ on the ALC end. Since $d(\chi\tau)$ is compactly supported, the improved decay of $d\sigma$ follows.
\endproof

\begin{remark}\label{rmk:Laplacian:2Forms:ALC}
We observe explicitly that in Proposition \ref{prop:Laplacian:2Forms:ALC} we can take $n=0$, \ie assume that $M_0$ is a smooth ALC manifold without any isolated conical singularities. Then the proposition guarantees the solvability of the equation $d^\ast d\sigma = d^\ast\xi$ with $d\sigma\in L^2$.
\end{remark}

\begin{remark}\label{rmk:Laplacian:3Forms:ALC}
Consider the analogous problem $\triangle\sigma = d^\ast\zeta$ for a $4$-form $\zeta\in L^2_{1,\bm{\nu}}$.
\begin{itemize}
\item Any solution $\sigma$ satisfies $dd^\ast\sigma=0$ (and therefore $d^\ast d\sigma = d^\ast\zeta$) if $\nu_i\geq -2$ and $\nu_\infty\leq -1$: indeed, the harmonic $2$-form $d^\ast\sigma$ is then closed by Lemma \ref{lem:Harmonic:2:forms:CS:ALC} (ii).
\item A solution always exists if $\nu_i \leq -\frac{7}{2}$ and $\nu_\infty \geq -3$, for then any harmonic $3$-form of rate $-\bm{dim}-\bm{\nu}+\bm{1}$ is closed by an integration by parts.
\end{itemize}
Unless we make further assumptions on the Laplace spectrum of the nearly K\"ahler manifolds $N_i$, these results appear to be optimal and therefore in general there is no choice of $\nu_i$ such that the equation $\triangle \sigma = d^\ast\zeta$ for a $4$-form $\zeta\in L^2_{1,\bm{\nu}}$ always has a solution with $dd^\ast\sigma=0$ and $d\sigma\in L^2_{1,\bm{\nu}}$.
\end{remark}

\section{Desingularising isolated conical singularities of ALC \gtmfd s}\label{sec:CS:ALC:Gluing}

The purpose of this final section of the paper is to prove the following general existence theorem.

\begin{theorem}\label{thm:Desing:CS:ALC}
Let $(M_0,\varphi_0)$ be a CS ALC \gtwo--manifold with singularities $p_1,\dots, p_n$ modelled on $\gtwo$--cones $\tu{C}(N_i)$, $i=1,\dots,n$. For each $i=1,\dots, n$ let $(M_i,\varphi_i)$ be an AC $\gtwo$--manifold asymptotic to the cone $\tu{C}(N_i)$ with rate less than or equal to $-3$. As in Remark \ref{rmk:AC:G2}, write
\[
\varphi_i = \xi_i + d\eta_i, \qquad \ast_{\varphi_i}\varphi_i = \zeta_i + r_i dr_i\wedge \ast_{N_i}\xi_i + d\theta_i
\]
for some $(\xi_i,\zeta_i)\in \mathcal{H}^3(N_i)\oplus \mathcal{H}^4(N_i)$ and $(\eta_i,\theta_i)=O(r_i^{-2-\delta})$ for some $\delta>0$.

Assume that $([\xi_1],\dots, [\xi_n])\in \bigoplus_i{H^3(N_i)}$ and $(0,[\zeta_1],\dots, [\zeta_n])\in H^4(N)\oplus \bigoplus_i{H^4(N_i)}$ lie in the image of, respectively, the restriction maps
\[
H^3(M_0')\ra H^3(N)\oplus \bigoplus\nolimits_i{H^3(N_i)}\stackrel{pr}{\ra} \bigoplus\nolimits_i{H^3(N_i)}, \qquad H^4(M_0')\ra H^4(N)\oplus \bigoplus\nolimits_i{H^4(N_i)}.
\]
Here $N\ra\Sigma$ is the circle fibration arising from the ALC end $\tu{BC}(\Sigma)$ of $M_0$.

Then the following holds. Let $M$ be the smooth manifold obtained by replacing neighbourhoods of the singular points $p_1,\dots, p_n$ with compact regions of the AC manifolds $M_1,\dots, M_n$ using the diffeomorphisms in Definitions \ref{def:CS:ALC} and \ref{def:AC:G2}. There exists $\epsilon_0>0$ with the following significance. For every $\epsilon\in (0,\epsilon_0)$ there exists a torsion-free ALC $\gtwo$--structure $\varphi_\epsilon$ on $M$ such that as $\epsilon\ra 0$
\begin{enumerate}
\item $\varphi_\epsilon$ converges smoothly to $\varphi_0$ on compact subsets of $M_0'$;
\item $\epsilon^{-3}\varphi_\epsilon\ra \varphi_i$ smoothly on the regions of $M$ identified with compact regions of $M_i$.
\end{enumerate}
\end{theorem}

The rest of the section is devoted to the proof of this theorem and then to concrete applications of it, in particular to establish our main Theorem \ref{mthm:Dihedral}. We will use the general Hodge-theoretic results for CS ALC $\gtwo$--spaces proven in the previous section to show how the topological assumptions in Theorem \ref{thm:Desing:CS:ALC} guarantee the existence of a closed and coclosed 3-form and closed 4-form on $M_0$ with prescribed asymptotics. As a further preliminary result, we will establish an extension to the ALC setting of Joyce's deformation result for closed $\gtwo$--structures with small torsion on compact manifolds to nearby torsion-free ones. We will then combine these preliminary results to conclude the proof of Theorem \ref{thm:Desing:CS:ALC} and apply it to Examples \ref{eg:CS:ALC}, \ref{eg:AC:G2:Bryant:Salamon} and \ref{eg:FHN:AC:G2}.

\subsection{Solving the linearised equation on a CS ALC {\gtmfd}}\label{sec:CS:ALC:Obstructions}

Let $(M_0,\varphi_0)$ be a CS ALC {\gtmfd} with singularities at $p_1,\dots, p_n$ with tangent cones $\tu{C}_1,\dots, \tu{C}_n$. Suppose that for each $i=1,\dots,n$ there exists an AC \gtmfd~$(M_i,\varphi_i)$ asymptotic to the \gtwo--cone $\tu{C}_i$ with rate $\mu'_i\leq -3$. We want to resolve the singularities of $(M_0,\varphi_0)$ by gluing in rescaled copies of $M_1,\dots, M_n$ in small neighbourhoods of $p_1,\dots, p_n$, respectively. Obstructions to this smoothing procedure arise from the fact, explained in Remark \ref{rmk:AC:G2}, that the restriction of $(\varphi_i,\ast_{\varphi_i}\varphi_i)$ to the end of $\tu{C}_i$ defines a pair of closed forms on an exterior region of $\tu{C}_i$ with a priori nontrivial cohomology classes. The harmonic representatives $(\xi_i,\zeta_i)$ of these cohomology classes on $\tu{C}_i$ in fact encode the leading order deviation of the AC \gtstr~$\varphi_i$ from the conical \gtstr~$\varphi_{\tu{C}_i}$. Now, suppose that
\begin{enumerate}
\item the CS \gtstr~$\varphi_0$ on $M_0$ can be deformed at the linear level to match $\varphi_{\tu{C}_i}+\xi_i$ near $p_i$ for all $i=1,\dots, n$, and
\item there exists a closed $4$-form $\psi$ on $M_0$ asymptotic to $\ast_{\varphi_{\tu{C}_i}} \varphi_{\tu{C}_i}+\zeta_i+ r_i dr_i\wedge \ast_{N_i}\xi_i$ near $p_i$ for all $i=1,\dots, n$.
\end{enumerate}
When $M_0$ is compact, Karigiannis \cite{Karigiannis} showed that these two conditions suffice to guarantee the success of the desingularisation procedure. Karigiannis also showed how to understand conditions (i) and (ii) in terms of the global topology of $M_0$. In this section we carry out a similar analysis in the case that $M_0$ is a CS ALC {\gtmfd} as in Definition \ref{def:CS:ALC}. 

Regard $M'_0=M_0\setminus\{ p_1,\dots, p_n\}$ as a manifold with $n+1$ boundary components given by the nearly K\"ahler manifolds $N_1,\dots, N_n$ cross-sections of the tangent cones $\tu{C}_1,\dots,\tu{C}_n$ at the singularities and~$N$. Hence we have natural restriction maps
\[
H^k(M'_0)\ra \bigoplus_{i=1}^n{H^k(N_i)}\oplus H^k(N).
\] 
Furthermore, recall that since $b_1 (\Sigma)=0=b_4(\Sigma)$ the Gysin sequence for the circle fibration $N\ra \Sigma$ yields exact sequences
\[
0\ra H^3 (\Sigma)\ra H^3(N) \ra H_-^2(\Sigma)\ra 0, \qquad 0\ra H^4(N)\ra H_-^3(\Sigma)\ra H^5(\Sigma)\ra 0.
\]
Here recall that $H_-^k(\Sigma)$ denotes the cohomology of $\Sigma$ with values in its orientation bundle.

The following theorems are the ALC analogue of \cite[Theorem 3.10]{Karigiannis} and yield the solution of the obstruction problem.

\begin{theorem}\label{thm:CS:ALC:Obstructions:4:form}
Let $(M_0,\varphi_0)$ be a CS ALC \gtmfd~with singularities at $p_1,\dots, p_n$ modelled on the \gtwo--cones $\tu{C}_1,\dots, \tu{C}_n$ respectively.

For each $i=1,\dots, n$ let $\zeta_i$ be a closed and coclosed $4$-form on $N_i$. Suppose that
\[
\left( [\zeta_1],\dots, [\zeta_n],0 \right)\in \bigoplus_{i=1}^n{H^4(N_i)}\oplus H^4(N)
\]
lies in the image of the restriction map $H^4(M'_0)\ra \bigoplus_{i=1}^n{H^4(N_i)} \oplus H^4(N)$. Then there exist a $4$-form $\zeta$ on $M'_0$ and $\mu>0$ such that $d\zeta =0$ and for all $j\geq 0$
\[
|\nabla^j_{\tu{C}_i}\left( \zeta -\zeta_i\right)|_{g_{\tu{C}_i}} = O(r_i^{-4+\mu-j}),
\]
as $r_i\ra 0$, $i=1,\dots, n$, and
\[
|\nabla^j_{\BC} \zeta |_{g_{\BC}} = O(r^{-3-\mu-j})
\]
as $r\ra\infty$.
\proof
Assume there exists a cohomology class in $H^4(M'_0)$ whose images in $H^4(N_i)$, $i=1,\dots, n$, and $H^4(N)$ are $[\zeta_i]$ and 0, respectively. We can then find a smooth closed $4$-form $\zeta'$ such that
\[
\zeta' = \zeta_i + dw_i , \qquad \zeta' = dw
\]
near $p_i$ and at infinity, respectively. Introduce cutoff functions $\chi_i$ and $\chi$ which vanish in a small enough neighbourhood of $p_i$ and outside of a large enough compact set, respectively, and are equal to $1$ on the ``bulk'' of $M_0$. We define $\zeta$ by
\[
\zeta = \zeta_i + d(\chi_i w_i) , \qquad \zeta' =  d(\chi w)
\]
near $p_i$ and infinity and $\zeta = \zeta'$ on the ``bulk'' of $M'_0$.
\endproof
\end{theorem}

\begin{remark*}
Here it is crucial to require that $\left( [\zeta_1],\dots, [\zeta_n],0 \right)$ and not only $\left( [\zeta_1],\dots, [\zeta_n] \right)$ lies in the image of $H^4(M_0')$. Indeed, under the weaker assumption we could have $\zeta = \theta\wedge\eta\wedge\tau +O(r^{-3-\delta})$ as $r\ra\infty$ for some nontrivial $\tau\in\mathcal{H}^2(\Sigma)$. Then $\zeta\notin L^2$, which implies that the approximate solution we will construct later has torsion which is too large. This could be overcome if $\zeta$ in the theorem could be made coclosed as well as closed, but this does not appear to be the case due to Remark~\ref{rmk:Laplacian:3Forms:ALC}.
\end{remark*}

\begin{theorem}\label{thm:CS:ALC:Obstructions:3:form}
Let $(M_0,\varphi_0)$ be an irreducible CS ALC \gtmfd~with singularities at $p_1,\dots, p_n$ modelled on the \gtwo--cones $\tu{C}_1,\dots, \tu{C}_n$ respectively.

For each $i=1,\dots, n$ let $\xi_i$ be a closed and coclosed $3$-form on $N_i$. Suppose that
\[
\sum_{i=1}^n{[\xi_i]}\in \bigoplus_{i=1}^n{H^3(N_i)}
\]
lies in the image of the restriction map $H^3(M'_0)\ra \bigoplus_{i=1}^n{H^3(N_i)}$. Then there exist a $3$-form $\xi\in \Omega^3_{27}(M'_0,\varphi_0)$ and $\mu>0$ such that $d\xi=0=d^\ast\xi$ and for all $i=1,\dots, n$
\begin{equation}\label{eq:CS:ALC:Obstructions:3form:i}
|\nabla^j_{\tu{C}_i}\left( \xi -\xi_i\right)|_{g_{\tu{C}_i}} = O(r_i^{-3+\mu-j}),
\end{equation}
as $r_i\ra 0$ for all $j\geq 0$. Consider the image of $[\xi]$ in $H^3(N)$ and its projection onto $H^2_-(\Sigma)$.
\begin{enumerate}
\item If the image of $[\xi]$ in $H^2_-(\Sigma)$ vanishes then there exists a harmonic $2$-form $\tau$ on $\Sigma$ such that
\begin{equation}\label{eq:CS:ALC:Obstructions:3form:infty:(i)}
|\nabla^j_{\BC}\left( \xi -\eta\wedge\tau\right)|_{g_{\BC}} = O(r^{-3-\mu-j})
\end{equation}
for all $j\geq 0$.
\item If the image of $[\xi]$ in $H^2_-(\Sigma)$ is represented by a nonzero harmonic $2$-form $\tau$ on $\Sigma$ then
\begin{equation}\label{eq:CS:ALC:Obstructions:3form:infty:(ii)}
|\nabla^j_{\BC}\left( \xi -\theta\wedge\tau\right)|_{g_{\BC}} = O(r^{-2-\mu-j})
\end{equation}
for all $j\geq 0$.
\end{enumerate}
\proof
As in the proof of Theorem \ref{thm:CS:ALC:Obstructions:4:form}, we can find a smooth closed $3$-form $\xi'$ on $M_0$ which satisfies the given boundary conditions \eqref{eq:CS:ALC:Obstructions:3form:i} and either \eqref{eq:CS:ALC:Obstructions:3form:infty:(i)} or \eqref{eq:CS:ALC:Obstructions:3form:infty:(ii)}. We want to perturb $\xi'$ so that it is also coclosed.

Consider then the equation $\triangle \sigma = -d^\ast\xi'$. Note that $\xi'\in L^2_{l,\bm{\nu}}$ where $\nu_i = -3-\delta$ and $\nu_\infty=-2+\delta$ for any $\delta>0$. The latter can be improved to $\nu_\infty=-3+\delta$ in case (i). However, in fact $d^\ast\xi' \in L^2_{l-1,\bm{\nu}'}$ where $\nu'_i = -3+\delta$ and $\nu'_\infty=-3-\delta$ in case (i) and $\nu'_\infty=-2-\delta$ in case (ii).

We apply Proposition \ref{prop:Laplacian:2Forms:ALC}: there is a solution $\sigma$ to $\triangle\sigma=-d^\ast\xi'$ such that $dd^\ast\sigma=0$ (in fact $d^\ast\sigma=0$ in view of Remark \ref{rmk:Harmonic:1:forms:CS:ALC}) and $d\sigma \in L^2_{l,\bm{\nu}'}$. In case (ii) the improved decay of $d\sigma$ follows from the fact (which is a consequence of Propositions \ref{prop:Harmonic:1Forms:Cone} and \ref{prop:Harmonic:2forms:cone}) that $-1$ is not an indicial root for the Laplacian on $\BC (\Sigma)$ acting on $2$-forms.

We now set $\xi = \xi' + d\sigma$, which is then closed and coclosed. Letting $l\ra\infty$, by weighted elliptic regularity we can pass from weighted $L^2$--decay to pointwise decay of $\xi$ and its derivatives in the statement of the theorem.
\endproof
\end{theorem}

\subsection{Adaptation of Joyce's deformation result to the ALC setting}

In this subsection we consider the problem of deforming a closed ALC \gtstr~with small torsion into a torsion-free \gtstr~while preserving the ALC asymptotics. As part of his construction of the first examples of compact \gtmfd s, Joyce gave a general and flexible result that allows one to deform closed \gtstr s with small torsion to \gtmetric s on \emph{compact} manifolds. We state and prove a version of Joyce's deformation result \cite[\S\S 11.6--11.8]{Joyce:Book} for closed \gtstr s with small torsion adapted to the ALC setting. The estimates are not optimal, but since this does not seem essential, we prefer to use Joyce's original argument and hence avoid having to prove uniform weighted estimates. In comparison to the original work of Joyce, the crucial new feature of the noncompact setting is that it is necessary not only to assume that the torsion is small but also that it decays fast enough.  

\begin{theorem}\label{thm:Joyce:ALC}
Fix a Riemannian cone $\tu{C}(\Sigma)$, and $\nu<-3$ sufficiently close to $-3$ so that there are no harmonic $2$-forms and closed and coclosed $3$-forms on $\tu{C}(\Sigma)$ homogeneous of order $\lambda$ in $[\nu+1,-2)$ and $[\nu,-3)$ respectively. Fix $C>0$, $p\geq 1$, $\alpha,\beta_2,\beta_p\in\R$ and $\kappa>0$. Then there exist $\epsilon_0,K>0$ such that the following holds whenever $0<\epsilon<\epsilon_0$.

Let $(M^7,g)$ be a (cyclic or dihedral) ALC manifold asymptotic to $\BC(\Sigma)$ for some circle bundle over a cone $\tu{C}(\Sigma)$. Suppose that the Riemannian metric $g$ is induced by a smooth closed ALC \gtstr~$\varphi$ and that there exists a smooth closed $3$-form $\chi$ such that $d\,{\ast_\varphi}\chi=d\,{\ast_\varphi} \varphi$. Assume that:
\begin{enumerate}
\item $\chi\in L^2_{l,\nu}$ for every $l\geq 0$.
\item $\| \chi\|_{L^2}\leq C\epsilon^{\alpha+\kappa}$.
\item $\| \chi\|_{C^0}\leq C\epsilon^{\kappa}$.
\item $\| d\chi\|_{L^2}\leq C\epsilon^{\beta_2+\kappa}$ and  $\| d \chi\|_{L^p}\leq C\epsilon^{\beta_p+\kappa}$.
\item For every $3$-form $\chi\in L^2_{1,\nu}$ the equation $\triangle \sigma = d^\ast\chi$ has a coclosed solution $\sigma$, unique up to the addition of a closed and coclosed $2$-form, with $\sigma\in L^2_{2,\nu+1+\delta}$ for every $\delta>0$ and $d\sigma\in L^2_{1,\nu}$.
\item $\|\nabla\rho\|_{L^2} \leq C \left( \| d\rho\|_{L^2} + \|d^\ast\rho\|_{L^2} + \epsilon^{-\alpha+\beta_2}\|\rho\|_{L^2}\right)$ for all $3$-forms $\rho\in L^2_1$.
\item $\|\nabla\rho\|_{L^p} \leq C \left( \| d\rho\|_{L^p} + \|d^\ast\rho\|_{L^p} + \epsilon^{-\alpha+\beta_p}\|\rho\|_{L^2}\right)$ for all $3$-forms $\rho\in L^2$ with $\nabla\rho\in L^p$.
\item $\|\rho\|_{C^0} \leq C \left( \epsilon^{-\beta_p}\| \nabla\rho\|_{L^p} + \epsilon^{-\alpha}\|\rho\|_{L^2}\right)$ for all $3$-forms $\rho\in L^2$ with $\nabla\rho\in L^p$.
\end{enumerate}
Then there exists a smooth $3$-form $\rho\in L^2_1\cap C^0$ which is $L^2$--orthogonal to closed and coclosed $L^2$--integrable $3$-forms and satisfies
\[
\| \rho\|_{C^0} + \epsilon^{-\alpha}\|\rho\|_{L^2} + \epsilon^{-\beta_2}\|\nabla\rho\|_{L^2} + \epsilon^{-\beta_p}\|\nabla\rho\|_{L^p}\leq K\epsilon^\kappa,
\]
such that $\varphi+\rho$ is a torsion-free ALC \gtstr~on $M$.
\proof
First of all, in \cite[Theorem G2]{Joyce:Book} (but also note the comment at the bottom of p 296) Joyce makes the special choices $p=14$, $\alpha=\tfrac{7}{2}$, $\beta_p=-\tfrac{1}{2}=\beta_2-3$ 
and $\kappa = \frac{1}{2}$.
Note also that in the compact setting $L^p\subset L^2$ for every $p\geq 2$ under a uniform bound on the total volume. It is immediate to check that, given assumptions (vi)--(viii), his argument in the compact setting works unchanged for any choice of $p,\alpha,\beta_2,\beta_p$ as in the statement of the Theorem.

We have to explain how Joyce's argument has to be modified due to the noncompactness of $M$. Joyce's strategy to prove the Theorem in the compact setting is to first solve the nonlinear elliptic system for a $2$-form $\sigma$
\[
\triangle \sigma = d^\ast\chi + d^\ast(f\chi) + d^\ast Q_{\varphi} (d\sigma),
\]
where $f\varphi=\tfrac{7}{3}\pi_1 (d\sigma)$. This is done using an iteration scheme that at each step requires us to solve the equation
\[
\triangle\sigma_j = d^\ast\chi_j,
\]
where $\chi_j = \psi + f_{j-1}\chi + Q_\varphi(d\sigma_{j-1})$ with $f_{j-1}\varphi=\tfrac{7}{3}\pi_1 (d\sigma_{j-1})$. If we knew that $\chi_{j}\in L^2_{1,\nu}$ then assumption (v) would guarantee the existence of a coclosed solution $\sigma_j$ with $d\sigma_j\in L^2_{1,\nu}$. Then Joyce's arguments would give us uniform estimates for $d\sigma_j$ in $L^p_1\cap L^2_1\cap C^0$ exploiting assumptions (vi)--(viii). The assumption $\chi_{j}\in L^2_{1,\nu}$ can be brought in to the iteration scheme by a simple induction argument exploiting (i) and (iii) and the fact that $d\sigma_{j-1}$ is bounded by the induction hypotheses. Indeed, by \cite[Proposition 10.3.5]{Joyce:Book} there exists a universal constant $c>0$ such that  
\[
|Q_{\varphi}(d\sigma_{j-1})| \leq c\, |d\sigma_{j-1}|^2, \qquad |\nabla Q_{\varphi}(d\sigma_{j-1})| \leq c\left( |d \sigma_{j-1}|^2|d^\ast\psi| + |\nabla d\sigma_{j-1}|\, | d\sigma_{j-1}| \right).
\]
Thus if $d\sigma_{j-1}\in L^2_{1,\nu}\cap C^0$ then $Q_{\varphi}(d\sigma_{j-1})\in L^2_{1,\nu}$.

Thus the equation $\triangle\sigma_j = d^\ast\chi_j$ can always be solved with $d\sigma_j\in L^2_{\nu}\subset L^2$. Then the proof of \cite[Theorem G2]{Joyce:Book} goes through to guarantee that for $\epsilon$ sufficiently small $\{d\sigma_j\}_{j\geq 0}$ is a Cauchy sequence in $L^2\cap C^0$ and $\{\nabla d\sigma_j\}_{j\geq 0}$ is a Cauchy sequence in $L^2\cap L^p$, with uniform estimates.

In the compact setting Joyce proves that $\sigma_j$ converges as well. This last fact is not strictly necessary and we prefer to avoid  adapting its proof to the noncompact setting. We instead define $\rho_j = d\sigma_j$, use Joyce's arguments to show that $\rho_j$ has a limit $\rho$ in $C^0\cap L^2_1$ with $\nabla\rho\in L^p$ and observe that $\rho$ is closed and $L^2$--orthogonal to closed and coclosed $3$-forms in $L^2$. Moreover, $\rho$ satisfies the estimates in the statement of the Theorem and the elliptic system
\[
d\rho =0, \qquad d^\ast\rho = d^\ast\chi + d^\ast(f\chi) + d^\ast Q_{\varphi} (\rho),
\]
where $f\varphi=\tfrac{7}{3}\pi_1 (\rho)$. Hence by (weighted) elliptic regularity $\rho$ is smooth and decays at least as~$r^{-3}$ (in fact, one would expect $\rho$ to decay as $r^\nu$, but we haven't provided a proof of the fact that~$\rho_j$ converges to $\rho$ in $L^2_\nu$ as this would require uniform weighted estimates for the solutions to $\triangle\sigma = d^\ast\chi$).

Now, following \cite[Theorem 10.3.7]{Joyce:Book}, for $\epsilon$ sufficiently small $\varphi+\rho$ is a closed ALC \gtstr{} on~$M$ such that
\[
d\ast_{\varphi+\rho}(\varphi+\rho) = df\wedge(\ast\varphi-\ast\chi) + d\alpha \wedge \varphi,
\] 
where $\alpha\wedge\varphi = 2\,{\ast}\pi_7 (\rho)$. To conclude that $df=0=d\alpha$, \ie that $\varphi+\rho$ is torsion-free, we use \cite[Proposition 10.3.4]{Joyce:Book}: although the result is stated under the assumption that $M$ is a compact manifold, one can easily check that the proof goes through unchanged in the noncompact setting provided $\rho$ and $\nabla\rho$ are $L^2$--integrable.
\endproof 
\end{theorem}

\begin{remark*}
The fact that $\rho\in L^2$ is closed and $L^2$--orthogonal to closed and coclosed $3$-forms in $L^2$ guarantees that $\rho$ is exact and therefore the closed $3$-form $\varphi+\rho$ lies in the same cohomology class as $\varphi$. Indeed, since $\rho\in L^2$ we have that $[\rho]\in H^3(M)$ lies in the image of $H^3_c(M)$,
and hence in the image of $\Psi : \mathcal{H}^3_{\lambda}(M) \to H^3(M)$ by %
Remark \ref{rmk:inj_across_L2}. More precisely, the proof of Theorem \ref{thm:ALC_hodge_fastdecay}(ii) (which underpins Theorem \ref{thm:Topological:Hodge:G2}(i)) shows that we can write the unique representative of~$[\rho]$ in $\mathcal{H}^3_{\lambda}(M)$ as $\rho + d\gamma$, with $\gamma$ extending smoothly to the compactification $X$. Now $d\gamma$ too is $L^2$-orthogonal to~$\mathcal{H}^3_{\lambda}(M)$, so we must have $\rho = -d\gamma$.
\end{remark*}

\subsection{Existence results}\label{sec:CS:ALC:Conclusion}

We have now all the ingredients in place for the proof of Theorem \ref{thm:Desing:CS:ALC}.

Let $(M_0,\varphi_0)$ be a CS ALC \gtwo--manifold with singularities $p_1,\dots, p_n$ modelled on $\gtwo$--cones $\tu{C}(N_i)$, $i=1,\dots,n$, and asymptotic to $\BC(\Sigma)$ at infinity. For each $i=1,\dots, n$ we fix an AC $\gtwo$--manifold $(M_i,\varphi_i)$ asymptotic to the cone $\tu{C}(N_i)$ with rate less than or equal to $-3$. As in Remark \ref{rmk:AC:G2}, we write
\[
\varphi_i = \xi_i + d\eta_i, \qquad \ast_{\varphi_i}\varphi_i = \zeta_i + r_i dr_i\wedge \ast_{N_i}\xi_i + d\theta_i
\]
for some $(\xi_i,\zeta_i)\in \mathcal{H}^3(N_i)\oplus \mathcal{H}^4(N_i)$ and $(\eta_i,\theta_i)=O(r_i^{-2-\delta})$ for some $\delta>0$.

We assume that $([\xi_1],\dots, [\xi_n])\in \bigoplus_i{H^3(N_i)}$ and $(0,[\zeta_1],\dots, [\zeta_n])\in H^4(N)\oplus \bigoplus_i{H^4(N_i)}$ lie in the image of, respectively, the restriction maps
\[
H^3(M_0')\ra H^3(N)\oplus \bigoplus\nolimits_i{H^3(N_i)}\stackrel{pr}{\ra} \bigoplus\nolimits_i{H^3(N_i)}, \qquad H^4(M_0')\ra H^4(N)\oplus \bigoplus\nolimits_i{H^4(N_i)}.
\]
Let $\zeta$ and $\xi$ be the forms provided by Theorems \ref{thm:CS:ALC:Obstructions:4:form} and, respectively, \ref{thm:CS:ALC:Obstructions:3:form}. We can then consider 1-parameter families of closed forms
\[
\varphi_{0,\epsilon}' = \varphi_0 + \epsilon^3\xi, \qquad \psi_{0,\epsilon}' = \ast_{\varphi_0}\varphi_0 + \epsilon^4\zeta - \epsilon^3\ast_{\varphi_0}\xi
\]
for $\epsilon>0$. Note that
\begin{equation}\label{eq:Decay:Torsion:Approx:Desing}
\ast_{\varphi_{0,\epsilon}'} \varphi_{0,\epsilon}' - \psi_{0,\epsilon}' = \ast_{\varphi_0}Q_{\varphi_0}(\epsilon^3 \xi) - \epsilon^4\zeta = O(\epsilon^6 r^{-4} + \epsilon^4 r^{-3-\delta})
\end{equation}
as $r\ra\infty$ for some $\delta>0$, with analogous estimates for all derivatives. On the other hand, as in \cite[Equations (82)--(85)]{Karigiannis} for $r_i\ra 0$ we have
\begin{equation}\label{eq:Approx:Desing}
\begin{aligned}
\varphi_{0,\epsilon}'= \varphi_{\tu{C}_i} + \epsilon^3\xi_i + d\eta_{0,i}, \qquad &\eta_{0,i} = O(r_i^{\mu_i+1}+\epsilon^3 r_i^{-2+\delta}),\\
\ast_{\varphi_{0,\epsilon}'}\varphi_{0,\epsilon}'= \ast_{\varphi_{\tu{C}_i}}\varphi_{\tu{C}_i} + \epsilon^4\zeta_i + \epsilon^3 r_i dr_i\wedge \ast_{N_i} \xi_i + d\theta_{0,i}, \qquad &\theta_{0,i} = O(r_i^{\mu_i+1}+\epsilon^3 r_i^{-2+\delta} + \epsilon^4 r_i^{-3+\delta}),
\end{aligned}
\end{equation}
with analogous estimates on the derivatives. Here $\mu_i>0$ is the rate of the CS singularity $p_i$ as in Definition \ref{def:CS:ALC}.

For $\epsilon$ sufficiently small we now proceed as in \cite[Definition 3.15]{Karigiannis} to construct a smooth $7$-manifold $M$ with a family of closed $\gtwo$--structures $\varphi'_\epsilon$ with suitably small torsion. This is done by interpolating between $\varphi_{0,\epsilon}'$ and a rescaling $\epsilon^3 \varphi_i$ in an annular region where $r_i\approx \epsilon^\gamma$ for $\gamma\in (0,1)$ given in \cite[Corollary 3.25]{Karigiannis}. More precisely, denote by $r_{M_i}$ a function on $M_i$ that coincides with the radius function on the cone $C_i$ along the AC end. Then the manifold $M$ is obtained by removing the set $r_i \leq 2\epsilon^\gamma$ in $M_0$ and replacing it with the region of $M_i$ where $r_{M_i}\leq 2\epsilon^{\gamma-1}$ under the rescaling $r_i=\epsilon r_{M_i}$. Under this diffeomorphism we have
\[
\epsilon^3\varphi_i = \varphi_{\tu{C}_i} + \epsilon^3 \xi_i + \epsilon^3 d\eta_i, \qquad \epsilon^3\eta_i = O(\epsilon^3 r_i^{-2-\delta})
\]
so that we can construct a closed $3$-form on $M$ that coincides with $\epsilon^3 \varphi_i$ for $r_{M_i}\leq \epsilon^{1-\gamma}$ and $\varphi_{0,\epsilon}'$ for $r_i \geq 2r_i$. Similarly, one interpolates between $\psi_{0,\epsilon}'$ and $\epsilon^4 \ast_{\varphi_{i}}\varphi_i$ in the same annular region to obtain a closed $4$-form $\psi_\epsilon'$ on $M$.

\begin{proof}[Proof of Theorem \ref{thm:Desing:CS:ALC}]
We will now explain why all the conditions of Theorem \ref{thm:Joyce:ALC} are satisfied with $\varphi=\varphi_{\epsilon}'$ and 3-form $\chi_\epsilon$ defined by
\[
\ast_{\varphi_\epsilon'} \chi_\epsilon = \ast_{\varphi_\epsilon'}\varphi_\epsilon'-\psi_\epsilon'. 
\]

Firstly, the condition (v) of Theorem \ref{thm:Joyce:ALC} is just Proposition \ref{prop:Laplacian:2Forms:ALC} and Remark \ref{rmk:Laplacian:2Forms:ALC} applied to the ALC manifold $(M,g_{\varphi'_\epsilon})$. Moreover, we have solutions to $\triangle\sigma = d^\ast\chi$ with $d^\ast\sigma=0$ instead of the weaker condition $dd^\ast\sigma$ since $d^\ast\sigma$ lies in the space of $L^2\mathcal{H}^1(M)$ of $L^2$--integrable closed and coclosed $1$-forms, isomorphic to $H^1_c(M)$ by Theorem \ref{thm:ALC_hodge_fastdecay}; the latter cohomology group vanishes since the Ricci-flat AC manifolds $M_i$ have necessarily vanishing first Betti number (since their fundamental groups are finite, see the analogous \cite[Proposition 5.10]{FHN:ALC:G2:from:AC:CY3} in the Calabi--Yau setting) and similarly for the Ricci-flat CS ALC manifold $M_0$ we must have $H^1_c(M_0)=0$ as explained in Remark \ref{rmk:Harmonic:1:forms:CS:ALC}.

Secondly, the estimates in (vi)--(viii) of Theorem \ref{thm:Joyce:ALC} hold for any $p>7$ with $\alpha=\tfrac{7}{2}$ and $\beta_p = \frac{7}{p}-1$ since the injectivity radius of $g_{\varphi_\epsilon'}$ is bounded below by $C\epsilon$ and the Riemannian curvature of $g_{\varphi_\epsilon'}$ is bounded above by $C\epsilon^{-2}$ as in \cite[Proposition 3.27]{Karigiannis}: then every point of $M$ has a neighbourhood of radius $\approx \epsilon$ with bounded geometry, so the estimates in the theorem follow by rescaling standard local elliptic estimates.

Finally, $\chi_\epsilon$ satisfies all the estimates in (i)--(iv) combining Karigiannis' estimates in \cite[Corollary 3.25]{Karigiannis} and the asymptotic expansion at infinity given in \eqref{eq:Decay:Torsion:Approx:Desing}.

The proof of Theorem \ref{thm:Desing:CS:ALC} is now just an application of Theorem \ref{thm:Joyce:ALC} combined with a standard bootstrap argument based on elliptic regularity to pass from the Sobolev estimates in Theorem \ref{thm:Joyce:ALC} to the $C^\infty_{loc}$--convergence in Theorem \ref{thm:Desing:CS:ALC}. 
\end{proof}

We now explain how to apply Theorem \ref{thm:Desing:CS:ALC} to Examples \ref{eg:CS:ALC}, \ref{eg:AC:G2:Bryant:Salamon} and \ref{eg:FHN:AC:G2}.

Up to taking finite quotients, the only currently known CS ALC $\gtwo$--holonomy space is the one described in Example \ref{eg:CS:ALC}, which we denote by $\widetilde{M}_0$. Recall that $\widetilde{M}_0$ has a unique conical singularity $p_1$ modelled on $\tu{C}(N_1)$ the cone over the homogeneous nearly K\"ahler structure on $N_1\simeq S^3\times S^3$ and that $\widetilde{M}_0'=\widetilde{M}_0\setminus \{p\}\simeq (0,\infty)
\times S^3\times S^3$. The ALC end of $\widetilde{M}_0$ is modelled on $\BC (\Sigma)$, where $\Sigma\simeq S^2\times S^3$ is endowed with its homogeneous Sasaki--Einstein structure and $N\ra \Sigma$ is the principal circle bundle with first Chern class the generator of $H^2(\Sigma;\Z)$.

We obtain infinitely many examples by considering finite quotients $M_0=\widetilde{M}_0/\Gamma$ by a finite subgroup $\Gamma$ of the group $(\sunitary{2}^2\times\tu{N})/\Z_2$ of diffeomorphisms of $\widetilde{M}_0$ that preserves its CS ALC $\gtwo$--structure. Recall that $\tu{N} \subset \sunitary{2}$ is the normaliser of $\unitary{1}$, $\tu{N} / \unitary{1} \cong \Z_2$, and any element of~$\tu{N} \setminus \unitary{1}$ has order 4.

\begin{remark*}
    Let us comment on the asymptotic geometry of $M_0$. Recall that the circle fibration on the link $N = S^3 \times S^3$ of the ALC end of $\wt M_0$ corresponds to the diagonal right action of $\unitary{1} \subset \tu{N} \subset \sunitary{2}$. Thus if $\Gamma$ is contained in this $\unitary{1}$ then the asymptotic link of $M_0$ is a principal circle bundle over the same base $\Sigma \simeq S^2 \times S^3$, but ``with shorter fibres''. (In particular, this is the case for $\Gamma = \{\pm1\} \subset \unitary{1}$.) In general, the base is the quotient of $\Sigma$ by the induced action of the image of $\Gamma$ in $\sunitary{2}/\Z_2 \times \tu{N}/\unitary{1}$.
    
    Thinking of $\Sigma$ as the link of the conifold, the action of $\tu{N}/\unitary{1}$ is by complex conjugation (which reverses orientation).    
The action of $\tu{N}$ descends to the quotient $(S^3 \times S^3)/{\pm 1}$, factoring through $\tu{N}/{\pm 1} \cong (\unitary{1}/{\pm1}) \rtimes \Z_2$, with the $\Z_2$ semi-direct factor acting as a standard involution.
\end{remark*}

Note that, no matter what finite subgroup $\Gamma$ we choose, the topological constraints in Theorem \ref{thm:Desing:CS:ALC} are automatically satisfied since $H^4(N_1)\oplus H^4 (N)=\{ 0\}$ and $H^3(M_0')\simeq H^3(N_1/\Gamma)$. Hence, we just need to explain choices of $\Gamma$ and AC $\gtwo$--manifold $M_1$ to feed into the construction of Theorem \ref{thm:Desing:CS:ALC}. For this it is important to note that the tangent cone $\tu{C}(N_1)$ at the conical singularity $p_1$ of~$\widetilde{M}_0$ has a larger automorphism group, namely $(\sunitary{2}^3/\Delta\Z_2)\rtimes S_3$, where $S_3$ is the symmetric group that exchanges the three factors.  

\begin{example}\label{eg:B7:D7}
We take $M_0=\widetilde{M}_0$, \ie take $\Gamma=\{ 1\}$, and $M_1$ the Bryant--Salamon AC $\gtwo$--metric on $S^3\times\R^4$ of Example \ref{eg:AC:G2:Bryant:Salamon}.
There are three topologically different ways to identify the asymptotic cone of $M_1$ with the cone at the singularity of $M_0$, exchanged by the $S_3$ subgroup of automorphisms of $\tu{C}(N_1)$. We then obtain three different families of ALC $\gtwo$--metrics by desingularising $M_0$ using one of the three versions of $M_1$. By Theorem \ref{thm:Uniqueness:Coho1:ALC}, each of these must be among the $\sunitary{2}^2\times\unitary{1}$--invariant $\gtwo$--metrics obtained in \cite[Theorem B]{FHN:Coho1:ALC} and conjectured earlier in the physics literature (the $\mathbb{B}_7$ family and two versions of the $\mathbb{D}_7$ family, arising from the 
two possible small resolutions of the conifold).  
\end{example}

\begin{example}\label{eg:C7}
Fix a pair $n,m$ of coprime positive integers. As explained in Example \ref{eg:FHN:AC:G2}, in \cite{FHN:Coho1:ALC} we constructed a (unique up to scale) AC $\gtwo$--holonomy metric $M_{m,n}$ asymptotic with rate~$-3$ to a $\gtwo$--cone with
cross-section $N_1=S^3\times S^3/\Gamma_{m+n}$ for a certain subgroup $\Gamma_{m+n}\simeq \Z_{2(n+m)}\subset (\sunitary{2}^2\times\unitary{1})/\Z_2$. The manifold $M_{m,n}$ is the total space of a circle bundle (determined by $m$ and $n$) over the canonical line bundle of $\C\PP^1\times\C\PP^1$. Taking $M_0=\widetilde{M}_0/\Gamma_{m+n}$ and $M_1=M_{m,n}$ in Theorem \ref{thm:Desing:CS:ALC} yields infinitely many families of ALC $\gtwo$--holonomy metrics which must coincide with the families of $\sunitary{2}^2\times\unitary{1}$--invariant $\gtwo$--metrics obtained in \cite[Theorem D]{FHN:Coho1:ALC}.
\end{example}

These examples recover known families of ALC~\gtwo-metrics of cyclic type.  The following example provides the first known instances of ALC $\gtwo$--manifolds of dihedral type, thus concluding the proof of Theorem \ref{mthm:Dihedral}.

\begin{example}\label{eg:A7}
Consider now the $m = n = 1$ case of Example \ref{eg:FHN:AC:G2}. The group $\Gamma_{cyc}:=\Gamma_{1,1}\simeq \Z_4$ is embedded in $(\sunitary{2}^2\times\unitary{1})/\Z_2\subset (\sunitary{2}^2\times \tu{N})/\Z_2$ via $\zeta\mapsto (1,\zeta^2,\zeta)$, \ie the image is generated by $(1, -1, i)$. Consider now a different embedding $\Gamma_{dih}$ of $\Z_4 \subset(\sunitary{2}^2\times \tu{N})/\Z_2$, generated by $(1,-1,\tau)$ for any $\tau \in \tu{N} \setminus \unitary{1}$. Then $M_0=\widetilde{M}_0/\Gamma_{dih}$ is again a CS ALC $\gtwo$--holonomy space 
but now it has an ALC end of dihedral type (the canonical double cover being $(S^3 \times S^3)/{\pm 1}$).

Now, the two subgroups $\Gamma_{cyc}$ and $\Gamma_{dih}$ are in fact conjugate in the larger symmetry group $\sunitary{2}^3/\triangle\Z_2$ of the cone $\tu{C}(N_1)$. Hence the isolated conical singularity of $M_0$ is modelled on the cone over $N_1/\Gamma_{cyc}$ and the AC manifold $M_{1,1}$ can be used in Theorem \ref{thm:Desing:CS:ALC}, yielding a 1-parameter family of ALC $\gtwo$--holonomy metrics of dihedral type whose existence was previously unknown. The underlying smooth manifold is diffeomorphic to $M_{1,1} \cong S^3\times\mathcal{O}_{S^2}(-4)$, where $\mathcal{O}_{S^2}(-4)$ is the total space of the complex line bundle over $S^2$ with Euler class $-4$. \end{example}

\begin{remark*}
    The $\gtwo$--metrics in the previous example produced via an application of Theorem \ref{thm:Desing:CS:ALC} can in fact be shown to be $\sunitary{2}^2$--invariant (by considering the symmetries of the building blocks and carrying out the deformation of Theorem \ref{thm:Desing:CS:ALC} equivariantly). Hence these metrics are still of cohomogeneity-one, \ie in principle described by solutions of an ODE system: however in \cite{FHN:Coho1:ALC} we made essential use of the larger symmetry group $\sunitary{2}^2\times\unitary{1}$ not enjoyed by these new metrics to reduce to a simpler ODE system. 
It would be interesting to use ODE methods to give a construction 
of the whole $1$-parameter family of $\mathbb{A}_7$ metrics. 
\end{remark*}

\bibliographystyle{amsinitial}
\bibliography{Def_ALC}

\end{document}